\documentclass[10pt]{article}
\usepackage{amsthm,amsfonts,amssymb,euscript,amsmath,comment,slashed}

\usepackage[affil-it]{authblk}
\usepackage{bbm}
\usepackage{graphics}
\usepackage{enumerate}
\usepackage[margin=1in,dvips]{geometry}

\numberwithin{equation}{section}

\usepackage{color}

\newcommand{\blue}{\textcolor{blue}}

\newcommand{\pfstep}[1]{\vspace{.5em} {\it \noindent #1.}}

\newtheorem{theorem}{Theorem}[section]
\newtheorem{proposition}[theorem]{Proposition}
\newtheorem{lemma}[theorem]{Lemma}
\newtheorem{remark}[theorem]{Remark}
\newtheorem{definition}[theorem]{Definition}
\newtheorem{corollary}[theorem]{Corollary}

\def\ls {\lesssim}

\def\th {\theta}
\def\rd {\partial}

\def\ep {\epsilon}

\def\f {\frac}
\def\ab {\bar{a}}

\def\cb {\bar{c}}
\def\i {\infty}
\def\de {\delta}
\newcommand{\ud}{\mathrm{d}}
\def\alp {\alpha}
\def\bt {\beta}
\def\vb {\langle v\rangle}
\def\wb {\langle x-tv \rangle}
\def\xb {\langle x \rangle}
\def\Mm {M_{\mathrm{max}}}
\def\Mi {M_{\mathrm{int}}}

\newcommand{\ba}{\begin{equation}}
\newcommand{\ea}{\end{equation}}

\newcommand{\bea}{\begin{eqnarray}}
\newcommand{\eea}{\end{eqnarray}}

\def\beaa{\begin{eqnarray*}}
\def\eeaa{\end{eqnarray*}}

\title{Stability of vacuum for the \\
Landau equation with moderately soft potentials}

\author{Jonathan Luk\thanks{jluk@stanford.edu}}
\affil{\small  Department of Mathematics, Stanford University, 450~Serra~Mall~Building~380,~Stanford~CA~94305-2125,~United~States~of~America \ }

\begin{document}

\maketitle

\begin{abstract}
Consider the spatially inhomogeneous Landau equation with moderately soft potentials (i.e.~with $\gamma \in (-2,0)$) on the whole space $\mathbb R^3$. We prove that if the initial data $f_{\mathrm{in}}$ are close to the vacuum solution $f_{\mathrm{vac}} \equiv 0$ in an appropriate norm, then the solution $f$ remains regular globally in time. This is the first stability of vacuum result for a binary collisional model featuring a long-range interaction.

Moreover, we prove that the solutions in the near-vacuum regime approach solutions to the linear transport equation as $t\to +\infty$. Furthermore, in general, solutions do not approach a traveling global Maxwellian as $t \to +\infty$.

Our proof relies on robust decay estimates captured using weighted energy estimates and the maximum principle for weighted quantities. Importantly, we also make use of a \emph{null structure} in the nonlinearity of the Landau equation which suppresses the most slowly-decaying interactions.
\end{abstract}

\section{Introduction}

Consider the Landau equation for the particle density $f(t,x,v)\geq 0$ in the whole space $\mathbb R^3$. Here, $t\in \mathbb R_{\geq 0}$, $x\in \mathbb R^3$ and $v\in \mathbb R^3$. The Landau equation reads
\begin{equation}\label{Landau}
\rd_t f +v_i\rd_{x_i}f=Q(f,f),
\end{equation}
where $Q(f,f)$ is the \emph{collision kernel} given by\footnote{Each of these terms depends also on $(t,x)$. For brevity, we have suppressed these dependence in \eqref{kernel}.}
\begin{equation}\label{kernel}
Q(f,f)(v):= \rd_{v_i}\int_{\mathbb R^3} a_{ij}(v-v_*)\big(f(v_*) (\rd_{v_j}f)(v)-f(v) (\rd_{v_j}f)(v_*)\big)\,\ud v_*,
\end{equation} 
and $a_{ij}$ is the non-negative symmetric matrix defined by
\begin{equation}\label{a.def}
a_{ij}(z):=\big(\delta_{ij}-\f{z_i z_j}{|z|^2}\big)|z|^{\gamma+2}.
\end{equation}
In all the expressions above (and in the remainder of the paper), we have used the convention that repeated lower case Latin indices are summed over $i,j=1,2,3$.

In this paper, we will be concerned with the case $\gamma \in (-2,0)$ in \eqref{a.def}. The case $\gamma \in (-2,0)$ is usually known as the case of \emph{moderately soft potentials}. Note that the $\gamma =-3$ case is the original case Landau wrote down, and the case that we consider can be thought of as a limiting case of the Boltzmann equation (without angular cutoff).

It will be convenient to also define
\begin{equation}\label{c.def}
c:=\rd^2_{z_i z_j} a_{ij}(z)=-2(\gamma+3)|z|^{\gamma}.
\end{equation}
and
\begin{equation}\label{bar.def}
\ab_{ij}:=a_{ij}\ast f,\quad \cb:=c\ast f,
\end{equation}
where $\ast$ denotes convolutions in $v$. The Landau equation \eqref{Landau} 
is then equivalent to
\begin{equation}\label{Landau.3}
\rd_t f+v_i\rd_{x_i}f=\ab_{ij} \rd^2_{v_i v_j} f-\cb f.
\end{equation}

We solve the Cauchy problem for the Landau equation \eqref{Landau}, i.e.~we study the solution arising from prescribed regular initial data:
\begin{equation}\label{data.def}
f(0,x,v)=f_{\mathrm{in}}(x,v)\geq 0.
\end{equation}

Our main result is that if $f_{\mathrm{in}}$ is sufficiently small and is sufficiently localized in both $x$ and $v$ (i.e.~if $f_{\mathrm{in}}$ is in the ``near-vacuum'' regime), then it gives rise to a unique global-in-time solution, which is moreover globally smooth. More precisely\footnote{For the precise definition of the multi-index notations, see Section~\ref{sec:notation}.},
\begin{theorem}\label{thm:main}
Let $\gamma \in (-2,0)$, $d_0>0$ and $\Mm = 
\begin{cases}
2+2\lceil \f{2}{2+\gamma} + 4\rceil & \mbox{ if $\gamma \in (-2,-1]$} \\
2+2\lceil \f{1}{|\gamma|}+ 4\rceil & \mbox{ if $\gamma \in (-1,0)$}
\end{cases}
$.

There exists an $\underline{\ep}_0 = \underline{\ep}_0(\gamma,d_0)>0$ such that if
\begin{equation*}
\begin{split}
\sum_{|\alp|+|\bt| \leq \Mm} &\: \|(1+|x|^2)^{\f{\Mm+5}{2}}\rd_x^{\alp}\rd_v^{\bt} (e^{2d_0(1+|v|^2)} f_{\mathrm{in}}) \|_{L^2_xL^2_v} \\
&\: + \sum_{|\alp|+|\bt|\leq \Mm-5} \|(1+|x|^2)^{\f{\Mm+5}{2}}(1+|v|^2)^{\f 12}\rd_x^{\alp}\rd_v^{\bt} (e^{2d_0(1+|v|^2)} f_{\mathrm{in}}) \|_{L^\i_xL^\i_v} \leq \ep 
\end{split}
\end{equation*}
for some $\ep \in [0,\underline{\ep}_0]$,
then there exists a unique global solution $f:[0,+\infty)\times \mathbb R^3\times \mathbb R^3$ to the Landau equation \eqref{Landau} in the energy space $C^0([0,T];H^{4,0}_{\mathrm{ul}})\cap L^2([0,T];H^{4,1}_{\mathrm{ul}})$ for any $T\in (0,+\infty)$ (see~Definition~\ref{def:ul}) which achieves the prescribed initial data $f_{\mathrm{in}}$.

Moreover, the solution is $C^\infty$ in $(0,+\infty)\times \mathbb R^3\times \mathbb R^3$ and, as long as $f_{\mathrm{in}}$ is not identically $0$, it holds that $f(t,x,v)>0$ for all $(t,x,v) \in (0,+\infty)\times \mathbb R^3\times \mathbb R^3$. 
\end{theorem}

\begin{remark}[The restriction on $\gamma \in (-2,0)$]\label{rmk:gamma}
Our argument requires $\gamma \in (-2,0)$ and indeed one sees that the number of derivatives needed in Theorem~\ref{thm:main} $\to +\infty$ as $\gamma$ approaches the endpoints. For $\gamma \in [-3,-2]$, the dispersion seems too weak for our argument; see Section~\ref{sec:method.heuristics}. On the other hand, for $\gamma \in (0,1]$, we lack at this point even a local-in-time theory which incorporates near-vacuum data. Finally, the $\gamma =0$ case seems already tractable, although it requires a slightly different argument. This will be treated in a future work.
\end{remark}

Theorem~\ref{thm:main} above can be thought of as the global nonlinear stability of the vacuum solution $f_{\mathrm{vac}} \equiv 0$. Such a result is known for the Boltzmann equation \emph{with an angular cutoff assumption} in the pioneering work of Illner--Shinbrot \cite{IlSh84}; see also the discussion in Section~\ref{sec:cutoff.Boltzmann}. However, the stability of the vacuum solution is not known for a collisional kinetic model featuring a long range interaction, such as in the case of the Landau equation or the Boltzmann equation without angular cutoff. One important difference between the cutoff Boltzmann equation and the Landau or non-cutoff Boltzmann equation is that \emph{ellipticity} is present in the latter models. This ellipticity manifests itself for instance in the smoothing of solutions; see \cite{ChDeHe09, AlDeViWe00, AMUXY10, HeSnTa17}. In the context of the stability of vacuum for the Landau equation, the ellipticity presents the following difficulty in understanding the long time dynamics of solutions: on the one hand, the collision kernel contains top-order elliptic terms which cannot be treated completely perturbatively as in the case of cutoff Boltzmann; on the other hand, in a neighborhood of the vacuum solution, the coefficient of the elliptic term does not seem to be coercive enough to provide useful control of the solution.

In the proof of Theorem~\ref{thm:main}, we show that despite the presence of elliptic terms, the main mechanism governing the long time behavior of the solutions is the dispersion associated with the transport operator. In particular, except for the terms with the top order derivatives that we necessarily treat with elliptic/parabolic methods, all the other terms arising from the collision kernel, including commutator terms coming from differentiating the elliptic part, are treated perturbatively. There are two main ingredients necessary to achieve this: (1) we prove robust decay estimates showing that solutions to the Landau equation in the near-vacuum regime obey similar decay estimates as the linear transport equation; (2) we show that there is a \emph{null structure} in the nonlinearity which suppresses that most slowly decaying terms in the collision kernel. 

Moreover, after closing all the estimates in the proof of Theorem~\ref{thm:main}, we show a posteriori that as long as we consider the solution in an appropriate weaker topology, the $Q(f,f)$ term in \eqref{kernel} --- including its top order contribution --- can be considered as a perturbation term. In other words, as far as the long term dynamics of the solution in this weaker topology is concerned, it is completely dominated by the transport part, and the elliptic term presents no correction of the long time dynamics. To formulate this result, let us first define, associated to a function $f(t,x,v)$, the function $f^\sharp(x,v)$:
\begin{equation}\label{fsharp.def}
f^\sharp(t,x,v):= f(t,x+tv,v).
\end{equation}
Note that $f$ is a solution to the linear transport equation, i.e.~$\rd_t f+v_i\rd_{x_i}f = 0$, if and only if $f^\sharp(t,x,v)$ is in fact \emph{independent} of $t$. Thus the statement that $f(t,x,v)$ approaches a solution to the linear transport equation can be captured by the following theorem:
\begin{theorem}\label{thm:asymptotics}
There exists $C>0$ depending only on $\gamma$ and $d_0$ such that the following holds. Assume the conditions of Theorem~\ref{thm:main} hold and suppose $f$ is a solution given by Theorem~\ref{thm:main}. Then there exists a unique function $f_\i^\sharp:\mathbb R^3\times\mathbb R^3\to \mathbb R$ such that for every $\ell \in \mathbb N\cup\{0\}$, 
$$\sup_{t\geq 0} (1+t)^{\min\{1,2+\gamma\}}\|(1+|v|^2)^{\f \ell 2} (1+|x|^2)^{\f {\Mm+4}2} (f^\sharp(t,x,v) - f^\sharp_\infty(x,v))\|_{L^\i_x L^\i_v} \leq C\ep^{\f 32}.$$
\end{theorem}

One particular consequence of Theorem~\ref{thm:asymptotics} is that the long-time asymptotics for the macroscopic quantities are to leading order determined by $f^\sharp_\infty$. We formulate this in the following corollary\footnote{Note that while we control the macroscopic momentum $m_i(t,x)$, we have no control of the macroscopic velocity $u_i(t,x) = \f{m_i}{\rho}(t,x)$ since we do not have any lower bounds for $\rho$.}:
\begin{corollary}\label{cor:macro}
There exists $C>0$ depending only on $\gamma$ and $d_0$ such that the following holds.  Assume the conditions of Theorem~\ref{thm:main} hold. Let $f_\infty^\sharp$ be as in Theorem~\ref{thm:asymptotics}. Define
$$\rho_\infty(t,x):= \int_{\mathbb R^3} f_\infty^\sharp(x-tv,v)\,\ud v,\,\, (m_i)_\infty(t,x):= \int_{\mathbb R^3} v_i f_\infty^\sharp(x-tv,v)\,\ud v,\,\, e_\infty(t,x):= \int_{\mathbb R^3} \f{|v|^2}{2} f_\infty^\sharp(x-tv,v)\,\ud v.$$
Then the mass density, the momentum density and the energy density, respectively defined by
$$\rho(t,x):= \int_{\mathbb R^3} f(t,x,v)\,\ud v,\quad m_i(t,x):= \int_{\mathbb R^3} v_i f(t,x,v)\,\ud v,\quad e(t,x):= \f 12\int_{\mathbb R^3} |v|^2 f(t,x,v)\,\ud v,$$
satisfy
\begin{equation}\label{cor:macro.1}
\|\rho \|_{L^\i_x}(t) + \|m_i \|_{L^\i_x}(t) + \|e \|_{L^\i_x}(t) \leq C\ep (1+t)^{-3}
\end{equation}
and
\begin{equation}\label{cor:macro.2}
(1+t)^{3}(\|\rho-\rho_\i\|_{L^\i_x}(t) + \|m_i-(m_i)_\i\|_{L^\i_x}(t) + \|e - e_\i\|_{L^\i_x}(t)) \leq C\ep^{\f 32}(1+t)^{-\min\{1,2+\gamma\}}
\end{equation}
for every $t\geq 0$.
\end{corollary}

Another natural question in the context of long-time asymptotics is whether the limit given in Theorem~\ref{thm:asymptotics} is associated to a global traveling Maxwellian (see Definition~\ref{def:GM} below). Recall that the $H$-functional $H[f] = \int_{\mathbb R^3}\int_{\mathbb R^3} f\log f\,\ud v\,\ud x$ is non-increasing along the flow by the Landau equation \eqref{Landau}. Moreover, the solutions to \eqref{Landau} for which the $H$-functional is constant (and have finite mass, entropy and second moments) are exactly the traveling global Maxwellians \cite{Levermore}. We show that despite these facts, general solutions to \eqref{Landau} do \underline{not} necessarily approach traveling global Maxwellians. 

In the case of the Boltzmann equation with an angular cutoff, the existence of solutions not approaching traveling global Maxwellians as $t\to +\infty$ was first demonstrated by Toscani \cite{To87} by showing that polynomial lower bound in the spatial variable can be propagated. In fact, for the Boltzmann equation with an angular cutoff, much more than non-convergence to traveling global Maxwellian is known: a scattering theory can be developed in a neighborhood of any sufficient small traveling global Maxwellian \cite{BaGaGoLe16}. In the case of the Landau equation, in view of the smoothing effect of the equation, it seems unlikely that a scattering theory of the type in \cite{BaGaGoLe16} still holds (cf.~\cite{Go16} for related discussions on the non-cutoff Boltzmann equation). Nonetheless, given the estimates in Theorem~\ref{thm:main}, we can construct solutions which do not approach traveling global Maxwellians using a perturbative argument.

Before we proceed to the formulation of this result (see~Theorem~\ref{thm:Maxwellian} below), we fix our notation and take the following definition of traveling global Maxwellians from \cite{Levermore}. 
\begin{definition}[Traveling global Maxwellians]\label{def:GM}
We say that a function $\mathcal M:[0,+\infty)\times \mathbb R^3\times \mathbb R^3\to \mathbb R_{>0}$ is a traveling global Maxwellian if
$$\mathcal M(t,x,v) = \f{m\sqrt{\det Q}}{(2\pi)^3} \exp \left(-\f 12 \begin{pmatrix} v\\ x-tv \end{pmatrix}^{T} \begin{pmatrix} \sigma I & \beta I+B \\ \beta I-B & \alp I \end{pmatrix} \begin{pmatrix} v\\ x-tv \end{pmatrix} \right),$$
for some $m\geq 0$, $\alp,\,\sigma>0$, $\bt\in \mathbb R$, $B\in \mathbb R^{3\times 3}$ skew symmetric matrix such that $Q=(\alp\sigma-\bt^2)I + B^2$ is positive definite.

Given a traveling global Maxwellian $\mathcal M$, define $\mathcal M^\sharp:[0,+\infty)\times \mathbb R^3\times \mathbb R^3$ by $\mathcal M^\sharp(t,x,v) := \mathcal M(t,x+tv,v)$. Note that by definition $\mathcal M^\sharp$ is independent of $t$. We will henceforth write $\mathcal M^\sharp(x,v) = \mathcal M^\sharp(t,x,v)$

We denote by $\mathfrak M$ the set of all traveling global Maxwellians.
\end{definition}

The following is our result that solutions in general do not asymptote to traveling global Maxwellians:
\begin{theorem}\label{thm:Maxwellian}
There exists $f_{\mathrm{in}}$ satisfying the assumptions of Theorem~\ref{thm:main} such that the limiting function $f_\infty$ given by Theorem~\ref{thm:asymptotics} (defined by $f_\infty(t,x,v):=f^\sharp_\infty(x-tv,v)$) does not correspond to the zero solution or a traveling global Maxwellian.
\end{theorem}
The proof of Theorem~\ref{thm:Maxwellian} will in fact show that for a very large class of initial data, the limits do not correspond to the zero solution or a traveling global Maxwellian; see Remark~\ref{rmk:generally.not.Max}.

The remainder of the introduction is structured as follows. In \textbf{Section~\ref{sec:method}}, we briefly discuss the method of the proof. Then in \textbf{Section~\ref{sec:related}} we discuss some related works. Finally, in \textbf{Section~\ref{sec:outline}}, we end the introduction with an outline of the remainder of the paper.

\subsection{Method of proof}\label{sec:method}

\subsubsection{Local existence}\label{sec:method.local}
In order to construct global-in-time solutions in the near-vacuum regime, one first needs to ensure that local-in-time solutions exist. This was carried out in a recent work of Henderson--Snelson--Tarfulea \cite{HeSnTa17} (see also \cite{AMUXY10, AMUXY11, AMUXY13} for related ideas for the Boltzmann equation without angular cutoff). 

We highlight two ingredients in \cite{HeSnTa17}:
\begin{enumerate}
\item Use of a function space adapted to a \emph{time-dependent} Gaussian in $|v|$: As $t$ increases, one only aims for an upper bound by a weaker Gaussian in the $v$-variable. This allows one to control the $v$-weights in coefficient $\bar{a}_{ij}$ in \eqref{Landau.3}. In order to handle the time-dependent Gaussian weight, one also needs to exploit the anisotropy (in $v$) of the coefficient $\bar{a}_{ij}$.
\item Use of $L^2$-based estimates: This in particular allows for an integration by parts argument to control the commutator terms without a loss of derivatives.
\end{enumerate}

As in \cite{HeSnTa17}, we will use time-dependent Gaussian weights in $|v|$. Our choice of Gaussian weights will decrease as $t$ increases, but it needs to decay in a sufficiently slow manner so that it is non-degenerate as $t\to +\infty$. More precisely, we define\footnote{We will from now on use the Japanese bracket notations; see Section~\ref{sec:notation}.}
\begin{equation}\label{g.def.intro}
g := e^{d(t) \vb^2} f,\quad d(t) := d_0 (1+ (1+t)^{-\de})
\end{equation}
for appropriate $d_0>0$, $\de>0$ and estimate $g$ instead of $f$. (Similar time-dependent weights have been used \cite{rjDtYhjZ2011, rjDtYhjZ2012, rjDtYhjZ2013} to gain control over $v$-weights.)

In order not to lose derivatives, we will in particular prove $L^2$-based energy estimates for $g$ and its derivatives. However, in our setting, we will also need additional ingredients to handle the large time behavior of $g$ (and its derivatives).

\subsubsection{Decay and heuristic argument}\label{sec:method.heuristics}

Before we proceed, we give a heuristic argument why one can expect that in the near-vacuum regime, the solutions to the Landau equation, as $t\to +\infty$ approach solutions to the linear transport equation
\begin{equation}\label{eq:transport}
\rd_t f + v_i \rd_{x_i} f = 0.
\end{equation}

Let us recall again the Landau equation (see \eqref{Landau.3})
\begin{equation}\label{eq:Landau.in.methods}
\rd_t f+v_i\rd_{x_i}f=\ab_{ij} \rd^2_{v_i v_j} f-\cb f.
\end{equation}
We now argue that \emph{if $f$ obeys decay estimates similar to those satisfied by solutions to the linear transport equation} \eqref{eq:transport}, then the RHS of \eqref{eq:Landau.in.methods} decays with a rate at least $(1+t)^{-1-}$, which in particular is \emph{integrable in time}. This at least shows that it is consistent to expect $f$ behaves like solutions to the linear transport equation.

\eqref{eq:transport} can be solved explicitly and a solution takes the form
\begin{equation}\label{transport.explicit}
f_{\mathrm{free}}(t,x,v) = f_{\mathrm{data}}(x-tv,v).
\end{equation}
Thus if the initial $f_{\mathrm{data}}$ is sufficiently localized in $x$ and $v$, then for all $\ell \in \mathbb N\cup \{0\}$,
\begin{equation}\label{easy.transport}
\int_{\mathbb R^3} \vb^\ell f_{\mathrm{free}}(t,x,v)\,\ud v\ls (1+t)^{-3},\quad |f_{\mathrm{free}}|(t,x,v) \ls 1.
\end{equation}
By \eqref{transport.explicit}, it also follows that taking $\rd_x$ derivatives does not worsen the decay estimate, but taking $\rd_v$ derivatives worsen the estimate by a power of $t$, i.e.
\begin{equation}\label{easy.transport.2}
\int_{\mathbb R^3} \vb^\ell |\rd_x^\alp \rd_v^\bt f_{\mathrm{free}}|(t,x,v)\,\ud v\ls (1+t)^{-3+|\bt|},\quad |\rd_x^\alp \rd_v^\bt f_{\mathrm{free}}|(t,x,v) \ls (1+t)^{|\bt|}.
\end{equation}
Assuming that $f$ obeys estimates as for $f_{\mathrm{free}}$, we now consider each of the terms on the RHS of \eqref{eq:Landau.in.methods}.

\textbf{The $\bar{c} f$ term.} An easy interpolation together with \eqref{easy.transport} imply that
$$\bar{c}(t,x,v) \ls \int_{\mathbb R^3} |v-v_*|^{\gamma} f(t,x,v_*)\,\ud v_* \ls \|f \|_{L^1_v}^{1+\f{\gamma}{3}}\|f\|_{L^\i_v}^{-\f{\gamma}{3}}\ls (1+t)^{-3-\gamma}.$$
At the same time, $f$ is uniformly bounded. Since $\gamma>-2$, this implies that $\bar{c}f$ is \emph{integrable in time}.

\textbf{The $\bar{a}_{ij} \rd^2_{v_i v_j} f$ term.} We focus on the decay in $t$ and neglect for the moment the additional $|v|$ weight which could in principle be handled using the time-dependent Gaussian weight as in Section~\ref{sec:method.local}. The $\bar{a}_{ij}$ term has decay
$$\bar{a}_{ij}(t,x,v) \ls  \int_{\mathbb R^3} (|v|+|v_*|)^{2+\gamma} f(t,x,v_*)\,\ud v_* \ls \vb^{2+\gamma} (1+t)^{-3}.$$
On the other hand, by \eqref{easy.transport.2}, $\rd^2_{v_i v_j} f$ is not bound even for solutions to the free transport, but instead \emph{grows} like $(1+t)^2$. Hence together it seems that $\bar{a}_{ij} \rd^2_{v_i v_j} f$ decays only as $(1+t)^{-1}$, which is barely non-integrable in time! 

\textbf{The null structure.} The key observation, however, is that while the decay estimates in \eqref{easy.transport}, \eqref{easy.transport.2} are in general sharp, they are sharp only when $\f xt \sim v$. For instance, given sufficiently regular and localized data, when $|v-\f{x}{t}|\gtrsim t^{-\alp}$ for some $\alp\in [0,1)$, $f(t,x,v)$ in fact decays in time (as opposed to merely being bounded). A similar improvement also occurs for velocity averages, as long as the velocity average is taken over a set with an appropriate lower bound on $|v-\f xt|$. 

Returning to our problem, at a spacetime point $(t,x)$, for the term $\int |v-v_*|^{2+\gamma}f(v_*)\,\ud v_*(\rd^2_{v_i v_j}f)(v)$, we must have one of the following three scenarios: (1) $v$ is not too close to $\f xt$, (2) $v_*$ is not too close to $\f xt$, or (3) $v$ and $v_*$ are close to each other. In cases (1) or (2), one has additional decay because of the gain away from $v\sim \f xt$ we described above; while in case (3) there is an improvement because of the small $|v-v_*|^{2+\gamma}$ factor! It therefore implies that
\begin{equation}\label{eq:improved.decay}
|\bar{a}_{ij}\rd^2_{v_iv_j} f|\ls \vb^{2+\gamma}(1+t)^{-1-}.
\end{equation}
The improved decay \eqref{eq:improved.decay} can be viewed as a consequence of a \emph{null structure} in the nonlinearity.

These rough heuristics already give hope that one can \emph{bootstrap} the decay estimates consistent with that of the transport equation.

Given the above discussion, the key ingredients for the proof are as follows:
\begin{enumerate}
\item Develop a robust method for proving decay estimates for solutions to the transport equation.
\item The robust decay estimates need to capture the improved decay in \eqref{eq:improved.decay}, which is important for exploiting the \emph{null structure} in the equation.
\item Moreover, the decay estimates have to be combined with $L^2$-based energy estimates (which is needed already for \emph{local} regularity theory; see~Section~\ref{sec:method.local}).
\end{enumerate}
We will discuss points 1, 2, and 3 respectively in Sections~\ref{sec:method.decay}, \ref{sec:hierarchy.of.norms} and \ref{sec:method.energy}. There is yet another issue arising from combining 1, 2, and 3, and will be discussed in Section~\ref{sec:descent.scheme}.

\subsubsection{A robust decay estimate for the transport equation and the maximum principle}\label{sec:method.decay}

\textbf{Main robust decay estimate.} Our robust decay estimate will be based on controlling a weighted $L^\i_xL^\i_v$ norm of $g$ and its derivatives (recall \eqref{g.def.intro}). The main idea is very simple for the linear transport equation. Given a sufficiently regular solution $f_{\mathrm{free}}$ to \eqref{eq:transport}, $\vb^\ell \wb^m f_{\mathrm{free}}$ also solves \eqref{eq:transport}. As a result $\| \vb^\ell \wb^m f_{\mathrm{free}} \|_{L^\i_xL^\i_v}(t)$ is uniformly bounded by its initial value. For $m>3$, this implies
\begin{equation}\label{stupid.dispersion}
|\int_{\mathbb R^3} \vb^\ell f_{\mathrm{free}} \,\ud v | \ls \| \vb^\ell \wb^m f_{\mathrm{free}} \|_{L^\i_xL^\i_v}(0) \int_{\mathbb R^3} \f{\ud v}{\wb^{m}}  \ls_m \f{\| \vb^\ell \wb^m f_{\mathrm{free}} \|_{L^\i_xL^\i_v}(0)}{(1+t)^3},
\end{equation}
and we have a decay estimate for weighted velocity averages of $f_{\mathrm{free}}$.

This type of estimate turns out to be sufficiently robust to be used in a nonlinear setting. We will prove the following weighted $L^\i_xL^\i_v$ bound for $g$ for some $m\geq 4$.\footnote{Here, and for the rest of this subsubsection, we have yet to make precise the powers $m$ that we will use. This will turn out to be a delicate issue; see Section~\ref{sec:hierarchy.of.norms} for further discussions.}
\begin{equation}\label{intro.Li.low.order}
\sup_{(x,v)\in \mathbb R^3\times \mathbb R^3} |\vb \wb^m g|(t,x,v) \ls \ep.
\end{equation}
Since $\bar{a}_{ij}$ and $\bar{c}$ are convolutions of $f$ with different kernels (see~\eqref{bar.def}), \eqref{intro.Li.low.order} implies quantitative decay estimates for the coefficients $\bar{a}_{ij}$ and $\bar{c}$ in a manner similar to \eqref{stupid.dispersion}.

\textbf{Commutators and higher order estimates.} To close our estimate we in fact need also to control also higher derivatives of $g$. For this purpose we use $\rd_x$, $\rd_v$ and $Y:=t\rd_x+\rd_v$ as commutators. $\rd_x$ and $Y$ both commute with the transport operator $\rd_t+v_i\rd_{x_i}$, but $\rd_v$ does not commute with the transport operator\footnote{For the estimate \eqref{intro.Li.goal}, in fact one can equivalently just prove $$|\wb^m \rd_x^\alp Y^\sigma g|(t,x,v) \ls \ep \vb^{-1}$$
(i.e.~without commuting with $\rd_v$) and \eqref{intro.Li.goal} follows from the triangle inequality. The actual reason that we also use $\rd_v$ as a commutator is more subtle and is related to the fact that we will use a hierarchy of weighted norms; see \eqref{how.to.weight.g} in Section~\ref{sec:hierarchy.of.norms}.}. This results in a loss of a power of $t$ for every commutation with $\rd_v$. In other words, we will aim at the following $L^\i_xL^\i_v$ estimate (see~\eqref{easy.transport.2} and \eqref{intro.Li.low.order}):
\begin{equation}\label{intro.Li.goal}
|\wb^m \rd_x^\alp \rd_v^\bt Y^\sigma g|(t,x,v) \ls \ep \vb^{-1} (1+t)^{|\bt|}.
\end{equation}
When $m\geq 4$, \eqref{intro.Li.goal} implies the following estimates for the coefficients (using an argument similar to \eqref{stupid.dispersion}, after appropriately accounting for the singularity in $v$ in the definition of $\bar{c}$):
\begin{equation}\label{intro.Li.coeff}
\sup_{i,j} |\rd_x^\alp \rd_v^\bt Y^\sigma \bar{a}_{ij}| \ls \ep \vb^{2+\gamma}(1+t)^{-3},\quad |\rd_x^\alp \rd_v^\bt Y^\sigma \bar{c}| \ls \ep (1+t)^{-3-\gamma}.
\end{equation}

\textbf{Estimating the error terms.} The estimates \eqref{intro.Li.goal} and \eqref{intro.Li.coeff} will be proven simultaneously in a bootstrap argument. In order to establish \eqref{intro.Li.goal}, we differentiate the equation for $g$ and control the terms on the RHS.

One of the error terms (which shows the typical difficulty) is $(\rd_x^{\alp'}\rd_v^{\bt'}Y^{\sigma'} \bar{a}_{ij}) (\rd^2_{v_iv_j} \rd_x^{\alp''} \rd_v^{\bt''} Y^{\sigma''} g)$ (where $\alp'+\alp''=\alp$, etc.). If we were to plug in \eqref{intro.Li.goal} and \eqref{intro.Li.coeff}, this error term is controlled by\footnote{Note that in the actual bootstrap setting we need some room and will only obtain a smallness constant of $\ep^{\f 32}$ instead of $\ep^2$. We will suppress this minor detail in the rest of the introduction.}
\begin{equation}\label{bad.intro.Li.term}
|\wb^m (\rd_x^{\alp'}\rd_v^{\bt'}Y^{\sigma'}\bar{a}_{ij}) (\rd^2_{v_iv_j} \rd_x^{\alp''} \rd_v^{\bt''} Y^{\sigma''} g)|(t,x,v)\ls \ep^2 \vb^{1+\gamma} (1+t)^{-1+|\bt|}.
\end{equation}
We make the following observations by comparing the $\vb$ weights and $t$ rates in \eqref{intro.Li.goal} and \eqref{bad.intro.Li.term}:
\begin{enumerate}
\item We need to prove an estimate \eqref{intro.Li.goal} which has \emph{better} $\vb$ weight compared to \eqref{bad.intro.Li.term}.
\item The decay rate on the RHS of \eqref{bad.intro.Li.term} is exactly borderline to obtain the decay rate in \eqref{intro.Li.goal}.
\end{enumerate}

For point (1) above, note that a gain in $\vb$ weight is possible is due to the $e^{d(t)\vb^2}$ weight in the definition of $g$ (see~\eqref{g.def.intro}). This gain has to be achieved, however, at the expense of a $(1+t)^\de$ decay rate (see~definition of $d(t)$).

For point (2) above regarding $t$-decay, already the borderline rate means that we cannot hope to straightforwardly recover \eqref{intro.Li.goal} when $|\bt|=0$. This is even more problematic since to handle the $\vb$ weights for point (1) above requires additional room for the $t$-decay rate. We must therefore improve the decay rate in \eqref{bad.intro.Li.term} by taking advantage of the null structure (recall the heuristic argument in Section~\ref{sec:method.heuristics}). This will be discussed in Section~\ref{sec:hierarchy.of.norms}.

\textbf{The maximum principle.} However, even with the ideas to be discussed in Section~\ref{sec:hierarchy.of.norms}, we will not be able to obtain sufficient $t$-decay to treat the main (non-commutator) term 
\begin{equation}\label{main.term.in.MP}
\wb^{m} \bar{a}_{ij} \rd^2_{v_iv_j} \rd_x^\alp \rd_v^\bt Y^\sigma g
\end{equation}
Instead, we will handle \eqref{main.term.in.MP} using a maximum principle argument: since $\bar{a}_{ij}$ is semi-positive definite, we show that the presence of the term \eqref{main.term.in.MP} can only give a favorable contribution. In other words, only the terms with $|\alp'|+|\bt'|+|\sigma'|\geq 1$ in \eqref{bad.intro.Li.term} will be treated as errors.

\textbf{Additional technical difficulties.} Unfortunately, even after taking into account all the above considerations, not all the $L^\i$ estimates we prove will be as strong as \eqref{intro.Li.goal}. This is related to the fact we need to couple our $L^\i$ estimates with $L^2$ estimates. The important point, however, is that \emph{at the lower order of derivatives}, i.e.~for $|\alp|+|\bt|+|\sigma|$ smaller than a particular threshold, we indeed obtain the estimate \eqref{intro.Li.goal}. We will return to this issue in Section~\ref{sec:descent.scheme}.

\subsubsection{Null structure and the hierarchy of weighted norms}\label{sec:hierarchy.of.norms}

Our robust proof of decay must also capture the null structure discussed in Section~\ref{sec:method.heuristics}! By naive inspection, one can already see that the $\wb$ weight (see~\eqref{intro.Li.goal} in Section~\ref{sec:method.decay}) ensures the solution to be localized at $v\sim \f xt$, which as discussed in Section~\ref{sec:method.heuristics} is exactly the mechanism which enforces the null structure. 

In order to exploit this gain, however, one needs to be able to put in \emph{extra} weights on the error terms, i.e.~in order to control $\wb^\ell \rd_x^\alp \rd_v^\bt Y^\sigma g$, we will need to have estimates for $\wb^{\ell+}\rd_x^{\alp'} \rd_v^{\bt'} Y^{\sigma'} g$, where $\ell+$ denotes a positive number strictly larger than $\ell$. In order to close the estimates, we need to exploit more subtle features of the problem and  introduce a \emph{hierarchy} of weighted norms. Namely, the weight of $\wb$ that we will use will depend on the number of $Y:=t\rd_x+\rd_v$ derivatives on $g$. In fact, we will control\footnote{Recall here $\Mm$ is the maximum number of derivatives in the assumptions of Theorem~\ref{thm:main}.}
\begin{equation}\label{how.to.weight.g}
\wb^{\Mm+5-|\sigma|}\rd_x^\alp\rd_v^\bt Y^\sigma g
\end{equation}
so that the more $Y$ derivatives we have, the weaker $\wb$ weight we put.

Here are the main observations that allow such a hierarchy of weighted estimates to be closed:
\begin{enumerate}
\item The main (i.e.~non-commutator) term can be considered as a ``good term'' (see the discussions on the maximum principle in Section~\ref{sec:method.decay})). Thus we only need to control terms where at least one derivative hits on $\bar{a}_{ij}$, i.e.
$$(\rd_x^{\alp'}\rd_v^{\bt'}Y^{\sigma'} \bar{a}_{ij})(\rd^2_{v_iv_j} \rd_x^{\alp''}\rd_v^{\bt''}Y^{\sigma''} g),$$
where $|\alp'|\geq 1$, $|\bt'|\geq 1$ or $|\sigma'|\geq 1$.
\item Next, we show that if there is at least one $\rd_v$ derivative on $\bar{a}_{ij}$, i.e.~if $|\bt'|\geq 1$, then the decay is $(1+t)^{-2-\nu}$ with some $\nu>0$ (depending on $\gamma$). This can be thought of as a better-than-expected estimate since without using the structure of $\bar{a}_{ij}$, one may naively expect that every $\rd_v$ derivative ``costs'' one power of $t$ so that one only has $|\rd_v \bar{a}_{ij}|\ls (1+t)^{-2}$.
\item In the case where there is at least one $\rd_x$ derivative on $\bar{a}_{ij}$, i.e.~if $|\alp'|\geq 1$, we write $\rd_x = t^{-1} (t\rd_x+\rd_v) - t^{-1} \rd_v = t^{-1}Y - t^{-1}\rd_v$. Note that 
\begin{itemize}
\item since $Y$ is one of our commutators, $t^{-1} Y$ effectively gains us a power of $t$;
\item $t^{-1} \rd_v$ also gains in terms of $t$ due to the gain associated to $\rd_v$ in point 2 above.
\end{itemize}
\item It thus remains to control the terms where there is at least one $Y=t\rd_x+\rd_v$ derivative hitting on $\bar{a}_{ij}$, i.e.~if $|\sigma'|\geq 1$. In this case, it must be that there is one fewer $Y$ hitting on $g$ as compared to the term that we are estimating! Our hierarchy of norms (see \eqref{how.to.weight.g}) is designed so that one can put an extra $\wb$ weight in this term and therefore one can use the \emph{null structure} to obtain an additional decay rate.
\end{enumerate}

\subsubsection{$L^2$ energy estimates}\label{sec:method.energy}

For regularity issues, we cannot work with $L^\i$ estimates alone, but will also need to work with $L^2$ based estimates (which is already the case for \emph{local-in-time} estimates; see Section~\ref{sec:method.local}.) Similar to the $L^\i$ estimates (see Section~\ref{sec:method.decay}), we use $\rd_x$, $\rd_v$, $Y:= t\rd_x+\rd_v$ as commutators. We then prove $L^2$ estimates for $\rd_x^\alp \rd_v^\bt Y^\sigma g$, again weighted with $\wb^{\Mm+5-|\sigma|}$ to exploit the null structure (see Section~\ref{sec:hierarchy.of.norms}).

\textbf{Main $L^2$ estimates.} Using the equation for $g$ one derives a weighted $L^2$ estimate which for $|\alp|+|\bt|+|\sigma|\leq \Mm$ controls the following three terms on any time interval $[0,T]$, up to some error terms:
\begin{equation}\label{intro.boundary.term}
\|\wb^{\Mm+5-|\sigma|} \rd_x^\alp \rd_v^\bt Y^\sigma g \|_{L^\i([0,T]; L^2_xL^2_v)}^2,
\end{equation}
\begin{equation}\label{intro.good.term.1}
\|\wb^{2\Mm+10-2|\sigma|} \bar{a}_{ij} (\rd_{v_i} \rd_x^\alp \rd_v^\bt Y^\sigma g)(\rd_{v_j} \rd_x^\alp \rd_v^\bt Y^\sigma g)\|_{L^1([0,T];L^1_xL^1_v)},
\end{equation}
and
\begin{equation}\label{intro.good.term.2}
\|\wb^{\Mm+5-|\sigma|} (1+t)^{-\f 12 -\f \de 2} \vb \rd_x^\alp \rd_v^\bt Y^\sigma g\|_{L^2([0,T];L^2_xL^2_v)}^2.
\end{equation}
The term \eqref{intro.boundary.term} is a fixed-time estimate while the terms \eqref{intro.good.term.1} and \eqref{intro.good.term.2} are non-negative terms integrated over $[0,T]\times \mathbb R^3\times \mathbb R^3$. 

The term \eqref{intro.good.term.1} arises from the main (non-commutator) term $\bar{a}_{ij}\rd^2_{v_iv_j} \rd_x^\alp \rd_v^\bt Y^\sigma g$. Note that since we do not establish any lower bound for $\bar{a}_{ij}$, \eqref{intro.good.term.1} could be too degenerate to be used to control error terms, but its good sign at least means that we need not view the main term as an error term.

\eqref{intro.good.term.2} has the favorable feature that it has a stronger $\vb$ weight compared to \eqref{intro.boundary.term}. This term is generated by the time-dependent Gaussian in the definition of \eqref{g.def.intro}. We will in fact use \eqref{intro.good.term.2} to bound most of the error terms.\footnote{The only error terms that we will not estimate with \eqref{intro.good.term.2} but will instead use \eqref{intro.boundary.term} are the terms arising from the commutator $[\rd_t+v_i \rd_{x_i}, \rd_x^\alp \rd_v^\bt Y^\sigma]$; see Section~\ref{sec:EE} for details.}

We note as in Section~\ref{sec:method.decay} that $\rd_v$ does not commute with $\rd_t+v_i\rd_{x_i}$ and therefore the decay rate worsens with every commutation of $\rd_v$. Denoting $E^2(T):= \sum_{|\alp|+|\bt|+|\sigma|\leq \Mm} (1+T)^{-2|\bt|} (\eqref{intro.boundary.term} + \eqref{intro.good.term.2})$, our goal will be to prove that
\begin{equation}\label{intro.E.goal}
E^2(T) \ls \ep^2.
\end{equation}
Note that this is consistent with the best $L^2_xL^2_v$ estimate that one can get for solutions to the linear transport equation.

\textbf{Controlling the error terms.} We consider an example of an error term when deriving the energy estimates (which shows the typical difficulties): 
\begin{equation}\label{example.error}
\sum_{\substack{\alp'+\alp''=\alp,\,\bt'+\bt''=\bt,\,\sigma'+\sigma''=\sigma \\|\alp'|+|\bt'|+|\sigma'|\geq 1}} \|\wb^{2\Mm+10-2|\sigma|}(\rd_x^\alp \rd_v^\bt Y^\sigma g) (\rd_x^{\alp'} \rd_v^{\bt'} Y^{\sigma'} \bar{a}_{ij}) (\rd^2_{v_i v_j} \rd_x^{\alp''} \rd_v^{\bt''} Y^{\sigma''} g) \|_{L^1([0,T];L^1_xL^1_v)}.
\end{equation}

Notice that as described above, the main (non-commutator term) can be viewed as a good term. Therefore we indeed only need to consider the cases $|\alp'|+|\bt'|+|\sigma'|\geq 1$.

The estimates are different depending on whether $|\alp'|+|\bt'|+|\sigma'|$ is small or large. When $|\alp'|+|\bt'|+|\sigma'|$ is small, we can use \eqref{intro.Li.coeff} to control $\rd_x^{\alp'} \rd_v^{\bt'} Y^{\sigma'} \bar{a}_{ij}$ and bound both $\rd_x^\alp \rd_v^\bt Y^\sigma g$ and $\rd^2_{v_i v_j} \rd_x^{\alp''} \rd_v^{\bt''} Y^{\sigma''} g$ in $L^2([0,T];L^2_xL^2_v)$ using the norm as in \eqref{intro.good.term.2}. One then sees that the decay rate is slightly insufficient (in fact it misses by a power of $(1+t)^{\de}$). As in the proof of the $L^\i_xL^\i_v$ estimates in Section~\ref{sec:hierarchy.of.norms}, to overcome the borderline decay, we need to make use of the null structure. Indeed, we note that $|\alp'|+|\bt'|+|\sigma'|\geq 1$ so that we can argue as in Section~\ref{sec:hierarchy.of.norms} to obtain a better decay rate. (Note that the ideas in Section~\ref{sec:hierarchy.of.norms} give a quantitatively better rate than the borderline case. Thus by choosing $\de$ sufficiently small, \eqref{example.error} can indeed be controlled by the norms in \eqref{intro.good.term.2}.)

Consider now the term \eqref{example.error} when $|\alp'|+|\bt'|+|\sigma'|$ is large. As a particular example, we have the following term when $|\alp|+|\bt|+|\sigma|=\Mm$:
\begin{equation}\label{the.annoying.intro.term}
\|\wb^{2\Mm+10-2|\sigma|}(\rd_x^\alp \rd_v^\bt Y^\sigma g) (\rd_x^{\alp} \rd_v^{\bt} Y^{\sigma} \bar{a}_{ij}) (\rd^2_{v_i v_j}  g) \|_{L^1([0,T];L^1_xL^1_v)}.
\end{equation}
Here, we are faced with another challenge regarding the decay rate. At the top order, we need to control $ \rd_x^{\alp} \rd_v^{\bt} Y^{\sigma} \bar{a}_{ij}$ in $L^2_x$ (as opposed to $L^\i_x$). As a result, we only obtain 
$$\|\vb^{-(2+\gamma)}\rd_x^{\alp}\rd_v^\bt Y^\sigma \bar{a}_{ij} \|_{L^2_xL^\i_v} \ls \ep (1+t)^{-\f 32+|\bt|}.$$
(This should be compared with the $L^\i_xL^\i_v$ estimate in \eqref{intro.Li.coeff} when $|\alp|+|\bt|+|\sigma|$ is lower order.) At the same time, we need to bound $\rd^2_{v_iv_j} g$ in some $L^\i_x$ norm. At first sight, one may hope, based on linear estimates for solutions to \eqref{eq:transport}, that 
$$\|\wb^{\Mm+5-|\sigma|} \vb \rd^2_{v_iv_j} g\|_{L^\i_xL^2_v} \ls \ep (1+t)^{\f 12}.$$
However, when $|\sigma|$ is small, we are very tight with the $\wb$ weights and in general we only obtain the following weaker estimate\footnote{Note that in the $(1+t)$ weight on the LHS, we have $\de$ instead of $\f \de 2$ (as one may expect). This is a technical point (see Section~\ref{sec:EE.prelim}) which plays no substantial role.} based on Sobolev embedding and \eqref{intro.good.term.2}:
$$\|(1+t)^{-\f 12- \de -|\bt|} \wb^{\Mm+5-|\sigma|} \vb \rd^2_{v_iv_j} g\|_{L^2([0,T];L^\i_xL^2_v)} \ls \ep (1+T)^2.$$
Combining these estimates and using \eqref{intro.E.goal}, it seems that H\"older's inequality only gives
\begin{equation*}
\begin{split}
&\: \|\wb^{2\Mm+10-2|\sigma|} (\rd_x^\alp \rd_v^\bt Y^\sigma g) (\rd_x^{\alp} \rd_v^{\bt} Y^{\sigma} \bar{a}_{ij}) (\rd^2_{v_i v_j}  g) \|_{L^1([0,T];L^1_xL^1_v)} \\
\ls &\: \|\wb^{\Mm+5-|\sigma|} \f{\vb}{(1+t)^{\f 12+\f \de 2}} \rd_x^\alp \rd_v^\bt Y^\sigma g\|_{L^2([0,T];L^2_xL^2_v)} \|\vb^{-(2+\gamma)} (1+t)^{\f 32-|\bt|} \rd_x^{\alp} \rd_v^{\bt} Y^{\sigma} \bar{a}_{ij} \|_{L^\i([0,T];L^2_xL^2_v)}\\
&\:\times \|(1+t)^{-\f 12-\de-|\bt|} \wb^{\Mm+5-|\sigma|} \vb \rd^2_{v_iv_j} g \|_{L^\i([0,T];L^\i_xL^\i_v)} \|(1+t)^{-\f 12+\f {3\de} 2+2|\bt|}\|_{L^\i([0,T])} \\
\ls &\: \ep^2 (1+T)^{\f 32+\f {3\de}2+2|\bt|},
\end{split}
\end{equation*}
which is much worse than the bound $(1+T)^{2|\bt|}$ that we aim at in \eqref{intro.E.goal}. 

To handle \eqref{the.annoying.intro.term}, note that while at the top order we need to put $\rd_x^{\alp}\rd_v^{\bt} Y^{\sigma} \bar{a}_{ij}$ in $L^2_xL^\i_v$, we must have $|\alp|+|\bt|+|\sigma| \geq 2$. In this case, we can further extend ideas as described in Section~\ref{sec:hierarchy.of.norms} to obtain better decay rates. (Note that unlike for the $L^\i$ estimates, ideas in Section~\ref{sec:hierarchy.of.norms} are no longer just used to beat the borderline terms, but are instead needed to achieve a more substantial improvement.) To implement this, we will in addition need to contend with certain singular $|v-v_*|^\gamma$ factors, which affect the decay rate. In order for the above ideas to work, we will then need to estimate some $\rd_x^{\alp}\rd_v^\bt Y^\sigma \bar{a}_{ij}$ terms in a few different mixed $L^2_xL^p_v$ spaces (for appropriate $p\in [2,\infty)$ depending on $\gamma$). See Sections~\ref{sec:coeff.L2} and \ref{sec:EE} for details.

\subsubsection{A descent scheme}\label{sec:descent.scheme}

As we have stressed above, even our decay estimates are based on the $L^\i$ bounds, in order not to lose derivatives, we need to combine the $L^\i$ estimates with $L^2$ energy estimates. In other words, the $L^\i$ estimates we described above in Section~\ref{sec:method.decay} do not close by themselves. Indeed, in carrying out the maximum principle argument, we encounter commutator terms that have one derivative more than the term that we are estimating. As a result, at the higher level of derivatives, we need to use the $L^2$ estimate together with Sobolev embedding to control these commutator terms. \textbf{This however creates a loss} in both $\vb$ and $t$ in the sense that the $L^\i$ decay rate thus obtained is \emph{weaker} than the corresponding decay rate for solutions to the linear transport equation.

In order to overcome this, we introduce a \emph{descent scheme}. More precisely, we allow the higher level $L^\i$ norms to have weaker decay in both $\vb$ and $t$ compared to \eqref{intro.Li.goal}, but as we descend in the number of derivatives, we obtain a slight improvement at every level, until we get to a sufficiently low level of derivatives for which we obtain the desired \eqref{intro.Li.goal}. To give a concrete example, consider the special case $\gamma = -1$. We will prove 
$$\sum_{|\alp|+|\bt|+|\sigma| = 10} \wb^{\Mm+5-|\sigma|}|\rd_x^\alp\rd_v^\bt Y^{\sigma}g|(t,x,v) \ls \ep (1+t)^{|\bt|+\f 32},$$ 
$$\sum_{|\alp|+|\bt|+|\sigma| = 9} \wb^{\Mm+5-|\sigma|}|\rd_x^\alp\rd_v^\bt Y^{\sigma}g|(t,x,v) \ls \ep \vb^{-1} (1+t)^{|\bt|+\f 34},$$
$$\sum_{|\alp|+|\bt|+|\sigma| \leq 8} \wb^{\Mm+5-|\sigma|}|\rd_x^\alp\rd_v^\bt Y^{\sigma}g|(t,x,v) \ls \ep \vb^{-1} (1+t)^{|\bt|}.$$
Here are two observations regarding the descent scheme:
\begin{enumerate}
\item Such a scheme can close since when controlling a nonlinear term, a term with higher order derivatives must multiply a term with lower order derivatives. Therefore, the loss that we allow in a descent scheme does not accumulate. (It is therefore also crucial that we indeed prove sharp estimates at the lower order!)
\item Moreover, when bounding the nonlinear terms, after using the ideas in Sections~\ref{sec:method.decay}--\ref{sec:method.energy}, every term that we encounter is quantitatively better than the borderline case. It is for this reason that every time we descend one order of derivative, we obtain a quantitative gain.
\end{enumerate}

We note that the full hierarchy for the descent scheme is more complicated for general $\gamma$. In fact, as $\gamma\to 0^-$ or $\gamma\to -2^+$, the number of steps for which we descend $\to +\infty$. (It is because of this fact that we need a large number of derivatives in Theorem~\ref{thm:main} as $\gamma\to 0^-$ or $\gamma\to -2^+$.)  We refer the reader to Section~\ref{sec:hierarchy} for the precise numerology.

\subsubsection{Long-time asymptotics}

The above concludes the discussions of the main difficulties of proving the global existence of near-vacuum solutions. Theorem~\ref{thm:asymptotics}, Corollary~\ref{cor:macro} and Theorem~\ref{thm:Maxwellian} more or less follow from the estimates that have been established. 

The only thing to note is that so far we have ``dropped'' the main elliptic term $\bar{a}_{ij}\rd^2_{v_iv_j}\rd_x^\alp \rd_v^\bt Y^\sigma g$ in either the maximum principle or energy estimate argument, showing that it can only give a better upper bound than that for the linear transport equation. To make statements about the precise asymptotic behavior of the solutions, however, we need to be able to control the main elliptic term.

The key point is to note that since all the estimates have now been closed, by carrying out an estimate on $f$ with a slightly weaker $\wb$ weight, we can use the null structure to show that even the main term $\bar{a}_{ij}\rd^2_{v_iv_j} f$ has faster than integrable time decay. We refer the reader to Section~\ref{sec:long.time} for details.

\subsection{Related works}\label{sec:related}

\subsubsection{Stability of vacuum for collisional kinetic models}\label{sec:cutoff.Boltzmann}

The earliest work on the stability of vacuum for a collisional kinetic model is that for the Boltzmann equation with an angular cutoff by Illner--Shinbrot \cite{IlSh84}. There are many extensions and refinements of \cite{IlSh84}; see for instance \cite{ToBe84, BaDeGo84, Ha85, BeTo85, To86, Po88, Guo01, Ar11, HeJi17}. We refer the readers also to the related \cite{To88, Go97, AlGa09} in which perturbations of traveling global Maxwellians were studied --- in this setting the long-time dynamics is also characterized by dispersion (compare Theorem~\ref{thm:asymptotics}).

To our knowledge the present work is the first stability of vacuum result for a collisional kinetic model \emph{with a long range interaction}. Note in particular that the analogous stability of vacuum problem for the non-cutoff Boltzmann equation remains open.

\subsubsection{Dispersion and stability for collisionless models} 
Stability of vacuum results in collisional models can be viewed in the larger context of stability results for nonlinear models in kinetic theory that are driven by dispersion. That dispersion of the transport operator is useful in establishing global result for close-to-vacuum data has been well-known early on for collisionless models; see \cite{BaDe85, GlSt87, GlSc88} for some early results, which are mostly based on the method of characteristics. See also \cite{BaDeGo84} for a discussion of the relation between these results and the stability of vacuum for the Boltzmann equation with angular cutoff. For more recent discussions, see \cite{Bi17, FaJoSm17.1, Sm16, Wa18.1, Wa18.3, Wa18.2, Wo18}, as well as remarkable proof of the stability of the Minkowski spacetime for the Einstein--Vlasov system \cite{Ta17, FaJoSm17, LiTa17}.

\subsubsection{Regularity theory for Landau equation}

It is an outstanding open problem whether regular initial data to the Landau equation give rise to globally regular solutions. The literature is too vast for an exhaustive discussion, but we highlight some relevant results here.

\textbf{Weak solutions.} Renormalized solutions to the Landau equation have been constructed in \cite{Vi96}. See also \cite{Li94, AlVi04}.

\textbf{Spatially homogeneous solutions.} In the Maxwellian molecule case ($\gamma = 0$) and the hard potentials case ($\gamma > 0$), the theory of spatially homogeneous solutions is very well-developed \cite{Vi98, DeVi00, DeVi00.2}. In the soft potentials case ($\gamma\in [-3,0)$), existence was studied in \cite{ArPe77, Vi98.2, De15}, and uniqueness was studied in \cite{FoGu09, Fo10}. See also \cite{Wu14, AlLiLi15, De15, Si17} for further a priori estimates in the soft potentials case. 

\textbf{Global nonlinear stability of Maxwellians.} The global nonlinear stability of Maxwellians on a periodic box was established in Guo's seminal \cite{Guo02}. This is part of Guo's program to use a nonlinear energy method to construct perturbative solutions in nonlinear kinetic models. The methods of Guo have moreover inspired many subsequent perturbative results for various kinetic models \cite{Guo02.2, Guo03, Guo03.2, StGu04, StGu06, StGu08, Guo12, StZh13}, including the remarkable works of the global nonlinear stability of Maxwellians for the non-cutoff Boltzmann equation \cite{GrSt11, AMUXY12, AMUXY12.2, AMUXY12.3}. See also \cite{CaMi17, CaTrWu17, CaTrWuErratum17} for more recent results on near-Maxwellian solutions.

\textbf{Conditional regularity theory.} A thread of recent works concern regularity of solutions to the Landau equation assuming a priori pointwise control of the mass density, energy density and entropy density \cite{GoImMoVa16, CaSiSn18, HeSn17, HeSnTa17, Sn18}. Our present paper in particular relies on the work \cite{HeSnTa17}, which proves the local existence, uniqueness and instantaneous smoothing of solutions using the theory developed in the papers mentioned above.

\textbf{Model problems.} Various simplified models for Landau equation have been introduced and some regularity results have been obtained for these models; see for instance \cite{KrSt12, GrKrSt12, GuGu16, ImMo18}.

\subsection{Outline of the paper}\label{sec:outline}

The remainder of the paper is structured as follows. 

In \textbf{Section~\ref{sec:notation}}, we introduce some notations that will be used throughout the paper. In \textbf{Section~\ref{sec:local.existence}}, we cite a recent local-in-time existence and uniqueness result of \cite{HeSnTa17}, which will be the starting point of our construction of global-in-time solutions. 

Sections~\ref{sec:bootstrap}--\ref{sec:everything} will be devoted to the proof of Theorem~\ref{thm:main}. In \textbf{Section~\ref{sec:bootstrap}}, we discuss the bootstrap argument used for the proof and introduce the bootstrap assumptions. In \textbf{Section~\ref{sec:coeff}}, we control the coefficients $\bar{a}_{ij}$, $\bar{c}$. In \textbf{Section~\ref{sec:MP}}, we use the maximum principle and an appropriate iteration argument to prove the $L^\i_xL^\i_v$ estimates. In \textbf{Section~\ref{sec:EE}}, we use energy methods to prove the $L^2_xL^2_v$ estimates. We then conclude the proof in \textbf{Section~\ref{sec:everything}}.

Finally, in \textbf{Section~\ref{sec:long.time}}, we discuss the long-time asymptotics of near-vacuum solutions and prove Theorem~\ref{thm:asymptotics}, Corollary~\ref{cor:macro} and Theorem~\ref{thm:Maxwellian}.

\subsection*{Acknowledgments} I thank Robert Strain for introducing this problem to me. I gratefully acknowledge the support of a Sloan fellowship, a Terman fellowship, and NSF grant DMS-1709458.

\section{Notations}\label{sec:notation}

We introduce some notations to be used throughout the paper.

\textbf{Norms.} We will use mixed $L^p$ norms, $1\leq p<\infty$ defined in the standard manner:
$$\|h\|_{L^p_v}:= (\int_{\mathbb R^3} |h|^p(v)\, \ud v)^{\f 1p}.$$
For $p=\infty$, define
$$\|h\|_{L^\i_v}:= \mbox{ess\,sup}_{v\in \mathbb R^3} |h|(v).$$
For mixed norms, the norm on the right is taken first. For instance,
$$\|h\|_{L^p_x L^q_v} := (\int_{\mathbb R^3}(\int_{\mathbb R^3} |h|^q(x,v)\,\ud v)^{\f pq}\,\ud x)^{\f 1p}$$
and
$$\|h\|_{L^r([0,T];L^p_x L^q_v)} := (\int_0^T (\int_{\mathbb R^3}(\int_{\mathbb R^3} |h|^q(t,x,v)\,\ud v)^{\f pq}\,\ud x)^{\f rp} \ud t)^{\f 1r}$$
with obvious modification when $p=\infty$, $q=\infty$ or $r=\infty$.
We will silently use that $\|h\|_{L^p_x L^q_v} \ls \|h\|_{L^q_v L^p_x}$ when $p\geq q$.

Given two Banach spaces $X_1$ and $X_2$, define the following norms for the sum $X_1+X_2$ and the intersection $X_1\cap X_2$:
$$\|h\|_{X_1+X_2} := \inf_{h = h_1+h_2} (\|h_1\|_{X_1} + \|h_2\|_{X_2}),\quad  \|h\|_{X_1\cap X_2} := \|h\|_{X_1}+\|h\|_{X_2}.$$

\textbf{Japanese brackets.} Define
$$\langle \cdot \rangle := (1+|\cdot|^2)^{\f 12}.$$

\textbf{Multi-indices.} $\alp = (\alp_1, \alp_2, \alp_3)\in (\mathbb N\cup\{0\})^3$ will be called a multi-index. Given a multi-index $\alp$, define $\rd_x^\alp = \rd_{x_1}^{\alp_1}\rd_{x_2}^{\alp_2}\rd_{x_3}^{\alp_3}$; and similarly for $\rd_v^\bt$ when $\bt$ is a multi-index. Let $|\alp| = \alp_1+\alp_2+\alp_3$. Multi-indices are added according to the rule that if $\alp' = (\alp'_1, \alp'_2, \alp'_3)$ and $\alp'' = (\alp''_1, \alp''_2, \alp''_3)$, then $\alp'+\alp'':=(\alp'_1+\alp''_1, \alp'_2+\alp''_2, \alp'_3+\alp''_3)$. Given a multi-index $\alp = (\alp_1,\alp_2,\alp_3)$, the length of the multi-index is defined by $|\alp| = \alp_1+\alp_2+\alp_3$.

We will often sum over all multi-indices up to a certain length. In this context, we will use the convention that $\sum_{|\alp|\leq -1} (\cdots) = 0$.

\section{Local existence}\label{sec:local.existence}

In this section, we recall the local existence result in \cite{HeSnTa17} (and state a small variant of it). 

To state the result in \cite{HeSnTa17}, we first recall their definition of uniformly local weighted Sobolev spaces.
\begin{definition}[Uniformly local weighted Sobolev spaces]\label{def:ul}
Let $\phi:\mathbb R^3\to \mathbb R$ be a smooth and compactly supported cut-off function such that $0\leq \phi\leq 1$ everywhere, $\phi(x)=1$ for $|x|\leq 1$ and $\phi(x) = 0$ for $|x|\geq 2$.

Define the $H^{k,\ell}_{\mathrm{ul}}$ norm on $\mathcal S(\mathbb R^3\times \mathbb R^3)$ by
$$\| h \|_{H^{k,\ell}_{\mathrm{ul}}}:= \sum_{|\alp|+|\bt|\leq k}(\sup_{a\in \mathbb R^3} \int_{\mathbb R^3}\int_{\mathbb R^3} |\phi(x-a) \vb^\ell \rd_x^\alp\rd_v^\bt h|^2 \,\ud v\,\ud x)^{\f 12}$$
and take $H^{k,\ell}_{\mathrm{ul}}$ to be the completion of $\mathcal S(\mathbb R^3\times \mathbb R^3)$ under this norm.
\end{definition}

The following theorem is taken from \cite{HeSnTa17}. (Note that in the statement of Theorem~1.1 in \cite{HeSnTa17}, the estimate \eqref{est.local.existence} is stated only for $e^{\f 12\rho_0\vb^2} f$ instead of $e^{(\rho_0-\kappa t)\vb^2} f$, but it is clear from the proof that \eqref{est.local.existence} indeed holds.)
\begin{theorem}[Henderson--Snelson--Tarfulea \cite{HeSnTa17}]\label{thm:local.existence}
Fix $\rho_0,\,M_0\in \mathbb R$ with $\rho_0>0$, $M_0>0$ and $k\in \mathbb N$ with $k\geq 4$. Suppose $e^{\rho_0\vb^2}f_{\mathrm{in}}$ satisfies the estimate
\begin{equation}\label{local.existence.assumption}
\|e^{\rho_0 \vb^2}f_{\mathrm{in}}\|_{H^{k}_{\mathrm{ul}}} \leq M_0.
\end{equation}
Then for any $\kappa>0$, there exists $T=T_{\gamma,\rho_0,M_0,\kappa}>0$ depending only on $\gamma$, $\rho_0$, $M_0$ and $\kappa$, such that there exists a unique solution $f\geq 0$ to the Landau equation \eqref{Landau} with initial data $f(0,x,v) = f_{\mathrm{in}}(x,v)$ and satisfying
\begin{equation}\label{est.local.existence}
\|e^{(\rho_0-\kappa t)\vb^2} f \|_{C^0([0,T];H^{k,0}_{\mathrm{ul}})\cap L^2([0,T];H^{k,1}_{\mathrm{ul}})}<+\infty.
\end{equation}
\end{theorem}

A more remarkable statement is that the solution constructed in Theorem~\ref{thm:local.existence} immediately acquires smoothness and positivity\footnote{as long as the initial $f_{\mathrm{in}}$ is not identically zero.} even the initial $f_{\mathrm{in}}$ may not be smooth and may contain vacuum regions.
\begin{theorem}[Henderson--Snelson--Tarfulea \cite{HeSnTa17}]\label{thm:smoothness.positivity}
The solution $f:[0,T]\times \mathbb R^3\times \mathbb R^3$ in Theorem~\ref{thm:local.existence} is $C^\infty$ when $t>0$. Moreover, if $f_{\mathrm{in}}$ is not identically zero, then $f(t,x,v)>0$ when $t>0$.
\end{theorem}

In what follows, we will need a slight variant of Theorem~\ref{thm:local.existence}. It can be proven in a very similar manner as Theorem~\ref{thm:local.existence} in \cite{HeSnTa17}, we therefore state it as a corollary and omit the proof. Note that the assumptions in Corollary~\ref{cor:local.ext} are weaker than those in Theorem~\ref{thm:local.existence} (i.e.~\eqref{cor:local.ext.assumption} implies \eqref{local.existence.assumption}). Therefore, by Theorem~\ref{thm:local.existence}, a unique local solution indeed exists under the assumptions of Corollary~\ref{cor:local.ext} (and moreover uniqueness holds even in the weaker space in Theorem~\ref{est.local.existence}). The point, however, is that Corollary~\ref{cor:local.ext} also gives a stronger estimate which will be useful.
\begin{corollary}\label{cor:local.ext}
Fix $\rho_0,\,M_0,\,\lambda\in \mathbb R$ with $\rho_0>0$, $M_0>0$ (and $\lambda$ arbitrary) and $k,\,N\in \mathbb N$ with $k\geq 4$. Suppose $e^{\rho_0\vb^2}f_{\mathrm{in}}$ satisfies the estimate
\begin{equation}\label{cor:local.ext.assumption}
\sum_{|\alp|+|\bt|\leq k}\|\langle x - \lambda v\rangle^N \rd_x^\alp\rd_v^\bt(e^{\rho_0\vb^2}f_{\mathrm{in}})\|_{L^2_x L^2_v} \leq M_0.
\end{equation}
Then for any $\kappa>0$, there exists $T=T_{\gamma,\rho_0,M_0,\lambda,N,\kappa}>0$ depending only on $\gamma$, $\rho_0$, $M_0$, $\lambda$, $N$ and $\kappa$, such that there exists a unique solution $f\geq 0$ to the Landau equation \eqref{Landau} with initial data $f(0,x,v) = f_{\mathrm{in}}(x,v)$ and satisfying
\begin{equation}\label{eq:local.ext}
\begin{split}
&\:\sum_{|\alp|+|\bt|\leq k}(\|\langle x - (\lambda+t) v\rangle^N \rd_x^\alp\rd_v^\bt(e^{(\rho_0-\kappa t)\vb^2}f)\|_{C^0([0,T];L^2_x L^2_v)} \\
&\:\qquad \qquad +\|\langle x - (\lambda+t) v\rangle^N \vb \rd_x^\alp\rd_v^\bt(e^{(\rho_0-\kappa t)\vb^2}f)\|_{L^2([0,T];L^2_x L^2_v)}) <+\infty.
\end{split}
\end{equation}
Moreover, given fixed $\gamma$, $\rho_0$, $M_0$, $N$ and $\kappa$, for any compact interval $K\subset \mathbb R$, $T=T_{\gamma,\rho_0,M_0,\lambda,N,\kappa}>0$ can be chosen uniformly for all $\lambda \in K$.
\end{corollary}

\begin{proof}[Comments on the proof]
The proof is essentially the same as in \cite{HeSnTa17}. Note that the norms in this corollary are different from those in Theorem~\ref{thm:local.existence} in two places: first, it involves a usual Sobolev space instead of a uniformly local one; second, there is an additional weight of $\langle x - (\lambda+t) v\rangle^N$.

The first difference in fact makes the proof easier, as one no longer needs to keep track of the cut-off functions $\phi$ (see Definition~\ref{def:ul}). The second difference only affects the proof minimally. This is because 
\begin{equation}\label{simple.weights.things}
(\rd_t + v_i \rd_{x_i}) \langle x - (\lambda+t) v\rangle = 0,\quad |\rd_{v_j} \langle x - (\lambda+t) v\rangle| \ls_{\lambda} 1.
\end{equation}
One can therefore prove weighted $L^2$ estimates with $\langle x - (\lambda+t) v\rangle^N$ weights and \eqref{simple.weights.things} guarantees that all extra terms arising from integrating by parts in the $L^2$ estimate can be easily controlled.

Finally, for $\lambda \in K$ and $K\subset \mathbb R$ a compact interval, $T$ can be chosen to depend only on $K$ but not the specific value of $\lambda$. This is simply because the constant in \eqref{simple.weights.things} can be chosen uniformly for all $\lambda \in K$.
\end{proof}

\section{Bootstrap assumptions and the bootstrap theorem}\label{sec:bootstrap}

We now begin the proof of Theorem~\ref{thm:main}. We will argue using a bootstrap argument. After introducing some preliminaries in \textbf{Section~\ref{sec:bootstrap.prelim}}, we will state our bootstrap assumptions and our main bootstrap theorem (\textbf{Theorem~\ref{thm:BA}}) in \textbf{Section~\ref{sec:BA}}.

\subsection{Preliminaries}\label{sec:bootstrap.prelim}

Let $\de$ be a small positive number (depending on $\gamma$) defined by
\begin{equation}\label{def:de}
\de := \min\{\f{2+\gamma}{4}, \f 1{10}\}.
\end{equation}

Instead of directly controlling $f$, define 
\begin{equation}\label{def:g}
g := e^{d(t) \vb^2} f,
\end{equation}
where 
\begin{equation}\label{def:d}
d(t) := d_0 (1+ (1+t)^{-\de}),
\end{equation} 
$d_0>0$ is the constant in the statement of Theorem~\ref{thm:main}, and $\de>0$ is the constant fixed above satisfying \eqref{def:de}. We will estimate $g$ instead of $f$. 

The function $g$ then satisfies the following equation:
\begin{equation}\label{eq:g}
\begin{split}
\rd_t g +v_i\rd_{x_i} g +\f{\de d_0}{(1+t)^{1+\de}} \vb^2 g- \ab_{ij}\rd^2_{v_i v_j} g= &\: -\cb_i g - 4 d(t) \bar{a}_{ij} v_i \rd_{v_j} g - 2d(t)(\de_{ij}-2d(t)v_iv_j)\bar{a}_{ij}g.
\end{split}
\end{equation}

Define the following shorthand
\begin{equation}\label{def:Y}
Y_i = t\rd_{x_i} + \rd_{v_i}.
\end{equation}

Introduce the following energies for $k=0,1,\dots, \Mm$ and for $T\geq 0$.
\begin{equation}\label{eq:energy.def}
\begin{split}
E_k(T):= &\: \sum_{|\alp|+|\bt|+|\sigma|= k}(1+T)^{-|\bt|}\|\wb^{\Mm+5-|\sigma|}(\rd_x^\alp \rd_v^\bt Y^\sigma g)\|_{L^\i([0,T];L^2_x L^2_v)}\\
&\: + \sum_{|\alp|+|\bt|+|\sigma|= k}(1+T)^{-|\bt|}\|(1+t)^{-\f 12-\f\de 2}\vb\wb^{\Mm+5-|\sigma|}(\rd_x^\alp \rd_v^\bt Y^{\sigma} g)\|_{L^2([0,T];L^2_x L^2_v)}.
\end{split}
\end{equation}
We note explicit the following features of the energies:
\begin{itemize}
\item $\wb$ weight depends on the number of $Y=t\rd_x+\rd_v$ derivatives: the more $Y$ derivatives we take, the weaker weight we have.
\item For every $\rd_v$ derivative we take, we give up a power of $(1+t)$.
\end{itemize}

\subsection{The bootstrap assumptions}\label{sec:BA}

Introduce the bootstrap assumption for the $E$ norms
\begin{equation}\label{BA}
E(T):= \sum_{k=0}^{\Mm} E_k(T) \leq \ep^{\f 34},
\end{equation}
where $\Mm = \begin{cases}
2+2\lceil \f{2}{2+\gamma} + 4\rceil & \mbox{ if $\gamma \in (-2,-1]$} \\
2+2\lceil \f{1}{|\gamma|}+ 4\rceil & \mbox{ if $\gamma \in (-1,0)$}
\end{cases}$ as in Theorem~\ref{thm:main}.

Introduce also the following bootstrap assumptions for the  $L^\i_xL^\i_v$ norms of derivatives of $g$:

When
$|\alp|+|\bt|+|\sigma| \leq \Mm-4-\max\{2,\lceil \f{2}{2+\gamma} \rceil \}$, 
\begin{equation}\label{BA.sim.1}
\|\wb^{\Mm+5-|\sigma|} \rd_x^{\alp} \rd_v^{\bt} Y^\sigma g \|_{L^\i_x L^\i_v}(T) \leq \ep^{\f 34}(1+T)^{|\bt|}.
\end{equation}

When $\Mm-3-\max\{2,\lceil \f{2}{2+\gamma} \rceil \}\leq |\alp|+|\bt|+|\sigma| =:k \leq \Mm-5$, then
\begin{equation}\label{BA.sim.2}
\|\wb^{\Mm+5-|\sigma|} \rd_x^{\alp} \rd_v^{\bt} Y^\sigma g \|_{L^\i_x L^\i_v}(T) \leq \ep^{\f 34}(1+T)^{\f 32 - (\Mm-4-k)\min\{\f 34, \f{3(2+\gamma)}{4}\} +|\bt|}.
\end{equation}

Our goal from now on until Section~\ref{sec:EE} will be to improve these bootstrap assumptions \eqref{BA}, \eqref{BA.sim.1} and \eqref{BA.sim.2} with $\ep^{\f34}$ replaced by $C\ep$ (which is indeed an improvement for $\ep$ sufficiently small) for some constant $C$ depending only on $d_0$ and $\gamma$. We formulate this as a theorem below. 

\begin{theorem}[Bootstrap theorem]\label{thm:BA}
Let $\gamma$, $d_0$ and $f_{\mathrm{in}}$ be as in Theorem~\ref{thm:main} and let $\de>0$ be as in \eqref{def:de}. There exist $\ep_0 = \ep_0(d_0,\gamma)>0$ and $C_0 = C_0(d_0,\gamma)>0$ with $C_0 \ep_0\leq \f 12 \ep_0^{\f 34}$ such that the following holds:

Suppose there exists $T_{Boot}>0$ and a solution $f:[0,T_{Boot})\times \mathbb R^3\times \mathbb R^3$ with $f(t,x,v)\geq 0$, $f$ smooth for $t>0$ and $f(0,x,v)=f_{\mathrm{in}}(x,v)$. Moreover, suppose that the estimates \eqref{BA}, \eqref{BA.sim.1} and\eqref{BA.sim.2} all hold 
for all $T \in [0,T_{Boot})$, then all of these estimate in fact hold for all $T \in [0,T_{Boot})$ with $\ep^{\f 34}$ replaced by $C\ep$.
\end{theorem}

From now on until Section~\ref{sec:EE}, we will prove Theorem~\ref{thm:BA} (see~Section~\ref{sec:end.of.bootstrap}). \textbf{In these section, we therefore always work under the assumptions of Theorem~\ref{thm:BA}.} To simplify the notations, from now on, unless explicitly stated otherwise, for two non-negative quantities $A$ and $B$, \textbf{$A \ls B$ means that there exists $C>0$ depending $d_0$ and $\gamma$ (and in particular independent of $\epsilon$) such that $A\leq CB$.}

\section{Estimates for the coefficients}\label{sec:coeff}

We work under the assumptions of Theorem~\ref{thm:BA}.

In this section, we prove $L^\infty_x$ and $L^2_x$ type bounds for the coefficients $\bar{a}_{ij}$, $\bar{c}$ and their derivatives. The $L^\i_x$ estimates will be proven in \textbf{Section~\ref{sec:coeff.Li}} while the $L^2_x$ estimates will be proven in \textbf{Section~\ref{sec:coeff.L2}}.

\subsection{The $L^\infty_x$ estimates for $\bar{a}_{ij}$, $\bar{c}$ and their derivatives}\label{sec:coeff.Li}

\subsubsection{Preliminary embedding estimates}\label{sec:prelim.Li}

We begin with a simple interpolation estimate.

\begin{lemma}\label{lem:Li.1}
Let $\nu \in (0,3)$ and $h:\mathbb R^3_v\to \mathbb R$ be an $L^1\cap L^\i$ function. Then the following estimate holds:
$$\sup_{v\in \mathbb R^3} \int_{\mathbb R^3} |v-v_*|^{-\nu} |h|(v_*) \, \ud v_* \ls \|h\|_{L^1_v}^{1-\f \nu 3} \|h\|_{L^\i_v}^{\f \nu 3}.$$
\end{lemma}

\begin{proof}
Assume that $h\not\equiv 0$ for otherwise the estimate is trivial.

Let $\lambda>0$ be a constant to be determined. We divide the integral into regions $|v-v_*|\leq \lambda$ and $|v-v_*|> \lambda$ and use H\"older's inequality in each of the regions to obtain
\begin{equation*}
\begin{split}
&\: \int_{\mathbb R^3} |v-v_*|^{-\nu} |h|(v_*) \, \ud v_*\\
\ls &\: \int_{\{|v-v_*|\leq \lambda\}} |v-v_*|^{-\nu} |h|(v_*) \, \ud v_* + \int_{\{|v-v_*|> \lambda\}} |v-v_*|^{-\nu} |h|(v_*) \, \ud v_* \\
\ls &\: \lambda^{-\nu+3} \|h\|_{L^\i_v} + \lambda^{-\nu} \|h\|_{L^1_v}.
\end{split}
\end{equation*}
Let $\lambda = \|h\|_{L^1_v}^{\f 13} \|h\|_{L^\i_v}^{-\f 13}$. Then 
$$\sup_{v\in \mathbb R^3} \int_{\mathbb R^3} |v-v_*|^{-\nu} |h|(v_*) \, \ud v_* \ls \|h\|_{L^1_v}^{1-\f \nu 3} \|h\|_{L^\i_v}^{\f \nu 3},$$
as claimed. \qedhere
\end{proof}

\begin{lemma}\label{lem:Li.2}
Let $h:[0,T_{Boot})\times \mathbb R^3 \times \mathbb R_v \to \mathbb R$ be a smooth function such that $\vb^4 h,\,\wb^4 h \in L^\i_xL^\i_v$ for all $t\in [0,T_{Boot})$. Then for all $t\in [0,T_{Boot})$,
$$\|h\|_{L^\i_x L^1_v}(t) \ls (1+t)^{-3}\left(\|\vb^4 h\|_{L^\i_x L^\i_v}(t) + \|\wb^4 h\|_{L^\i_x L^\i_v}(t) \right).$$
\end{lemma}
\begin{proof}

\pfstep{Step~1: The case $t\leq 1$} In this case we simply use the H\"older's inequality to estimate as follows:
\begin{equation*}
\begin{split}
\|h\|_{L^1_v}(t,x) \ls &\: \left(\int_{\mathbb R^3}  \vb^{-4}\,\ud v \right)\left(\sup_{(x,v)\in \mathbb R^3} \vb^4|h|(t,x,v)\right) \ls \|\vb^4 h\|_{L^\i_x L^\i_v}(t).
\end{split}
\end{equation*}
Taking supremum over all $x \in \mathbb R^3$ yields the desired estimate.

\pfstep{Step~2: The case $t> 1$} Let $\lambda>0$ be a constant to be chosen. We divide the region of integration according to $|v-\f{x}{t}|\leq \lambda$ and $|v-\f{x}{t}|> \lambda$.
\begin{equation*}
\begin{split}
\|h\|_{L^1_v}(t,x) \ls &\: \int_{\{|v-\f{x}{t}|\leq \lambda\}}  |h|(t,x,v)\,\ud v + \int_{\{|v-\f{x}{t}|> \lambda\}} |h|(t,x,v) \, \ud v \\
\ls &\: \lambda^3 \|h\|_{L^\i_x L^\i_v}(t) + t^{-4} \||x-tv|^4 h\|_{L^\i_x L^\i_v}(t) \int_{\{|v-\f{x}{t}|> \lambda\}} |v-\f{x}{t}|^{-4} \, \ud v \\
\ls &\: \|\wb^4 h\|_{L^\i_x L^\i_v}(t) \left(\lambda^3 + \lambda^{-1} t^{-4}\right).
\end{split}
\end{equation*}
Let $\lambda = t^{-1}$ and taking the supremum over all $x\in \mathbb R^3$, we obtain the desired estimate.
\end{proof}

\begin{lemma}\label{lem:Li.3}
Let $\nu \in (0,3)$ and $h:[0,T_{Boot})\times \mathbb R^3 \times \mathbb R^3 \to \mathbb R$ be a smooth function such that $\vb^4 h,\,\wb^4 h \in L^\i_xL^\i_v$ for all $t\in [0,T_{Boot})$. Then for all $t\in [0,T_{Boot})$,
$$\int_{\mathbb R^3} |v-v_*|^{-\nu} |h|(t,x,v_*) \, \ud v_* \ls (1+t)^{-3+\nu}\left(\|\vb^4 h\|_{L^\i_x L^\i_v}(t) + \|\wb^4 h\|_{L^\i_x L^\i_v}(t) \right).$$
\end{lemma}
\begin{proof}
This follows from combining Lemmas~\ref{lem:Li.1} and \ref{lem:Li.2}. \qedhere
\end{proof}

\subsubsection{Estimates for weighted $v$-integrals of $f$}\label{sec:Li.v.weighted}

We now use the preliminary estimates derived in Section~\ref{sec:prelim.Li} to bound general weighted $v$-integrals of $f$; see Lemmas~\ref{lem:Li.5} and \ref{lem:Li.6} below. 

\begin{lemma}\label{lem:Li.4}
For every $\ell \in \mathbb N$ and $m\in \mathbb N\cup \{0\}$, the following estimate holds with an implicit constant depending on $\ell$, $\gamma$ and $d_0$ for any $(t,x,v) \in [0,T_{Boot})\times \mathbb R^3\times \mathbb R^3$:
$$\vb^\ell \wb^m |\rd_x^\alp \rd_v^\bt Y^\sigma f|(t,x,v) \ls \sum_{|\bt'|\leq |\bt|,\,|\sigma'|\leq |\sigma|} \wb^m |\rd_x^\alp \rd_v^{\bt'} Y^{\sigma'} g|(t,x,v).$$
\end{lemma}
\begin{proof}
This follows immediately from differentiating \eqref{def:g} and using $\vb^{\underline{\ell}} e^{-d(t)\vb^2}\ls_{\underline{\ell}} 1$ for all $\underline{\ell} \in \mathbb N$. \qedhere
\end{proof}

\begin{lemma}\label{lem:Li.5}
If $|\alp|+|\bt|+|\sigma| \leq \Mm-4-\max\{2,\lceil \f{2}{2+\gamma} \rceil \}$, then
\begin{equation*}
\begin{split}
\|\vb^4 \wb  \rd_x^\alp \rd_v^\bt Y^\sigma f\|_{L^\i_x L^1_v}(t)  \ls \ep^{\f 34} (1+t)^{-3+|\bt|}.
\end{split}
\end{equation*}
If $\Mm-3-\max\{2,\lceil \f{2}{2+\gamma} \rceil \} \leq |\alp|+|\bt|+|\sigma|=:k \leq \Mm-5$, then
$$\|\vb^4 \wb \rd_x^\alp \rd_v^\bt Y^\sigma f\|_{L^\i_x L^1_v}(t) \ls \ep^{\f 34} (1+t)^{-\f 32+|\bt|-(\Mm-4-k)\min\{\f 34,\f{3(2+\gamma)}{4}\}}.$$
\end{lemma}
\begin{proof}
By Lemmas~\ref{lem:Li.2} and \ref{lem:Li.4}, we obtain 
\begin{equation*}
\begin{split}
&\: \|\vb^4 \wb \rd_x^\alp \rd_v^\bt Y^\sigma f\|_{L^\i_x L^1_v}(t) \\
\ls &\: (1+t)^{-3}(\|\vb^8 \wb \rd_x^\alp \rd_v^\bt Y^\sigma f\|_{L^\i_x L^\i_v}(t) + \|\vb^4 \wb^5 \rd_x^\alp \rd_v^\bt Y^\sigma f\|_{L^\i_x L^\i_v}(t)) \\
\ls &\: (1+t)^{-3} (\sum_{|\bt'|\leq |\bt|,\,|\sigma'|\leq |\sigma|} \|\wb^5 \rd_x^{\alp} \rd_v^{\bt'} Y^{\sigma'} g \|_{L^\i_x L^\i_v} ).
\end{split}
\end{equation*}
The desired conclusion then follows from \eqref{BA.sim.1} and \eqref{BA.sim.2}. \qedhere
\end{proof}

\begin{lemma}\label{lem:Li.6}
Let $\nu \in (0,3)$. If $|\alp|+|\bt|+|\sigma| \leq \Mm-4-\max\{2,\lceil \f{2}{2+\gamma} \rceil \}$, then
$$\|\int_{\mathbb R^3} |v-v_*|^{-\nu} |\rd_x^\alp \rd_v^\bt Y^\sigma f|(t,x,v_*) \,\ud v_*\|_{L^\i_x L^\i_v}(t) \ls \ep^{\f 34}(1+t)^{-3+\nu+|\bt|}.$$
If $\Mm-3-\max\{2,\lceil \f{2}{2+\gamma} \rceil \} \leq |\alp|+|\bt|+|\sigma|=:k \leq \Mm-5$, then
$$\|\int_{\mathbb R^3} |v-v_*|^{-\nu} |\rd_x^\alp \rd_v^\bt Y^\sigma f|(t,x,v_*) \,\ud v_*\|_{L^\i_x L^\i_v}(t) \ls \ep^{\f 34} (1+t)^{-\f 32+\nu+|\bt|-(\Mm-4-k)\min\{\f 34,\f{3(2+\gamma)}{4}\}}.$$
\end{lemma}
\begin{proof}
By Lemmas~\ref{lem:Li.3} and \ref{lem:Li.4},
\begin{equation*}
\begin{split}
&\: \|\int_{\mathbb R^3} |v-v_*|^{-\nu} |\rd_x^\alp \rd_v^\bt Y^\sigma f|(t,x,v_*) \,\ud v_*\|_{L^\i_x L^\i_v}(t)\\
\ls &\: (1+t)^{-3+\nu} (\|\vb^4 \rd_x^\alp \rd_v^\bt Y^\sigma f\|_{L^\i_x L^\i_v}(t) + \|\wb^4 \rd_x^\alp \rd_v^\bt Y^\sigma f\|_{L^\i_x L^\i_v}(t)) \\
\ls &\: (1+t)^{-3+\nu}(\sum_{|\bt'|\leq |\bt|,\,|\sigma'|\leq |\sigma|} \|\wb^4 \rd_x^\alp \rd_v^{\bt'} Y^{\sigma'} g\|_{L^\i_x L^\i_v}(t)).
\end{split}
\end{equation*}
The desired conclusion then follows from \eqref{BA.sim.1} and \eqref{BA.sim.2}. \qedhere
\end{proof}

\subsubsection{$L^\i_x L^\i_v$ estimates for $\bar{a}_{ij}$ and its derivatives}

\begin{proposition}\label{prop:a.expressions}
In the following, suppose $|\alp|+|\bt|+|\sigma|\leq \Mm$.

The coefficient $\bar{a}_{ij}$ and its higher derivatives satisfy the following pointwise bounds:
\begin{equation}\label{a.0.1}
\max_{i,j}|\rd_x^\alp \rd_v^\bt Y^\sigma \bar{a}_{ij}|(t,x,v) \ls \int_{\mathbb R^3} |v-v_*|^{2+\gamma} |\rd_x^\alp \rd_v^\bt Y^\sigma f|(t,x,v_*)\,\ud v_*,
\end{equation}
\begin{equation}\label{a.0.2}
\max_j |\rd_x^\alp \rd_v^\bt Y^\sigma (\bar{a}_{ij}v_i)|(t,x,v) \ls \vb^{\max\{2+\gamma,1\}} \int_{\mathbb R^3} \langle v_*\rangle^{\max\{2+\gamma,1\}} |\rd_x^\alp \rd_v^\bt Y^\sigma f|(t,x,v_*)\,\ud v_*,
\end{equation}
\begin{equation}\label{a.0.3}
|\rd_x^\alp \rd_v^\bt Y^\sigma (\bar{a}_{ij}v_i v_j)|(t,x,v) \ls \int_{\mathbb R^3} (|v|^{2+\gamma}|v_*|^2 + |v_*|^{4+\gamma}) |\rd_x^\alp \rd_v^\bt Y^\sigma f|(t,x,v_*)\,\ud v_*.
\end{equation}
The first $v$-derivatives of $\bar{a}_{ij}$ and their higher derivatives satisfy the following pointwise bounds:
\begin{equation}\label{a.1.1}
\max_{i,j,\ell}|\rd_x^\alp \rd_v^\bt Y^\sigma \rd_{v_\ell} \bar{a}_{ij}|(t,x,v) \ls \int_{\mathbb R^3} |v-v_*|^{1+\gamma} |\rd_x^\alp \rd_v^\bt Y^\sigma f|(t,x,v_*)\,\ud v_*,
\end{equation}
\begin{equation}\label{a.1.2}
\max_{i,j,\ell} |\rd_x^\alp \rd_v^\bt Y^\sigma \rd_{v_\ell} (\bar{a}_{ij}v_i)|(t,x,v) \ls \int_{\mathbb R^3} (|v||v-v_*|^{1+\gamma}+|v-v_*|^{2+\gamma}) |\rd_x^\alp \rd_v^\bt Y^\sigma f|(t,x,v_*)\,\ud v_*,
\end{equation}
Finally, the second $v$-derivatives of $\bar{a}_{ij}$ and their higher derivatives satisfy the following pointwise bounds:
\begin{equation}\label{a.2.1}
\max_{i,j,\ell,m}|\rd_x^\alp \rd_v^\bt Y^\sigma \rd^2_{v_\ell v_m} \bar{a}_{ij}|(t,x,v) \ls \int_{\mathbb R^3} |v-v_*|^{\gamma} |\rd_x^\alp \rd_v^\bt Y^\sigma f|(t,x,v_*)\,\ud v_*
\end{equation}
\end{proposition}

\begin{proof}
\pfstep{Step~0: Preliminaries}
We will repeatedly use the following easily verified facts when\footnote{Here, $|\bt'|\leq 2$ is to ensure that $\rd_v^{\bt'} a_{ij}$ is in $L^1_{loc,v_*}$ so that the computation can be justified.} $|\bt'|\leq 2$:
$$[\rd_x^\alp \rd_v^\bt Y^\sigma \rd_v^{\bt'}] \bar{a}_{ij} = \int_{\mathbb R^3} (\rd_v^{\bt'} a_{ij}(v-v_*)) (\rd_x^\alp \rd_v^\bt Y^\sigma f)(t,x,v_*)\,\ud v_*,$$
$$[\rd_x^\alp \rd_v^\bt Y^\sigma \rd_v^{\bt'}] (\bar{a}_{ij}v_i) = \int_{\mathbb R^3} (\rd_v^{\bt'} (a_{ij}(v-v_*)v_i)) (\rd_x^\alp \rd_v^\bt Y^\sigma f)(t,x,v_*)\, \ud v_*,$$
$$[\rd_x^\alp \rd_v^\bt Y^\sigma \rd_v^{\bt'}] (\bar{a}_{ij}v_iv_v) =  \int_{\mathbb R^3} (\rd_v^{\bt'} (a_{ij}(v-v_*)v_iv_j)) (\rd_x^\alp \rd_v^\bt Y^\sigma f)(t,x,v_*)\, \ud v_*,$$

As a result, the proof of the proposition essentially boils down to checking the derivatives of the kernel. This is what we will check below. In other words, when we say ``Proof of \eqref{a.0.1}'', we mean that we will estimate the kernel so that when plugging in the above, we obtain \eqref{a.0.1}.
 
\pfstep{Step~1: Proof of \eqref{a.0.1}} 
To obtain \eqref{a.0.1}, we only need an estimate
$$|a_{ij}(v-v_*)| \leq |v-v_*|^{2+\gamma},$$
which is obvious by \eqref{a.def}.

\pfstep{Step~2: Proof of \eqref{a.0.2}} For \eqref{a.0.2}, we start with \eqref{a.def} and compute 
\begin{equation}\label{a.0.2.main.step}
\begin{split}
a_{ij}(v-v_*)v_i =&\: |v-v_*|^{2+\gamma} \left(v_j - \f{(v\cdot (v-v_*))(v-v_*)_j}{|v-v_*|^2} \right) \\
=&\: |v-v_*|^{\gamma} \left(v_j |v|^2- 2v_j (v\cdot v_*) + v_j|v_*|^2 - (|v|^2 - (v\cdot v_*))(v-v_*)_j \right) \\
=&\: |v-v_*|^{\gamma} \left(- v_j (v\cdot v_*) + v_j|v_*|^2 \blue{+} (|v|^2 - (v\cdot v_*))(v_*)_j \right) \\
=&\: |v-v_*|^{\gamma} \left(-v_j (v_* \cdot (v-v_*)) \blue{+} (v\cdot(v-v_*))(v_*)_j \right) .
\end{split}
\end{equation}

We now split into various cases. First, suppose $1+\gamma\geq 0$. Then \eqref{a.0.2.main.step} and the triangle inequality implies that
$$\sup_j |a_{ij}(v-v_*)v_i| \ls |v-v_*|^{1+\gamma} |v| |v_*| \ls |v|^{2+\gamma} |v_*| + |v| |v_*|^{2+\gamma}.$$

If $1+\gamma<0$, we further split into two cases. If $|v-v_*|\leq 1$, then a trivial estimate using \eqref{a.def} implies
$$\sup_j |a_{ij}(v-v_*)v_i| \ls |v-v_*|^{2+\gamma} |v| \ls |v|.$$
If $1+\gamma<0$ and $|v-v_*|>1$, then by \eqref{a.0.2.main.step}, we obtain
$$\sup_j |a_{ij}(v-v_*)v_i| \ls |v-v_*|^{1+\gamma} |v| |v_*| \ls |v||v_*|.$$

\pfstep{Step~3: Proof of \eqref{a.0.3}} For \eqref{a.0.3}, we compute
\begin{equation}\label{eq:avv}
\begin{split}
a_{ij}(v-v_*)v_i v_j =&\: |v-v_*|^{2+\gamma} \left(|v|^2 - \f{(v\cdot (v-v_*))^2}{|v-v_*|^2} \right) \\
=&\: |v-v_*|^{\gamma} \left(|v|^2 (|v|^2 + |v_*|^2 -2(v\cdot v_*)) - |v|^4 + 2|v|^2(v\cdot v_*)- (v\cdot v_*)^2 \right)\\
=&\: |v-v_*|^{\gamma} \left(|v|^2 |v_*|^2 - (v\cdot v_*)^2 \right).
\end{split}
\end{equation}
Using the Pythagorean theorem, we obtain the following estimate:
\begin{equation}\label{Pythagorean}
\begin{split}
|v|^2|v_*|^2 - (v\cdot v_*)^2 = &\: \f{1}{|v|^2} \left||v|^2 v_* -(v\cdot v_*) v\right|^2 = \f{1}{|v|^2} \left|(v\cdot (v-v_*)) v_* -(v\cdot v_*)(v-v_*)\right|^2 \leq 2|v-v_*|^2|v_*|^2.
\end{split}
\end{equation}
Putting \eqref{eq:avv} and \eqref{Pythagorean} together, and noting $2+\gamma>0$, we thus obtain
$$|a_{ij}(v-v_*)v_i v_j| \ls |v-v_*|^{2+\gamma} |v_*|^2 \ls |v|^{2+\gamma}|v_*|^2 + |v_*|^{4+\gamma}.$$

\pfstep{Step~4: Proof of \eqref{a.1.1}} By homogeneity of $\bar{a}_{ij}$, it is easy to see that
\begin{equation}\label{easy.dva}
|\rd_{v_k}\bar{a}_{ij}(v-v_*)|\ls |v-v_*|^{1+\gamma},
\end{equation}
 which implies \eqref{a.1.1}.

\pfstep{Step~5: Proof of \eqref{a.1.2}} Arguing again by the homogeneity of $\bar{a}_{ij}$, it follows that $|\rd_{v_k}[\bar{a}_{ij}(v-v_*) v_i]|\ls |v-v_*|^{1+\gamma}|v|+|v-v_*|^{2+\gamma}$, which then implies \eqref{a.1.2}.

\pfstep{Step~6: Proof of \eqref{a.2.1}} Finally, for the second derivatives of $a_{ij}$, we use homogeneity to obtain
$$|\rd_{v_\ell} \rd_{v_k} a_{ij}(v-v_*)| \ls |v-v_*|^\gamma,$$
which implies \eqref{a.2.1}. \qedhere

\end{proof}

Using Proposition~\ref{prop:a.expressions}, as well as estimates in Section~\ref{sec:Li.v.weighted}, we derive estimates for $\bar{a}_{ij}$ and its derivatives in the next few propositions. Our first proposition is the most general, but as we will see, we will need various refinements later to close our bootstrap argument.

\begin{proposition}\label{prop:ab.Li.1}
If $|\alp|+|\bt|+|\sigma| \leq \Mm-4-\max\{2,\lceil \f{2}{2+\gamma} \rceil \}$, then for $(t,x,v)\in [0,T_{Boot})\times \mathbb R^3\times \mathbb R^3$,
$$|\rd_x^\alp \rd_v^\bt Y^\sigma \bar{a}_{ij}|(t,x,v) \ls \ep^{\f 34} \vb^{2+\gamma}(1+t)^{-3+|\bt|}.$$
If $\Mm-3-\max\{2,\lceil \f{2}{2+\gamma} \rceil \} \leq |\alp|+|\bt|+|\sigma| =:k \leq \Mm-5$, then for $(t,x,v)\in [0,T_{Boot})\times \mathbb R^3\times \mathbb R^3$,
$$|\rd_x^\alp \rd_v^\bt Y^\sigma \bar{a}_{ij}|(t,x,v) \ls \ep^{\f 34} \vb^{2+\gamma} (1+t)^{-\f 32+|\bt|-(\Mm-4-k)\min\{\f 34,\f{3(2+\gamma)}{4}\}}.$$
\end{proposition}
\begin{proof}
This follows from \eqref{a.0.1} in Proposition~\ref{prop:a.expressions}, the bound $|v-v_*|^{2+\gamma}\ls |v|^{2+\gamma}+|v_*|^{2+\gamma}$ (since $2+\gamma>0$), and Lemma~\ref{lem:Li.5}. \qedhere
\end{proof}

The next proposition is a variant of Proposition~\ref{prop:ab.Li.1}: it gives an improved $t$-decay rate under the assumption $|\bt|\geq 1$.

\begin{proposition}\label{prop:ab.Li.2}
If $|\alp|+|\bt|+|\sigma| \leq \Mm-4-\max\{2,\lceil \f{2}{2+\gamma} \rceil \}$ \underline{and $|\bt|\geq 1$}, then for $(t,x,v)\in [0,T_{Boot})\times \mathbb R^3\times \mathbb R^3$,
$$|\rd_x^\alp \rd_v^\bt Y^\sigma \bar{a}_{ij}|(t,x,v) \ls \ep^{\f 34} \vb^{2+\gamma}(1+t)^{-\min\{4,5+\gamma\}+|\bt|}.$$
If $\Mm-3-\max\{2,\lceil \f{2}{2+\gamma} \rceil \} \leq |\alp|+|\bt|+|\sigma|=:k \leq \Mm-5$ \underline{and $|\bt|\geq 1$}, then for $(t,x,v)\in [0,T_{Boot})\times \mathbb R^3\times \mathbb R^3$,
$$|\rd_x^\alp \rd_v^\bt Y^\sigma \bar{a}_{ij}|(t,x,v) \ls \ep^{\f 34} \vb^{2+\gamma}(1+t)^{-\min\{\f 52, \f 72+\gamma\}+|\bt|-(\Mm-4-k)\min\{\f 34,\f{3(2+\gamma)}{4}\}}.$$

\end{proposition}
\begin{proof}
Assume throughout the proof that $|\bt|\geq 1$. We start with \eqref{a.1.1} in Proposition~\ref{prop:a.expressions} and consider separately $1+\gamma\geq 0$ and $1+\gamma< 0$.

Suppose $1+\gamma \geq 0$. Then we have
$$|\rd_x^\alp \rd_v^\bt Y^\sigma \bar{a}_{ij}|(t,x,v) \ls \sum_{|\bt'|\leq |\bt|-1} \int_{\mathbb R^3} (|v|^{1+\gamma}+|v_*|^{1+\gamma}) |\rd_x\rd_v^{\bt'} Y^\sigma f|(t,x,v_*)\,\ud v_*.$$
The desired estimate then follows from Lemma~\ref{lem:Li.5}.

Suppose $1+\gamma < 0$. Then we have
$$|\rd_x^\alp \rd_v^\bt Y^\sigma \bar{a}_{ij}|(t,x,v) \ls \sum_{|\bt'|\leq |\bt|-1} \int_{\mathbb R^3} |v-v_*|^{1+\gamma} |\rd_x\rd_v^{\bt'} Y^\sigma f|(t,x,v_*)\,\ud v_*.$$
The desired estimate then follows from Lemma~\ref{lem:Li.6} with $\nu = -(1+\gamma)$.

(We note explicitly that indeed an estimate with an even better $\vb$ weight still holds, but we will be content with the stated weaker estimate since this allows for an easier comparison with the estimates in Proposition~\ref{prop:ab.Li.weighted}, which will in turn allow us to handle our estimates more systematically later.) \qedhere
\end{proof}

The next proposition is another variant of Proposition~\ref{prop:ab.Li.1} which gives an improved decay rate under the assumption $|\alp|\geq 1$. (Note that the estimate is very weak as $t\to 0$.)
\begin{proposition}\label{prop:ab.Li.3}
If $|\alp|+|\bt|+|\sigma| \leq \Mm-4-\max\{2,\lceil \f{2}{2+\gamma} \rceil \}$ \underline{and $|\alp|\geq 1$}, then for $(t,x,v)\in [0,T_{Boot})\times \mathbb R^3\times \mathbb R^3$,
$$|\rd_x^\alp \rd_v^\bt Y^\sigma \bar{a}_{ij}|(t,x,v) \ls \ep^{\f 34} \vb^{2+\gamma}t^{-1}(1+t)^{-\min\{3,4+\gamma\}+|\bt|}.$$
If $\Mm-3-\max\{2,\lceil \f{2}{2+\gamma} \rceil \} \leq |\alp|+|\bt|+|\sigma|=:k \leq \Mm-5$ \underline{and $|\alp|\geq 1$}, then for $(t,x,v)\in [0,T_{Boot})\times \mathbb R^3\times \mathbb R^3$,
$$|\rd_x^\alp \rd_v^\bt Y^\sigma \bar{a}_{ij}|(t,x,v) \ls \ep^{\f 34} \vb^{2+\gamma}t^{-1}(1+t)^{-\min\{\f 32, \f 52+\gamma\}+|\bt|-(\Mm-4-k)\min\{\f 34,\f{3(2+\gamma)}{4}\}}.$$
\end{proposition}
\begin{proof}
We rely on the following simple pointwise bound, which is obtained by writing $\rd_x = t^{-1} (t\rd_x + \rd_v) - t^{-1}\rd_v = t^{-1} Y - t^{-1}\rd_v$:
\begin{equation*}
\begin{split}
&\: |\rd_x^\alp \rd_v^\bt Y^\sigma \bar{a}_{ij}(t,x,v)|  \\
\leq &\: t^{-1} \sum_{\substack{|\alp'|= |\alp|-1\\ |\sigma'|= |\sigma|+1}} |\rd_x^{\alp'} \rd_v^{\bt} Y^{\sigma'} \bar{a}_{ij}(t,x,v)| + t^{-1} \sum_{\substack{|\alp'|= |\alp|-1\\ |\bt'|= |\bt|+1}} |\rd_x^{\alp'} \rd_v^{\bt'} Y^{\sigma} \bar{a}_{ij}(t,x,v)|.
\end{split}
\end{equation*}
The desired estimate is then an immediate consequence of Propositions~\ref{prop:ab.Li.1} (for the first term) and \ref{prop:ab.Li.2} (for the second term). \qedhere
\end{proof}

We have another variant of Proposition~\ref{prop:ab.Li.1}, which again has better $t$-decay rate as $t\to +\infty$. Unlike Propositions~\ref{prop:ab.Li.2} and \ref{prop:ab.Li.3}, this does not require $|\bt|\geq 1$ or $|\alp|\geq 1$, but there is a loss in $\wb$ weights. (Note also that the estimate is very weak as $t\to 0$.)
\begin{proposition}\label{prop:ab.Li.null.cond}
If $|\alp|+|\bt|+|\sigma|\leq \Mm-4-\max\{2,\lceil \f{2}{2+\gamma}\rceil\}$, then for $(t,x,v)\in [0,T_{Boot})\times \mathbb R^3\times \mathbb R^3$,
$$\max_{i,j}|\rd_x^\alp \rd_v^\bt Y^\sigma \bar{a}_{ij}|(t,x,v)\ls \ep^{\f 34} \wb^{\min\{1,2+\gamma\}}\vb^{\max\{0,1+\gamma\}}t^{-\min\{2+\gamma,1\}}(1+t)^{-3+|\bt|}.$$
If $\Mm-3-\max\{2,\lceil \f{2}{2+\gamma}\rceil\}\leq|\alp|+|\bt|+|\sigma| =: k \leq \Mm-5$, then for $(t,x,v)\in [0,T_{Boot})\times \mathbb R^3\times \mathbb R^3$,
$$\max_{i,j}|\rd_x^\alp \rd_v^\bt Y^\sigma \bar{a}_{ij}|(t,x,v)\ls \ep^{\f 34} \wb^{\min\{1,2+\gamma\}}\vb^{\max\{0,1+\gamma\}}t^{-\min\{2+\gamma,1\}}(1+t)^{-\f 32+|\bt|-(\Mm-4-k)\min\{\f 34,\f{3(2+\gamma)}{4}\}}.$$
\end{proposition}
\begin{proof}
The idea is to make use of the weight $|v-v_*|^{2+\gamma}$ and write $|v-v_*|^{2+\gamma} \ls  t^{-\min\{2+\gamma,1\}}(|x-tv|+|x-tv_*|)^{\min\{2+\gamma,1\}} |v-v_*|^{\max\{0,1+\gamma\}}$. Hence, by \eqref{a.0.1} in Proposition~\ref{prop:a.expressions},
\begin{equation*}
\begin{split}
&\: |\rd_x^\alp \rd_v^\bt Y^\sigma \bar{a}_{ij}|(t,x,v)\\
\ls &\: \int_{\mathbb R^3} |v-v_*|^{2+\gamma}|\rd_x^\alp \rd_v^\bt Y^\sigma f|(t,x,v_*)\,\ud v_* \\
\ls &\: t^{-\min\{2+\gamma,1\}}\int_{\mathbb R^3} (|x-tv|+|x-tv_*|)^{\min\{2+\gamma,1\}} |v-v_*|^{\max\{0,1+\gamma\}}|\rd_x^\alp \rd_v^\bt Y^\sigma f|(t,x,v_*)\,\ud v_* \\
\ls &\: t^{-\min\{2+\gamma,1\}}\wb^{\min\{2+\gamma,1\}} \vb^{\max\{0,1+\gamma\}}\\
&\: \qquad \times\int_{\mathbb R^3} \langle x-tv_*\rangle^{\min\{2+\gamma,1\}} \langle v_* \rangle^{\max\{0,1+\gamma\}}|\rd_x^\alp \rd_v^\bt Y^\sigma f|(t,x,v_*)\,\ud v_*.
\end{split}
\end{equation*}
The desired estimate then follows from Lemma~\ref{lem:Li.5}. \qedhere
\end{proof}

Our final estimate in the subsubsection is an analogue of Proposition~\ref{prop:ab.Li.1}, but we now also allow contracting $\bar{a}_{ij}$ with $v$'s.
\begin{proposition}\label{prop:ab.Li.weighted}
If $|\alp|+|\bt|+|\sigma|\leq \Mm-4-\max\{2,\lceil \f{2}{2+\gamma} \rceil \}$, then for $(t,x,v)\in [0,T_{Boot})\times \mathbb R^3\times \mathbb R^3$,
$$\max_{i}|\rd_x^\alp \rd_v^\bt Y^\sigma (\bar{a}_{ij}v_j)|(t,x,v)+ |\rd_x^\alp \rd_v^\bt Y^\sigma (\bar{a}_{ij}v_iv_j)|(t,x,v) \ls \ep^{\f 34} \vb^{\max\{2+\gamma,1\}}(1+t)^{-3+|\bt|}.$$
If $\Mm-3-\max\{2,\lceil \f{2}{2+\gamma} \rceil \} \leq |\alp|+|\bt|+|\sigma| =:k \leq \Mm-5$, then for $(t,x,v)\in [0,T_{Boot})\times \mathbb R^3\times \mathbb R^3$,
\begin{equation*}
\begin{split}
&\: \max_{i}|\rd_x^\alp \rd_v^\bt Y^\sigma (\bar{a}_{ij}v_j)|(t,x,v) + |\rd_x^\alp \rd_v^\bt Y^\sigma (\bar{a}_{ij}v_iv_j)|(t,x,v)\\
\ls &\: \ep^{\f 34} \vb^{\max\{2+\gamma,1\}} (1+t)^{-\f 32+|\bt|-(\Mm-4-k)\min\{\f 34,\f{3(2+\gamma)}{4}\}}.
\end{split}
\end{equation*}
\end{proposition}
\begin{proof}
The follows from combining \eqref{a.0.2} and \eqref{a.0.3} in Proposition~\ref{prop:a.expressions} and Lemma~\ref{lem:Li.5}. \qedhere
\end{proof}

\subsubsection{$L^\i_x L^\i_v$ estimates for $\bar{c}$ and its derivatives}

\begin{proposition}\label{prop:cb.Li}
If $|\alp|+|\bt|+|\sigma|\leq \Mm-4-\max\{2,\lceil \f{2}{2+\gamma}\rceil\}$, then for $(t,x,v)\in [0,T_{Boot})\times \mathbb R^3\times \mathbb R^3$,
$$|\rd_x^\alp \rd_v^\bt Y^\sigma \bar{c}|(t,x,v)\ls \ep^{\f 34} (1+t)^{-3-\gamma+|\bt|}.$$
If $\Mm-3-\max\{2,\lceil \f{2}{2+\gamma}\rceil\}\leq|\alp|+|\bt|+|\sigma| =: k \leq \Mm-5$, then for $(t,x,v)\in [0,T_{Boot})\times \mathbb R^3\times \mathbb R^3$,
$$|\rd_x^\alp \rd_v^\bt Y^\sigma \bar{c}|(t,x,v)\ls \ep^{\f 34} (1+t)^{-\f 32-\gamma+|\bt|-(\Mm-4-k)\min\{\f 34,\f{3(2+\gamma)}{4}\}}.$$
\end{proposition}
\begin{proof}
By \eqref{c.def} and \eqref{bar.def}, we have
$$|\rd_x^\alp \rd_v^\bt Y^\sigma \bar{c}|(t,x,v) \ls \int_{\mathbb R^3} |v-v_*|^{\gamma}  |\rd_x^\alp \rd_v^\bt Y^\sigma f|(t,x,v_*)\,\ud v_*.$$
The conclusion then follows from Lemma~\ref{lem:Li.6} with $\nu = -\gamma$. \qedhere
\end{proof}

\subsection{The $L^2_x$ estimates for $\bar{a}_{ij}$, $\bar{c}$ and their derivatives}\label{sec:coeff.L2}

\subsubsection{Preliminary estimates}

\begin{lemma}\label{lem:decay.using.weights}
Let $h:[0,T_{Boot})\times \mathbb R^3\times \mathbb R^3$ be a smooth function. Then 
\begin{equation}\label{eq:decay.using.weights}
\|h\|_{L^1_v}(t,x)\ls  (1+t)^{-\f {3}{2}} \|\vb^2 \wb^2 h(t,x,v)\|_{L^2_v}.
\end{equation}
\end{lemma}
\begin{proof}
\pfstep{Step~1: $0\leq t< 1$} Suppose $t\in [0,1)$. This is the easy case: we simply use H\"older's inequality to obtain
\begin{equation*}
\begin{split}
 \int_{\mathbb R^3} |h| (t,x,v)\, \ud v \ls &\: (\int_{\mathbb R^3} \vb^4 h^2 (t,x,v)\, \ud v)^{\f 12}(\int_{\mathbb R^3} \vb^{-4}\, \ud v)^{\f{1}{2}}\\
\ls &\: (\int_{\mathbb R^3} \vb^4 h^2 (t,x,v)\, \ud v)^{\f 12}.
\end{split}
\end{equation*}
This implies \eqref{eq:decay.using.weights} for $0 \leq t < 1$.

\pfstep{Step~2: $t\geq 1$} Suppose now $t\geq 1$. We again use H\"older's inequality, except that we need to partition the region of integration in order to obtain decay from the $|x-tv|$-weights. More precisely,
\begin{equation*}
\begin{split}
&\: \int_{\mathbb R^3} |h| (t,x,v)\, \ud v \\
\ls &\: \int_{\mathbb R^3\cap \{|v-\f{x}{t}|\leq t^{-1}\}} |h| (t,x,v)\, \ud v + \int_{\mathbb R^3\cap \{|v-\f{x}{t}|\geq t^{-1}\}} |h|(t,x,v)\, \ud v \\
\ls &\: (\int_{\mathbb R^3\cap \{|v-\f{x}{t}|\leq t^{-1}\}} h^2(t,x,v)\, \ud v)^{\f 12}(\int_{\mathbb R^3\cap \{|v-\f{x}{t}|\leq t^{-1}\}}\, \ud v)^{\f 12} \\
&\: + (\int_{\mathbb R^3\cap \{|v-\f{x}{t}|\geq t^{-1}\}} |v-\f{x}{t}|^4 h^2(t,x,v)\, \ud v)^{\f 12} (\int_{\mathbb R^3\cap \{|v-\f{x}{t}|\geq t^{-1}\}} |v-\f{x}{t}|^{-4} \,\ud v)^{\f 12}\\
\ls &\: t^{-\f 32}  (\int_{\mathbb R^3} h^2(t,x,v)\, \ud v)^{\f 12} + t^{-2}\cdot (\int_{\mathbb R^3} \wb^4 h^2(t,x,v)\, \ud v)^{\f p2} \cdot t^{\f 12}\\
\ls &\: t^{-\f 32} (\int_{\mathbb R^3} \wb^4 h^2(t,x,v)\, \ud v)^{\f 12}.
\end{split}
\end{equation*} 
This yields \eqref{eq:decay.using.weights} for $t\geq 1$. \qedhere
\end{proof}

Lemma~\ref{lem:decay.using.weights} implies the following $L^2_xL^1_v$ estimate.
\begin{lemma}\label{lem:L1L2}
Let $h:[0,T_{Boot})\times \mathbb R^3\times \mathbb R^3$ be a smooth function. Then 
$$\|h\|_{L^2_x L^1_v}(t) \ls (1+t)^{-\f 32}\|\vb^2\wb^2 h \|_{L^2_xL^2_v}.$$
\end{lemma}
\begin{proof}
By Lemma~\ref{lem:decay.using.weights}, we have
\begin{equation*}
\begin{split}
\|h\|_{L^2_x L^1_v}(t) = &\: (\int_{\mathbb R^3} \|h\|_{L^1_v}^2(t,x)\,\ud x)^{\f 12} \\
\ls &\: (1+t)^{-\f 32} (\int_{\mathbb R^3} \|\vb^2 \wb^2 h\|_{L^2_v}^2(t,x)\,\ud x)^{\f 12} \\
= &\: (1+t)^{-\f 32}\|\vb^2 \wb^2 h \|_{L^2_xL^2_v}.
\end{split}
\end{equation*}
\end{proof}

\begin{lemma}\label{lemma:HLS.type}
Let $h:\mathbb R^3\to \mathbb R$ be a smooth function.

For $\nu \in (\f 32, 3)$,
\begin{equation}\label{eq:HLS.type.1}
\left\| \int_{\mathbb R^3} |v-v_*|^{-\nu} |h|(v_*) \, \ud v_* \right\|_{L^2_v} \ls \|h\|_{L^1_v}^{-\f{2\nu}{3}+2} \|h\|_{L^2_v}^{\f{2\nu}3-1}.
\end{equation}
For $\nu \in [0, \f 32]$,
\begin{equation}\label{eq:HLS.type.2}
\left\| \int_{\mathbb R^3} |v-v_*|^{-\nu} |h|(v_*)\, dv_* \right\|_{L^{\f{15}{4\nu}}_v} \ls \|h\|_{L^1_v}^{1-\f{2\nu}{15}}\|h\|_{L^2_v}^{\f{2\nu}{15}}.
\end{equation}
\end{lemma}
\begin{proof}
\pfstep{Step~1: Proof of \eqref{eq:HLS.type.1}} Without loss of generality, we assume that $h$ is not identically $0$ (for otherwise the estimates are trivial).

Let $\lambda>0$ be a constant to be determined. We estimate as follows (see the justification of each step after the estimates):
\begin{align}
&\: \|\int_{\mathbb R^3} |v-v_*|^{-\nu} |h|(v_*) \, \ud v_*\|_{L^2_v} \notag \\
= &\: \left(\int_{\mathbb R^3} (\int_{\{v_*: |v-v_*|\leq \lambda\}} |v-v_*|^{-\nu} |h|(v_*) \, \ud v_* )^2\, \ud v \right)^{\f 12} + \left(\int_{\mathbb R^3} (\int_{\{v_*: |v-v_*|> \lambda\}} |v-v_*|^{-\nu} |h|(v_*) \, \ud v_* )^2\, \ud v \right)^{\f 12} \label{HLS.RHS.1}\\
\ls &\: \left(\int_{\mathbb R^3} \left(\int_{\{v_*: |v-v_*|\leq \lambda\}} |v-v_*|^{-\nu} |h|^2(v_*) \, \ud v_* \right)\left(\int_{\{v_*: |v-v_*|\leq \lambda\}} |v-v_*|^{-\nu} \, \ud v_* \right)\, \ud v \right)^{\f 12} \label{HLS.RHS.2}\\
&\: + \int_{\mathbb R^3} (\int_{\{v: |v-v_*|> \lambda\}} |v-v_*|^{-2\nu} \, \ud v )^{\f 12} |h|(v_*)\, \ud v_* \label{HLS.RHS.3}\\
\ls &\: \lambda^{-\f \nu 2+ \f 32} \|h\|_{L^2_v} (\int_{\{v: |v-v_*|\leq \lambda\}} |v-v_*|^{-\nu} \,\ud v )^{\f 12} + \lambda^{-\nu+\f 32} \|h\|_{L^1_v} \label{HLS.RHS.4}\\
\ls &\: \lambda^{-\nu+3} \|h\|_{L^2_v} + \lambda^{-\nu+\f 32} \|h\|_{L^1_v}. \label{HLS.RHS.5}
\end{align}

In \eqref{HLS.RHS.1}, we divided the integral into regions $|v-v_*|\leq \lambda$ and $|v-v_*|> \lambda$; in \eqref{HLS.RHS.2}, we used the Cauchy--Schwarz inequality; in \eqref{HLS.RHS.3}, we used the Minkowski inequality; in the first term in \eqref{HLS.RHS.4}, we noted that $(\int_{\{v_*: |v-v_*|\leq \lambda\}} |v-v_*|^{-\nu} \, \ud v_*)^{\f 12} \ls \lambda^{-\f \nu 2+\f 32}$ and then used Fubini's theorem; in the second term in \eqref{HLS.RHS.4} we simply used $(\int_{\{v: |v-v_*|> \lambda\}} |v-v_*|^{-2\nu} \, \ud v )^{\f 12} \ls \lambda^{-\nu+\f 32}$; in \eqref{HLS.RHS.5} we used $(\int_{\{v: |v-v_*|\leq \lambda\}} |v-v_*|^{-\nu} \,\ud v )^{\f 12} \ls \lambda^{- \f \nu 2+\f 32}$. (Note that in \eqref{HLS.RHS.4} and \eqref{HLS.RHS.5}, we have relied on $\nu \in (\f 32,3)$ in our estimates.)

Let $\lambda := \|h\|_{L^1_v}^{\f 23} \|h\|_{L^2_v}^{-\f 23}$ (which is possible since $h$ is not identically $0$). We then obtain
$$\left\|\int_{\mathbb R^3} |v-v_*|^{-\nu} |h|(v_*) \, \ud v_* \right\|_{L^2_v} \ls \|h\|_{L^1_v}^{-\f{2\nu}{3}+2} \|h\|_{L^2_v}^{\f{2\nu}3-1},$$
as desired.

\pfstep{Step~2: Proof of \eqref{eq:HLS.type.2}} For this inequality we use the Hardy--Littlewood--Sobolev inequality in $\mathbb R^3$: for $0<\nu<3$, $1<p<q<+\infty$, and $\f 1q = \f 1p -\f{(3-\nu)} 3$, 
\begin{equation}\label{HLS}
\left\| \int_{\mathbb R^3} |v-v_*|^{-\nu} |h|(v_*)\, dv_* \right\|_{L^q_v} \ls \| h\|_{L^p_v}.
\end{equation}
For $\nu \in [0,\f 32]$, we now apply \eqref{HLS} with\footnote{Note that $\nu=0$ is technically not allowed in \eqref{HLS}, but for the specific $(p,q)$ under consideration, the inequality is trivially true.} $\f 1p= -\f{1}{15}\nu + 1$ and $\f 1q= \f{4}{15}\nu$. It then follows that from H\"older's inequality that
$$\left\| \int_{\mathbb R^3} |v-v_*|^{-\nu} |h|(v_*)\, dv_* \right\|_{L^{\f{15}{4\nu}}_v} \ls \| h\|_{L^{\f{15}{15-\nu}}_v} \ls \|h\|_{L^1_v}^{1-\f{2\nu}{15}}\|h\|_{L^2_v}^{\f{2\nu}{15}},$$
as claimed. \qedhere
\end{proof}

Combining Lemmas~\ref{lem:decay.using.weights} and \ref{lemma:HLS.type}, and taking the $L^2_x$ norm, we obtain
\begin{lemma}\label{lem:h.HLS.type.decay}
Let $h:[0,T_{Boot})\times \mathbb R^3\times \mathbb R^3$ be a smooth function. 

For $\nu \in (\f 32, 3)$,
\begin{equation}\label{eq:h.HLS.type.decay.1}
\left\| \int_{\mathbb R^3} |v-v_*|^{-\nu} |h|(t,x,v_*) \, \ud v_* \right\|_{L^2_xL^2_v} \ls (1+t)^{\nu-3} \|\vb^2\wb^2 h\|_{L^2_xL^2_v}(t,x).
\end{equation}
For $\nu \in [0, \f 32]$,
\begin{equation}\label{eq:h.HLS.type.decay.2}
\left\| \int_{\mathbb R^3} |v-v_*|^{-\nu} |h|(t,x,v_*)\, dv_* \right\|_{L^2_x L^{\f{15}{4\nu}}_v} \ls (1+t)^{-\f 32+\f{\nu}{5}} \|\vb^2\wb^2 h\|_{L^2_xL^2_v}(t,x).
\end{equation}
\end{lemma}

\subsubsection{Estimates for weighted $v$-integrals of $f$}

\begin{proposition}\label{prop:main.L2.decay}
Let $|\alp|+|\bt|+|\sigma|\leq \Mm$. Then the following three estimates\footnote{The reader may find the notation in \eqref{eq:L2.decay.notsing} slightly confusing since the LHS does not depend on $v$. We use such notation so that we have a more unified estimate later; see Proposition~\ref{prop:L2.p*}.} hold for all $t\in [0,T_{Boot})$:
\begin{equation}\label{eq:L2.decay.notsing}
\left\| \int_{\mathbb R^3} \langle v_*\rangle^4 |\rd_x^\alp \rd_v^\bt Y^\sigma f|(t,x,v_*)\, \ud v_* \right\|_{L^2_x L^\i_v} \ls \ep^{\f 34}(1+t)^{-\f 32+|\bt|}.
\end{equation}
For $\nu \in (\f 32, 3)$, 
\begin{equation}\label{eq:HLS.type.decay.1}
\left\| \int_{\mathbb R^3} |v-v_*|^{-\nu}\langle v_*\rangle^4 \langle x-tv_*\rangle^2 |\rd_x^\alp \rd_v^\bt Y^\sigma f|(t,x,v_*) \, \ud v_* \right\|_{L^2_x L^2_v} \ls \ep^{\f 34}(1+t)^{\nu-3+|\bt|}.
\end{equation}
For $\nu \in [0, \f 32]$,
\begin{equation}\label{eq:HLS.type.decay.2}
\left\| \int_{\mathbb R^3} |v-v_*|^{-\nu} \langle v_*\rangle^4 \langle x-tv_*\rangle^2 |\rd_x^\alp \rd_v^\bt Y^\sigma f|(t,x,v_*)\, dv_* \right\|_{L^2_x L^{\f{15}{4\nu}}_v} \ls \ep^{\f 34}(1+t)^{-\f 32+\f{\nu}{5}+|\bt|} \ls \ep^{\f 34}(1+t)^{-\f 65+|\bt|}.
\end{equation}
\end{proposition}
\begin{proof}
\eqref{eq:L2.decay.notsing} follows from Lemmas~\ref{lem:Li.4}, \ref{lem:L1L2} and the bootstrap assumption \eqref{BA}.

\eqref{eq:HLS.type.decay.1} follows from Lemma~\ref{lem:Li.4}, \eqref{eq:h.HLS.type.decay.1} in Lemma~\ref{lem:h.HLS.type.decay} and the bootstrap assumption \eqref{BA}.

Finally, the first inequality in \eqref{eq:HLS.type.decay.2} follows from Lemma~\ref{lem:Li.4}, \eqref{eq:h.HLS.type.decay.2} in Lemma~\ref{lem:h.HLS.type.decay} and the bootstrap assumption \eqref{BA}. The very last inequality in \eqref{eq:HLS.type.decay.2} is simply an assertion that $-\f 32+\f{\nu}{5}\leq -\f 32 +\f 3{10} = -\f 65$ when $\nu \in [0,\f 32]$. \qedhere
\end{proof}

The different $L^p$ spaces used in Proposition~\ref{prop:main.L2.decay} motivates the following definitions. The notation is intended to be suggestive of the following: we will control one $v$-derivative of $\bar{a}_{ij}$ in $L^{p_*}_v$ and we will control two $v$-derivatives of $\bar{a}_{ij}$ in $L^{p_{**}}_v$. (Zeroth $v$-derivatives of $\bar{a}_{ij}$ will be estimated in $L^\i_v$.)
\begin{definition}\label{p*.def}
Define $p_*$ and $p_{**}$ by
$$p_*:= \begin{cases}
\infty &\mbox{ if }\gamma \in [-1,0) \\
-\f{15}{4(\gamma+1)} &\mbox{ if }\gamma \in (-2,-1)
\end{cases},\qquad
p_{**}:= \begin{cases}
-\f{15}{4\gamma} &\mbox{ if }\gamma \in [-\f 32,0) \\
2 &\mbox{ if }\gamma \in (-2,-\f 32)
\end{cases}.
$$
Note that $p_*,\,p_{**}\in [2,\infty]$ (for any $\gamma \in (-2,0)$).
\end{definition}

With this convention for $p_*$ and $p_{**}$, let us rephrase the last two inequalities in Proposition~\ref{prop:main.L2.decay}:
\begin{proposition}\label{prop:L2.p*}
Let $|\alp|+|\bt|+|\sigma|\leq \Mm$. Then the following two estimates hold for all $t\in [0,T_{Boot})$:
\begin{equation}\label{eq:gamma.Lp*}
\left\|\vb^{-\max\{0,1+\gamma\}} \int_{\mathbb R^3} |v-v_*|^{1+\gamma} \langle v_*\rangle^2 \langle x-tv_*\rangle^2 |\rd_x^\alp \rd_v^\bt Y^\sigma f|(t,x,v_*) \, \ud v_* \right\|_{L^2_x L^{p_*}_v} \ls \ep^{\f 34}(1+t)^{-\f 65+|\bt|},
\end{equation}
\begin{equation}\label{eq:gamma.Lp**}
\left\| \int_{\mathbb R^3} |v-v_*|^{\gamma} \langle v_*\rangle^2 \langle x-tv_*\rangle^2 |\rd_x^\alp \rd_v^\bt Y^\sigma f|(t,x,v_*)\, dv_* \right\|_{L^2_x L^{p_{**}}_v} \ls \ep^{\f 34}(1+t)^{-\min\{\f 65, 3+\gamma\}+|\bt|}.
\end{equation}
\end{proposition}
\begin{proof}
To prove \eqref{eq:gamma.Lp*}, we consider separately $\gamma\in [-1,0)$ and $\gamma\in (-2,-1)$. If $\gamma \in [-1,0)$, $p_*=\infty$. Also, $|v-v_*|^{1+\gamma}\ls |v|^{1+\gamma}+|v_*|^{1+\gamma}$. Hence, by \eqref{eq:L2.decay.notsing} in Proposition~\ref{prop:main.L2.decay}, we obtain
\begin{equation*}
\begin{split}
&\: \left\|\vb^{-\max\{0,1+\gamma\}} \int_{\mathbb R^3} |v-v_*|^{1+\gamma} \langle v_*\rangle^2 \langle x-tv_*\rangle^2 |\rd_x^\alp \rd_v^\bt Y^\sigma f|(t,x,v_*) \, \ud v_* \right\|_{L^2_x L^{p_*}_v} \\
\ls &\: \left\| \int_{\mathbb R^3} \langle v_*\rangle^{3+\gamma} \langle x-tv_*\rangle^2 |\rd_x^\alp \rd_v^\bt Y^\sigma f|(t,x,v_*) \, \ud v_* \right\|_{L^2_x L^{\i}_v} \\
\ls &\: \ep^{\f 34}(1+t)^{-\f 32+|\bt|},
\end{split}
\end{equation*}
which is slightly better than \eqref{eq:gamma.Lp*}.

Consider now the case $\gamma \in (-2,-1)$. In this case, $p_* = -\f{54}{5(\gamma+1)}$ and $1+\gamma \in (-1,0)$. Hence, by \eqref{eq:HLS.type.decay.2} in Proposition~\ref{prop:main.L2.decay}, we obtain
\begin{equation*}
\begin{split}
&\: \left\|\vb^{-\max\{0,1+\gamma\}} \int_{\mathbb R^3} |v-v_*|^{1+\gamma} \langle v_*\rangle^2 \langle x-tv_*\rangle^2 |\rd_x^\alp \rd_v^\bt Y^\sigma f|(t,x,v_*) \, \ud v_* \right\|_{L^2_x L^{p_*}_v} \\
\ls &\: \left\| \int_{\mathbb R^3} |v-v_*|^{1+\gamma} \langle v_*\rangle^2 \langle x-tv_*\rangle^2 |\rd_x^\alp \rd_v^\bt Y^\sigma f|(t,x,v_*) \, \ud v_* \right\|_{L^2_x L^{-\f{54}{5\gamma}}_v} \\
\ls &\: \ep^{\f 34}(1+t)^{-\f 65+|\bt|}.
\end{split}
\end{equation*}
which is as in \eqref{eq:gamma.Lp*}. We have thus concluded the proof of \eqref{eq:gamma.Lp*}. 

We now prove \eqref{eq:gamma.Lp**}. Now since $\gamma<0$, we can directly use \eqref{eq:HLS.type.decay.1} and \eqref{eq:HLS.type.decay.2} in Proposition~\ref{prop:main.L2.decay} to obtain
\begin{equation*}
\begin{split}
&\: \left\| \int_{\mathbb R^3} |v-v_*|^{\gamma} \langle v_*\rangle^2 \langle x-tv_*\rangle^2 |\rd_x^\alp \rd_v^\bt Y^\sigma f|(t,x,v_*)\, dv_* \right\|_{L^2_x L^{p_{**}}_v} \\
\ls &\: \ep^{\f 34}(1+t)^{-\min\{\f 65, 3+\gamma\}+|\bt|},
\end{split}
\end{equation*}
as desired. \qedhere
\end{proof}

\subsubsection{$L^2_x$ estimates for $\bar{a}_{ij}$ and its derivatives}
With the above preparation, we now prove the $L^2_x$ estimates for $\bar{a}_{ij}$ and its derivatives.

\begin{proposition}\label{prop:ab.L2}
If $|\alp|+|\bt|+|\sigma|\leq \Mm$, then
$$\max_{i,j} \|\vb^{-(2+\gamma)}\rd_x^\alp \rd_v^\bt Y^\sigma \bar{a}_{ij} \|_{L^2_x L^\i_v}(t) \ls \ep^{\f 34} (1+t)^{-\f 32+|\bt|}.$$
\end{proposition}
\begin{proof}
This follows from \eqref{a.0.1} in Proposition~\ref{prop:a.expressions} and \eqref{eq:L2.decay.notsing} in Proposition~\ref{prop:main.L2.decay}. \qedhere
\end{proof}

The next proposition improves the decay rate in $t$, but requires $|\bt|\geq 2$ (compare Proposition~\ref{prop:ab.Li.2}).
\begin{proposition}\label{prop:ab.L2.improved}
If $|\alp|+|\bt|+|\sigma|\leq \Mm$ \underline{and $|\bt|\geq 2$}, then
$$\|\rd_x^\alp \rd_v^\bt Y^\sigma \bar{a}_{ij} \|_{L^2_x L^{p_{**}}_v} \ls \ep^{\f 34} (1+t)^{-\min\{\f{16}5,5+\gamma\}+|\bt|},$$
where $p_{**}$ is as in Definition~\ref{p*.def}.
\end{proposition}
\begin{proof}
By \eqref{a.2.1} in Proposition~\ref{prop:a.expressions} and \eqref{eq:gamma.Lp**} in Proposition~\ref{prop:L2.p*}, we have
\begin{equation*}
\begin{split}
&\: \|\rd_x^\alp \rd_v^\bt Y^\sigma \bar{a}_{ij}\|_{L^2_x L^{p_{**}}_v}(t)\\
\ls &\: \sum_{|\bt'|\leq |\bt|-2} \left\| \int_{\mathbb R^3} |v-v_*|^{-\gamma} |\rd_x^\alp \rd_v^{\bt'}Y^{\sigma} f|(t,x,v_*) \,\ud v_* \right\|_{L^2_x L^{p_{**}}_v}(t)\\
\ls &\: \sum_{|\bt'|\leq |\bt|-2} \ep^{\f 34} (1+t)^{-\min\{\f 65, 3+\gamma\}+|\bt'|}
\ls  \ep^{\f 34} (1+t)^{-\min\{\f {16}5, 5+\gamma\}+|\bt|}.
\end{split}
\end{equation*}
\qedhere
\end{proof}

The next proposition improves the decay rate in $t$, but requires $|\alp|\geq 2$ (compare Proposition~\ref{prop:ab.Li.3}). It is also very weak as $t\to 0$.
\begin{proposition}\label{prop:ab.L2.improved.2}
If $|\alp|+|\bt|+|\sigma|\leq \Mm$ \underline{and $|\alp|\geq 2$}, then
$$\|\vb^{-(2+\gamma)} \rd_x^\alp \rd_v^\bt Y^\sigma \bar{a}_{ij} \|_{L^2_xL^\i_v +L^2_x L^2_v} \ls \ep^{\f 34} t^{-2}(1+t)^{-\min\{\f{6}5,3+\gamma\}+|\bt|},$$
where $p_{**}$ is as in Definition~\ref{p*.def}.
\end{proposition}
\begin{proof}
\pfstep{Step~1: Preliminary estimate for the $|\bt|\geq 1$ case} The purpose of this step is to establish the following claim, which can be viewed as an analogue of Proposition~\ref{prop:ab.L2.improved}, but with only $|\bt|\geq 1$.

\textbf{Claim:} If $|\alp|+|\bt|+|\sigma|\leq \Mm$ \underline{and $|\bt|\geq 1$}, then
$$\|\vb^{-(2+\gamma)} \rd_x^\alp \rd_v^\bt Y^\sigma \bar{a}_{ij} \|_{L^2_x L^{p_{*}}_v} \ls \ep^{\f 34} (1+t)^{-\min\{\f{11}5,3+\gamma\}+|\bt|},$$
where $p_{*}$ is as in Definition~\ref{p*.def}.

To prove this claim, it suffices to combine \eqref{a.1.1} in Proposition~\ref{prop:a.expressions} and \eqref{eq:gamma.Lp*} in Proposition~\ref{prop:L2.p*}.

\pfstep{Step~2: Main argument} Arguing as in the proof of Proposition~\ref{prop:ab.Li.3}, we write $\rd_x = t^{-1} Y - t^{-1}\rd_v$. Using this identity twice, we obtain the pointwise estimate
\begin{equation*}
\begin{split}
&\: |\rd_x^\alp \rd_v^\bt Y^\sigma \bar{a}_{ij}(t,x,v)|  \\
\ls &\: t^{-2} \sum_{\substack{|\alp'|= |\alp|-2\\ |\sigma'|= |\sigma|+2}} |\rd_x^{\alp'} \rd_v^{\bt} Y^{\sigma'} \bar{a}_{ij}(t,x,v)| + t^{-2} \sum_{\substack{|\alp'|= |\alp|-2\\ |\bt'|=|\bt|+1\\ |\sigma'|= |\sigma|+1}} |\rd_x^{\alp'} \rd_v^{\bt'} Y^{\sigma'} \bar{a}_{ij}(t,x,v)| \\
&\: + t^{-2} \sum_{\substack{|\alp'|= |\alp|-2\\ |\bt'|= |\bt|+2}} |\rd_x^{\alp'} \rd_v^{\bt'} Y^{\sigma} \bar{a}_{ij}(t,x,v)|.
\end{split}
\end{equation*}
We now estimate the first term with Proposition~\ref{prop:ab.L2} (with $(\alp',\bt,\sigma')$ in place of $(\alp,\bt,\sigma)$), estimate the second term with the Claim in Step~1 (with $(\alp',\bt',\sigma)$ in place of $(\alp,\bt,\sigma)$, noting that $|\bt'|\geq 1$), and estimate the last term with Proposition~\ref{prop:ab.L2.improved} (with $(\alp',\bt',\sigma)$ in place of $(\alp,\bt,\sigma)$, noting that $|\bt'|\geq 2$), we obtain
\begin{equation*}
\begin{split}
&\: \|\vb^{-(2+\gamma)} \rd_x^\alp \rd_v^\bt Y^\sigma \bar{a}_{ij} \|_{L^2_xL^\i_v + L^2_xL^{p_*}_v+L^2_x L^{p_{**}}_v} \\
\ls &\: \ep^{\f 34} \max\{t^{-2}(1+t)^{-\f 32+|\bt|}, t^{-2}(1+t)^{-\f 65+|\bt|},t^{-2}(1+t)^{-\min\{\f{6}5,3+\gamma\}+|\bt|}\}.
\end{split}
\end{equation*}
Since $\f 32 \geq \f 65$, this implies 
\begin{equation}\label{eq:L2.improved.alp.step2}
\|\vb^{2+\gamma} \rd_x^\alp \rd_v^\bt Y^\sigma \bar{a}_{ij} \|_{L^2_xL^\i_v + L^2_xL^{p_*}_v+L^2_x L^{p_{**}}_v} \ls \ep^{\f 34} t^{-2}(1+t)^{-\min\{\f{6}5,3+\gamma\}+|\bt|}.
\end{equation}

\pfstep{Step~3: Calculus lemma and conclusion of the proof} Noticing the elementary embedding
$$L^q \subset L^p + L^r$$
when $1\leq p\leq q \leq r \leq +\infty$, and using $p_*,\,p_{**}\in [2,+\infty]$, the conclusion thus follows from \eqref{eq:L2.improved.alp.step2}. \qedhere
\end{proof}

The next estimate gives a better decay rate as $t\to +\infty$, but it is very weak as $t\to 0$ and requires an additional weight of $\wb^{2}$ (compare Proposition~\ref{prop:ab.Li.null.cond}).
\begin{proposition}\label{prop:ab.L2.null.cond}
For $p_{**}\in [2,\infty]$ as in Definition~\ref{p*.def}, if $|\alp|+|\bt|+|\sigma|\leq \Mm$, then for $t\in [0,T_{Boot})$,
$$\max_{i,j} \|\wb^{-2}\rd_x^{\alp} \rd_v^\bt Y^\sigma \bar{a}_{ij}\|_{L^2_xL^\i_v + L^2_x L^{p_{**}}_v}(t) \ls \ep^{\f 34}t^{-2}(1+t)^{-\min\{\f{6}{5},3+\gamma\}+|\bt|}.$$
\end{proposition}
\begin{proof}
By \eqref{a.0.1} in Proposition~\ref{prop:a.expressions} and the triangle inequality,
\begin{equation*}
\begin{split}
&\: |\rd_x^{\alp}\rd_v^\bt Y^{\sigma} \bar{a}_{ij}|(t,x,v) \\
\ls &\: \int_{\mathbb R^3} |v-v_*|^{2+\gamma} |\rd_x^{\alp} \rd_v^{\bt} Y^\sigma f|(t,x,v_*)\,\ud v_* \\
\ls &\: \int_{\mathbb R^3} |v-v_*|^{\gamma} (|v-\f{x}{t}|^2 + |v_*-\f xt|^2) |\rd_x^{\alp} \rd_v^{\bt} Y^\sigma f|(t,x,v_*)\,\ud v_* \\
\ls &\: t^{-2}\int_{\mathbb R^3} |v-v_*|^{\gamma} (|x-tv|^2 + |x-tv_*|^2) |\rd_x^{\alp} \rd_v^{\bt} Y^\sigma f|(t,x,v_*)\,\ud v_*.
\end{split}
\end{equation*}
Therefore,
\begin{equation*}
\begin{split}
&\:\max_{i,j} \|\wb^{-2}\rd_x^{\alp} \rd_v^\bt Y^\sigma \bar{a}_{ij}\|_{L^2_x L^{p_{**}}_v}(t) \\
\ls &\: t^{-2} \|\int_{\mathbb R^3} |v-v_*|^{\gamma} \langle x-tv_*\rangle^2 |\rd_x^{\alp} \rd_v^{\bt} Y^\sigma f|(t,x,v_*)\,\ud v_*\|_{L^2_x L^{p_{**}}_v}(t).
\end{split}
\end{equation*}
The desired conclusion then follows from \eqref{eq:gamma.Lp**} in Proposition~\ref{prop:L2.p*}. \qedhere
\end{proof}

We next prove estimates for $\bar{a}_{ij}$ for its derivatives \emph{when contracted with $v$}  (compare Proposition~\ref{prop:ab.Li.weighted}). When $|\alp|\geq 1$ or $|\bt|\geq 1$, we have an improvement in the decay rate (although the estimate is very weak as $t\to 0$).
\begin{proposition}\label{prop:ab.L2.weighted}
If $|\alp|+|\bt|+|\sigma|\leq \Mm$, then
\begin{equation}\label{prop:ab.Lw.weighted.basic}
\max_j\|\vb^{-\max\{2+\gamma,1\}}\rd_x^\alp \rd_v^\bt Y^\sigma (\bar{a}_{ij} v_i)\|_{L^2_x L^\i_v}(t) \ls \ep^{\f 34} (1+t)^{-\f 32+|\bt|}
\end{equation}
and
$$\|\vb^{-(2+\gamma)}\rd_x^\alp \rd_v^\bt Y^\sigma (\bar{a}_{ij} v_i v_j)\|_{L^2_x L^\i_v}(t) \ls \ep^{\f 34} (1+t)^{-\f 32+|\bt|}.$$
If $|\alp|+|\bt|+|\sigma|\leq \Mm$ \underline{and $\max\{|\alp|,|\bt|\}\geq 1$}, then
\begin{equation}\label{prop:ab.L2.weighted.improved}
\max_j \|\vb^{-\max\{2+\gamma,1\}}\rd_x^\alp \rd_v^\bt Y^\sigma (\bar{a}_{ij} v_i)\|_{L^2_x L^\i_v + L^2_x L^{p_*}_v}(t) \ls \ep^{\f 34} t^{-1}(1+t)^{-\min\{\f 65,3+\gamma\}+|\bt|}.
\end{equation}
where $p_{*}$ is as in Definition~\ref{p*.def}.
\end{proposition}
\begin{proof}
The first two inequalities follow from combining \eqref{a.0.2} and \eqref{a.0.3} in Proposition~\ref{prop:a.expressions} with \eqref{eq:L2.decay.notsing} in Proposition~\ref{prop:main.L2.decay}.

We now prove \eqref{prop:ab.L2.weighted.improved}. First we consider the $|\bt|\geq 1$ case. Using \eqref{a.1.2} in Proposition~\ref{prop:a.expressions}, \eqref{eq:L2.decay.notsing} in Proposition~\ref{prop:main.L2.decay} and \eqref{eq:gamma.Lp*} in Proposition~\ref{prop:L2.p*}, we obtain
\begin{equation}\label{prop:ab.L2.weighted.improved.v}
\begin{split}
&\: \max_j \|\vb^{-\max\{2+\gamma,1\}}\rd_x^\alp \rd_v^\bt Y^\sigma (\bar{a}_{ij} v_i)\|_{L^2_x L^\i_v + L^2_x L^{p_*}_v}(t) \\
\ls &\: \ep^{\f 34} (1+t)^{-\f 32+|\bt|-1} + \ep^{\f 34}(1+t)^{-\min\{\f 65, 3+\gamma\}+|\bt|-1} \\
\ls &\: \ep^{\f 34} (1+t)^{-\min\{\f{11}5,4+\gamma\}+|\bt|},
\end{split}
\end{equation}
which implies \eqref{prop:ab.L2.weighted.improved} when $|\bt|\geq 1$.

For the $|\alp|\geq 1$ case, using $\rd_x = t^{-1} Y -t^{-1}\rd_v$, we have the pointwise bound
$$|\rd_x^\alp \rd_v^\bt Y^\sigma (\bar{a}_{ij} v_i)|(t,x,v) \ls t^{-1}\sum_{\substack{|\alp'|= |\alp|-1\\ |\sigma'|= |\sigma|+1}} |\rd_x^{\alp'} \rd_v^{\bt} Y^{\sigma'} (\bar{a}_{ij} v_i)|(t,x,v) + t^{-1} \sum_{\substack{|\alp'|= |\alp|-1\\ |\bt'|= |\bt|+1}} |\rd_x^{\alp'} \rd_v^{\bt'} Y^{\sigma} (\bar{a}_{ij} v_i)|(t,x,v).$$
The desired estimate \eqref{prop:ab.L2.weighted.improved} (in the $|\alp|\geq 1$ case) then follows from \eqref{prop:ab.Lw.weighted.basic} and \eqref{prop:ab.L2.weighted.improved.v} applied to the first and second term respectively. \qedhere
\end{proof}

\subsubsection{$L^2_x$ estimates for $\bar{c}$ and its derivatives}

\begin{proposition}\label{prop:cb.L2}
For $p_{**}\in [2,\infty]$ as in Definition~\ref{p*.def}, if $|\alp|+|\bt|+|\sigma|\leq \Mm$, then for $t\in [0,T_{Boot})$,
$$\|\rd_x^\alp \rd_v^\bt Y^\sigma \bar{c} \|_{L^2_x L^{p_{**}}_v}(t) \ls \ep^{\f 34} (1+t)^{-\min\{\f 65, 3+\gamma\}+|\bt|}.$$
\end{proposition}
\begin{proof}
This is an easy consequence of \eqref{eq:gamma.Lp**} in Proposition~\ref{prop:L2.p*} since
\begin{equation*}
\begin{split}
\|\rd_x^\alp \rd_v^\bt Y^\sigma \bar{c} \|_{L^2_x L^{p_{**}}_v}(t) \ls \left\| \int_{\mathbb R^3} |v-v_*|^{\gamma} |\rd_x^\alp \rd_v^\bt Y^\sigma f|(t,x,v_*)\, dv_* \right\|_{L^2_x L^{p_{**}}_v} \ls \ep^{\f 34}(1+t)^{-\min\{\f 65, 3+\gamma\}+|\bt|}.
\end{split}
\end{equation*}
\end{proof}

\section{The maximum principle argument and the $L^\i_xL^\i_v$ estimates}\label{sec:MP}

We continue to work under the assumptions of Theorem~\ref{thm:BA}.

In this section, we prove $L^\i_xL^\i_v$ bounds for $g$ and its derivatives. These estimates are based on an application of the maximum principle In the process, we need to obtain a hierarchy of estimates in a descent scheme; see Section~\ref{sec:descent.scheme}. By the end of the section, we will have improved in particular the constants in the bootstrap assumptions \eqref{BA.sim.1} and \eqref{BA.sim.2} (see~Proposition~\ref{sec:recover.Li}).

This section is structured as follows. First, in \textbf{Section~\ref{sec:prelim.Li.Sobolev}}, we prove some preliminary $L^\infty_x L^\infty_v$ estimates, which are Sobolev-embedding based (and are by themselves too weak to close the argument). In \textbf{Section~\ref{sec:max.prin}}, we derive a general maximum principle for linear inhomogeneous equation that is suitable in our setting. We then introduce our hierarchy of estimates in \textbf{Section~\ref{sec:hierarchy}}. In the same section, we initiate an induction argument aim at proving this hierarchy of estimates. In the next few subsection, the goal will be to use the maximum principle in Section~\ref{sec:max.prin} in the context of the induction argument introduced in Section~\ref{sec:hierarchy}. This consists of a few steps. (a) In \textbf{Section~\ref{sec:est.inho.Li}}, we classify the different types of inhomogeneous terms arising in the equation for $g$ and its derivatives. (b) In \textbf{Section~\ref{sec:Lierror}}, we then control the error terms that we classified in Section~\ref{sec:est.inho.Li}. (c) In \textbf{Section~\ref{sec:Li.everything}}, we put together the bounds from (a) and (b) above to conclude the induction. Finally, we end the section with \textbf{Section~\ref{sec:recover.Li}} in which we improve the constants in the bootstrap assumptions \eqref{BA.sim.1} and \eqref{BA.sim.2}.

\subsection{Preliminary $L^\infty_x L^\infty_v$ estimates}\label{sec:prelim.Li.Sobolev}

We begin with some preliminary $L^\infty_x L^\infty_v$ estimates (see already Proposition~\ref{prop:prelim.Linfty}), which are completely based on Sobolev embedding. These estimates are not optimal in either $t$ or $\vb$.

\begin{lemma}\label{lem:stupid.Sobolev.embedding}
Let $h:[0,T_{\mathrm{Boot}})\times \mathbb R^3\times \mathbb R^3\to \mathbb R$ be a $C^\infty$ function. Then for every $t\in [0,T_{Boot})$, the following estimate holds:
$$\|h\|_{L^\i_x L^\i_v} \ls \left(\| h\|_{L^2_x L^2_v} + \sum_{|\alp|=4} \|\rd_x^{\alp} h\|_{L^2_x L^2_v}\right)^{\f 58} \left(\sum_{|\bt|= 4} \|\rd_v^\bt h\|_{L^2_x L^2_v}\right)^{\f 38}.$$
\end{lemma}
\begin{proof}
Without loss of generality, we assume that $h\not\equiv 0$ (for otherwise the estimate is trivial). 

Standard Sobolev embedding in $\mathbb R^6$ gives
$$\|h\|_{L^\i_x L^\i_v} \ls \sum_{|\alp|+|\bt|+|\sigma|\leq 4} \|\rd_x^{\alp} \rd_v^\bt h\|_{L^2_x L^2_v}.$$
An easy argument (for instance using Plancherel's theorem) allows us to control the RHS with only the lowest and the highest derivatives, i.e.
\begin{equation}\label{scaled.6D}
\|h\|_{L^\i_x L^\i_v} \ls \| h\|_{L^2_x L^2_v} + \sum_{|\alp|=4} \|\rd_x^{\alp} h\|_{L^2_x L^2_v}+ \sum_{|\bt|= 4} \|\rd_v^\bt h\|_{L^2_x L^2_v}.
\end{equation}
We scale $h$ in the $v$-variable, i.e.~introduce, for $\lambda>0$,
$$h_\lambda(x,v) := h(x, \lambda^{-1}v).$$
One then computes that
$$\|h_\lambda \|_{L^\i_x L^\i_v} = \|h\|_{L^\i_x L^\i_v},\quad \| h_{\lambda} \|_{L^2_x L^2_v} = \lambda^{\f 32}\| h\|_{L^2_x L^2_v},$$
$$\sum_{|\alp|=4} \|\rd_x^{\alp} h_{\lambda} \|_{L^2_x L^2_v} = \lambda^{\f 32} \sum_{|\alp|=4} \|\rd_x^{\alp} h\|_{L^2_x L^2_v},\quad \sum_{|\bt|= 4} \|\rd_v^\bt h_\lambda\|_{L^2_x L^2_v} = \lambda^{-\f 52} \sum_{|\bt|= 4} \|\rd_v^\bt h\|_{L^2_x L^2_v}.$$
Applying \eqref{scaled.6D} to $h_\lambda$ and using the above computations then imply that for any $\lambda > 0$,
\begin{equation}\label{scaled.6D.2}
\|h\|_{L^\i_x L^\i_v} \ls \lambda^{\f 32} \left(\| h\|_{L^2_x L^2_v} + \sum_{|\alp|=4} \|\rd_x^{\alp} h\|_{L^2_x L^2_v}\right) + \lambda^{-\f 52} \sum_{|\bt|= 4} \|\rd_v^\bt h\|_{L^2_x L^2_v}.
\end{equation}
Let $\lambda = \left(\| h\|_{L^2_x L^2_v} + \sum_{|\alp|=4} \|\rd_x^{\alp} h\|_{L^2_x L^2_v}\right)^{-\f 14} \left(\sum_{|\bt|= 4} \|\rd_v^\bt h\|_{L^2_x L^2_v}\right)^{\f 14}$. (We can do this since $h\not\equiv 0$.) By \eqref{scaled.6D.2},
$$\|h\|_{L^\i_x L^\i_v} \ls \left(\| h\|_{L^2_x L^2_v} + \sum_{|\alp|=4} \|\rd_x^{\alp} h\|_{L^2_x L^2_v}\right)^{\f 58} \left(\sum_{|\bt|= 4} \|\rd_v^\bt h\|_{L^2_x L^2_v}\right)^{\f 38},$$
as claimed.
\end{proof}

We apply Lemma~\ref{lem:stupid.Sobolev.embedding} in our context to estimate the derivative of $g$.
\begin{proposition}\label{prop:prelim.Linfty}
For $|\alp|+|\bt|+|\sigma|\leq \Mm-4$,
$$\|\wb^{\Mm+5-|\sigma|}\rd_x^\alp \rd_v^\bt Y^\sigma g\|_{L^\i_x L^\i_v}(t) \ls \ep^{\f 34}(1+t)^{\f 32+|\bt|}.$$
\end{proposition}
\begin{proof}
The goal is to apply Lemma~\ref{lem:stupid.Sobolev.embedding} to 
\begin{equation}\label{special.h}
h = \wb^{\Mm+5-|\sigma|} (\rd_x^\alp \rd_v^\bt Y^\sigma g).
\end{equation}
For the rest of the proof, we fix $h$ as in \eqref{special.h}. We now compute the derivatives of $h$ (in terms of weighted derivatives of $g$). The $\rd_x$ derivatives are easier to compute: since $|\rd_x^{\alp'} \wb^{\Mm+5-|\sigma|}| \ls \wb^{\Mm+5-|\sigma|-|\alp'|}$ for $|\alp'|\leq 4$, we have
\begin{equation}\label{Sobolev.x.derivative}
\begin{split}
\sum_{|\alp'| = 4} |\rd_x^{\alp'} h| \ls &\: \sum_{|\alp''|\leq |\alp|+4}\wb^{\Mm+5-|\sigma|-|\alp''|} |\rd_x^{\alp''} \rd_v^\bt Y^\sigma g|\\
\ls &\: \sum_{|\alp''|\leq |\alp|+4}\wb^{\Mm+5-|\sigma|} |\rd_x^{\alp''} \rd_v^\bt Y^\sigma g|.
\end{split}
\end{equation}
For the $\rd_v$ derivative, note that when $\rd_v$ acts on $\wb^{\Mm+5-|\sigma|}$, we get a power of $t$, i.e.~$|\rd_v^{\bt'} \wb^{\Mm+5-|\sigma|}| \ls t^{|\bt'|}\wb^{\Mm+5-|\sigma|-|\bt'|}$ for $|\bt'|\leq 4$. Hence,
\begin{equation}\label{Sobolev.v.derivative}
\begin{split}
\sum_{|\bt'| = 4} |\rd_v^{\bt'} h| \ls &\: \sum_{|\bt''|+|\bt'''|\leq |\bt|+4} t^{|\bt'''|}\wb^{\Mm+5-|\sigma|-|\bt'''|} |\rd_x^{\alp} \rd_v^{\bt''} Y^\sigma g| \\
\ls &\: \sum_{|\bt''|+|\bt'''|\leq |\bt|+4} t^{|\bt'''|}\wb^{\Mm+5-|\sigma|} |\rd_x^{\alp} \rd_v^{\bt''} Y^\sigma g|.
\end{split}
\end{equation}
Applying Lemma~\ref{lem:stupid.Sobolev.embedding} to $h$ and using \eqref{Sobolev.x.derivative}, \eqref{Sobolev.v.derivative} and the bootstrap assumption \eqref{BA}, we obtain
\begin{equation*}
\begin{split}
&\: \|\wb^{\Mm+5-|\sigma|}(\rd_x^\alp \rd_v^\bt Y^\sigma g)\|_{L^\i_x L^\i_v}(t) = \|h\|_{L^\i_xL^\i_v}(t)\\
\ls &\: (\| h\|_{L^2_x L^2_v}(t) + \sum_{|\alp'|=4} \|\rd_x^{\alp'} h\|_{L^2_x L^2_v}(t))^{\f 58} (\sum_{|\bt'|= 4} \|\rd_v^{\bt'} h\|_{L^2_x L^2_v}(t) )^{\f 38}\\
\ls &\: (\sum_{|\alp''|\leq |\alp|+4} \|\wb^{\Mm+5-|\sigma|}(\rd_x^{\alp''} \rd_v^\bt Y^\sigma g)\|_{L^2_x L^2_v}(t))^{\f 58}\\
&\:\quad \times(\sum_{|\bt''|+|\bt'''|\leq |\bt|+4} t^{|\bt'''|} \|\wb^{\Mm+5-|\sigma|}(\rd_x^\alp \rd_v^{\bt''} Y^\sigma g)\|_{L^2_x L^2_v}(t))^{\f 38} \\
\ls &\: \ep^{\f 34}(1+t)^{\f{5|\bt|}{8}}(\sum_{|\bt''|+|\bt'''|\leq |\bt|+4} t^{\f{3|\bt'''|}{8}}(1+t)^{\f {3|\bt''|}8}) \ls 
\ep^{\f 34}(1+t)^{\f{5|\bt|}{8} + \f{3|\bt|}{8} + \f{3\cdot 4}{8}} = \ep^{\f 34}(1+t)^{|\bt|+\f{3}{2}},
\end{split}
\end{equation*}
as claimed. \qedhere
\end{proof}

\subsection{The maximum principle}\label{sec:max.prin}

The goal of this subsection is to establish a general maximum principle; see Proposition~\ref{prop:max.prin}. Before we precisely state the maximum principle, let us already give some remarks:
\begin{enumerate}
\item Despite the various technicalities, the main point of Proposition~\ref{prop:max.prin} is to get a bound for the solution $h$ to the equation \eqref{h.equation}. The bound that we derive (see~\eqref{H.bound} and \eqref{max.prin.conclusion}) is such that we either gain two powers in $\vb$ and lose $(1+\de)$-power in $(1+t)$, or we have no gain in $\vb$ and lose exactly one power of $(1+t)$.

Such a statement is straightforward if \eqref{h.equation} is replaced by the transport equation $\rd_t h + v_i\rd_{x_i} h + \f{\de d_0}{(1+t)^{1+\de}}\vb^2 h = H$. The key point of Proposition~\ref{prop:max.prin} is therefore to ensure that the term $-\bar{a}_{ij} \rd_{v_i v_j}^2 h$ on the LHS of \eqref{h.equation} does not destroy the transport estimate.
\item In order to carry out the argument, we need some a priori control on $h$; see \eqref{h.max.a.priori}. This is a very weak bound which can have very bad dependence on $t$ (compare this with the conclusion of Proposition~\ref{prop:max.prin}, which gives a much stronger bound), but importantly for every fixed $t$ we need the estimate to be uniform in $x$ and $v$ so as to control various cutoffs we introduce.
\item In additional to an estimate on $h$, we also need an a priori bound on $\rd_v h$; see \eqref{dvh.bound}. This is a technical condition necessary to carry out a cut-off argument. The bounds that we need are sufficiently weak to be consistent with the descent scheme (see~Section~\ref{sec:induction}).
\end{enumerate}

The following is our main general maximum principle. The reader can keep in mind that Proposition~\ref{prop:max.prin} will be applied for $h$ being appropriate derivatives of $g$.
\begin{proposition}[Maximum principle]\label{prop:max.prin}
Let $N\in \mathbb N$ with $N\leq \Mm+5$. Let $h:[0,T_{\mathrm{Boot}})\times \mathbb R^3\times \mathbb R^3\to \mathbb R$ be a $C^\infty$ function such that the following four conditions hold for some $p_H\in [1,2)$, $r_H\geq -1-\de$, and $C_{H}\geq 1$:
\begin{enumerate}
\item $h$ is bounded on compact subintervals of $[0,T_{\mathrm{Boot}})$: For every $T\in [0,T_{\mathrm{Boot}})$, there exists a constant $C_{T}>0$ such that for every $(t,x,v)\in [0,T]\times \mathbb R^3\times \mathbb R^3$,
\begin{equation}\label{h.max.a.priori}
\wb^{N}|h(t,x,v)|\leq C_{T}.
\end{equation}
\item $\rd_{v_i} h$ satisfy the following estimate: for $i=1,2,3$ and for every $(t,x,v)\in [0,T_{\mathrm{Boot}})\times \mathbb R^3\times \mathbb R^3$,
\begin{equation}\label{dvh.bound}
\wb^{N}|\rd_{v_i} h(t,x,v)| \leq C_{H} \ep^{\f 34} \vb^{\min\{p_H-2-\gamma,p_H-1\}}(1+t)^{r_H+2+\min\{2+\gamma,1\}}.
\end{equation}
\item $h$ satisfies the following equation:
\begin{equation}\label{h.equation}
\rd_t h + v_i\rd_{x_i} h + \f{\de d_0}{(1+t)^{1+\de}}\vb^2 h- \bar{a}_{ij} \rd_{v_i v_j}^2 h = H,
\end{equation}
where $H:[0,T_{Boot})\times \mathbb R^3\times \mathbb R^3$  is a smooth function satisfying the bound
\begin{equation}\label{H.bound}
\wb^{N}|H|(t,x,v) \leq 
\begin{cases}
C_{H}\ep \vb^{p_H}(1+t)^{r_H} + C_H \ep \vb^{p_H-2}(1+t)^{r_H+\de} & \mbox{if $r_H+\de\geq 0$}\\
C_{H}\ep \vb^{p_H}(1+t)^{r_H} & \mbox{if $r_H+\de \in [-1,0)$}
\end{cases}.
\end{equation}
\item The initial data for $h$ satisfy the bound
\begin{equation}\label{h.data.bound}
\vb\xb^N|h|(0,x,v) \leq C_{H}\epsilon.
\end{equation}
\end{enumerate}

Then for $\ep_0$ is sufficiently small (depending only on $\gamma$ and $d_0$, and in particular independent of $C_{T}$, $C_{H}$, $p_H$ and $r_H$ above), the following estimate holds for all $(t,x,v)\in [0,T_{Boot})\times \mathbb R^3\times \mathbb R^3$:
\begin{equation}\label{max.prin.conclusion}
\begin{split}
|\wb^{N} h|(t,x,v) \leq (3+\f{6}{d_0\de}) C_H \ep \vb^{p_H-2}(1+t)^{r_H+1+\de}.
\end{split}
\end{equation}
\end{proposition}
\begin{proof}

\pfstep{Step~1: Deriving an equation with weights} Define 
\begin{equation}\label{def:wN}
w_N(t,x,v):= \vb\wb^N e^{-d(t)\vb^2},
\end{equation}
and
\begin{equation}\label{def:thN}
\widetilde{h}_N(t,x,v) := (w_N h)(t,x,v).
\end{equation}

We now derive an equation for $\widetilde{h}_N(t,x,v)$ (see already \eqref{h.tilde.computations}). To simplify the notations, let us suppress the explicit dependence on $(t,x,v)$ when there is no risk of confusion.

We first compute
\begin{equation*}
\begin{split}
w_N \rd_{v_i} h = \rd_{v_i} \widetilde{h}_N - (\rd_{v_i} \log w_N)\widetilde{h}_N,
\end{split}
\end{equation*}
which implies
\begin{equation}\label{d2hwithweights}
\begin{split}
&\: w_N \rd^2_{v_iv_j} h \\
=&\: \rd_{v_i} (\rd_{v_j} \widetilde{h}_N - (\rd_{v_j} \log w_N)\widetilde{h}_N) - (\rd_{v_i} \log w_N)(\rd_{v_j} \widetilde{h}_N - (\rd_{v_j} \log w_N)\widetilde{h}_N)\\
=&\: \rd^2_{v_i v_j} \widetilde{h}_N - (\rd_{v_i}\log w_N)(\rd_{v_j}\widetilde{h}_N) - (\rd_{v_j}\log w_N)(\rd_{v_i}\widetilde{h}_N) - [(\rd^2_{v_i v_j}\log w_N) -(\rd_{v_i}\log w_N)(\rd_{v_j}\log w_N)]\widetilde{h}_N\\
=&\: \rd^2_{v_i v_j} \widetilde{h}_N - (\rd_{v_i}\log w_N)w_N(\rd_{v_j}h)  - (\rd_{v_j}\log w_N)w_N(\rd_{v_i}h) - [(\rd^2_{v_i v_j}\log w_N) +(\rd_{v_i}\log w_N)(\rd_{v_j}\log w_N)]\widetilde{h}_N.
\end{split}
\end{equation}
On the other hand, we have
$$(\rd_t + v_i\rd_{x_i})(\log w_N) = \f{d_0\de}{(1+t)^{1+\de}}\vb^2$$
so that
\begin{equation}\label{transporthwithweights}
w_N (\rd_t + v_i\rd_{x_i})h  = (\rd_t + v_i\rd_{x_i})\widetilde{h}_N -\f{d_0\de}{(1+t)^{1+\de}}\vb^2 \widetilde{h}_N.
\end{equation}

By \eqref{h.equation}, \eqref{d2hwithweights} and \eqref{transporthwithweights}, and using $\bar{a}_{ij} = \bar{a}_{ji}$, we obtain
\begin{equation}\label{h.tilde.computations}
\begin{split}
&\: \rd_t \widetilde{h}_N + v_i\rd_{x_i} \widetilde{h}_N -\bar{a}_{ij} \rd_{v_i v_j}^2 \widetilde{h}_N \\
= &\: -2\bar{a}_{ij} w_N(\rd_{v_i} h)(\rd_{v_j} \log w_N)-\bar{a}_{ij} \widetilde{h}_N [\rd^2_{v_i v_j} \log w_N+(\rd_{v_i}\log w_N)(\rd_{v_j}\log w_N)] + w_N H.
\end{split}
\end{equation}

\pfstep{Step~2: Cutoff at infinity} In order to avoid the difficulty with applying the maximum principle in non-compact domains, we cut off the function $h$. Introduce a smooth cutoff function $\chi:\mathbb R\to \mathbb R_{\geq 0}$ such that $\chi(x) = 1$ for $|x|\leq 1$, $\chi(x) = 0$ for $|x|\geq 2$ and $\|\chi'\|_{L^\i},\,\|\chi''\|_{L^\i} \leq 10$. We cut off $\widetilde{h}_N$ and define
$$\widetilde{h}_{N,R}(t,x,v) = \chi(\f{|x|^2}{R^6}) \chi(\f{|v|^2}{R^2}) \widetilde{h}_N(t,x,v)$$
for $R>1$ large and to be chosen. (Note that $|x|$ is allowed to be much larger, with $|x|\ls R^3$, as opposed to $|v|$, for which the cutoff only allows $|v|\ls R$.)

\textbf{Fix an arbitrary $T\in [0,T_{Boot})$.} We will allow our choice of $R$ in the cutoffs to depend on $T$. In order to emphasize that the implicit constant depends on $T$ (in addition to $d_0$ and $\gamma$), we will use the convention $\ls_T$. 

Our next goal will be to estimate $|\rd_t \widetilde{h}_{N,R} + v_i\rd_{x_i} \widetilde{h}_{N,R} -\bar{a}_{ij} \rd_{v_i v_j}^2 \widetilde{h}_{N,R}|(t,x,v)$. To this end, we will use \eqref{h.tilde.computations} and estimate all the error terms arising from differentiating the cut-off functions.

We first note the following simple estimate which we will repeatedly use. Since $d(t)\geq d_0$, it follows that $e^{-d(t)\vb^2}\vb\ls 1$. Therefore, by \eqref{h.max.a.priori}, we have
\begin{equation}\label{widetilde.h}
|\widetilde{h}_N|(t,x,v) \ls_T 1.
\end{equation}

Now we compute
$$v_i\rd_{x_i}\widetilde{h}_{N,R} = \chi(\f{|x|^2}{R^6}) \chi(\f{|v|^2}{R^2}) v_i \rd_{x_i}\widetilde{h}_N + 2 (v\cdot x) R^{-6} \chi'(\f{|x|^2}{R^6}) \chi(\f{|v|^2}{R^2}) \widetilde{h}_N.$$
Now on the support of the cutoff functions, we have $|v\cdot x| \ls R\cdot R^3 = R^4$. Combining this with \eqref{widetilde.h}, we obtain
\begin{equation}\label{cutoff.term.1}
|v_i\rd_{x_i}\widetilde{h}_{N,R}(t,x,v) - \chi(\f{|x|^2}{R^6}) \chi(\f{|v|^2}{R^2}) v_i \rd_{x_i}\widetilde{h}_N(t,x,v)| \ls_{T} R^{-2}.
\end{equation}

On the other hand, we compute
\begin{equation}\label{cutoff.term.2}
\begin{split}
&\: \bar{a}_{ij} \rd^2_{v_i v_j} \widetilde{h}_{N,R} \\
= &\: \bar{a}_{ij} \chi(\f{|x|^2}{R^6}) \chi(\f{|v|^2}{R^2}) \rd^2_{v_i v_j} \widetilde{h}_N + \bar{a}_{ij} \chi(\f{|x|^2}{R^6}) \big(2 \de_{ij} R^{-2} \chi'(\f{|v|^2}{R^2}) + 4 v_i v_j R^{-4} \chi''(\f{|v|^2}{R^2}) \big) \widetilde{h}_N \\
&\: + 4 \bar{a}_{ij} \chi(\f{|x|^2}{R^6}) \chi'(\f{|v|^2}{R^2}) v_i R^{-2} \rd_{v_j} \widetilde{h}_N\\
= &\: \bar{a}_{ij} \chi(\f{|x|^2}{R^6}) \chi(\f{|v|^2}{R^2}) \rd^2_{v_i v_j} \widetilde{h}_N + \bar{a}_{ij} \chi(\f{|x|^2}{R^6}) \big(2 \de_{ij} R^{-2} \chi'(\f{|v|^2}{R^2}) + 4 v_i v_j R^{-4} \chi''(\f{|v|^2}{R^2}) \big) \widetilde{h}_N \\
&\: + 4 \bar{a}_{ij} \chi(\f{|x|^2}{R^6}) \chi'(\f{|v|^2}{R^2}) v_i R^{-2} (w_N\rd_{v_j} h + (\rd_{v_j}\log w_N)\widetilde{h}_N).
\end{split}
\end{equation}
We now control the difference $\bar{a}_{ij} \rd^2_{v_i v_j} \widetilde{h}_{N,R} - \bar{a}_{ij} \chi(\f{|x|^2}{R^6}) \chi(\f{|v|^2}{R^2}) \rd^2_{v_i v_j} \widetilde{h}_N$ using \eqref{cutoff.term.2}. First, by \eqref{widetilde.h}, Propositions~\ref{prop:ab.Li.1} and \ref{prop:ab.Li.weighted}, we obtain
\begin{equation}\label{cutoff.term.3}
\begin{split}
&\: |\bar{a}_{ij} \chi(\f{|x|^2}{R^6}) \big(2 \de_{ij} R^{-2} \chi'(\f{|v|^2}{R^2}) + 4 v_i v_j R^{-4} \chi''(\f{|v|^2}{R^2}) \big) \widetilde{h}_N| \ls_T R^{\max\{2+\gamma,1\}} R^{-2} \ls_T R^{\max\{-1,\gamma\}}.
\end{split}
\end{equation}
To control the term with $\rd_{v_j}h$, we use \eqref{dvh.bound} and Proposition~\ref{prop:ab.Li.weighted} to obtain
\begin{equation}\label{cutoff.term.4}
\begin{split}
&\: |\bar{a}_{ij} \chi(\f{|x|^2}{R^6}) \chi'(\f{|v|^2}{R^2}) v_i R^{-2} w_N\rd_{v_j} h| \\
\ls_T &\: C_H \ep^{\f 74} R^{\max\{2+\gamma,1\}} R^{-2}R^{\min\{p_H-2-\gamma,p_H-1\}} \ls_T C_H\ep^{\f 74} R^{p_H-2}.
\end{split}
\end{equation}
To handle the remaining term in \eqref{cutoff.term.2}, we compute
\begin{equation}\label{eq:dw}
\rd_{v_i}\log w_N = \f 12\rd_{v_i}(\log (1+|v|^2)) + \f N2\rd_{v_i}(\log (1+|x-tv|^2)) -2d(t)v_i = \f{v_i}{1+|v|^2} - \f{N t(x-tv)_i}{1+|x-tv|^2}-2d(t)v_i.
\end{equation}
Using \eqref{widetilde.h}, \eqref{eq:dw}, Propositions~\ref{prop:ab.Li.1} and \ref{prop:ab.Li.weighted}, we obtain
\begin{equation}\label{cutoff.term.5}
\begin{split}
&\: |4 \bar{a}_{ij} \chi(\f{|x|^2}{R^6}) \chi'(\f{|v|^2}{R^2}) v_i R^{-2} (\rd_{v_j}\log w_N)\widetilde{h}_N| \\
\ls_T &\: (\max_j\bar{a}_{ij}v_i) \chi(\f{|x|^2}{R^6}) \chi'(\f{|v|^2}{R^2}) R^{-2} |\widetilde{h}_N| + \bar{a}_{ij}v_iv_j \chi(\f{|x|^2}{R^6}) \chi'(\f{|v|^2}{R^2}) R^{-2} |\widetilde{h}_N| \\
\ls_T &\: R^{\max\{2+\gamma,1\}-2} \ls_T R^{\max\{-1,\gamma\}}.
\end{split}
\end{equation}
Combining \eqref{cutoff.term.3}, \eqref{cutoff.term.4} and \eqref{cutoff.term.5} and plugging the estimates into \eqref{cutoff.term.2}, we obtain
\begin{equation}\label{cutoff.term.6}
|\bar{a}_{ij} \rd^2_{v_i v_j} \widetilde{h}_{N,R}(t,x,v) - \bar{a}_{ij} \chi(\f{|x|^2}{R^6}) \chi(\f{|v|^2}{R^2}) \rd^2_{v_i v_j} \widetilde{h}_N(t,x,v)| \ls_{C_H,T} R^{\max\{-1,\gamma,p_H-2\}}.
\end{equation}

Combining \eqref{cutoff.term.1} and \eqref{cutoff.term.6}, we get that for any $T\in [0,T_{Boot})$, there exists $C'_T>0$ (depending on $T$, $C_H$ in addition to $d_0$ and $\gamma$) such that for every $(t,x,v) \in [0,T)\times B(0,\sqrt{2}R^3)\times B(0,\sqrt{2}R)$
\begin{equation}\label{tildeh.main.est}
\begin{split}
 |\rd_t \widetilde{h}_{N,R} + v_i\rd_{x_i} \widetilde{h}_{N,R} -\bar{a}_{ij} \rd_{v_i v_j}^2 \widetilde{h}_{N,R}|(t,x,v) 
\leq &\: \chi(\f{|x|^2}{R^6}) \chi(\f{|v|^2}{R^2})\times |\mbox{RHS of \eqref{h.tilde.computations}}| + C'_T R^{\min\{-1,\gamma, p_H-2\}}.
\end{split}
\end{equation}

At this point, we fix a sequence of cutoff parameters $\{R_n\}_{n=1}^\infty$. For $n\in \mathbb N$, let $T_n = T_{Boot} - \f 1n$. We define $R_n > 0$ so that
\begin{equation}\label{Rn.condition}
C'_{T_n}R_n^{\min\{-1,\gamma, p_H-2\}} \leq \f {1}{4n} \f{d_0\de}{(1+T_{Boot})^{1+\de}}\vb^2.
\end{equation}
Such a sequence of $R_n$ exists since $\gamma<0$ and $p_H<2$. We assume moreover without loss of generality that $R_n$ is increasing and $R_n \to +\infty$ so that 
\begin{equation}\label{exhaustion.1}
[0,T_n]\times B(0,R^3_n) \times B(0,R_n)\subset [0,T_{n+1}]\times B(0,R^3_{n+1}) \times B(0,R_{n+1}),\,\forall n\in \mathbb N
\end{equation}
and
\begin{equation}\label{exhaustion.2}
\cup_{n=1}^\infty \left([0,T_n]\times B(0,R^3_n) \times B(0,R_n) \right) = [0,T_{Boot})\times \mathbb R^3\times \mathbb R^3.
\end{equation}

\pfstep{Step~3: Continuity argument and estimating $\chi(\f{|x|^2}{R_n^6}) \chi(\f{|v|^2}{R_n^2})\times (\mbox{RHS of \eqref{h.tilde.computations}})$} Let $R_n$ be as in the previous step (so that \eqref{Rn.condition}, \eqref{exhaustion.1} and \eqref{exhaustion.2} hold). Our goal in this step is to bound $\chi(\f{|x|^2}{R_n^6}) \chi(\f{|v|^2}{R_n^2})\times (\mbox{RHS of \eqref{h.tilde.computations}})$. To carry out these estimates, we introduce a continuity argument.

Define $\widetilde{d}:[0,+\infty)\to \mathbb R$ by
\begin{equation}\label{def:td}
\widetilde{d}(t) := d_0(1-(1+t)^{-\de}),
\end{equation}
and define
$T'_n \in [0,T_n]$ by
\begin{equation}\label{T'.def}
\begin{split}
T'_n := \sup\{t\in [0,T_n]\colon |\widetilde{h}_{N,R_n} (s,x,v)|\leq  (6+\f{12}{d_0\de})C_H\ep &\,e^{-d(s)\vb^2}\vb^{p_H-1}(1+s)^{r_H+1+\de} + \f{2e^{\widetilde{d}(s)\vb^2}}{n},\\
&\:\forall (s,x,v)\in [0,t]\times B(0,\sqrt{2}R^3_n) \times B(0,\sqrt{2}R_n)\}.
\end{split}
\end{equation}
By \eqref{h.data.bound} and the continuity of $\widetilde{h}_{N,R_n}$ (and the fact that we are only considering a compact set), $T_n'>0$. Moreover, again using the continuity of $\widetilde{h}_{N,R_n}$, the following estimate holds for $(t,x,v)\in [0,T'_n]\times B(0,\sqrt{2}R^3_n) \times B(0,\sqrt{2}R_n)$:
\begin{equation}\label{eq:BA.h}
|\widetilde{h}_{N,R_n} (t,x,v)|\leq (6+\f{12}{d_0\de})C_H\ep e^{-d(t)\vb^2}\vb^{p_H-1}(1+t)^{r_H+1+\de} + \f{2e^{\widetilde{d}(t)\vb^2}}{n}.
\end{equation}

From now on until Step~7, we will carry out our estimates using the bound \eqref{eq:BA.h}. The goal will be to prove an estimate that is better than \eqref{eq:BA.h} so that we conclude by continuity that $T_n'=T_n$.

In the remainder of Step~3, we bound $\chi(\f{|x|^2}{R_n^6}) \chi(\f{|v|^2}{R_n^2})\times |\mbox{RHS of \eqref{h.tilde.computations}}|$ using \eqref{eq:BA.h}. We first carry out preliminary calculations in Step~3(a), and then in Steps~3(b) to 3(e) we consider each of the terms on RHS of \eqref{h.tilde.computations}. Finally, we will combine everything in Step~3(f) and show that under \eqref{eq:BA.h}, the estimate \eqref{cutoff.RHS} holds.

\pfstep{Step~3(a): Preliminary computations} 
Using \eqref{eq:dw}, we compute
\begin{equation}\label{eq:d2w}
\begin{split}
&\:\rd^2_{v_i v_j}\log w_N = \rd_{v_j}\left( \f{v_i}{1+|v|^2} - \f{N t(x-tv)_i}{1+|x-tv|^2}-2d(t)v_i\right)\\
 = &\:\f{\de_{ij}(1+|v|^2) -2v_i v_j}{(1+|v|^2)^2} - \f{N t^2\de_{ij}(1+|x-tv|^2) - 2Nt^2 (x-tv)_i (x-tv)_j}{(1+ |x-tv|^2)^2} -2d(t)\de_{ij}.
\end{split}
\end{equation}

\pfstep{Step~3(b): Estimating $2\chi(\f{|x|^2}{R_n^6}) \chi(\f{|v|^2}{R_n^2})\bar{a}_{ij} w_N(\rd_{v_i} h)(\rd_{v_j} \log w_N)$} Using \eqref{eq:dw} and bounding $\chi\leq 1$, we obtain
\begin{equation}\label{max:2a.1}
\begin{split}
&\: |2\chi(\f{|x|^2}{R_n^6}) \chi(\f{|v|^2}{R_n^2})\bar{a}_{ij} w_N(\rd_{v_i} h)(\rd_{v_j} \log w_N)|(t,x,v) \\
\ls &\: (\max_j|\bar{a}_{ij}|) w_N |\rd_{v_i}h|(t,x,v) +  |\bar{a}_{ij}v_j|w_N |\rd_{v_i} h|(t,x,v) + t \wb^{-1}(\max_j|\bar{a}_{ij}|) w_N |\rd_{v_i}h|(t,x,v).
\end{split}
\end{equation}
The first two terms in \eqref{max:2a.1} can be controlled in a similar manner: we use Propositions~\ref{prop:ab.Li.1} and \ref{prop:ab.Li.weighted} to bound $\bar{a}_{ij}$ and $\bar{a}_{ij}v_j$ respectively, and then use also \eqref{dvh.bound} to obtain
\begin{equation}\label{max:2a.2}
\begin{split}
&\: (\max_j|\bar{a}_{ij}|) w_N |\rd_{v_i}h|(t,x,v) +  |\bar{a}_{ij}v_j|w_N |\rd_{v_i} h|(t,x,v)\\
\ls &\: \ep^{\f 34}\vb^{\max\{2+\gamma,1\}}(1+t)^{-3} \cdot e^{-d(t)\vb^2}\vb \cdot C_H\ep^{\f 34} \vb^{\min\{p_H-2-\gamma,p_H-1\}}(1+t)^{r_H+2+\min\{2+\gamma,1\}} \\
\ls &\: C_H\ep^{\f 32}e^{-d(t)\vb^2}\vb^{p_H+1}(1+t)^{r_H-1+\min\{2+\gamma,1\}}.
\end{split}
\end{equation}
For the last term in \eqref{max:2a.1}, we use instead Proposition~\ref{prop:ab.Li.null.cond} to bound $\wb^{-1}|\bar{a}_{ij}|$. Combining Proposition~\ref{prop:ab.Li.null.cond} with \eqref{dvh.bound}, we obtain
\begin{equation}\label{max:2a.3}
\begin{split}
&\: t\wb^{-1}(\max_j|\bar{a}_{ij}|) w_N |\rd_{v_i}h|(t,x,v)\\
\ls &\: t\cdot \ep^{\f 34}\vb^{\max\{0,1+\gamma\}} t^{-\min\{1,2+\gamma\}} (1+t)^{-3} \cdot e^{-d(t)|v|^2}\vb\cdot C_H\ep^{\f 34}\vb^{\min\{p_H-2-\gamma,p_H-1\}}(1+t)^{r_H+2+\min\{2+\gamma,1\}} \\
\ls &\: C_H\ep^{\f 32}e^{-d(t)\vb^2}\vb^{p_H}(1+t)^{r_H}.
\end{split}
\end{equation}
Plugging \eqref{max:2a.2} and \eqref{max:2a.3} into \eqref{max:2a.1}, and estimating crudely, we obtain
\begin{equation}\label{max:2a.final}
|2\chi(\f{|x|^2}{R_n^6}) \chi(\f{|v|^2}{R_n^2})\bar{a}_{ij} w_N(\rd_{v_i} h)(\rd_{v_j} \log w_N)|(t,x,v)\ls C_H\ep^{\f 32}e^{-d(t)\vb^2}\vb^{p_H+1}(1+t)^{r_H}.
\end{equation}

\pfstep{Step~3(c): Estimating $\chi(\f{|x|^2}{R_n^6}) \chi(\f{|v|^2}{R_n^2})\bar{a}_{ij} \widetilde{h}_N \rd^2_{v_i v_j} \log w_N$}
By \eqref{eq:d2w},
\begin{equation}\label{max:2b.1}
\begin{split}
|\chi(\f{|x|^2}{R_n^6}) \chi(\f{|v|^2}{R_n^2})\bar{a}_{ij} \widetilde{h}_N \rd^2_{v_i v_j} \log w_N|(t,x,v) \ls &\: \max_{i,j}(|\bar{a}_{ij}| |\widetilde{h}_{N,R_n}|(t,x,v) +  t^2\wb^{-2}|\bar{a}_{ij}| |\widetilde{h}_{N,R_n}|(t,x,v)).
\end{split}
\end{equation}
To handle the first term in \eqref{max:2b.1}, we use Proposition~\ref{prop:ab.Li.1} and \eqref{eq:BA.h} to obtain
\begin{equation}\label{max:2b.2}
\begin{split}
&\: \max_{i,j}|\bar{a}_{ij}| |\widetilde{h}_{N,R_n}|(t,x,v) \\
\ls &\:\ep^{\f 34}\vb^{2+\gamma}(1+t)^{-3}\cdot (C_H\ep e^{-d(t)\vb^2}\vb^{p_H-1}(1+t)^{r_H+1+\de} + \f{e^{\widetilde{d}(t)\vb^2}}{n}) \\
\ls &\: C_H\ep^{\f 74}e^{-d(t)\vb^2}\vb^{p_H+1+\gamma}(1+t)^{r_H-2+\de} + \f{\ep^{\f 34}e^{\widetilde{d}(t)\vb^2}}{n}\vb^{2+\gamma}(1+t)^{-3}.
\end{split}
\end{equation}
To bound the second term in \eqref{max:2b.1}, we will in fact prove an estimate when the $t$ and $\wb$ weights are even slightly worse (since such a stronger estimate will be useful later). By Proposition~\ref{prop:ab.Li.null.cond} and \eqref{eq:BA.h},
\begin{equation}\label{max:2b.3}
\begin{split}
&\: \max_{i,j}t(1+t)\wb^{-1}|\bar{a}_{ij}| |\widetilde{h}_{N,R_n}|(t,x,v) \\
\ls &\: t(1+t)\cdot \ep^{\f 34}\vb^{\max\{0,1+\gamma\}}t^{-\min\{1,2+\gamma\}}(1+t)^{-3}\cdot (C_H\ep e^{-d(t)\vb^2}\vb^{p_H-1}(1+t)^{r_H+1+\de} + \f{e^{\widetilde{d}(t)\vb^2}}{n} ) \\
\ls &\: C_H\ep^{\f 74}e^{-d(t)\vb^2}\vb^{p_H}(1+t)^{r_H-\min\{2+\gamma,1\}+\de} + \f{\ep^{\f 34}e^{\widetilde{d}(t)\vb^2}}{n} (1+t)^{-1-\min\{2+\gamma,1\}} \vb^{\max\{0,1+\gamma\}}.
\end{split}
\end{equation}

Taking the worse bounds from \eqref{max:2b.2} and \eqref{max:2b.3} and plugging into \eqref{max:2b.1}, we obtain
\begin{equation}\label{max:2b.final}
\begin{split}
&\: |\chi(\f{|x|^2}{R_n^6}) \chi(\f{|v|^2}{R_n^2})\bar{a}_{ij} \widetilde{h}_N \rd^2_{v_i v_j} \log w_N|(t,x,v) \\
\ls &\: C_H\ep^{\f 74}e^{-d(t)\vb^2}\vb^{p_H+1}(1+t)^{r_H-\min\{2+\gamma,1\}+\de} + \f{\ep^{\f 34}e^{\widetilde{d}(t)\vb^2}}{n} (1+t)^{-1-\min\{2+\gamma,1\}} \vb^{2+\gamma} \\
\ls &\: C_H\ep^{\f 74}e^{-d(t)\vb^2}\vb^{p_H+1}(1+t)^{r_H} + \f{\ep^{\f 34}e^{\widetilde{d}(t)\vb^2}}{n} (1+t)^{-1-\de} \vb^{2},
\end{split}
\end{equation}
where in the last line we have used \eqref{def:de}.

\pfstep{Step~3(d): Estimating $\chi(\f{|x|^2}{R_n^6}) \chi(\f{|v|^2}{R_n^2})\bar{a}_{ij} \widetilde{h}_N (\rd_{v_i}\log w_N)(\rd_{v_j}\log w_N)$} By \eqref{eq:dw},
\begin{equation}\label{max:2c.1}
\begin{split}
&\: |\chi(\f{|x|^2}{R_n^6}) \chi(\f{|v|^2}{R_n^2})\bar{a}_{ij} \widetilde{h}_N (\rd_{v_i}\log w_N)(\rd_{v_j}\log w_N)|(t,x,v) \\
\ls &\: |\bar{a}_{ij} \widetilde{h}_{N,R_n} (\rd_{v_i}\log w_N)(\rd_{v_j}\log w_N)|(t,x,v) \\
\ls &\: ((\max_{i,j}\bar{a}_{ij}) + (\max_{i} \bar{a}_{ij} v_j) + \bar{a}_{ij}v_iv_j) |\widetilde{h}_{N,R_n}|(t,x,v) +  t\wb^{-1}((\max_{i,j}\bar{a}_{ij}) + (\max_{i} \bar{a}_{ij} v_j))|\widetilde{h}_{N,R_n}|(t,x,v) \\
&\: + t^2 \wb^{-2}(\max_{i,j}\bar{a}_{ij})|\widetilde{h}_{N,R_n}|(t,x,v).
\end{split}
\end{equation}
The first term on the RHS of \eqref{max:2c.1} can be estimated in a similar way as \eqref{max:2b.2}, except that we also apply Proposition~\ref{prop:ab.Li.weighted} to bound $(\max_{i} \bar{a}_{ij} v_j)$ and $\bar{a}_{ij}v_iv_j$. We then obtain a similar estimate as \eqref{max:2b.2} (after accounting for the difference in $\vb$ weights between Propositions~\ref{prop:ab.Li.1} and \ref{prop:ab.Li.weighted}), i.e.
\begin{equation}\label{max:2c.2}
\begin{split}
&\: ((\max_{i,j}\bar{a}_{ij}) + (\max_{i} \bar{a}_{ij} v_j) + \bar{a}_{ij}v_iv_j) |\widetilde{h}_{N,R_n}|(t,x,v) \\
\ls &\: C_H\ep^{\f 74}e^{-d(t)\vb^2}\vb^{\max\{p_H+1+\gamma,p_H\}}(1+t)^{r_H-2+\de} + \f{\ep^{\f 34}e^{\widetilde{d}(t)\vb^2}}{n}\vb^{\max\{2+\gamma,1\}}(1+t)^{-3}.
\end{split}
\end{equation}
For the second term on the RHS of \eqref{max:2c.1}, we first use the simple bound $(\max_{i,j}\bar{a}_{ij}) + (\max_{i} \bar{a}_{ij} v_j) \ls \max_{i,j} |\bar{a}_{ij}|\vb$ and then apply \eqref{max:2b.3} to obtain
\begin{equation}\label{max:2c.2.5}
\begin{split}
&\: t\wb^{-1}((\max_{i,j}\bar{a}_{ij}) + (\max_{i} \bar{a}_{ij} v_j))|\widetilde{h}_{N,R_n}|(t,x,v)\\
\ls &\: C_H\ep^{\f 74}e^{-d(t)\vb^2}\vb^{p_H+1}(1+t)^{r_H-\min\{2+\gamma,1\}+\de} + \f{\ep^{\f 34}e^{\widetilde{d}(t)\vb^2}}{n} (1+t)^{-1-\min\{2+\gamma,1\}} \vb^{\max\{1,2+\gamma\}}.
\end{split}
\end{equation}
The last term on the RHS of \eqref{max:2c.1} can also be controlled by \eqref{max:2b.3} so that we obtain
\begin{equation}\label{max:2c.3}
\begin{split}
&\: t^2 \wb^{-2}(\max_{i,j}\bar{a}_{ij})|\widetilde{h}_{N,R_n}|(t,x,v) \\
\ls &\: C_H\ep^{\f 74}e^{-d(t)\vb^2}\vb^{p_H}(1+t)^{r_H-\min\{2+\gamma,1\}+\de} + \f{\ep^{\f 34}e^{\widetilde{d}(t)\vb^2}}{n} (1+t)^{-1-\min\{2+\gamma,1\}} \vb^{\max\{0,1+\gamma\}}.
\end{split}
\end{equation}

We now take the worse bounds in \eqref{max:2c.2}, \eqref{max:2c.2.5} and \eqref{max:2c.3} and plug them into \eqref{max:2c.1} to obtain
\begin{equation}\label{max:2c.final}
\begin{split}
&\: |\chi(\f{|x|^2}{R_n^6}) \chi(\f{|v|^2}{R_n^2})\bar{a}_{ij} \widetilde{h}_N (\rd_{v_i}\log w_N)(\rd_{v_j}\log w_N)|(t,x,v) \\
\ls &\: C_H\ep^{\f 74}e^{-d(t)\vb^2}\vb^{\max\{p_H+1+\gamma,p_H\}}(1+t)^{r_H-\min\{2+\gamma,1\}+\de} + \f{\ep^{\f 34}e^{\widetilde{d}(t)\vb^2}}{n} (1+t)^{-1-\min\{2+\gamma,1\}} \vb^{\max\{2+\gamma,1\}}\\
\ls &\: C_H\ep^{\f 74}e^{-d(t)\vb^2}\vb^{p_H+1}(1+t)^{r_H} + \f{\ep^{\f 34}e^{\widetilde{d}(t)\vb^2}}{n} (1+t)^{-1-\de} \vb^{2},
\end{split}
\end{equation}
where in the last line we have used \eqref{def:de}.

\pfstep{Step~3(e): Estimating $w_NH$} Using \eqref{def:wN} and plugging in \eqref{H.bound}, we obtain
\begin{equation}\label{max:2d.final}
|w_N H|(t,x,v)\leq 
\begin{cases}
C_H\ep e^{-d(t)\vb^2}\vb^{p_H+1}(1+t)^{r_H} + C_H\ep e^{-d(t)\vb^2}\vb^{p_H-1}(1+t)^{r_H+\de} & \mbox{if $r_H+\de\geq 0$}\\
C_H\ep e^{-d(t)\vb^2}\vb^{p_H+1}(1+t)^{r_H}  & \mbox{if $r_H+\de\in [-1,0)$}.
\end{cases}
\end{equation}

\pfstep{Step~3(f): Putting everything together}
By \eqref{h.tilde.computations} and the estimates \eqref{max:2a.final}, \eqref{max:2b.final}, \eqref{max:2c.final} and \eqref{max:2d.final}, after choosing $\ep_0$ (and therefore also $\ep$) sufficiently small and using \eqref{def:de}, we obtain
\begin{equation}\label{cutoff.RHS}
\begin{split}
&\: \chi(\f{|x|^2}{R_n^6}) \chi(\f{|v|^2}{R_n^2})\times|\mbox{RHS of \eqref{h.tilde.computations}}|(t,x,v)\\
\leq &\: \begin{cases}
2C_H\ep e^{-d(t)\vb^2}\vb^{p_H+1}(1+t)^{r_H} + C_H\ep e^{-d(t)\vb^2}\vb^{p_H-1}(1+t)^{r_H+\de} & \\
\qquad + \f{1}{4n}e^{\widetilde{d}(t)\vb^2}\vb^2 (1+t)^{-1-\de} & \mbox{if $r_H+\de\geq 0$}\\
2C_H\ep e^{-d(t)\vb^2}\vb^{p_H+1}(1+t)^{r_H} + \f{1}{4n}e^{\widetilde{d}(t)\vb^2} \vb^2 (1+t)^{-1-\de} & \mbox{if $r_H+\de\in [-1,0)$}.
\end{cases}
\end{split}
\end{equation}

\pfstep{Step~4: The functions $u_{N}^{(n),\pm}$} For every $n\in \mathbb N$, define $u_{N}^{(n), \pm} \colon [0,T_n]\times \overline{B(0,\sqrt{2}R_n^3)}\times \overline{B(0,\sqrt{2}R_n)} \to \mathbb R$ by
\begin{equation}\label{u.def}
\begin{split}
&\: u_{N}^{(n), \pm}(t,x,v) \\
:= &\: \begin{cases}
 \widetilde{h}_{N,R_n}(t,x,v) \pm \f {e^{\widetilde{d}(t)\vb^2}}n \pm 3C_H\ep \int_0^t e^{-d(s)\vb^2} \vb^{p_H+1} (1+s)^{r_H}\, \ud s & \\
 \qquad \pm 2C_H\ep e^{-d(t)\vb^2}\vb^{p_H-1} (1+t)^{r_H+1+\de} & \quad \mbox{if $r_H+\de \geq 0$}\\
\widetilde{h}_{N,R_n}(t,x,v) \pm \f {e^{\widetilde{d}(t)\vb^2}}n \pm 3C_H\ep \int_0^t e^{-d(s)\vb^2} \vb^{p_H+1} (1+s)^{r_H}\, \ud s & \quad \mbox{if $r_H+\de \in [-1,0)$}.
\end{cases}
\end{split}
\end{equation}
We will apply the maximum principle (respectively, the minimum principle) to $u_N^{(n),-}$ (respectively, $u_N^{(n),+}$) in Step~5 (respectively, Step~6) below. In preparation for these steps, we show that (a) $u_{N}^{(n), \pm}$ satisfy appropriate differential inequalities, and (b) $u_{N}^{(n), \pm}$ satisfy appropriate boundary conditions. This is carried out in Steps~4(a) and 4(b) below. Step~4(a) is further divided into Step~4(a)(i) and Step~4(a)(ii) according to whether $r_H+\de<0$ or $r_H+\de\geq 0$.

\pfstep{Step~4(a)(i): Differential inequalities for $u_{N}^{(n),\pm}$ when $r_H+\de<0$} Assume now that $r_H+\de<0$. (This case is slightly simpler than the $r_H+\de\geq 0$ case.) Differentiating \eqref{u.def} and recalling the definitions of $d(t)$ and $\widetilde{d}(t)$ in \eqref{def:d} and \eqref{def:td}, we have
\begin{equation}\label{eq:u}
\begin{split}
&\: \rd_t u_{N}^{(n),\pm}(t,x,v) + v_i\rd_{x_i} u_{N}^{(n),\pm}(t,x,v) -\bar{a}_{ij} \rd_{v_i v_j}^2 u_{N}^{(n),\pm}(t,x,v) \\
=&\: \underbrace{\rd_t \widetilde{h}_{N,R_n}(t,x,v) + v_i\rd_{x_i}\widetilde{h}_{N,R_n}(t,x,v) -\bar{a}_{ij} \rd_{v_i v_j}^2 \widetilde{h}_{N,R_n}(t,x,v)}_{=:\mbox{h contribution}} \\
&\: \underbrace{\pm \f {e^{\widetilde{d}(t)\vb^2}}n \f{d_0\de}{(1+t)^{1+\de}}\vb^2}_{=:\mbox{good term}_1} \underbrace{\mp \f {e^{\widetilde{d}(t)\vb^2}}n\bar{a}_{ij} \widetilde{d}(t)(2\de_{ij} + 4\widetilde{d}(t)v_iv_j) }_{=:\mbox{error term}_1} \underbrace{\pm  3C_H\ep e^{-d(t)\vb^2} \vb^{p_H+1}(1+t)^{r_H}}_{=:\mbox{good term}_2} \\
&\: \underbrace{\pm 3C_H\ep\bar{a}_{ij}(t,x,v)  \int_0^t  d(s)(2\de_{ij} - 4d(s) v_iv_j) e^{-d(s)\vb^2}\vb^{p_H+1}(1+s)^{r_H}\, \ud s}_{=:\mbox{error term}_{2,1}}\\
&\: \underbrace{\pm 6C_H(p_H+1)\ep\bar{a}_{ij}(t,x,v) \int_0^t d(s)v_iv_j e^{-d(s)\vb^2} \vb^{p_H-1}(1+s)^{r_H}\, \ud s}_{=:\mbox{error term}_{2,2}}\\
&\: \underbrace{\mp 3C_H(p_H+1)\ep\bar{a}_{ij}(t,x,v) \int_0^t (\de_{ij} + \f{(p_H -1)v_i v_j}{\vb^2}) e^{-d(s)\vb^2} \vb^{p_H-1}(1+s)^{r_H}\, \ud s}_{=:\mbox{error term}_{2,3}}.
\end{split}
\end{equation}

Let us first explain the strategy in handling RHS of \eqref{eq:u}. For the ``h contribution'' term, we use the equation \eqref{h.tilde.computations} and the estimates obtained in \eqref{cutoff.RHS}. We then have two good terms, which are good in the sense that they are $\geq 0$ in the ``$+$'' case and $\leq 0$ in the ``$-$'' case. Finally, we have four error terms. We will show that the term ``error term$_1$'' can be controlled by ``good term$_1$''; while the terms ``error term$_{2,1}$'', ``error term$_{2,2}$'' and ``error term$_{2,3}$'' can all be controlled by ``good term$_2$''.

To carry this out, we first compare ``good term$_1$'' and ``error term$_1$''. Noting that $\widetilde{d}(t) \leq 2d_0$, estimating $\bar{a}_{ii}$ and $\bar{a}_{ij}v_iv_j$ by Propositions~\ref{prop:ab.Li.1} and \ref{prop:ab.Li.weighted} respectively, we obtain
\begin{equation*}
\begin{split}
|\mbox{error term}_1|\ls &\: \f{e^{\widetilde{d}(t)\vb^2}}{n} \ep^{\f 34} \vb^{\max\{2+\gamma,1\}}(1+t)^{-3}.
\end{split}
\end{equation*}
Choosing $\ep_0$ sufficiently small (so that $\ep$ is also small), we see that ``good term$_1$'' controls ``error term$_1$'' and in fact
\begin{equation}\label{eq:MP.pf.term1}
\f {e^{\widetilde{d}(t)\vb^2}}n \f{d_0\de}{(1+t)^{1+\de}}\vb^2 -|\mbox{error term}_1| \geq \f {3e^{\widetilde{d}(t)\vb^2}}{4n} \f{d_0\de}{(1+t)^{1+\de}}\vb^2.
\end{equation}

Next, we compare ``good term$_2$'', ``error term$_{2,1}$'', ``error term$_{2,2}$'' and ``error term$_{2,3}$''. For this purpose we first consider the integral $\int_0^t e^{-d(s)\vb^2} \vb^{p_H+3} (1+s)^{r_H} \ud s$. Let us assume for the moment that $r_H+1+\de\neq 0$. Then since $d'(t) = -\f{d_0 \de}{(1+t)^{1+\de}}$ and $d(t)$ is monotonically decreasing, we obtain the following estimate after integrating by parts:
\begin{equation}\label{int.gain.weight.1}
\begin{split}
&\: C_H\ep \int_0^t e^{-d(s)\vb^2} \vb^{p_H+3} (1+s)^{r_H} \ud s\\
= &\: -C_H\ep \vb^{p_H+3} \int_0^t  \big(\f{1}{d'(s) \vb^2}\f{\ud}{\ud s} e^{- d(s)\vb^2}\big)(1+s)^{r_H}\, \ud s \\
\leq &\: \f{C_H\ep}{d_0 \de} \vb^{p_H+1} \int_0^t e^{-d(s)\vb^2} \left|\f{\ud}{\ud s} (1+s)^{r_H+1+\de} \right|\, \ud s + \f{C_H\ep}{d_0 \de}\vb^{p_H+1} (e^{-d(t)\vb^2} (1+t)^{r_H+1+\de} -  e^{-2d_0\vb^2}) \\
\leq &\: \f{C_H\ep |r_H+1+\de|}{d_0 \de} \vb^{p_H+1} e^{-d(t)\vb^2} \int_0^t  (1+s)^{r_H+\de}\, \ud s + \f{C_H\ep}{d_0 \de}\vb^{p_H+1}e^{-d(t)\vb^2} (1+t)^{r_H+1+\de}\\
\leq &\: \f{C_H\ep |r_H+1+\de|}{d_0 \de|r_H+1+\de|} \vb^{p_H+1} e^{-d(t)\vb^2} \max\{1, (1+t)^{r_H+1+\de}\} + \f{C_H\ep}{d_0 \de}\vb^{p_H+1} e^{-d(t)\vb^2} (1+t)^{r_H+1+\de} \\
\leq &\: \f{2 C_H\ep}{d_0 \de} \vb^{p_H+1} e^{-d(t)\vb^2} \max\{1, (1+t)^{r_H+1+\de}\}\leq \f{2 C_H\ep}{d_0 \de} \vb^{p_H+1} e^{-d(t)\vb^2} (1+t)^{r_H+1+\de},
\end{split}
\end{equation}
since $r_H+1+\de\geq 0$ by assumption. In the case $r_H+1+\de = 0$, we argue in a similar, but simpler, way:
\begin{equation}\label{int.gain.weight.2}
\begin{split}
&\: C_H\ep \int_0^t e^{-d(s)\vb^2} \vb^{p_H+3} (1+s)^{r_H} \ud s\\
= &\: -C_H\ep \vb^{p_H+3} \int_0^t  \big(\f{1}{d'(s) \vb^2}\f{\ud}{\ud s} e^{- d(s)\vb^2}\big)(1+s)^{r_H}\, \ud s \\
= &\: \f{C_H\ep}{d_0 \de}\vb^{p_H+1}e^{-d(t)\vb^2} \leq \f{2 C_H\ep}{d_0 \de} \vb^{p_H+1} e^{-d(t)\vb^2} (1+t)^{r_H+1+\de}.
\end{split}
\end{equation}
We now estimate the terms ``error term$_{2,1}$'', ``error term$_{2,2}$'' and ``error term$_{2,3}$'' in \eqref{eq:u}. Noting that $d(t) \leq 2d_0$ for all $t\geq 0$, we apply Propositions~\ref{prop:ab.Li.1} and \ref{prop:ab.Li.weighted}, \eqref{int.gain.weight.1} and \eqref{int.gain.weight.2} to obtain
\begin{equation*}
\begin{split}
&\: |\mbox{error term$_{2,1}|+|$error term$_{2,2}|+|$error term$_{2,3}|$} \\
\ls &\: C_H \ep^{\f 74} \int_0^t e^{-d(s)\vb^2} \vb^{p_H+3} (1+s)^{r_H} \ud s \ls C_H \ep^{\f 74} \vb^{p_H+1} e^{-d(t)\vb^2} (1+t)^{r_H+1+\de}.
\end{split}
\end{equation*}
Choosing $\ep_0$ (and therefore also $\ep$) sufficiently small, we can bound ``error term$_{2,1}$'', ``error term$_{2,2}$'' and ``error term$_{2,3}$'' by ``good term$_2$'' and in fact
\begin{equation}\label{eq:MP.pf.term2}
\begin{split}
&\: 3C_H\ep e^{-d(t)\vb^2} \vb^{p_H+1}(1+t)^{r_H} - |\mbox{error term$_{2,1}|+|$error term$_{2,2}|+|$error term$_{2,3}|$}\\
> &\: 2C_H\ep e^{-d(t)\vb^2} \vb^{p_H+1}(1+t)^{r_H}.
\end{split}
\end{equation}

We now consider the ``$+$'' case in \eqref{eq:u}. By \eqref{tildeh.main.est}, \eqref{cutoff.RHS}, \eqref{eq:MP.pf.term1} and \eqref{eq:MP.pf.term2}, we obtain
\begin{equation}\label{eq:du+.prelim}
\begin{split}
&\: \rd_t u_{N}^{(n),+}(t,x,v) + v_i\rd_{x_i} u_{N}^{(n),+}(t,x,v) -\bar{a}_{ij} \rd_{v_i v_j}^2 u_{N}^{(n),+}(t,x,v)\\
> &\: -2C_H\ep e^{-d(t)\vb^2}\vb^{p_H+1}(1+t)^{r_H} - \f{e^{\widetilde{d}(t)\vb^2}}{4n}\vb^2 (1+t)^{-1-\de} -C'_{T_n} R_n^{\min\{-1,\gamma, p_H-2\}} \\
&\: + \f {3e^{\widetilde{d}(t)\vb^2}}{4n} \f{d_0\de}{(1+t)^{1+\de}}\vb^2 +2C_H\ep e^{-d(t)\vb^2} \vb^{p_H+1}(1+t)^{r_H} \\
\end{split}
\end{equation}
Now note that the first term on the RHS of \eqref{eq:du+.prelim} is $\leq$ the last term on the RHS of \eqref{eq:du+.prelim}. For the remaining three terms, note that by \eqref{Rn.condition}, for any $(t,x,v)\in [0,T_{Boot})\times \mathbb R^3\times \mathbb R^3$, we have
\begin{equation*}
\begin{split}
&\: - \f{e^{\widetilde{d}(t)\vb^2}}{4n}\vb^2 (1+t)^{-1-\de} -C'_{T_n} R_n^{\min\{-1,\gamma, p_H-2\}} + \f {3e^{\widetilde{d}(t)\vb^2}}{4n} \f{d_0\de}{(1+t)^{1+\de}}\vb^2 \\
\geq &\: \f {e^{\widetilde{d}(t)\vb^2}}{2n} \f{d_0\de}{(1+t)^{1+\de}}\vb^2 - \f {1}{4n} \f{d_0\de}{(1+T_{Boot})^{1+\de}}\vb^2 \geq \f {d_0\de}{4n(1+T_{Boot})^{1+\de}}.
\end{split}
\end{equation*}
Putting all these together, we thus obtain
\begin{equation}\label{eq:du+}
\begin{split}
\rd_t u_{N}^{(n),+}(t,x,v) + v_i\rd_{x_i} u_{N}^{(n),+}(t,x,v) -\bar{a}_{ij} \rd_{v_i v_j}^2 u_{N}^{(n),+}(t,x,v) \geq &\: \f {d_0\de}{4n(1+T_{Boot})^{1+\de}}.
\end{split}
\end{equation}

In a completely analogous manner, in the ``$-$'' case we obtain
\begin{equation}\label{eq:du-}
\begin{split}
\rd_t u_{N}^{(n),-}(t,x,v) + v_i\rd_{x_i} u_{N}^{(n),-}(t,x,v) -\bar{a}_{ij} \rd_{v_i v_j}^2 u_{N}^{(n),-}(t,x,v) \leq -\f {d_0\de}{4n(1+T_{Boot})^{1+\de}}.
\end{split}
\end{equation}

\pfstep{Step~4(a)(ii): Differential inequalities for $u_{N}^{(n),\pm}$ when $r_H+\de \geq 0$} We now consider the case $r_H+\de\geq 0$. Our goal will be to show that \eqref{eq:du+} and \eqref{eq:du-} also hold in this case.

To proceed, we carry out a computation as in \eqref{eq:u}, using \eqref{u.def}, \eqref{def:d} and \eqref{def:td}. Note that in the $r_H+\de\geq 0$ case that we are considering, there are a few more terms arising from various derivatives of $2C_H\ep e^{-d(t)\vb^2}\vb^{p_H-1} (1+t)^{r_H+1+\de}$.
\begin{equation}\label{eq:u.2}
\begin{split}
&\: \rd_t u_{N}^{(n),\pm}(t,x,v) + v_i\rd_{x_i} u_{N}^{(n),\pm}(t,x,v) -\bar{a}_{ij} \rd_{v_i v_j}^2 u_{N}^{(n),\pm}(t,x,v) \\
=&\: \mbox{RHS of \eqref{eq:u}} \\
&\: \underbrace{\pm 2C_H \ep (r+\de+1) e^{-d(t)\vb^2} \vb^{p_H-1} (1+t)^{r_H+\de}}_{=:\mbox{good term}_3} \underbrace{\pm 2C_H \ep d_0\de e^{-d(t)\vb^2} \vb^{p_H+1} (1+t)^{r_H}}_{=:\mbox{good term}_4} \\
&\: \underbrace{\pm 2C_H\ep \bar{a}_{ij} (2d(t)\de_{ij}-4(d(t))^2 v_iv_j) e^{-d(t)\vb^2}\vb^{p_H-1}(1+t)^{r_H+\de+1}}_{=:\mbox{error term}_{4,1}} \\
&\: \underbrace{\pm 2C_H\ep \bar{a}_{ij} (-\f{(p_H-1)(4d(t)v_iv_j+\de_{ij})}{\vb^2}  - \f{(p_H-1)(p_H-3)v_iv_j}{\vb^{4}}) e^{-d(t)\vb^2}\vb^{p_H-1}(1+t)^{r_H+\de+1}}_{=:\mbox{error term}_{4,2}}.
\end{split}
\end{equation}

We first treat the terms that arise in RHS of \eqref{eq:u}. Note that this can almost be treated exactly as in \eqref{eq:u}, except that now since $r_H+\de \geq 0$, in the ``h contribution'' term, we have an extra contribution on the RHS of \eqref{cutoff.RHS}, which is bounded in magnitude by $C_H \ep e^{-d(t)\vb^2} \vb^{p_H-1} (1+t)^{r_H+\de}$. Now note that this can be absorbed by ``good term$_3$'' since when $r+\de\geq 0$ we have
\begin{equation*}
\begin{split}
&\: 2C_H \ep (r+\de+1) e^{-d(t)\vb^2} \vb^{p_H-1} (1+t)^{r_H+\de} - C_H \ep e^{-d(t)\vb^2} \vb^{p_H-1} (1+t)^{r_H+\de} \\
\geq &\: C_H \ep e^{-d(t)\vb^2} \vb^{p_H-1} (1+t)^{r_H+\de}.
\end{split}
\end{equation*}

The other terms on the RHS of \eqref{eq:u} can be treated exactly as in Step~4(a)(i).

To handle the RHS of \eqref{eq:u.2}, it therefore remains to show the terms ``error term$_{4,1}$'' and ``error term$_{4,2}$'' can be absorbed by ``good term$_4$''. Note that by Propositions~\ref{prop:ab.Li.1} and \ref{prop:ab.Li.weighted} and \eqref{def:de},
\begin{equation}\label{bad.term.4}
\begin{split}
&\: C_H\ep (\bar{a}_{ii} + \bar{a}_{ij} v_iv_j) e^{-d(t)\vb^2}\vb^{p_H-1}(1+t)^{r_H+\de+1} \\
\ls &\: C_H\ep \cdot \ep^{\f 34} \vb^{\max\{2+\gamma,1\}} (1+t)^{-3} \cdot e^{-d(t)\vb^2}\vb^{p_H-1}(1+t)^{r_H+\de+1} \\
\ls &\: C_H\ep^{\f 74} e^{-d(t)\vb^2}\vb^{p_H-1+\max\{2+\gamma,1\}}(1+t)^{r_H+\de-2} \\
\ls &\: C_H\ep^{\f 74} e^{-d(t)\vb^2}\vb^{p_H+1}(1+t)^{r_H}.
\end{split}
\end{equation}
Now it is easy to observe that
\begin{equation}\label{bad.term.4.1}
|\mbox{error term}_{4,1}| + |\mbox{error term}_{4,2}| \ls C_H\ep (\bar{a}_{ii} + \bar{a}_{ij} v_iv_j) e^{-d(t)\vb^2}\vb^{p_H-1}(1+t)^{r_H+\de+1}.
\end{equation}
It therefore follows from \eqref{bad.term.4} and \eqref{bad.term.4.1} that
\begin{equation}\label{bad.term.4.2}
\begin{split}
&\: 2C_H \ep d_0\de e^{-d(t)\vb^2} \vb^{p_H+1} (1+t)^{r_H} - |\mbox{error term}_{4,1}| - |\mbox{error term}_{4,2}| \\
\geq &\: C_H\ep d_0\de e^{-d(t)\vb^2}\vb^{p_H+1}(1+t)^{r_H}.
\end{split}
\end{equation}

Since all the other terms are exactly as in the $r_H+\de<0$ case, we therefore conclude that both inequalities \eqref{eq:du+} and \eqref{eq:du-} hold also in the case $r_H+\de\geq 0$.

\pfstep{Step~4(b): Boundary conditions for $u_N^{(n),\pm}$} Since for every $t\in [0,T_n]$, $\widetilde{h}_{N,R_n}(t,x,v)$ is compactly support in $\overline{B(0,\sqrt{2}R_n^3)}\times \overline{B(0,\sqrt{2}R_n)}$, $\widetilde{h}_{N,R_n}\restriction_{\rd(B(0,\sqrt{2}R^3)\times B(0,\sqrt{2}R))}(t)=0$. Noting also the obvious signs for the other terms in the definition of $u_{N,R}^{(n),\pm}$ in \eqref{u.def}, we obtain that for every $t\in [0,T'_n]$,
\begin{equation}\label{u.boundary}
u_{N}^{(n),+}\restriction_{\rd(B(0,\sqrt{2}R_n^3)\times B(0,\sqrt{2}R_n))}(t) \geq 0,\quad u_{N}^{(n),-}\restriction_{\rd(B(0,\sqrt{2}R_n^3)\times B(0,\sqrt{2}R_n))}(t) \leq 0.
\end{equation}

\pfstep{Step~5: Maximum principle argument} We now apply the maximum principle to obtain an upper bound for $u_{N}^{(n),-}$. To simplify notation, let $\mathcal D_n:=[0,T'_n]\times \overline{B(0,\sqrt{2}R^3_n)}\times \overline{B(0,\sqrt{2}R_n)}$.

Our goal is to show that 
\begin{equation}\label{eq:u-.goal}
\sup_{(t,x,v)\in \mathcal D_n} u_{N}^{(n),-}(t,x,v) \leq C_H \ep.
\end{equation}

Since $\mathcal D_n$ is compact, $u_{N}^{(n),-}$ achieves a maximum in $\mathcal D_n$. It is clearly sufficient to bound the maximum of $u_N^{(n),-}$. At least one\footnote{It is possible that more than one of the following possibilities hold since these sets are not disjoint and moreover the maximum can be achieved at two distinct points.} of the following holds:
\begin{enumerate}
\item The maximum is achieved on the initial time slice $\mathcal D_n\cap \{(t,x,v): t=0\}$.
\item The maximum is achieved on the set $(0,T_n')\times \rd(B(0,\sqrt{2}R^3_n)\times B(0,\sqrt{2}R_n))$.
\item The maximum is achieved at an interior point $\mathcal D_n^{\mathrm{o}}$.
\item The maximum is achieved in the interior of the future boundary $\mathcal D_n\cap \{(T'_n,x,v)\colon |x|<\sqrt{2}R^3_n\mbox{ or }|v|< \sqrt{2}R_n\}$.
\end{enumerate}

In Case~1, the estimate \eqref{eq:u-.goal} is trivially true by \eqref{h.data.bound}. 

In Case~2, the estimate \eqref{eq:u-.goal} is also trivially true by \eqref{u.boundary}.

We next argue that Case~3~is impossible. If there exists an interior maximum point $p$ of $u_N^{(n),-}$, it holds that $\rd_t u_N^{(n),-}(p) = \rd_{x_i} u_N^{(n),-}(p) = \rd_{v_i} u_N^{(n),-}(p) = 0$ for $i=1,2,3$, and $\bar{a}_{ij}\rd^2_{v_i v_j} u_N^{(n),-}(p)\leq 0$. Hence, \eqref{eq:du-} evaluated at $p$ implies that $0\leq -\f {d_0\de}{4n(1+T_{Boot})^{1+\de}}$, which is a contradiction.

Finally, we argue that Case~4~is also impossible. Suppose there is a point $p\in \mathcal D_n\cap \{(T'_n,x,v)\colon |x|<\sqrt{2}R^3_n\mbox{ or }|v|< \sqrt{2}R_n\}$ so that $u_N^{(n),-}$ assumes its maximum at $p$. Then $\rd_{x_i} u_N^{(n),-}(p) = \rd_{v_i} u_N^{(n),-}(p) = 0$ for $i=1,2,3$, and $\bar{a}_{ij}\rd^2_{v_i v_j} u_N^{(n),-}(p) \leq 0$. Hence, \eqref{eq:du-} evaluated at point $p$ implies that
$$\rd_t u_N^{(n),-}(p) \leq -\f {d_0\de}{4n(1+T_{Boot})^{1+\de}}.$$
As a consequence, by considering the Taylor expansion of $u_N^{(n),-}$ at $p$, one concludes that $u_N^{(n),-}(p)$ is in fact not a maximum, contradicting the definition of $p$.

Combining the considerations in the four cases above, we have established \eqref{eq:u-.goal}.

\pfstep{Step~6: Minimum principle} In an entirely analogous manner as Step~5, but considering instead the minimum of $u_N^{(n),+}$, we obtain
\begin{equation}\label{eq:u+.goal}
\inf_{(t,x,v)\in \mathcal D_n} u_{N}^{(n),+}(t,x,v) \geq -C_H \ep.
\end{equation}

\pfstep{Step~7: Completion of the continuity argument} By the definition of $u_N^{(n),\pm}$ in \eqref{u.def}, the estimates \eqref{eq:u-.goal} and \eqref{eq:u+.goal}, and the triangle inequality,
\begin{equation}\label{h.tilde.final}
\begin{split}
|\widetilde{h}_{N,R_n}(t,x,v)|
\leq &\: \begin{cases}
C_H\ep + \f {e^{\widetilde{d}(t)\vb^2}}n + 3C_H\ep \int_0^t e^{-d(s)\vb^2} \vb^{p_H+1} (1+s)^{r_H}\, \ud s \\
\qquad + 2C_H\ep e^{-d(t)\vb^2}\vb^{p_H-1} (1+t)^{r_H+1+\de} & \quad \mbox{if $r_H+\de \geq 0$}\\
C_H\ep + \f {e^{\widetilde{d}(t)\vb^2}}n + 3C_H\ep \int_0^t e^{-d(s)\vb^2} \vb^{p_H+1} (1+s)^{r_H}\, \ud s & \quad \mbox{if $r_H+\de \in [-1,0)$}
\end{cases}
\end{split}
\end{equation}
for every $(t,x,v)\in \mathcal D_n$. 

To proceed, we need to control the term $3C_H\ep \int_0^t e^{-d(s)\vb^2} \vb^{p_H+1} (1+s)^{r_H}\, \ud s$ in \eqref{h.tilde.final}. By \eqref{int.gain.weight.1} and \eqref{int.gain.weight.2} (and multiplying by $\vb^{-2}$), we obtain
$$3C_H\ep \int_0^t e^{-d(s)\vb^2} \vb^{p_H+1} (1+s)^{r_H}\, \ud s \leq \f{6C_H\ep}{d_0\de}\vb^{p_H-1} e^{-d(t)\vb^2}(1+t)^{r_H+1+\de}.$$
Plugging this into \eqref{h.tilde.final}, we obtain
\begin{equation}\label{eq:h.BA.improved}
\begin{split}
|\widetilde{h}_{N,R_n}(t,x,v)|
\leq \begin{cases} (3+\f{6}{d_0 \de})C_H\ep \vb^{p_H-1} e^{-d(t)\vb^2} (1+t)^{r_H+1+\de} + \f {e^{\widetilde{d}(t)\vb^2}}n & \quad \mbox{if $r_H+\de\geq 0$} \\
(1+\f{6}{d_0 \de})C_H\ep \vb^{p_H-1} e^{-d(t)\vb^2} (1+t)^{r_H+1+\de} + \f {e^{\widetilde{d}(t)\vb^2}}n & \quad \mbox{if $r_H+\de \in [-1,0)$}.
\end{cases}
\end{split}
\end{equation}
for every $(t,x,v)\in \mathcal D_n$. Note that this \emph{improves} the constant in \eqref{eq:BA.h}.

We now complete the continuity argument initiated in Step~3; namely, we show that for every $n\in \mathbb N$, $T'_n = T_n$ (where $T'_n$ is as in \eqref{T'.def}). Suppose not, then by \eqref{eq:h.BA.improved} and continuity of $\widetilde{h}_{N,R_n}$, \eqref{eq:BA.h} must hold for some short time beyond $T'_n$. This then contradicts the definition of $T_n'$. 

Therefore, we have proven that $\mathcal D_n= [0,T_n]\times \overline{B(0,\sqrt{2}R_n^3)}\times \overline{B(0,\sqrt{2}R_n)}$ and thus \eqref{eq:h.BA.improved} holds in the whole region $[0,T_n]\times \overline{B(0,\sqrt{2}R_n^3)}\times \overline{B(0,\sqrt{2}R_n)}$.

\pfstep{Step~8: Putting everything together} Fix a point $(t,x,v)\in [0,T_{Boot})\times\mathbb R^3\times \mathbb R^3$. By \eqref{exhaustion.1} and \eqref{exhaustion.2}, there exists $n_0\in \mathbb N$ such that $(t,x,v)\in [0,T_n]\times B(0,R^3_n) \times B(0,R_n)$ for all $n\geq n_0$.

By \eqref{eq:h.BA.improved} and the fact that $\mathcal D_n= [0,T_n]\times \overline{B(0,\sqrt{2}R_n^3)}\times \overline{B(0,\sqrt{2}R_n)}$, we know that
$$|\widetilde{h}_N(t,x,v)|\leq |\widetilde{h}_{N,R_n}(t,x,v)|\leq (3+\f{6}{d_0 \de}) C_H\ep \vb^{p_H-1} e^{-d(t)\vb^2} (1+t)^{r_H+1+\de} + \f {e^{\widetilde{d}(t)\vb^2}}n$$
for every $n\geq n_0$. Taking $n\to +\infty$, we thus obtain
$$|\widetilde{h}_{N}(t,x,v)|\leq (3+\f{6 }{d_0 \de})C_H\ep \vb^{p_H-1} e^{-d(t)\vb^2} (1+t)^{r_H+1+\de}.$$
Recalling the definition of $\widetilde{h}_{N}$ in \eqref{def:wN} and \eqref{def:thN}, this implies
$$\wb^N|h|(t,x,v) \ls (3+\f{6 }{d_0 \de})C_H\ep \vb^{p_H-2} (1+t)^{r_H+1+\de}.$$
Since $(t,x,v)$ is arbitrary, we have proven \eqref{max.prin.conclusion}. \qedhere
\end{proof}

\subsection{The hierarchy of $L^\i_xL^\i_v$ estimates and the induction argument}\label{sec:hierarchy}

We now discuss the (hierarchy of) $L^\i_xL^\i_v$ estimates that we will prove. Recall from Section~\ref{sec:descent.scheme} that our $L^\i_xL^\i_v$ estimates will be proved in a descent scheme in which the estimates are better at lower levels of derivatives. In this section, we make precise the numerology of the estimates (see \textbf{Sections~\ref{sec:Mi} and \ref{sec:Z.def}}). We then initiate in \textbf{Section~\ref{sec:induction}} an induction argument to prove these estimates.

\subsubsection{Definition of $\Mi$}\label{sec:Mi}

Define $\Mi$ by
\begin{equation}\label{def:Mi}
\Mi := \begin{cases}
\Mm - \lceil (\f{2}{2+\gamma} + 4)\rceil & \mbox{ if $\gamma \in (-2,-1]$} \\
\Mm - \lceil (\f{1}{|\gamma|}+ 4)\rceil & \mbox{ if $\gamma \in (-1,0)$}
\end{cases}.
\end{equation}
The parameter $\Mi$ is used to indicate an ``intermediate'' number of derivatives, below which we have sharp $Z$ estimates (i.e.~with $\zeta = \theta = 0$ in \eqref{eq:Z.norm}; see Section~\ref{sec:Z.def} for definitions). Note that by definition
\begin{equation}\label{eq:MmMi}
\Mm+2\geq 2\Mi.
\end{equation}
An easy computation shows that 
\begin{equation}\label{eq:MmMi.2}
\begin{cases}
\Mm-\Mi\geq 7 & \quad \mbox{if $\gamma \in (-2,-1]$}\\
\Mm-\Mi\geq 6 & \quad \mbox{if $\gamma \in (-1,0)$}
\end{cases}
\end{equation}
Moreover, $\Mm-\Mi\to \infty$ as $\gamma\to -2$ or $\gamma\to 0$.

\subsubsection{Definition of the $Z_{k,\zeta,\th}$ norms}\label{sec:Z.def}

For $\zeta \in [0,\f 32)$ and $\theta\in [0,1)$, introduce also the following $L^\infty$-type norm:
\begin{equation}\label{eq:Z.norm}
Z_{k,\zeta,\theta}(T):= \sum_{|\alp|+|\bt|+|\sigma|= k}\|(1+t)^{-\zeta-|\bt|}\vb^{1-\theta}\wb^{\Mm+5-|\sigma|}(\rd_x^\alp \rd_v^\bt Y^\sigma g)\|_{L^\i([0,T];L^\i_x L^\i_v)}.
\end{equation}
Like the $E_k$ norms (cf.~\eqref{eq:energy.def}), the $Z_{k,\zeta,\theta}$ norms have weights of $\wb$ dependent on the number of $Y$ derivatives and the norms become worse in $t$ for every $\rd_v$ derivative. Moreover, the $Z_{k,\zeta,\th}$ norms depend on two parameters $\zeta$ and $\theta$. The $Z$ norms are the strongest when $\zeta = \theta = 0$, and the parameters $\zeta$ and $\theta$ exactly parametrize the ``loss'' compared to the strongest case in the growth rate in $t$ and the weight in $\vb$ respectively. In addition, note that when $\theta=0$, the $Z$ norm is one $\vb$ weight stronger compared to the $E$ norm.

The values of $\zeta$ and $\th$ that we will use depend on $k$ (in addition to $\gamma$). We define below $\zeta_k$ and $\th_k$. For each $k = 0,1,\dots, \Mm-4$, our goal will be to control $Z_{k,\zeta_k,\th_k}$. (In the process of bounding the $Z_{k,\zeta_k,\th_k}$, we will first need some weaker $Z_{k,\zeta_k,\th_m}$ bounds for $m\geq k$; see Section~\ref{sec:induction}).

\textbf{Suppose $0\leq k\leq \Mi$.} Define
\begin{equation}\label{BA.Z}
\zeta_k = \th_k = 0.
\end{equation}

\textbf{Suppose $\gamma \in (-2,-1]$ and $\Mi+1\leq k\leq \Mm-5$.} Define 
\begin{equation}\label{BA.Z.-2}
\zeta_k = \f 32- \f{3(2+\gamma)}{4}\cdot(\Mm-4-k),\quad \th_k = 0.
\end{equation}

\textbf{Suppose $\gamma \in (-1, 0)$ and $\Mi+1 \leq k \leq \Mm-6$.} Define
\begin{equation}\label{BA.Z.0.1}
\zeta_k = 0, \quad \th_k = 1+(\Mm-4-k)\gamma.
\end{equation}

\textbf{Suppose $\gamma \in (-1, 0)$ and $k = \Mm-5$.} Define
\begin{equation}\label{BA.Z.0.2}
\zeta_{k} = \f 34, \quad \th_k = 1 +\gamma.
\end{equation}

\textbf{Suppose $k = \Mm-4$.} Define
\begin{equation}\label{BA.Z.top}
\zeta_k = \f 32,\quad \th_k = 1.
\end{equation}

\subsubsection{The induction argument}\label{sec:induction}

In order to prove the desired estimates for the $Z_{k,\zeta,\th}$ norms, we prove the following statement with an induction (decreasing) in $m = 0,1,\dots, \Mm-4$:
\begin{equation}\label{main.induction}
\begin{cases}
Z_{k,\zeta_k,\th_m} \ls \ep^{\f 34} &\mbox{for $k\leq m$}\\
Z_{k,\zeta_k,\th_k} \ls \ep^{\f 34} &\mbox{for $k\geq m$},
\end{cases}
\end{equation}
i.e.~for $k\geq m$ we prove the sharp estimates, while for $k\leq m$ we prove estimates corresponding to the sharp $\zeta_k$ but with only a weaker $\vb$ weight corresponding to $\th_m$.

First, note that the base case of the induction, i.e.~the $m=\Mm-4$ case, holds thanks to Proposition~\ref{prop:prelim.Linfty} and the bootstrap assumptions \eqref{BA.sim.1} and \eqref{BA.sim.2}. We then carry out the induction step below.

\textbf{From now on until Section~\ref{sec:Li.everything}, we assume \eqref{main.induction} holds for $m = m_*+1$ for some $m = 0,1,\dots,\Mm-5$.} Our goal will be to show that \eqref{main.induction} holds also for $m=m_*$. (In fact, we will show a slightly stronger estimate with $\ep^{\f 34}$ replaced by $\ep$.)

Before we proceed, observe that in the induction step it suffices to control $\rd_x^\alp \rd_v^\bt Y^\sigma g$ for $|\alp|+|\bt|+|\sigma|\leq m_*$ (since the required estimate when $|\alp|+|\bt|+|\sigma|> m_*$ is tautologically part of the induction hypotheses).

\subsection{Classifying the inhomogeneous terms in the equation for $\rd_x^\alp \rd_v^\bt Y^\sigma g$}\label{sec:est.inho.Li}

We continue to work under the induction hypotheses in Section~\ref{sec:induction}.

Our goal in this subsection is to control $|(\rd_t + v_i\rd_{x_i} -\bar{a}_{ij} \rd^2_{v_i v_j}) (\rd_x^\alp \rd_v^\bt Y^\sigma g)|$ in preparation to apply the maximum principle in Proposition~\ref{prop:max.prin}. 

\begin{proposition}\label{prop:mainLiLi}
Suppose $|\alp|+|\bt|+|\sigma|\leq \Mm-5$. Let $h:=\rd_x^\alp \rd_v^\bt Y^\sigma g$. Then, under the assumptions of Theorem~\ref{thm:BA}, $(\rd_t + v_i\rd_{x_i} +\f{\de d_0}{(1+t)^{1+\de}} \vb^2 -\bar{a}_{ij} \rd^2_{v_i v_j}) h$ obeys the following estimates for all $(t,x,v)\in [0,T_{Boot})\times \mathbb R^3\times \mathbb R^3$:
\begin{equation*}
\begin{split}
&\: |(\rd_t + v_i\rd_{x_i} +\f{\de d_0}{(1+t)^{1+\de}} \vb^2 -\bar{a}_{ij} \rd^2_{v_i v_j}) h|(t,x,v) \\
\ls &\: (I_p^{\alp,\bt,\sigma} + II_p^{\alp,\bt,\sigma} + III_p^{\alp,\bt,\sigma} + IV_p^{\alp,\bt,\sigma} + V_p^{\alp,\bt,\sigma} + VI_p^{\alp,\bt,\sigma})(t,x,v),
\end{split}
\end{equation*}
where\footnote{All of the following terms are functions of $(t,x,v)$. We suppress the explicit dependence on $(t,x,v)$ when there is no risk of confusion.}
\begin{equation}\label{def:Ip}
I_p^{\alp,\bt,\sigma}:= \sum_{\substack{|\alp'|+|\alp''| = |\alp|,\, |\bt'|+|\bt''| = |\bt| +2\\ |\sigma'|+|\sigma''|=|\sigma| \\ 1\leq |\alp'|+|\bt'|+|\sigma'| \leq |\alp|+|\bt|+|\sigma|}} \max_{i,j}|\rd_x^{\alp'}\rd_v^{\bt'} Y^{\sigma'} \bar{a}_{ij}| |\rd_x^{\alp''} \rd_v^{\bt''} Y^{\sigma''} g|,
\end{equation}
\begin{equation}\label{def:IIp}
II_p^{\alp,\bt,\sigma}:= \sum_{\substack{|\alp'|+|\alp''| = |\alp|,\, |\sigma'|+|\sigma''|= |\sigma|\\ |\bt'|+|\bt''| = |\bt| +1,\, |\bt'|\leq |\bt|}} \max_j|\rd_x^{\alp'}\rd_v^{\bt'} Y^{\sigma'} (\bar{a}_{ij}v_i)| |\rd_x^{\alp''} \rd_v^{\bt''} Y^{\sigma''} g|,
\end{equation}
\begin{equation}\label{def:IIIp}
\begin{split}
III_p^{\alp,\bt,\sigma}:= &\: \sum_{\substack{|\alp'|+|\alp''| = |\alp|\\ |\bt'|+|\bt''| = |\bt| \\|\sigma'|+|\sigma''| = |\sigma|}} (|\rd_x^{\alp'}\rd_v^{\bt'} Y^{\sigma'} \bar{a}_{ii}| + |\rd_x^{\alp'} \rd_v^{\bt'} Y^{\sigma'} (\bar{a}_{ij}v_iv_j)|) |\rd_x^{\alp''} \rd_v^{\bt''} Y^{\sigma''} g|,
\end{split}
\end{equation}
\begin{equation}\label{def:IVp}
\begin{split}
IV_p^{\alp,\bt,\sigma}:= &\: \sum_{\substack{|\alp'|+|\alp''| = |\alp|\\ |\bt'|+|\bt''| = |\bt| \\|\sigma'|+|\sigma''| = |\sigma|}} |\rd_x^{\alp'}\rd_v^{\bt'} Y^{\sigma'} \bar{c}||\rd_x^{\alp''} \rd_v^{\bt''} Y^{\sigma''} g|,
\end{split}
\end{equation}
\begin{equation}\label{def:Vp}
\begin{split}
V_p^{\alp,\bt,\sigma}:= &\: \sum_{|\alp'|\leq |\alp|+1,\,|\bt'|\leq |\bt|-1 }|\rd_x^{\alp'} \rd_v^{\bt'} Y^{\sigma} g|,
\end{split}
\end{equation}
\begin{equation}\label{def:VIp}
\begin{split}
VI_p^{\alp,\bt,\sigma}:= &\: \sum_{\substack{|\bt'|\leq |\bt|,\,|\sigma'|\leq |\sigma|\\ |\bt'|+|\sigma'|\leq |\bt|+|\sigma|-1}}\f{\vb}{(1+t)^{1+\de}}|\rd_x^{\alp} \rd_v^{\bt'} Y^{\sigma'} g|.
\end{split}
\end{equation}
Here, by our convention (see~Section~\ref{sec:notation}), if $|\bt|+|\sigma|=0$, then the terms $V_p^{\alp,\bt,\sigma}$ and $VI_p^{\alp,\bt,\sigma}$ are not present.
\end{proposition}
\begin{proof}
Differentiating \eqref{eq:g} by $\rd_x^\alp \rd_v^\bt Y^\sigma$, we obtain
\begin{equation}\label{eq:g.diff}
\begin{split}
&\: \rd_t \rd_x^\alp \rd_v^\bt Y^\sigma g +v_i\rd_{x_i} \rd_x^\alp \rd_v^\bt Y^\sigma g +\f{\de d_0}{(1+t)^{1+\de}} \vb^2 \rd_x^\alp \rd_v^\bt Y^\sigma g -\ab_{ij}\rd^2_{v_i v_j} \rd_x^\alp \rd_v^\bt Y^\sigma g \\
= &\: \underbrace{\left[ \rd_t+v_i \rd_{x_i}, \rd_x^\alp \rd_v^\bt Y^\sigma \right] g}_{=:\mathrm{Term}_1} + \underbrace{\f{\de d_0}{(1+t)^{1+\de}} \left(\rd_x^\alp \rd_v^\bt Y^\sigma (\vb^2g) - \vb^2 \rd_x^\alp \rd_v^\bt Y^\sigma g \right)}_{=:\mathrm{Term}_2} \\
&\: + \underbrace{\left(\rd_x^\alp \rd_v^\bt Y^\sigma (\ab_{ij}\rd^2_{v_i v_j}  g) -\ab_{ij}\rd^2_{v_i v_j} \rd_x^\alp \rd_v^\bt Y^\sigma g \right)}_{=:\mathrm{Term}_3} \underbrace{-\rd_x^\alp \rd_v^\bt Y^\sigma (\cb_i g)}_{=:\mathrm{Term}_4} - \underbrace{4 d(t) \rd_x^\alp \rd_v^\bt Y^\sigma (\bar{a}_{ij} v_i \rd_{v_j} g)}_{=:\mathrm{Term}_5} \\
&\: - \underbrace{2d(t)\rd_x^\alp \rd_v^\bt Y^\sigma ((\de_{ij}-2d(t)v_iv_j)\bar{a}_{ij}g)}_{=:\mathrm{Term}_6}.
\end{split}
\end{equation}
The terms $\mathrm{Term}_1$--$\mathrm{Term}_3$ are commutator terms, while the terms $\mathrm{Term}_4$--$\mathrm{Term}_6$ arise from differentiating the RHS of \eqref{eq:g}.

We estimate each term on the RHS of \eqref{eq:g.diff}. 

For $\mathrm{Term}_1$, note that both $\rd_x$ and $Y$ commute with the transport operator $\rd_t+ v_i\rd_{x_i}$. Hence, 
\begin{equation*}
\mathrm{Term}_1 = \left[ \rd_t+v_i \rd_{x_i}, \rd_x^\alp \rd_v^\bt Y^\sigma \right] g = \sum_{\substack{\bt'+\bt'' = \bt \\ |\bt'| = 1}} \rd_x^{\bt'}\rd_v^{\bt''} \rd_x^\alp Y^\sigma g.
\end{equation*}
Therefore,
\begin{equation}\label{eq:Term1.bound}
|\mathrm{Term}_1|\ls \sum_{|\alp'|\leq |\alp|+1,\,|\bt'|\leq |\bt|-1 }|\rd_x^{\alp'} \rd_v^{\bt'} Y^\sigma g|.
\end{equation}
This gives the contribution $V_p^{\alp,\bt,\sigma}$ in \eqref{def:Vp}.

For $\mathrm{Term}_2$, the commutator term arises from $\rd_v$ or $t\rd_x+\rd_v$ acting on $\vb^2$. We thus have
\begin{equation}\label{eq:Term2.bound}
\begin{split}
|\mathrm{Term}_2| \ls &\: \sum_{\substack{|\bt'|\leq |\bt|,\,|\sigma'|\leq |\sigma|\\ |\bt'|+|\sigma'|\leq |\bt|+|\sigma|-1}}  \f{\de d_0}{(1+t)^{1+\de}}|v| |\rd_x^\alp \rd_v^{\bt'} Y^{\sigma'} g| +\sum_{\substack{|\bt'|\leq |\bt|,\,|\sigma'|\leq |\sigma|\\ |\bt'|+|\sigma'|\leq |\bt|+|\sigma|-2}}  \f{\de d_0}{(1+t)^{1+\de}}|\rd_x^\alp \rd_v^{\bt'} Y^{\sigma'} g|.
\end{split}
\end{equation}
A very rough estimate then shows that this gives the contribution $VI_p^{\alp,\bt,\sigma}$ in \eqref{def:VIp}.

For $\mathrm{Term}_3$, distributing the derivatives we get
\begin{equation*}
\begin{split}
\mathrm{Term}_3 = &\: \sum_{\substack{\alp'+\alp'' = \alp\\ \bt'+\bt'' = \bt \\ \sigma'+\sigma''=\sigma \\ |\alp'|+|\bt'|+|\sigma'| \geq 1}} (\rd_x^{\alp'}\rd_v^{\bt'} Y^{\sigma'} \bar{a}_{ij}) \rd^2_{v_iv_j}\rd_x^{\alp''} \rd_v^{\bt''} Y^{\sigma''} g.
\end{split}
\end{equation*}
After relabeling the multi-indices, this gives the contribution $I_p^{\alp,\bt,\sigma}$ \eqref{def:I}.

For $\mathrm{Term}_4$, we distribute the derivatives to get 
\begin{equation}\label{eq:Term4.bound}
|\mathrm{Term}_4| \ls \sum_{\substack{|\alp'|+|\alp''| = |\alp| \\ |\bt'|+|\bt''|=|\bt| \\ |\sigma'|+|\sigma''|= |\sigma|}}|\rd_x^{\alp'} \rd_v^{\bt'} Y^{\sigma'} \bar{c}| |\rd_x^{\alp''} \rd_v^{\bt''} Y^{\sigma''} g|.
\end{equation}
This gives the contribution $IV_p^{\alp,\bt,\sigma}$ in \eqref{def:IVp}.

For $\mathrm{Term}_5$, since $d(t)$ is uniformly bounded, we distribute the derivatives to obtain
\begin{equation}\label{eq:Term5.bound}
|\mathrm{Term}_5| \ls \sum_{\substack{|\alp'|+|\alp''| = |\alp|,\, |\sigma'|+|\sigma''| = |\sigma|\\ |\bt'|+|\bt''| = |\bt| +1,\, |\bt'|\leq |\bt|}} \max_j|\rd_x^{\alp'} \rd_v^{\bt'} Y^{\sigma'} (\bar{a}_{ij}v_i)| |\rd_x^{\alp''} \rd_v^{\bt''} Y^{\sigma''} g|,
\end{equation}
which gives the contribution $II_p^{\alp,\bt,\sigma}$.

Finally, for $\mathrm{Term}_6$, using again the uniform boundedness of $d(t)$, it is easy to see that
\begin{equation}\label{eq:Term6.bound}
|\mathrm{Term}_6| \ls \sum_{\substack{|\alp'|+|\alp''| = |\alp|\\ |\bt'|+|\bt''| = |\bt| \\ |\sigma'|+|\sigma''| = |\sigma|}} (|\rd_x^{\alp'}\rd_v^{\bt'} Y^{\sigma'} \bar{a}_{ii}| + |\rd_x^{\alp'} \rd_v^{\bt'} Y^{\sigma'} (\bar{a}_{ij}v_iv_j)|) |\rd_x^{\alp''} \rd_v^{\bt''} Y^{\sigma''} g|,
\end{equation}
which gives the contribution $III_p^{\alp,\bt,\sigma}$. \qedhere
\end{proof}

\subsection{Controlling the error terms}\label{sec:Lierror}

We continue to work under the induction hypotheses in Section~\ref{sec:induction}.

The goal of this subsection is to control the terms \eqref{def:Ip}--\eqref{def:IVp} in Proposition~\ref{prop:mainLiLi}.

\begin{proposition}\label{prop:Ip}
Let $m_*$ be as in the induction hypotheses in Section~\ref{sec:induction} and suppose $k:=|\alp|+|\bt|+|\sigma|\leq m_*$. Then for $I_p^{\alp,\bt,\sigma}$ as in \eqref{def:Ip} and for $(t,x,v)\in [0,T_{Boot})\times \mathbb R^3\times \mathbb R^3$,
$$\wb^{\Mm+5-|\sigma|}I_p^{\alp,\bt,\sigma}(t,x,v)\ls \ep^{\f 32} \vb^{1+\th_{m_*}} (1+t)^{-1-\de+\zeta_{k}+|\bt|}.$$
\end{proposition}
\begin{proof}
From now on, take $\alp', \,\alp'',\, \bt',\,\bt'',\,\sigma',\,\sigma''$ which obey $|\alp'|+|\alp''|= |\alp|$, $|\bt'|+|\bt''|= |\bt|+2$, $|\sigma'|+|\sigma''|=|\sigma|$, $1\leq |\alp'|+|\bt'|+|\sigma'|\leq |\alp|+|\bt|+|\sigma|=:k$. We will always silently assume that these conditions are satisfied.

\textbf{Short-time estimates: $t\leq 1$.} We first bound $\wb^{\Mm+5-|\sigma|}I_p^{\alp,\bt,\sigma}(t,x,v)$ when $t\leq 1$. Estimating $\max_{i,j}|\rd_x^{\alp'}\rd_v^{\bt'} Y^{\sigma'}\bar{a}_{ij}|$ by Proposition~\ref{prop:ab.Li.1} and bounding $|\rd_x^{\alp''} \rd_v^{\bt''} Y^{\sigma''} g|(t,x,v)$ by the induction hypotheses, we obtain
\begin{equation}\label{short.time.gamma.small}
\begin{split}
&\: \wb^{\Mm+5-|\sigma|}\max_{i,j}|\rd_x^{\alp'}\rd_v^{\bt'} Y^{\sigma'}\bar{a}_{ij}| |\rd_x^{\alp''} \rd_v^{\bt''} Y^{\sigma''} g|(t,x,v)\\
\ls &\: \ep^{\f 34} \vb^{2+\gamma} \cdot \ep^{\f 34} \vb^{-1+\th_{m_*+1}}\ls \ep^{\f 32} \vb^{1+\gamma+\th_{m_*+1}}.
\end{split}
\end{equation}
To see that \eqref{short.time.gamma.small} implies the desired estimates when $t\leq 1$, note that by \eqref{BA.Z}--\eqref{BA.Z.top}, $\th_{m_*+1}+\gamma \leq \th_{m_*}$.

\textbf{Long-time estimates: $t\geq 1$.} We now assume $t\geq 1$. Notice that we have $|\alp'|+|\bt'|+|\sigma'|\geq 1$. We divide below into the (non-mutually exclusive) cases $|\sigma'|\geq 1$ (Case~1) and $\max\{|\alp'|,\,|\bt'|\}\geq 1$ (Case~2).

\pfstep{Case~1: $|\sigma'|\geq 1$} Since $|\sigma'|\geq 1$, we have $|\sigma''|\leq |\sigma|-1$. This implies $\wb^{\Mm+5-|\sigma|} \ls \wb^{\Mm+5-|\sigma''|}\wb^{-1}$

Since $k\leq m_* \leq \Mm-5$, by \eqref{eq:MmMi}, either $|\alp'|+|\bt'|+|\sigma'| \leq \Mi$ or $|\alp''|+|\bt''|+|\sigma''|\leq \Mi$. Term 1 below handles the situations where $|\alp''|+|\bt''|+|\sigma''|\leq \Mi$, while term 2 below handles the situations where $|\alp'|+|\bt'|+|\sigma'|\leq \Mi$. To obtain term 1, we note that $|\alp'|+|\bt'|+|\sigma'|\leq k$ and apply Proposition~\ref{prop:ab.Li.null.cond} (to control $\wb^{-1}|\rd_x^{\alp'}\rd_v^{\bt'}Y^{\sigma'} \bar{a}_{ij}|$) and the induction hypotheses (to control $\rd_x^{\alp''} \rd_v^{\bt''} Y^{\sigma''} g$). To obtain term 2, we note that $|\alp''|+|\bt''|+|\sigma''|\leq k+1$ and apply Proposition~\ref{prop:ab.Li.null.cond} (to control $\wb^{-1}|\rd_x^{\alp'}\rd_v^{\bt'} Y^{\sigma'} \bar{a}_{ij}|$) and the induction hypotheses (to control $\rd_x^{\alp''} \rd_v^{\bt''} Y^{\sigma''} g$).
\begin{equation}\label{estimate.we.prove.in.case.1}
\begin{split}
&\: \wb^{\Mm+5-|\sigma|}\max_{i,j}|\rd_x^{\alp'} \rd_v^{\bt'} Y^{\sigma'} \bar{a}_{ij}| |\rd_x^{\alp''} \rd_v^{\bt''} Y^{\sigma''} g|(t,x,v)\\
\ls &\: (\max_{i,j}\wb^{-1}|\rd_x^{\alp'} \rd_v^{\bt'} Y^{\sigma'} \bar{a}_{ij}|(t,x,v))(\wb^{\Mm+5-|\sigma''|} |\rd_x^{\alp''} \rd_v^{\bt''} Y^{\sigma''} g|(t,x,v))\\
\ls &\: \underbrace{\ep^{\f 34} \vb^{\max\{0,1+\gamma\}}(1+t)^{-3-\min\{2+\gamma,1\}+\zeta_k+|\bt'|} \cdot \ep^{\f 34} (1+t)^{|\bt''|}}_{\mbox{Case 1, Term 1}} \\
&\: + \underbrace{\ep^{\f 34} \vb^{\max\{0,1+\gamma\}}(1+t)^{-3-\min\{2+\gamma,1\}+|\bt'|} \cdot \ep^{\f 34} (1+t)^{\zeta_{k+1}+|\bt''|}}_{\mbox{Case 1, Term 2}}\\
= &\: \ep^{\f 32} \vb (1+t)^{-1-\min\{2+\gamma,1\}+\zeta_{k+1}+|\bt|},
\end{split}
\end{equation}
where in the last line we have used $|\bt'|+|\bt''|\leq |\bt|+2$ and $\gamma <0$. Now note that by \eqref{BA.Z}--\eqref{BA.Z.top}, $\zeta_{k+1}-\zeta_k\leq \min\{\f 34, \f{3(2+\gamma)}{4}\}$. Therefore, combining this with \eqref{def:de}, we obtain
\begin{equation}\label{zeta.compare}
-1-\min\{2+\gamma,1\}+\zeta_{k+1}+|\bt| \leq -1-\de+\zeta_k+|\bt|.
\end{equation}
On the other hand, note also that since by \eqref{BA.Z}--\eqref{BA.Z.top} $\th_{m_*}\geq 0$, we have $1\leq 1+\th_{m_*}$. Hence, the estimate we obtained in \eqref{estimate.we.prove.in.case.1} above is better than is required in the statement of the proposition.

\pfstep{Case~2: $|\alp'|\geq 1$ or $|\bt'|\geq 1$} By \eqref{eq:MmMi}, either $|\alp''|+|\bt''|+|\sigma''|\leq \Mi$ or $|\alp'|+|\bt'|+|\sigma'|\leq \Mi$. These are handled respectively as term~1 and term~2 below. In term~1, we also have $|\alp'|+|\bt'|+|\sigma'|\leq k$; while in term~2 we also have $|\alp''|+|\bt''|+|\sigma''|\leq k+1$. For each of these two terms, we bound $\rd_x^{\alp'}\rd_v^{\bt'} Y^{\sigma'} \bar{a}_{ij}$ using Proposition~\ref{prop:ab.Li.2} or \ref{prop:ab.Li.3} (applicable since $|\bt'|\geq 1$ or $|\alp'|\geq 1$) and estimate $\rd_x^{\alp''} \rd_v^{\bt''} Y^{\sigma''} g$ using the induction hypotheses. 
\begin{equation}\label{estimates.we.prove.in.case.2}
\begin{split}
&\: \wb^{\Mm+5-|\sigma|}\max_{i,j}|\rd_x^{\alp'} \rd_v^{\bt'} Y^{\sigma'} \bar{a}_{ij}| |\rd_x^{\alp''} \rd_v^{\bt''} Y^{\sigma''} g|(t,x,v)\\
\ls &\: (\max_{i,j}|\rd_x^{\alp'}\rd_v^{\bt'} \bar{a}_{ij}|(t,x,v))(\wb^{\Mm+5-|\sigma''|} |\rd_x^{\alp''} \rd_v^{\bt''} Y^{\sigma''} g|(t,x,v))\\
\ls &\: \underbrace{\ep^{\f 34} \vb^{2+\gamma}(1+t)^{-3-\min\{2+\gamma,1\}+\zeta_k+|\bt'|} \cdot \ep^{\f 34} \vb^{-1+\th_{m_*+1}} (1+t)^{|\bt''|}}_{\mbox{Case 2, Term 1}} \\
&\: + \underbrace{\ep^{\f 34} \vb^{2+\gamma}(1+t)^{-3-\min\{2+\gamma,1\}+|\bt'|} \cdot \ep^{\f 34} \vb^{-1+\th_{m_*+1}} (1+t)^{\zeta_{k+1}+ |\bt''|}}_{\mbox{Case 2, Term 2}}\\
\ls &\: \ep^{\f 32} \vb^{1+\gamma+\th_{m_*+1}} (1+t)^{-1-\min\{2+\gamma,1\}+\zeta_{k+1}+|\bt|},
\end{split}
\end{equation}
where we have used $|\bt'|+|\bt''|\leq |\bt|+2$. Now note that by \eqref{BA.Z}--\eqref{BA.Z.top}, $\th_{m_*+1}+\gamma \leq \th_{m_*}$. Therefore, using also \eqref{zeta.compare}, we see that \eqref{estimates.we.prove.in.case.2} is better than is required in the statement of the proposition. \qedhere

\end{proof}

\begin{proposition}\label{prop:IIp}
Let $m_*$ be as in the induction hypotheses in Section~\ref{sec:induction} and suppose $k:=|\alp|+|\bt|+|\sigma|\leq m_*$. Then for $II_p^{\alp,\bt,\sigma}$ as in \eqref{def:IIp} and for $(t,x,v)\in [0,T_{Boot})\times \mathbb R^3\times \mathbb R^3$,
$$\wb^{\Mm+5-|\sigma|}II_p^{\alp,\bt,\sigma}(t,x,v)\ls \ep^{\f 32} \vb^{1+\th_{m_*}} (1+t)^{-1-\de+\zeta_{k}+|\bt|}.$$
\end{proposition}
\begin{proof}
We can argue in a very similar manner as Proposition~\ref{prop:Ip} so only the key points will be sketched. From now on take $\alp'$, $\alp''$, $\bt'$, $\bt''$, $\sigma'$ and $\sigma''$ as in $II_p^{\alp,\bt,\sigma}$ in \eqref{def:IIp}.

First, for $t\leq 1$, we have a similar bound as \eqref{short.time.gamma.small} using Proposition~\ref{prop:ab.Li.weighted} and the induction hypotheses:
\begin{equation*}
\begin{split}
&\: \wb^{\Mm+5-|\sigma|}\max_{j}|\rd_x^{\alp'}\rd_v^{\bt'}Y^{\sigma'} (\bar{a}_{ij} v_i)| |\rd_x^{\alp''} \rd_v^{\bt''} Y^{\sigma''} g|(t,x,v)\\
\ls &\: \ep^{\f 34} \vb^{\max\{2+\gamma,1\}} \cdot \ep^{\f 34} \vb^{-1+\th_{m_*+1}}\ls \ep^{\f 32} \vb^{\max\{1+\gamma+\th_{m_*+1},\th_{m_*+1}\}}\ls \ep^{\f 34} \vb^{1+\th_{m_*}}.
\end{split}
\end{equation*}
In the last step above, we have used 
\begin{equation}\label{complicated.v.weights}
\max\{1+\gamma+\th_{m_*+1},\th_{m_*+1}\}\leq 1+\th_{m_*},
\end{equation}
which can be checked using \eqref{BA.Z}--\eqref{BA.Z.top}.

For the $t\geq 1$, we argue as in Case~2 (i.e.~the $|\alp'|\geq 1$ or $|\bt'|\geq 1$ case) in the proof of Proposition~\ref{prop:Ip}. We note the following:
\begin{enumerate}
\item We have $a_{ij}v_i$ instead of $a_{ij}$ so that we will use Proposition~\ref{prop:ab.Li.weighted} in place of Propositions~\ref{prop:ab.Li.2} and \ref{prop:ab.Li.3}. (Note that the application of Proposition~\ref{prop:ab.Li.weighted} does not require $|\alp'|\geq 1$ or $|\bt'|\geq 1$.)
\item By comparing Proposition~\ref{prop:ab.Li.weighted} with Propositions~\ref{prop:ab.Li.2} and \ref{prop:ab.Li.3}, we see that the estimates we obtain for $\max_{j}|\rd_x^{\alp'}\rd_v^{\bt'}Y^{\sigma'} (\bar{a}_{ij} v_i)|$ are different from those we obtain in the $|\alp'|+|\bt'|\geq 1$ case for $\max_{i,j}|\rd_x^{\alp'}\rd_v^{\bt'}Y^{\sigma'} \bar{a}_{ij}|$ in the following ways:
\begin{enumerate}
\item We have $\vb^{\max\{2+\gamma,1\}}$ instead of $\vb^{2+\gamma}$, and
\item we have a $t$ rate that is $(1+t)^{\min\{2+\gamma,1\}}$ worse.
\end{enumerate}
\item On the other hand, for the $II_p^{\alp,\bt,\sigma}$ terms, we have $|\bt'|+|\bt''|\leq |\bt|+1$ instead of $|\bt'|+|\bt''|\leq |\bt|+2$ as in $I_p^{\alp,\bt,\sigma}$. This gives a gain of $(1+t)^{-1}$.
\end{enumerate}
Combining all these observations and making the necessary changes of the proof in Case~2 of Proposition~\ref{prop:Ip}, we obtain the following analogue of \eqref{estimates.we.prove.in.case.2}:
\begin{equation*}
\begin{split}
&\: \wb^{\Mm+5-|\sigma|}\max_{j}|\rd_x^{\alp'} \rd_v^{\bt'} Y^{\sigma'} (\bar{a}_{ij}v_i)| |\rd_x^{\alp''} \rd_v^{\bt''} Y^{\sigma''} g|(t,x,v)\\
\ls &\: \ep^{\f 32} \vb^{\max\{1+\gamma+\th_{m_*+1},\th_{m_*+1}\}} (1+t)^{-2+\zeta_{k+1}+|\bt|}.
\end{split}
\end{equation*}
Now using \eqref{complicated.v.weights} and \eqref{zeta.compare}, we obtain the desired conclusion. \qedhere
\end{proof}

\begin{proposition}\label{prop:IIIp}
Let $m_*$ be as in the induction hypotheses in Section~\ref{sec:induction} and suppose $k:=|\alp|+|\bt|+|\sigma|\leq m_*$. Then for $III_p^{\alp,\bt,\sigma}$ as in \eqref{def:IIIp} and for $(t,x,v)\in [0,T_{Boot})\times \mathbb R^3\times \mathbb R^3$,
$$\wb^{\Mm+5-|\sigma|}III_p^{\alp,\bt,\sigma}(t,x,v)\ls \ep^{\f 32} \vb^{1+\th_{m_*}} (1+t)^{-1-\de+\zeta_{k}+|\bt|}.$$
\end{proposition}
\begin{proof}
The $III_p^{\alp,\bt,\sigma}$ is even better behaved than the $II_p^{\alp,\bt,\sigma}$ term. To see this, note that by Propositions~\ref{prop:ab.Li.1} and \ref{prop:ab.Li.weighted}, $\rd_x^{\alp'}\rd_v^{\bt'}Y^{\sigma'}\bar{a}_{ii}$ and $\rd_x^{\alp'}\rd_v^{\bt'}Y^{\sigma'}(\bar{a}_{ij}v_iv_j)$ obey all the estimates that $\rd_x^{\alp'}\rd_v^{\bt'}Y^{\sigma'}(\bar{a}_{ij}v_i)$ satisfy. Note moreover that for the $III_p^{\alp,\bt,\sigma}$ terms, we have $|\bt'|+|\bt''|\leq |\bt|$ instead of $|\bt'|+|\bt''|\leq |\bt|+1$ as in $II_p^{\alp,\bt,\sigma}$, which therefore gives a better estimate in terms of $(1+t)$. Hence the term $III_p^{\alp,\bt,\sigma}$ obeys all the estimates that $II_p^{\alp,\bt,\sigma}$ satisfy. The conclusion then follows from Proposition~\ref{prop:IIp}. \qedhere
\end{proof}

\begin{proposition}\label{prop:IVp}
Let $m_*$ be as in the induction hypotheses in Section~\ref{sec:induction} and suppose $k:=|\alp|+|\bt|+|\sigma|\leq m_*$. Then for $IV_p^{\alp,\bt,\sigma}$ as in \eqref{def:IVp} and for $(t,x,v)\in [0,T_{Boot})\times \mathbb R^3\times \mathbb R^3$,
$$\wb^{\Mm+5-|\sigma|}IV_p^{\alp,\bt,\sigma}(t,x,v)\ls \ep^{\f 32} (1+t)^{-1-\de+\zeta_{k}+|\bt|}.$$
\end{proposition}
\begin{proof}
In this proof, we will also take $\alp',\,\alp'',\,\bt',\,\bt'',\,\sigma',\,\sigma''$ as in the term $IV_p^{\alp,\bt,\sigma}$, i.e.~$|\alp'|+|\alp''|\leq |\alp|$, $|\bt'|+|\bt''|\leq |\bt|$ and $|\sigma'|+|\sigma''|\leq |\sigma|$.

By \eqref{eq:MmMi}, we have either $|\alp''|+|\bt''|+|\sigma''|\leq \Mi$ or $|\alp'|+|\bt'|+|\sigma'|\leq \Mi$. We treat these two cases respectively in term 1 and term 2 below. Note that in both terms we also have $|\alp'|+|\bt'|+|\sigma'|\leq k$ and $|\alp''|+|\bt''|+|\sigma''|\leq k$. In each of these terms, we use Proposition~\ref{prop:cb.Li} to control $|\rd_x^{\alp'}\rd_v^{\bt'}Y^{\sigma'}\bar{c}|$ and use the induction hypotheses to control $\wb^{\Mm+5-|\sigma''|} |\rd_x^{\alp''} \rd_v^{\bt''} Y^{\sigma''} g|$. 
\begin{equation*}
\begin{split}
&\: \wb^{\Mm+5-|\sigma|}|\rd_x^{\alp'}\rd_v^{\bt'}Y^{\sigma'} \bar{c}| |\rd_x^{\alp''} \rd_v^{\bt''} Y^{\sigma''} g|(t,x,v)\\
\ls &\: |\rd_x^{\alp'}\rd_v^{\bt'}Y^{\sigma'} \bar{c}|(t,x,v) (\wb^{\Mm+5-|\sigma''|} |\rd_x^{\alp''} \rd_v^{\bt''} Y^{\sigma''} g|(t,x,v))\\
\ls &\: \underbrace{\ep^{\f 34} (1+t)^{-3-\gamma+\zeta_k+|\bt'|} \cdot \ep^{\f 34}  (1+t)^{|\bt''|}}_{\mbox{Term 1}} + \underbrace{\ep^{\f 34} (1+t)^{-3-\gamma+|\bt'|} \cdot \ep^{\f 34}  (1+t)^{\zeta_k+ |\bt''|}}_{\mbox{Term 2}}\\
\ls &\: \ep^{\f 32} (1+t)^{-3-\gamma+\zeta_k+|\bt|},
\end{split}
\end{equation*}
where in the last line we have used $|\bt'|+|\bt''|=|\bt|$. The statement of the proposition hence follows from \eqref{def:de}. \qedhere
\end{proof}

\subsection{Concluding the induction argument}\label{sec:Li.everything}

We continue to work under the induction hypotheses in Section~\ref{sec:induction}. Our goal in this subsection will be to conclude the induction argument.

Combining Proposition~\ref{prop:mainLiLi} and Propositions~\ref{prop:Ip}--\ref{prop:IVp}, we immediately obtain
\begin{proposition}\label{prop:MP}
Let $m_*$ be as in the induction hypotheses in Section~\ref{sec:induction} and suppose $k:=|\alp|+|\bt|+|\sigma|\leq m_*$. Then for $(t,x,v)\in [0,T_{Boot})\times \mathbb R^3\times \mathbb R^3$,
\begin{equation}\label{prop:MP.1}
\begin{split}
&\: \wb^{\Mm+5-|\sigma|} |(\rd_t + v_i\rd_{x_i} -\bar{a}_{ij} \rd^2_{v_i v_j}) (\rd_x^\alp \rd_v^\bt Y^\sigma g)|(t,x,v) \\
\ls &\: \ep^{\f 32} \vb^{1+\th_{m_*}} (1+t)^{-1-\de+\zeta_{k}+|\bt|} +\sum_{\substack{|\alp'|\leq |\alp|+1 \\ |\bt'|\leq |\bt|-1 }} \wb^{\Mm+5-|\sigma|}|\rd_x^{\alp'} \rd_v^{\bt'} Y^{\sigma} g|(t,x,v)\\
&\:+\sum_{\substack{|\bt'|\leq |\bt|,\,|\sigma'|\leq |\sigma|\\ |\bt'|+|\sigma'|\leq |\bt|+|\sigma|-1}}\f{\vb}{(1+t)^{1+\de}}\wb^{\Mm+5-|\sigma|}|\rd_x^{\alp} \rd_v^{\bt'} Y^{\sigma'} g|(t,x,v).
\end{split}
\end{equation}
\end{proposition}
\begin{proof}
It suffices to bound the terms $I_p^{\alp,\bt,\sigma},\dots, VI_p^{\alp,\bt,\sigma}$ in Proposition~\ref{prop:mainLiLi}. Propositions~\ref{prop:Ip}--\ref{prop:IVp} exactly show that the terms $I_p^{\alp,\bt,\sigma}$, $II_p^{\alp,\bt,\sigma}$, $III_p^{\alp,\bt,\sigma}$ and $IV_p^{\alp,\bt,\sigma}$ are bounded above by the first term on the RHS of \eqref{prop:MP.1}. Finally, the terms $V_p^{\alp,\bt,\sigma}$ and $VI_p^{\alp,\bt,\sigma}$ are exactly the last two terms on the RHS of \eqref{prop:MP.1}. \qedhere
\end{proof}

\begin{proposition}\label{prop:MP.cond}
Let $m_*$ be as in the induction hypotheses in Section~\ref{sec:induction} and suppose $k:=|\alp|+|\bt|+|\sigma|\leq m_*$. Let $h:=\rd_x^\alp \rd_v^\bt Y^\sigma g$. Then the assumptions 1, 2 and 4 in Proposition~\ref{prop:max.prin} hold with $(N,r_H,p_H)$ defined as follows (and depend on $\gamma$, $d_0$, $|\alp|$, $|\bt|$ and $|\sigma|$):
\begin{equation*}
\begin{split}
N =&\: \Mm+5-|\sigma|,\\
r_H =&\: -1-\de+\zeta_k+|\bt|,\\
p_H=&\:1+\th_{m_*},
\end{split}
\end{equation*}
and $C_H\geq 1$ some constant depending only on $\gamma$ and $d_0$.
\end{proposition}
\begin{proof}
\pfstep{Step~1: Verifying assumption~1} This is an immediate corollary of the preliminary $L^\infty$ estimate in Proposition~\ref{prop:prelim.Linfty}.

\pfstep{Step~2: Verifying assumption~2} In view of the definition of $(r_H, p_H)$, \eqref{dvh.bound} and the induction hypotheses in Section~\ref{sec:induction}, we need to check that 
$$\th_{m_*+1}-1\leq \min\{\th_{m_*}-1-\gamma,\th_{m_*}\},\quad \zeta_{k+1}+1+|\bt| \leq 1-\de+\zeta_k+\min\{2+\gamma,1\}+|\bt|.$$
The first inequality can be checked explicitly using \eqref{BA.Z}--\eqref{BA.Z.top}, while the second inequality is equivalent to \eqref{zeta.compare} that we have already checked.

\pfstep{Step~3: Verifying assumption~4} This is an immediate consequence of the assumptions of Theorem~\ref{thm:main}. \qedhere
\end{proof}

\begin{proposition}\label{prop:Li.final}
Let $m_*$ be as in the induction hypotheses in Section~\ref{sec:induction} and suppose $k:=|\alp|+|\bt|+|\sigma|\leq m_*$. Then for $(t,x,v)\in [0,T_{Boot})\times \mathbb R^3\times \mathbb R^3$,
\begin{equation}\label{MP.con}
\wb^{\Mm+5-|\sigma|}|\rd_x^\alp \rd_v^\bt Y^\sigma g|(t,x,v) \ls \ep \vb^{-1+\th_{m_*}} (1+t)^{\zeta_k+|\bt|}.
\end{equation}
\end{proposition}
\begin{proof}
The idea of the proof is to use the maximum principle in Proposition~\ref{prop:max.prin} together with the estimates established in Propositions~\ref{prop:MP} and \ref{prop:MP.cond}. In Proposition~\ref{prop:MP.cond}, we have checked assumptions 1, 2 and 4 of Proposition~\ref{prop:max.prin}. We want to use Proposition~\ref{prop:MP} to verify assumption 3 of Proposition~\ref{prop:max.prin}. The only remaining issue is to handle the last two terms in \eqref{prop:MP.1}. For this reason, we proceed by an induction argument on $|\bt|+|\sigma|$.

\pfstep{Base case: $|\bt|+|\sigma|=0$} In this case, the last two terms on the RHS of \eqref{prop:MP.1} are not present. Hence we simply have, for every $|\alp|\leq m_*$,
$$\wb^{\Mm+5}|(\rd_t + v_i\rd_{x_i} -\bar{a}_{ij} \rd^2_{v_i v_j}) (\rd_x^\alp g)|(t,x,v) \ls \ep^{\f 32} \vb^{1+\th_{m_*}} (1+t)^{-1-\de+\zeta_{k}}.$$
Therefore, by Propositions~\ref{prop:max.prin} (with $(N, r_H, p_H) = (\Mm+5-|\sigma|, -1-\de+\zeta_k+|\bt|,1+\th_{m_*})$) and \ref{prop:MP.cond}, we obtain
$$\wb^{\Mm+5}|\rd_x^\alp g|(t,x,v)\ls \ep \vb^{-1+\th_{m_*}} (1+t)^{\zeta_{k}},$$
as desired.

\pfstep{Induction step} Assume as an induction hypothesis that there exists $B\in\mathbb N$ such that if $|\alp|+|\bt|+|\sigma| \leq m_*$ and $|\bt|+|\sigma|\leq B-1$, then \eqref{MP.con} holds.

We now take $\alp$, $\bt$, $\sigma$ such that $|\alp|+|\bt|+|\sigma|=:k\leq m_*$ and $|\bt|+|\sigma|=B$. Our goal is to show that \eqref{MP.con} holds for this choice of $(\alp,\bt,\sigma)$.

It is easy to see that after plugging in the estimates in the induction hypothesis into \eqref{prop:MP.1} in Proposition~\ref{prop:MP}, we obtain
\begin{equation}\label{inductive.step.for.MP}
\begin{split}
&\: \wb^{\Mm+5-|\sigma|} |(\rd_t + v_i\rd_{x_i} -\bar{a}_{ij} \rd^2_{v_i v_j}) (\rd_x^\alp\rd_v^\bt Y^\sigma g)|(t,x,v) \\
\ls &\: \begin{cases}
\ep^{\f 32} \vb^{1+\th_{m_*}} (1+t)^{-1-\de+\zeta_{k}+|\bt|} +\ep \vb^{-1+\th_{m_*}} (1+t)^{-1+\zeta_k+|\bt|} + \ep \vb^{\th_{m_*}} (1+t)^{-1-\de+\zeta_{k-1}+|\bt|}&\mbox{if $|\bt|\geq 1$}\\
\ep^{\f 32} \vb^{1+\th_{m_*}} (1+t)^{-1-\de+\zeta_{k}} + \ep \vb^{\th_{m_*}} (1+t)^{-1-\de+\zeta_{k-1}}&\mbox{if $|\bt|=0$}
\end{cases}\\
\ls &\: \begin{cases}
\ep \vb^{1+\th_{m_*}} (1+t)^{-1-\de+\zeta_{k}+|\bt|} + \ep \vb^{-1+\th_{m_*}} (1+t)^{-1+\zeta_k+|\bt|}&\mbox{if $|\bt|\geq 1$}\\
\ep \vb^{1+\th_{m_*}} (1+t)^{-1-\de+\zeta_{k}+|\bt|}&\mbox{if $|\bt|=0$}
\end{cases}.
\end{split}
\end{equation}

Using the estimate in \eqref{inductive.step.for.MP} and the bounds in Proposition~\ref{prop:MP.cond}, we now apply the maximum principle in Propositions~\ref{prop:max.prin} (again with $(N, r_H, p_H) = (\Mm+5-|\sigma|, -1-\de+\zeta_k+|\bt|,1+\th_{m_*})$). Note that when $|\bt|\geq 1$, we have $\zeta_k+|\bt|\geq 0$ so we can allow the second term on the last line of \eqref{inductive.step.for.MP}. This yields the desired estimates \eqref{MP.con} for $|\alp|+|\bt|+|\sigma|\leq m_*$ and $|\bt|+|\sigma|=B$. 

By induction on $|\bt|+|\sigma|$, we have thus proven \eqref{MP.con} for all $\alp$, $\bt$, $\sigma$ such that $|\alp|+|\bt|+|\sigma|\leq m_*$. \qedhere

\end{proof}

Proposition~\ref{prop:Li.final} concludes the induction argument (on $m_*$) initiated in Section~\ref{sec:induction}. As a consequence, \eqref{main.induction} holds for all $1\leq m \leq \Mm-4$. Moreover, our arguments showed that \eqref{main.induction} holds with $\ep^{\f 34}$ replaced by $\ep$. This in turn implies that \eqref{MP.con} holds for all $k:=|\alp|+|\bt|+|\sigma|\leq m$ with $1\leq m\leq \Mm-5$. We summarize this in the following corollary:
\begin{corollary}\label{cor:MP.con}
Let $k:=|\alp|+|\bt|+|\sigma|$ with $1\leq k\leq \Mm-5$. Then for $(t,x,v)\in [0,T_{Boot})\times \mathbb R^3\times \mathbb R^3$,
\begin{equation*}
\wb^{\Mm+5-|\sigma|}|\rd_x^\alp \rd_v^\bt Y^\sigma g|(t,x,v) \ls \ep \vb^{-1+\th_k} (1+t)^{\zeta_k+|\bt|}.
\end{equation*}
\end{corollary}

\subsection{Recovering the bootstrap assumptions \eqref{BA.sim.1} and \eqref{BA.sim.2}}\label{sec:recover.Li}

The $L^\i_xL^\i_v$ estimates obtained in Corollary~\ref{cor:MP.con} in particular improve the constants in the bootstrap assumptions \eqref{BA.sim.1} and \eqref{BA.sim.2}. We record this in the following proposition.
\begin{proposition}\label{prop:Li.improved}
If $|\alp|+|\bt|+|\sigma| \leq \Mm-4-\max\{2,\lceil \f{2}{2+\gamma} \rceil \}$, then for every $t\in [0,T_{Boot})$,
\begin{equation}\label{BA.sim.1.improved}
\|\wb^{\Mm+5-|\sigma|} \rd_x^{\alp} \rd_v^{\bt} Y^{\sigma} g \|_{L^\i_x L^\i_v}(t) \ls \ep(1+t)^{|\bt|}.
\end{equation}
If $\Mm-3-\max\{2,\lceil \f{2}{2+\gamma} \rceil \}\leq |\alp|+|\bt|+|\sigma|=:k \leq \Mm-5$, then for every $t\in [0,T_{Boot})$,
\begin{equation}\label{BA.sim.2.improved}
\|\wb^{\Mm+5-|\sigma|} \rd_x^{\alp} \rd_v^{\bt} Y^{\sigma} g \|_{L^\i_x L^\i_v}(t) \ls \ep (1+t)^{\f 32 - (\Mm-4-k)\min\{\f 34, \f{3(2+\gamma)}{4}\} +|\bt|}.
\end{equation}
\end{proposition}
\begin{proof}
This is an immediate consequence of Corollary~\ref{cor:MP.con}: first, note that $-1+\th_k\leq 0$ so that we can drop the $\vb$ weights; then compare the definition of $\zeta_k$ in \eqref{BA.Z}--\eqref{BA.Z.top} with the $t$-rates in \eqref{BA.sim.1.improved} and \eqref{BA.sim.2.improved}. \qedhere
\end{proof}

\section{Energy estimates}\label{sec:EE}
We continue to work under the assumptions of Theorem~\ref{thm:BA}.

In this section, our goal is to prove $L^2_xL^2_v$ for $g$ and its derivatives. We begin in \textbf{Section~\ref{sec:EE.prelim}} by obtaining some preliminary estimates which will later be used to control some error terms. In \textbf{Section~\ref{sec:main.EE}}, we prove our main energy estimates, and classify the error terms that arise. In \textbf{Section~\ref{sec:EE.error}}, we control all the error terms from Section~\ref{sec:main.EE}.In \textbf{Section~\ref{sec:EE.everything}}, we then put together the estimates in Sections~\ref{sec:main.EE} and \ref{sec:EE.everything} and conclude the energy estimates using an induction argument. Finally, in \textbf{Section~\ref{sec:end.of.bootstrap}}, we complete the proof of Theorem~\ref{thm:BA}.

\subsection{Preliminary estimates}\label{sec:EE.prelim}

\begin{lemma}\label{lem:move.t.weights}
For $T\in [0,T_{\mathrm{Boot}})$ and $|\alp|+|\bt|+|\sigma|\leq \Mm$,
$$\|(1+t)^{-\f 12-\f{2\de} 2-|\bt|} \vb \wb^{\Mm+5-|\sigma|}\rd_x^{\alp} \rd_v^{\bt} Y^{\sigma}g\|_{L^2([0,T];L^2_x L^2_v)} \ls \ep^{\f 34}.$$
\end{lemma}
\begin{proof}
We can assume that $T> 1$ for otherwise the inequality is an immediate consequence of the bootstrap assumption \eqref{BA}.

We split the integration in time into dyadic intervals. More precisely, let $k = \lceil \log_2 T \rceil+1$. Define $\{T_i\}_{i=0}^{k}$ with $T_0<T_1<T_2<\dots <T_k$, where $T_0=0$, $T_i = 2^{i-1}$ when $i =1,\dots,k-1$ and $T_k = T$.
Note that by the bootstrap assumption \eqref{BA}, for any $i=1,\dots,k$, 
$$\|(1+t)^{-\f 12-\f{\de} 2} \vb \wb^{\Mm+5-|\sigma|}\rd_x^{\alp} \rd_v^{\bt}Y^{\sigma}  g\|_{L^2([T_{i-1},T_i];L^2_x L^2_v)} \ls \ep^{\f 34} 2^{|\bt|i}.$$
Therefore,
\begin{equation*}
\begin{split}
&\: \|(1+t)^{-\f 12-\f{2\de} 2-|\bt|} \vb \wb^{\Mm+5-|\sigma|} \rd_x^{\alp} \rd_v^{\bt}Y^{\sigma} g\|_{L^2([0,T];L^2_x L^2_v)}\\
\ls &\: (\sum_{i=1}^k \|(1+t)^{-\f 12-\f{2\de} 2-|\bt|} \vb \wb^{\Mm+5-|\sigma|} \rd_x^{\alp} \rd_v^{\bt}Y^{\sigma} g\|_{L^2([T_{i-1},T_i];L^2_x L^2_v)}^2)^{\f 12} \\
\ls &\: (\sum_{i=1}^k 2^{-2|\bt|i-\de i} \|(1+t)^{-\f 12-\f{\de} 2} \vb \wb^{\Mm+5-|\sigma|} \rd_x^{\alp} \rd_v^{\bt}Y^{\sigma} g\|_{L^2([T_{i-1},T_i];L^2_x L^2_v)}^2)^{\f 12} \\
\ls &\: \ep^{\f 34}(\sum_{i=1}^k 2^{-2|\bt|i-\de i}\cdot 2^{2|\bt|i})^{\f 12} = \ep^{\f 34}(\sum_{i=1}^k 2^{-\de i})^{\f 12} \ls \ep^{\f 34},
\end{split}
\end{equation*}
which is what we claimed. \qedhere
\end{proof}

We next prove an interpolation estimate for the lower order (i.e.~with $|\alp|+|\bt|+|\sigma|\leq \Mi$) norms, which is an immediate consequence of Proposition~\ref{prop:Li.final} and Lemma~\ref{lem:move.t.weights}.

\begin{proposition}\label{main.bulk.est}
Let $|\alp|+|\bt|+|\sigma|\leq \Mi$ (where $\Mi$ is as in \eqref{def:Mi}). Then for every $T\in [0,T_{Boot})$,
\begin{equation}\label{main.bulk.est.1}
\|(1+t)^{-\f 12-\de-|\bt|}\vb \wb^{\Mm+5-|\sigma|}\rd_x^\alp \rd_v^\bt Y^\sigma g\|_{L^2([0,T];L^\i_x L^2_v\cap L^\i_x L^\i_v)} \ls \ep^{\f 34}.
\end{equation}
As a consequence, for any $p\in [2,\infty]$, the following holds for every $T\in [0,T_{Boot})$
\begin{equation}\label{main.bulk.est.2}
\|(1+t)^{-\f 12-\de-|\bt|}\vb \wb^{\Mm+5-|\sigma|}\rd_x^\alp \rd_v^\bt Y^\sigma g\|_{L^2([0,T];L^\i_x L^2_v\cap L^\i_x L^p_v)} \ls \ep^{\f 34}.
\end{equation}
\end{proposition}
\begin{proof}
For the remainder of the proof take $|\alp|+|\bt|+|\sigma|\leq \Mi$.

\pfstep{Step~1: Proof of \eqref{main.bulk.est.1}} By definition, it suffices to prove this estimate separately for the $L^2([0,T];L^\i_x L^2_v)$ norm and the $L^2([0,T];L^\i_x L^\i_v)$ norm.

To control the $L^2([0,T];L^\i_x L^2_v)$ norm, let us in fact prove an estimate for the stronger norm $L^2([0,T]; L^2_v L^\i_x)$. For each $(t,v)\in [0,T]\times \mathbb R^3$, standard Sobolev embedding in $\mathbb R^3$ (for the $x$ variables) yields
\begin{equation*}
\begin{split}
&\: \|(1+t)^{-\f 12-\de-|\bt|}\vb \wb^{\Mm+5-|\sigma|} \rd_x^\alp \rd_v^\bt Y^\sigma g\|_{L^\i_x}(t,v)\\
\ls &\: \sum_{|\alp'|\leq 2} \|\rd_x^{\alp'}((1+t)^{-\f 12-\de-|\bt|}\vb \wb^{\Mm+5-|\sigma|} \rd_x^\alp \rd_v^\bt Y^\sigma g)\|_{L^2_x}(t,v) \\
\ls &\: \sum_{|\alp''|\leq |\alp|+2} \|(1+t)^{-\f 12-\de-|\bt|}\vb \wb^{\Mm+5-|\sigma|} \rd_x^{\alp''}\rd_v^\bt Y^\sigma g\|_{L^2_x}(t,v)
\end{split}
\end{equation*}
Taking the $L^2_v$ norm and then the $L^2([0,T])$ norm on both sides, and using Lemma~\ref{lem:move.t.weights} (which is applicable, since if $|\alp|+|\bt|+|\sigma|\leq \Mi$, then $|\alp|+|\bt|+|\sigma|+2\leq \Mm$ by \eqref{eq:MmMi.2}), we obtain
$$\|(1+t)^{-\f 12-\de-|\bt|}\vb \wb^{\Mm+5-|\sigma|} \rd_x^\alp \rd_v^\bt Y^\sigma g\|_{L^2([0,T];L^2_vL^\i_x)} \ls \ep^{\f 34}.$$

Next we control the $L^2([0,T];L^\i_x L^\i_v)$ norm. Using Corollary~\ref{cor:MP.con} and the fact that $\zeta_k=\th_k=0$ when $k\leq \Mi$ according to \eqref{BA.Z}, we obtain
\begin{equation*}
\begin{split}
\|(1+t)^{-\f 12-\de-|\bt|}\vb \wb^{\Mm+5-|\sigma|} \rd_x^\alp \rd_v^\bt Y^\sigma g\|_{L^2([0,T];L^\i_vL^\i_x)} 
\ls \ep \|(1+t)^{-\f 12-\de} \|_{L^2([0,T])} \ls \ep.
\end{split}
\end{equation*}
This concludes the proof of \eqref{main.bulk.est.1}.

\pfstep{Step~2: Proof of \eqref{main.bulk.est.2}} \eqref{main.bulk.est.2} is an immediate consequence of \eqref{main.bulk.est.1} and H\"older's inequality. \qedhere
\end{proof}

\subsection{Main energy estimates}\label{sec:main.EE}

In this subsection, we prove the main energy estimates. We first prove a general energy estimate for solutions to equation of the form 
$$\rd_t h + v_i \rd_{x_i} h + \f{\de d_0}{(1+t)^{1+\de}}\vb^2 h -\bar{a}_{ij} \rd^2_{v_i v_j} h = H.$$
This estimate will be then be applied to $g$ and its derivatives; see Proposition~\ref{prop:EE.with.error}.

\begin{proposition}\label{EE.general}
Let $\ell \in \mathbb N\cup \{ 0\}$, $\ell\leq \Mm+5$. Suppose $h:[0,T_{Boot})\times \mathbb R^3\times \mathbb R^3\to \mathbb R$ is a $C^\infty$ solution to
\begin{equation}\label{eq:general}
\rd_t h + v_i \rd_{x_i} h + \f{\de d_0}{(1+t)^{1+\de}}\vb^2 h -\bar{a}_{ij} \rd^2_{v_i v_j} h = H,
\end{equation}
with $\wb^{\ell} h\in L^2_xL^2_v$ for all $t\in [0,T_{Boot})$, and $H:[0,T_{Boot})\times \mathbb R^3\times \mathbb R^3$ is a $C^\infty$ function such that $\wb^{\ell}H\in L^2([0,T_{Boot}); L^2_xL^2_v)$.

Then for any $T\in [0,T_{Boot})$,
\begin{equation*}
\begin{split}
&\:\|\langle x-tv\rangle^{\ell} h\|_{L^\i([0,T];L^2_xL^2_v)}^2 + \|(1+t)^{-\f 12-\f \de 2} \vb \wb^{\ell} h\|_{L^2([0,T];L^2_xL^2_v)}^2\\
\ls &\: \|\xb^\ell h\|_{L^2_xL^2_v}^2(0) + \int_0^T \left|\int_{\mathbb R^3}\int_{\mathbb R^3} \wb^{2\ell} h H(t,x,v)\, \ud v\, \ud x \right|\ud t.
\end{split}
\end{equation*}
\end{proposition}
\begin{proof}
Let $T$ be as in the statement of the proposition and take $T_* \in (0,T]$ to be arbitrary. The idea is to multiply \eqref{eq:general} by
$\wb^{2\ell} h$, integrate in $[0,T_*]\times \mathbb R^3\times \mathbb R^3$,
and integrate by parts. First note that we have
$$\left(\f{\rd}{\rd t} + v^i \f{\rd}{\rd x^i} \right)\left(\wb^{2\ell} \right)= 0.$$
Hence, performing the integration discussed above and integrate by parts in $t$ and $x$, we obtain
\begin{align}
&\: \f 12 \int_{\mathbb R^3} \int_{\mathbb R^3}\langle x-T_*v\rangle^{2\ell} h^2(T_*,x,v)\,\ud v\,\ud x - \f 12\int_{\mathbb R^3} \int_{\mathbb R^3} \xb^{2\ell} h^2(0,x,v)\,\ud v\,\ud x \label{main.EE.boundary}\\
&\: + \int_0^{T_*} \int_{\mathbb R^3} \int_{\mathbb R^3} \wb^{2\ell} \f{\de d_0}{(1+t)^{1+\de}}\vb^2 h^2(t,x,v)\,\ud v\,\ud x\,\ud t \label{main.EE.bulk}\\
&\: + \int_0^{T_*} \int_{\mathbb R^3} \int_{\mathbb R^3} \wb^{2\ell} \bar{a}_{ij} h \rd^2_{v_i v_j} h(t,x,v)\,\ud v\,\ud x\,\ud t \label{main.EE.IBP}\\
=&\: \int_0^{T_*} \int_{\mathbb R^3} \int_{\mathbb R^3} \wb^{2\ell} h H (t,x,v)\,\ud v\,\ud x\,\ud t. \label{main.EE.H}
\end{align}

For the term \eqref{main.EE.IBP}, we integrate by parts in $v$ (multiple times) and use that $\bar{a}_{ij}$ is symmetric to obtain
\begin{align}
&\: \int_0^{T_*} \int_{\mathbb R^3} \int_{\mathbb R^3} \wb^{2\ell} \bar{a}_{ij} h \rd^2_{v_i v_j} h(t,x,v)\,\ud v\,\ud x\,\ud t \notag\\ 
= &\: -\int_0^{T_*} \int_{\mathbb R^3} \int_{\mathbb R^3} \rd_{v_i}(\wb^{2\ell} \bar{a}_{ij} h) \rd_{v_j} h(t,x,v)\,\ud v\,\ud x\,\ud t \notag\\
= &\: -\int_0^{T_*} \int_{\mathbb R^3} \int_{\mathbb R^3} \wb^{2\ell} \bar{a}_{ij} (\rd_{v_i} h)(\rd_{v_j} h)(t,x,v)\,\ud v\,\ud x\,\ud t \notag\\
&\: +\int_0^{T_*} \int_{\mathbb R^3} \int_{\mathbb R^3} \ell \wb^{2\ell-2} t(x_i-tv_i) \bar{a}_{ij} \rd_{v_j} h^2(t,x,v) \,\ud v\,\ud x\,\ud t \notag\\
&\: - \f 12\int_0^{T_*} \int_{\mathbb R^3} \int_{\mathbb R^3} \wb^{2\ell}(\rd_{v_i} \bar{a}_{ij}) \rd_{v_j} h^2(t,x,v)\,\ud v\,\ud x\,\ud t \notag\\
= &\: -\int_0^{T_*} \int_{\mathbb R^3} \int_{\mathbb R^3} \wb^{2\ell} \bar{a}_{ij} (\rd_{v_i} h)(\rd_{v_j} h)(t,x,v)\,\ud v\,\ud x\,\ud t \label{EE.error.main}\\
&\: + \ell\int_0^{T_*} \int_{\mathbb R^3} \int_{\mathbb R^3} t^2 \de_{ij}\wb^{2\ell-2} \bar{a}_{ij} h^2(t,x,v) \,\ud v\,\ud x\,\ud t \label{EE.error.1}\\
&\: + 2\ell(\ell-1) \int_0^{T_*} \int_{\mathbb R^3} \int_{\mathbb R^3} \wb^{2\ell-4} t^2(x_i-tv_i)(x_j-tv_j)\bar{a}_{ij} h^2(t,x,v) \,\ud v\,\ud x\,\ud t \label{EE.error.2} \\
&\: - 2\ell \int_0^{T_*} \int_{\mathbb R^3} \int_{\mathbb R^3}  \wb^{2\ell-2} t(x_i-tv_i) (\rd_{v_j}\bar{a}_{ij})  h^2(t,x,v) \,\ud v\,\ud x\,\ud t\label{EE.error.3} \\
&\: + \f 12\int_0^{T_*} \int_{\mathbb R^3} \int_{\mathbb R^3} \wb^{2\ell}(\rd^2_{v_i v_j} \bar{a}_{ij}) h^2(t,x,v)\,\ud v\,\ud x\,\ud t. \label{EE.error.4}
\end{align}

We now analyze each of \eqref{EE.error.main}--\eqref{EE.error.4}. For \eqref{EE.error.main}, we simply note that
$$\mbox{\eqref{EE.error.main}} \leq 0.$$
For \eqref{EE.error.1}, we apply H\"older's inequality and Proposition~\ref{prop:ab.Li.null.cond} to obtain
\begin{equation*}
\begin{split}
|\mbox{\eqref{EE.error.1}}|\ls &\: \int_0^{T_*} t^2 \|\wb^{-2}\vb^{-1}\bar{a}_{ij}\|_{L^\i_x L^\i_v}(t) \|\vb \wb^{\ell} h\|_{L^2_xL^2_v}^2(t) \,\ud t \\
\ls &\: \ep^{\f 34} \int_0^{T_*} (1+t)^{-1-\min\{2+\gamma,1\}} \|\vb \wb^{\ell} h\|_{L^2_xL^2_v}^2(t) \,\ud t.
\end{split}
\end{equation*}
For \eqref{EE.error.2}, we argue similarly as for \eqref{EE.error.1}, since clearly $\left|\wb^{2\ell-4}(x_i-tv_i)(x_j-tv_j)\right|\leq \wb^{2\ell-2}$. We therefore apply H\"older's inequality and Proposition~\ref{prop:ab.Li.null.cond} to obtain
\begin{equation*}
\begin{split}
|\mbox{\eqref{EE.error.2}}|\ls &\: \int_0^{T_*} t^2 \|\wb^{-2}\vb^{-1}\bar{a}_{ij}\|_{L^\i_x L^\i_v}(t) \|\vb \wb^{\ell} h\|_{L^2_xL^2_v}^2(t) \,\ud t \\
\ls &\: \ep^{\f 34} \int_0^{T_*} (1+t)^{-1-\min\{2+\gamma,1\}} \|\vb \wb^{\ell} h\|_{L^2_xL^2_v}^2(t) \,\ud t.
\end{split}
\end{equation*}
For \eqref{EE.error.3}, first note that $\wb^{2\ell-2}|x_i-tv_i|\leq \wb^{2\ell-1}$. Therefore, applying H\"older's inequality and Proposition~\ref{prop:ab.Li.null.cond}, we obtain
\begin{equation*}
\begin{split}
|\mbox{\eqref{EE.error.3}}|\ls &\: \int_0^{T_*} t (\sum_{|\bt|=1}\|\wb^{-1}\vb^{-1}\rd_v^\bt\bar{a}_{ij}\|_{L^\i_x L^\i_v}(t)) \|\vb \wb^\ell h\|_{L^2_xL^2_v}(t) \,\ud t \\
\ls &\: \ep^{\f 34} \int_0^{T_*} (1+t)^{-1-\min\{2+\gamma,1\}} \|\vb \wb^{\ell} h\|_{L^2_xL^2_v}^2(t) \,\ud t.
\end{split}
\end{equation*}
For \eqref{EE.error.4}, there is no gain in $\wb$ factors for an application of Proposition~\ref{prop:ab.Li.null.cond}. Instead, we take advantage of the $\rd_v$ derivatives on $\bar{a}_{ij}$. More precisely, we apply H\"older's inequality and Proposition~\ref{prop:ab.Li.2} to obtain
\begin{equation*}
\begin{split}
|\mbox{\eqref{EE.error.4}}|\ls &\: \int_0^{T_*} (\sum_{|\bt|=2}\|\vb^{-1}\rd_v^\bt\bar{a}_{ij}\|_{L^\i_x L^\i_v}(t)) \|\vb \wb^{\ell} h\|_{L^2_xL^2_v}(t) \,\ud t \\
\ls &\: \ep^{\f 34} \int_0^{T_*} (1+t)^{-1-\min\{2+\gamma,1\}} \|\vb \wb^{\ell} h\|_{L^2_xL^2_v}(t) \,\ud t.
\end{split}
\end{equation*}
This concludes the discussion on the term \eqref{main.EE.IBP}.

Finally, we look at the term \eqref{main.EE.H}, which can easily be controlled by
$$|\mbox{\eqref{main.EE.H}}|\leq \int_0^{T_*} \left|\int_{\mathbb R^3}\int_{\mathbb R^3} \wb^{2\ell} h H(t,x,v)\, \ud v\, \ud x \right|\ud t.$$

Returning to the main identity \eqref{main.EE.boundary}--\eqref{main.EE.H}, we therefore obtain
\begin{equation*}
\begin{split}
&\: \|\langle x-T_*v \rangle^{\ell} h\|_{L^2_x L^2_v}^2(T_*)  + \de d_0 \|(1+t)^{-\f 12-\f \de 2} \vb \wb^{\ell} h(t,x,v)\|_{L^2([0,T_*];L^2_x L^2_v)}^2 \\
\ls &\: \|\xb^{\ell} h\|_{L^2_x L^2_v}^2(0) + \ep^{\f 34} \|(1+t)^{-\f 12 -\f 12 \min\{2+\gamma,1\}}\vb \wb^{\ell} h\|_{L^2([0,T_*];L^2_xL^2_v)}^2 \\
&\: +\int_0^{T_*} \left|\int_{\mathbb R^3}\int_{\mathbb R^3} \wb^{2\ell} h H(t,x,v)\, \ud v\, \ud x \right|\ud t.
\end{split}
\end{equation*}
Note that since $\de<\min\{2+\gamma,1\}$ (see~\eqref{def:de}), by choosing $\ep_0$ sufficiently small (and therefore $\ep$ sufficiently small), the second term on the RHS can be absorbed into the second term on the LHS. Using also $T_*\leq T$, this gives
\begin{equation*}
\begin{split}
&\: \|\langle x-T_*v \rangle^{\ell} h\|_{L^2_x L^2_v}^2(T_*)  + \de d_0 \|(1+t)^{-\f 12-\f \de 2} \vb \wb^{\ell} h(t,x,v)\|_{L^2([0,T_*];L^2_x L^2_v)}^2 \\
\ls &\: \|\xb^{\ell} h\|_{L^2_x L^2_v}^2(0) +\int_0^{T} \left|\int_{\mathbb R^3}\int_{\mathbb R^3} \wb^{2\ell} h H(t,x,v)\, \ud v\, \ud x \right|\ud t.
\end{split}
\end{equation*}
Finally, taking the supremum over all $T_*\in (0,T]$, we obtain the desired estimate. \qedhere
\end{proof}

\begin{proposition}\label{prop:EE.with.error}
Let $|\alp|+|\bt|+|\sigma|\leq \Mm$. Then the following estimate holds for all $T\in [0,T_{Boot})$:
\begin{equation*}
\begin{split}
&\: \|\wb^{\Mm+5-|\sigma|}\rd_x^\alp \rd_v^\bt Y^\sigma g\|_{L^\infty([0,T];L^2_x L^2_v)}^2+ \|(1+t)^{-\f 12-\f \de 2} \vb \wb^{\Mm+5-|\sigma|}\rd_x^\alp \rd_v^\bt Y^\sigma g\|_{L^2([0,T];L^2_x L^2_v)}^2\\
\ls &\: \ep^2 + (I_e^{\alp,\bt,\sigma}+II_e^{\alp,\bt,\sigma}+III_e^{\alp,\bt,\sigma}+IV_e^{\alp,\bt,\sigma}+V_e^{\alp,\bt,\sigma}+VI_e^{\alp,\bt,\sigma}+VII_e^{\alp,\bt,\sigma}+VIII_e^{\alp,\bt,\sigma})(T),
\end{split}
\end{equation*}
where\footnote{For the sake of brevity, we suppress the arguments $(t,x,v)$ of $g$, $a$, $c$ and their derivatives.}
\begin{equation}\label{def:I}
\begin{split}
I_e^{\alp,\bt,\sigma}:= \max_{i,j}\sum_{\substack{|\alp'|+|\alp''|+|\alp'''| \leq 2|\alp|\\|\bt'|+|\bt''|+|\bt'''| \leq 2|\bt|+2 \\ |\sigma'|+|\sigma''|+|\sigma'''|\leq 2|\sigma| \\ |\alp'''|+|\bt'''|+|\sigma'''| = |\alp|+|\bt|+|\sigma|\\ 1\leq |\alp'|+|\bt'|+|\sigma'|\leq |\alp|+|\bt|+|\sigma|\\ |\alp''|+|\bt''|+|\sigma''|\leq |\alp|+|\bt|+|\sigma|}}&\: \| \wb^{2\Mm+10-2|\sigma|}|\rd_x^{\alp'''}\rd_v^{\bt'''}Y^{\sigma'''} g|\\
&\:\quad \times |\rd_x^{\alp'}\rd_v^{\bt'}Y^{\sigma'} \bar{a}_{ij}| |\rd_x^{\alp''}\rd_v^{\bt''}Y^{\sigma''} g| \|_{L^1([0,T];L^1_xL^1_v)},
\end{split}
\end{equation}
\begin{equation}\label{def:II}
\begin{split}
II_e^{\alp,\bt,\sigma}:= \max_{i,j}\sum_{\substack{|\alp'|+|\alp''|+|\alp'''| \leq 2|\alp|\\|\bt'|+|\bt''|+|\bt'''| \leq 2|\bt|+1 \\ |\sigma'|+|\sigma''|+|\sigma'''|\leq 2|\sigma| \\ |\alp'''|+|\bt'''|+|\sigma'''| = |\alp|+|\bt|+|\sigma| \\ |\alp'|+|\bt'|+|\sigma'| = 1}}&\:  \| t \wb^{2\Mm+9-2|\sigma|} |\rd_x^{\alp'''}\rd_v^{\bt'''}Y^{\sigma'''} g| \\
&\:\quad \times |\rd_x^{\alp'}\rd_v^{\bt'}Y^{\sigma'} \bar{a}_{ij}| |\rd_x^{\alp''} \rd_v^{\bt''} Y^{\sigma''} g| \|_{L^1([0,T];L^1_xL^1_v)},
\end{split}
\end{equation}
\begin{equation}\label{def:III}
\begin{split}
III_e^{\alp,\bt,\sigma}:= \max_j \sum_{\substack{|\alp'|+|\alp''|= |\alp|\\|\bt'|+|\bt''|= |\bt|+1 \\ |\sigma'|+|\sigma''|=|\sigma| \\ 1\leq |\alp'|+|\bt'|+|\sigma'|\leq |\alp|+|\bt|+|\sigma| \\ |\alp''|+|\bt''|+|\sigma''|\leq |\alp|+|\bt|+|\sigma|}}&\: \| \wb^{2\Mm+10-2|\sigma|}|\rd_x^{\alp}\rd_v^{\bt}Y^{\sigma} g| \\
&\:\quad\times |\rd_x^{\alp'}\rd_v^{\bt'}Y^{\sigma'} (\bar{a}_{ij}v_i)| |\rd_x^{\alp''} \rd_v^{\bt''} Y^{\sigma''} g| \|_{L^1([0,T];L^1_xL^1_v)},
\end{split}
\end{equation}
\begin{equation}\label{def:IV}
\begin{split}
IV_e^{\alp,\bt,\sigma}:= \max_j \| t \wb^{2\Mm+9-2|\sigma|}|\rd_x^{\alp}\rd_v^{\bt}Y^\sigma g| |\bar{a}_{ij}v_i| |\rd_x^{\alp} \rd_v^{\bt} Y^\sigma g| \|_{L^1([0,T];L^1_xL^1_v)},
\end{split}
\end{equation}
\begin{equation}\label{def:V}
\begin{split}
V_e^{\alp,\bt,\sigma}:= &\: \sum_{\substack{|\alp'|+|\alp''|\leq |\alp|\\|\bt'|+|\bt''|\leq |\bt| \\ |\sigma'|+|\sigma''|\leq |\sigma|}} \| \wb^{2\Mm+10-2|\sigma|}|\rd_x^{\alp}\rd_v^{\bt} g| |\rd_x^{\alp'}\rd_v^{\bt'}Y^{\sigma'} \bar{a}_{ii}| |\rd_x^{\alp''} \rd_v^{\bt''} Y^{\sigma''} g| \|_{L^1([0,T];L^1_xL^1_v)}\\
&\: + \sum_{\substack{|\alp'|+|\alp''|\leq |\alp|\\|\bt'|+|\bt''|\leq |\bt| \\ |\sigma'|+|\sigma''|\leq |\sigma|}} \| \wb^{2\Mm+10-2|\sigma|}|\rd_x^{\alp}\rd_v^{\bt}Y^\sigma g||\rd_x^{\alp'}\rd_v^{\bt'}Y^{\sigma'} (\bar{a}_{ij}v_iv_j)| |\rd_x^{\alp''} \rd_v^{\bt''} Y^{\sigma''} g| \|_{L^1([0,T];L^1_xL^1_v)},
\end{split}
\end{equation}
\begin{equation}\label{def:VI}
\begin{split}
VI_e^{\alp,\bt,\sigma}:= \sum_{\substack{|\alp'|+|\alp''|\leq |\alp|\\|\bt'|+|\bt''|\leq |\bt|\\|\sigma'|+|\sigma''|\leq |\sigma|}} \| \wb^{2\Mm+10-2|\sigma|}|\rd_x^\alp\rd_v^\bt Y^\sigma g|| \rd_x^{\alp'}\rd_v^{\bt'}Y^{\sigma'} \cb| |\rd_x^{\alp''} \rd_v^{\bt''} Y^{\sigma''} g| \|_{L^1([0,T];L^1_xL^1_v)},
\end{split}
\end{equation}
\begin{equation}\label{def:VII}
\begin{split}
VII_e^{\alp,\bt,\sigma}:= \sum_{|\alp'|\leq |\alp|+1,\,|\bt'|\leq |\bt|-1} \| \wb^{2\Mm+10-2|\sigma|} |\rd_x^\alp\rd_v^\bt Y^\sigma g| |\rd_x^{\alp'} \rd_v^{\bt'} Y^\sigma g| \|_{L^1([0,T];L^1_xL^1_v)},
\end{split}
\end{equation}
and
\begin{equation}\label{def:VIII}
\begin{split}
VIII_e^{\alp,\bt,\sigma}:= \sum_{\substack{|\bt'|\leq |\bt|,\,|\sigma'|\leq |\sigma|\\ |\bt'|+|\sigma'|\leq |\bt|+|\sigma|-1}} \|\f{\vb^{\f 12}}{(1+t)^{1+\de}} \wb^{2\Mm+10-2|\sigma|}  |\rd_x^\alp\rd_v^\bt Y^\sigma g||\rd_x^{\alp} \rd_v^{\bt'} Y^{\sigma'} g| \|_{L^1([0,T];L^1_xL^1_v)}.
\end{split}
\end{equation}
Here, by our convention (see~Section~\ref{sec:notation}), if $|\bt|+|\sigma|=0$, then the terms $VII_e^{\alp,\bt,\sigma}$ and $VIII_e^{\alp,\bt,\sigma}$ are not present.
\end{proposition}

\begin{proof}

In view of the general estimate in Proposition~\ref{EE.general} and the data bound in the assumption of Theorem~\ref{thm:main}, it suffices to show that\footnote{We remark that technically at the top level, i.e.~when $|\alp|+|\bt|+|\sigma|=\Mm$, our bootstrap assumptions by themselves are not strong enough to ensure that the RHS is in $L^2([0,T];L^2_xL^2_v)$ to apply Proposition~\ref{EE.general}. Nevertheless, by a standard argument which approximates the initial data with slightly more regular data and proves that higher regularity persists, this can be justified. We omit the details.}
$$\int_0^T \left|\int_{\mathbb R^3}\int_{\mathbb R^3} \wb^{2\Mm+10-2|\sigma|} \rd_x^\alp\rd_v^\bt Y^\sigma g \times (\mbox{RHS of \eqref{eq:g.diff}}) \,\ud v\,\ud x\right |\,\ud t$$
is bounded by the terms $I_e$ through $VII_e$.

Consider each term on the RHS of \eqref{eq:g.diff}.

\pfstep{Step~1: Controlling $\mathrm{Term}_1$} By \eqref{eq:Term1.bound}, the contribution from $\mathrm{Term}_1$ can be bounded by the term $VII_e$ in \eqref{def:VII}.

\pfstep{Step~2: Controlling $\mathrm{Term}_2$} By \eqref{eq:Term2.bound}, the contribution from $\mathrm{Term}_2$ can be bounded by the term $VIII_e$ in \eqref{def:VIII}.

\pfstep{Step~3: Controlling $\mathrm{Term}_3$} To handle this term requires additional integrations by parts. (This is in contrast to the $L^\infty$ estimate in Proposition~\ref{prop:mainLiLi}, since we now cannot lose derivatives.)

Consider first the case $(|\alp'|,|\bt'|,|\sigma'|)=(1,0,0)$. In this case, $\rd_x^{\alp'} = \rd_{x_\ell}$ for some $\ell$ and $\rd_x^{\alp'}\rd_v^{\bt'}Y^{\sigma'} = \rd_{x_\ell}$, $\rd_x^\alp = \rd_{x_\ell} \rd_x^{\alp''}$, $\rd_v^{\bt''}= \rd_v^\bt$, $Y^{\sigma''} = Y^\sigma$.

We now carry out the (two) integrations by parts. To simplify notation, let us not write the integrals, but use the notation $\simeq$ to denote that the equality holds after integrating with respect to $\ud v\, \ud x$.
\begin{equation*}
\begin{split}
&\: \wb^{2\Mm+10-2|\sigma|}\rd_x^\alp\rd_v^\bt Y^\sigma g (\rd_{x_\ell}\bar{a}_{ij}) \rd^2_{v_i v_j} \rd_x^{\alp''} \rd_v^{\bt} Y^\sigma g \\
\simeq &\: -\wb^{2\Mm+10-2|\sigma|}\rd_{v_i} \rd_{x_\ell}\rd_x^{\alp''}\rd_v^\bt Y^\sigma g (\rd_{x_\ell}\bar{a}_{ij}) \rd_{v_j} \rd_x^{\alp''} \rd_v^{\bt} Y^\sigma g \\
&\: +4t(\Mm+5-|\sigma|)(x_i-tv_i)\wb^{2\Mm+8-2|\sigma|}\rd_x^\alp\rd_v^\bt Y^\sigma g (\rd_{x_\ell}\bar{a}_{ij}) \rd_{v_j} \rd_x^{\alp''} \rd_v^{\bt} Y^\sigma g \\
&\: -\wb^{2\Mm+10-2|\sigma|}\rd_x^\alp\rd_v^\bt Y^\sigma g (\rd_{v_i}\rd_{x_\ell}\bar{a}_{ij}) \rd_{v_j} \rd_x^{\alp''} \rd_v^{\bt} Y^\sigma g \\
\simeq &\: \f 12\wb^{2\Mm+10-2|\sigma|}\rd_{v_i} \rd_x^{\alp''}\rd_v^\bt Y^\sigma g (\rd_{x_\ell}^2\bar{a}_{ij}) \rd_{v_j} \rd_x^{\alp''} \rd_v^{\bt} Y^\sigma g \\
&\: +2(\Mm+5-|\sigma|)(x_\ell-tv_\ell)\wb^{2\Mm+8-2|\sigma|}\rd_{v_i} \rd_x^{\alp''}\rd_v^\bt Y^\sigma g (\rd_{x_\ell}\bar{a}_{ij}) \rd_{v_j} \rd_x^{\alp''} \rd_v^{\bt} Y^\sigma g \\
&\: +4t(\Mm+5-|\sigma|)(x_i-tv_i)\wb^{2\Mm+8-2|\sigma|}\rd_x^\alp\rd_v^\bt Y^\sigma g (\rd_{x_\ell}\bar{a}_{ij}) \rd_{v_j} \rd_x^{\alp''} \rd_v^{\bt} Y^\sigma g \\
&\: -\wb^{2\Mm+10-2|\sigma|}\rd_x^\alp\rd_v^\bt Y^\sigma g (\rd_{v_i}\rd_{x_\ell}\bar{a}_{ij}) \rd_{v_j} \rd_x^{\alp''} \rd_v^{\bt} Y^\sigma g.
\end{split}
\end{equation*}
Take the $L^1([0,T];L^1_xL^1_v)$ norm of each of these terms. The first, second and fourth terms can be bounded by $I_e$ while the third term can be bounded by $II_e$.

Next, we consider the case $(|\alp'|,|\bt'|,|\sigma'|)=(0,1,0)$. In this case, $\rd_v^{\alp'} = \rd_{v_\ell}$ for some $\ell$ and $\rd_x^{\alp'}\rd_v^{\bt'} Y^{\sigma'}= \rd_{v_\ell}$, $\rd_x^{\alp''}=\rd_x^\alp$, $\rd_v^\bt = \rd_{v_\ell} \rd_x^{\bt''}$, $Y^{\sigma''} = Y^\sigma$.

As above, $\simeq$ means that two expressions are equal after integrating with respect to $\ud v\, \ud x$.
\begin{equation*}
\begin{split}
&\: \wb^{2\Mm+10-2|\sigma|}\rd_x^\alp\rd_v^\bt Y^\sigma g (\rd_{v_\ell}\bar{a}_{ij}) \rd^2_{v_i v_j} \rd_x^{\alp} \rd_v^{\bt''} Y^\sigma g \\
\simeq &\: -\wb^{2\Mm+10-2|\sigma|}\rd_{v_i} \rd_x^{\alp}\rd_{v_\ell}\rd_v^{\bt''} Y^\sigma g (\rd_{v_\ell}\bar{a}_{ij}) \rd_{v_j} \rd_x^{\alp} \rd_v^{\bt''} Y^\sigma g \\
&\: +4t(\Mm+5-|\sigma|)(x_i-tv_i)\wb^{2\Mm+8-2|\sigma|}\rd_x^\alp\rd_v^\bt Y^\sigma g (\rd_{v_\ell}\bar{a}_{ij}) \rd_{v_j} \rd_x^{\alp} \rd_v^{\bt''} Y^\sigma g \\
&\: -\wb^{2\Mm+10-2|\sigma|}\rd_x^\alp\rd_v^\bt Y^\sigma g (\rd_{v_i}\rd_{v_\ell}\bar{a}_{ij}) \rd_{v_j} \rd_x^{\alp} \rd_v^{\bt''} Y^\sigma g \\
\simeq &\: \f 12\wb^{2\Mm+10-2|\sigma|}\rd_{v_i} \rd_x^{\alp}\rd_v^{\bt''} Y^\sigma g (\rd_{v_\ell}^2\bar{a}_{ij}) \rd_{v_j} \rd_x^{\alp} \rd_v^{\bt''}Y^\sigma g \\
&\: -2t(\Mm+5-|\sigma|)(x_\ell-tv_\ell)\wb^{2\Mm+8-2|\sigma|}\rd_{v_i} \rd_x^{\alp}\rd_v^{\bt''} Y^\sigma g (\rd_{x_\ell}\bar{a}_{ij}) \rd_{v_j} \rd_x^{\alp} \rd_v^{\bt''} Y^\sigma g \\
&\: +4t(\Mm+5-|\sigma|)(x_i-tv_i)\wb^{2\Mm+8-2|\sigma|}\rd_x^\alp\rd_v^\bt Y^\sigma g (\rd_{v_\ell}\bar{a}_{ij}) \rd_{v_j} \rd_x^{\alp} \rd_v^{\bt''} Y^\sigma g \\
&\: -\wb^{2\Mm+10-2|\sigma|}\rd_x^\alp\rd_v^\bt Y^\sigma g (\rd_{v_i}\rd_{v_\ell}\bar{a}_{ij}) \rd_{v_j} \rd_x^{\alp} \rd_v^{\bt''} Y^\sigma g.
\end{split}
\end{equation*}
Take the $L^1([0,T];L^1_xL^1_v)$ norm of each of these terms. The first and fourth terms can be bounded by $I_e$ while the second and third terms can be bounded by $II_e$.

Finally, we consider the case $(|\alp'|,|\bt'|,|\sigma'|)=(0,0,1)$. In this case, $Y^{\alp'} = t\rd_{x_\ell}+\rd_{v_\ell} =:Y_\ell$ for some $\ell$ and $\rd_x^{\alp'}\rd_v^{\bt'} Y^{\sigma'}= Y_\ell$, $\rd_x^{\alp''}=\rd_x^\alp$, $\rd_v^{\bt''} = \rd_x^{\bt}$, $Y^{\sigma} = Y_\ell Y^{\sigma''}$.

As above, $\simeq$ means that two expressions are equal after integrating with respect to $\ud v\, \ud x$.
\begin{equation*}
\begin{split}
&\: \wb^{2\Mm+10-2|\sigma|}\rd_x^\alp\rd_v^\bt Y^{\sigma''}Y_\ell g (Y_{\ell}\bar{a}_{ij}) \rd^2_{v_i v_j} \rd_x^{\alp} \rd_v^{\bt} Y^{\sigma''} g \\
\simeq &\: -\wb^{2\Mm+10-2|\sigma|}\rd_{v_i} \rd_x^{\alp}\rd_v^{\bt} Y_\ell Y^{\sigma''}g (Y_{\ell}\bar{a}_{ij}) \rd_{v_j} \rd_x^{\alp} \rd_v^{\bt} Y^{\sigma''} g \\
&\: +4t(\Mm+5-|\sigma|)(x_i-tv_i)\wb^{2\Mm+8-2|\sigma|}\rd_x^\alp\rd_v^\bt Y^{\sigma} g (Y_\ell\bar{a}_{ij}) \rd_{v_j} \rd_x^{\alp} \rd_v^{\bt} Y^{\sigma''} g \\
&\: -\wb^{2\Mm+10-2|\sigma|}\rd_x^\alp\rd_v^\bt Y^\sigma g (\rd_{v_i}Y_{\ell}\bar{a}_{ij}) \rd_{v_j} \rd_x^{\alp} \rd_v^{\bt} Y^{\sigma''} g \\
\simeq &\: \f 12\wb^{2\Mm+10-2|\sigma|}\rd_{v_i} \rd_x^{\alp}\rd_v^{\bt} Y^{\sigma''} g (Y_{\ell}^2\bar{a}_{ij}) \rd_{v_j} \rd_x^{\alp} \rd_v^{\bt}Y^{\sigma''} g \\
&\: +4t(\Mm+5-|\sigma|)(x_i-tv_i)\wb^{2\Mm+8-2|\sigma|}\rd_x^\alp\rd_v^\bt Y^\sigma g (Y_\ell\bar{a}_{ij}) \rd_{v_j} \rd_x^{\alp} \rd_v^{\bt} Y^{\sigma''} g \\
&\: -\wb^{2\Mm+10-2|\sigma|}\rd_x^\alp\rd_v^\bt Y^\sigma g (\rd_{v_i}Y_{\ell}\bar{a}_{ij}) \rd_{v_j} \rd_x^{\alp} \rd_v^{\bt} Y^{\sigma''} g.
\end{split}
\end{equation*}
Take the $L^1([0,T];L^1_xL^1_v)$ norm of each of these terms. The first and third terms can be bounded by $I_e$ while the second term can be bounded by $II_e$.

\pfstep{Step~4: Controlling $\mathrm{Term}_4$} For $\mathrm{Term}_4$, it is straightforward to see that \eqref{eq:Term4.bound} implies that the corresponding contribution is bounded by $VI_e^{\alp,\bt,\sigma}$ in \eqref{def:VI}.

\pfstep{Step~5: Controlling $\mathrm{Term}_5$} For $\mathrm{Term}_5$, 
\begin{equation}\label{Term.5.prelim}
\begin{split}
&\: 4 (d(t))^2 \rd_x^\alp \rd_v^\bt Y^\sigma (\bar{a}_{ij} v_i \rd_{v_j} g) \\
= &\: 4 (d(t))^2 \sum_{\substack{\alp'+\alp''=\alp \\ \bt'+\bt''=\bt \\ \sigma'+\sigma''=\sigma \\ |\alp'|+|\bt'|+|\sigma'|\geq 1}} (\rd_x^{\alp'} \rd_v^{\bt'} Y^{\sigma'} (\bar{a}_{ij} v_i)) (\rd_{v_j}\rd_x^{\alp''} \rd_v^{\bt''} g) + 4 (d(t))^2  \bar{a}_{ij} v_i \rd_{v_j} \rd_x^\alp \rd_v^\bt Y^\sigma g.
\end{split}
\end{equation}
The first term in \eqref{Term.5.prelim} gives a contribution of the type $III_e$ in \eqref{def:III}. 

For the second term in \eqref{Term.5.prelim}, we need an integration by parts in $\rd_{v_j}$:
\begin{equation}\label{Term.5.serious}
\begin{split}
&\: \left| \int_{\mathbb R^3}\int_{\mathbb R^3} 4 (d(t))^2 \wb^{2\Mm+10-2|\sigma|} (\rd_x^\alp \rd_v^\bt Y^\sigma g) \bar{a}_{ij} v_i (\rd_{v_j} \rd_x^\alp \rd_v^\bt Y^\sigma g) \,\ud v\, \ud x\right| \\
\leq &\: 2 (d(t))^2 \|  \wb^{2\Mm+10-2|\sigma|} \rd_x^\alp \rd_v^\bt Y^\sigma g (\rd_{v_j}(\bar{a}_{ij} v_i))  \rd_x^\alp \rd_v^\bt Y^\sigma g  \|_{L^1_vL^1_x} \\
& + 4(\Mm+5-|\sigma|) (d(t))^2 \|  \wb^{2\Mm+8-2|\sigma|}(x_j-tv_j) \rd_x^\alp \rd_v^\bt Y^\sigma g (\bar{a}_{ij} v_i)  \rd_x^\alp \rd_v^\bt Y^\sigma g \|_{L^1_vL^1_x}.
\end{split}
\end{equation}
After bounding $(d(t))^2\ls 1$, $|x_j-tv_j|\ls \wb$, and integrating over $t\in [0,T]$, the first term on the RHS of \eqref{Term.5.serious} gives a contribution of the type $III_e$ in \eqref{def:III} and the second term on the RHS of \eqref{Term.5.serious} gives a contribution of the type $IV_e$ in \eqref{def:IV}.

\pfstep{Step~6: Controlling $\mathrm{Term}_6$} By \eqref{eq:Term6.bound}, $\mathrm{Term}_6$ can be bounded by $V_e^{\alp,\bt,\sigma}$ in \eqref{def:V}. \qedhere
\end{proof}

\subsection{Controlling the error terms}\label{sec:EE.error}

\begin{proposition}\label{prop:I}
Let $|\alp|+|\bt|+|\sigma|\leq \Mm$. Then the term $I_e^{\alp,\bt,\sigma}$ in \eqref{def:I} is bounded as follows for every $T\in [0,T_{Boot})$:
$$I_e^{\alp,\bt,\sigma}\ls \ep^2(1+T)^{2|\bt|}.$$
\end{proposition}
\begin{proof}
From now on we take a particular term in $I_e^{\alp,\bt,\sigma}$, and assume that $\alp'$, $\alp''$, $\alp'''$, $\bt'$, $\bt''$, $\bt'''$, $\sigma'$, $\sigma''$ and $\sigma'''$ obey the required conditions in the sum in $I_e^{\alp,\bt,\sigma}$.

\textbf{Short-time estimates: $T\leq 1$.} We first consider the estimates for $T\leq 1$. We will consider separately the cases $|\alp'|+|\bt'|+|\sigma'|\leq \Mi$ and $|\alp'|+|\bt'|+|\sigma'|>\Mi$.

\pfstep{Case~1: $|\alp'|+|\bt'|+|\sigma'|\leq \Mi$} By H\"older's inequality, Proposition~\ref{prop:ab.Li.1} and Lemma~\ref{lem:move.t.weights},
\begin{equation*}
\begin{split}
&\: \max_{i,j} \| \wb^{2\Mm+10-2|\sigma|}|\rd_x^{\alp'''}\rd_v^{\bt'''}Y^{\sigma'''} g| |\rd_x^{\alp'}\rd_v^{\bt'}Y^{\sigma'} \bar{a}_{ij}| |\rd_x^{\alp''} \rd_v^{\bt''}Y^{\sigma''} g| \|_{L^1([0,T];L^1_xL^1_v)} \\
\ls &\: \max_{i,j} \|\vb\wb^{\Mm+5-|\sigma'''|}\rd_x^{\alp'''}\rd_v^{\bt'''}Y^{\sigma'''} g\|_{L^2([0,T];L^2_xL^2_v)}\\
&\: \times \|\vb^{-2}\rd_x^{\alp'}\rd_v^{\bt'}Y^{\sigma'} \bar{a}_{ij}\|_{L^\i([0,T];L^\i_xL^\i_v)} \times  \|\vb\wb^{\Mm+5-|\sigma''|}\rd_x^{\alp''} \rd_v^{\bt''}Y^{\sigma''} g \|_{L^2([0,T];L^2_xL^2_v)} \\
\ls &\: \ep^{\f 34} \times \ep^{\f 34}\times \ep^{\f 34} = \ep^{\f 94}.
\end{split}
\end{equation*}

\pfstep{Case~2: $|\alp'|+|\bt'|+|\sigma'|> \Mi$} Note that in this case $|\alp''|+|\bt''|+|\sigma''|\leq \Mi$. By H\"older's inequality, Propositions~\ref{prop:ab.L2} and \ref{main.bulk.est},
\begin{equation*}
\begin{split}
&\: \max_{i,j} \| \wb^{2\Mm+10-2|\sigma|}|\rd_x^{\alp'''}\rd_v^{\bt'''}Y^{\sigma'''} g| |\rd_x^{\alp'}\rd_v^{\bt'}Y^{\sigma'} \bar{a}_{ij}| |\rd_x^{\alp''} \rd_v^{\bt''}Y^{\sigma''} g| \|_{L^1([0,T];L^1_xL^1_v)} \\
\ls &\: \max_{i,j} \|\vb\wb^{\Mm+5-|\sigma'''|}\rd_x^{\alp'''}\rd_v^{\bt'''}Y^{\sigma'''} g\|_{L^2([0,T];L^2_xL^2_v)}\\
&\: \times \|\vb^{-2}\rd_x^{\alp'}\rd_v^{\bt'}Y^{\sigma'} \bar{a}_{ij}\|_{L^\i([0,T];L^2_xL^\i_v)} \times  \|\vb\wb^{\Mm+5-|\sigma''|}\rd_x^{\alp''} \rd_v^{\bt''}Y^{\sigma''} g \|_{L^2([0,T];L^\i_xL^2_v)} \\
\ls &\: \ep^{\f 34} \times \ep^{\f 34} \times \ep^{\f 34} = \ep^{\f 94}.
\end{split}
\end{equation*}

\textbf{Long-time estimates: $T\geq 1$.} We now move to the estimates for $T\geq 1$. Again, we separately consider the cases $|\alp'|+|\bt'|+|\sigma'|\leq \Mi$ and $|\alp'|+|\bt'|+|\sigma'|> \Mi$.

\pfstep{Case~1: $|\alp'|+|\bt'|+|\sigma'|\leq \Mi$} In this case, we control $\rd_x^{\alp'}\rd_v^{\bt'}Y^{\sigma'}\bar{a}_{ij}$ in $L^\i_xL^\i_v$ (with appropriate weights). Since $\Mi \leq \Mm-4-\max\{2,\lceil \f 2{2+\gamma} \rceil\}$, we can apply the lower order $L^\i_xL^\i_v$ estimates for $\rd_x^{\alp'}\rd_v^{\bt'}Y^{\sigma'}\bar{a}_{ij}$ without a loss. It is crucial in our estimate to also exploit the fact $|\alp'|+|\bt'|+|\sigma'|\geq 1$, so that we have $L^\i_xL^\i_v$ estimates which are \emph{better} than that in Proposition~\ref{prop:ab.Li.1} (which would have incurred a logarithmic loss). To use this, we separately consider the (non-mutually exclusive) subcases $|\sigma'|\geq 1$ and $\max\{|\alp'|,\,|\bt'|\} \geq 1$ below.

\pfstep{Case~1(a): $|\alp'|+|\bt'|+|\sigma'|\leq \Mi$ and $|\sigma'|\geq 1$} The key point for $|\sigma'|\geq 1$ is that $|\sigma'''|+|\sigma''|\leq 2|\sigma|-1$. As a consequence, $\wb^{2\Mm+10-2|\sigma|} \ls \wb^{\Mm+5-|\sigma'''|}\wb^{\Mm+5-|\sigma''|}\wb^{-1}$. Note also that $|\bt'|+|\bt''|+|\bt'''|\leq 2|\bt|+2$. With these bounds, we use H\"older's inequality, Proposition~\ref{prop:ab.Li.null.cond} and Lemma~\ref{lem:move.t.weights} to obtain
\begin{equation*}
\begin{split}
&\: \max_{i,j} \| \wb^{2\Mm+10-2|\sigma|}|\rd_x^{\alp'''}\rd_v^{\bt'''}Y^{\sigma'''} g| |\rd_x^{\alp'}\rd_v^{\bt'}Y^{\sigma'} \bar{a}_{ij}| |\rd_x^{\alp''} \rd_v^{\bt''}Y^{\sigma''} g| \|_{L^1([0,T];L^1_xL^1_v)} \\
\ls &\: \max_{i,j}(1+T)^{2|\bt|} \| (1+t)^{-\f 12-\de-|\bt'''|}\vb\wb^{\Mm+5-|\sigma'''|}\rd_x^{\alp'''}\rd_v^{\bt'''}Y^{\sigma'''} g\|_{L^2([0,T];L^2_xL^2_v)}\\
&\: \times \|(1+t)^{1+2\de-|\bt'|+2}\vb^{-2}\wb^{-1}\rd_x^{\alp'}\rd_v^{\bt'}Y^{\sigma'} \bar{a}_{ij}\|_{L^\i([0,T];L^\i_xL^\i_v)}\\
&\:\times  \|(1+t)^{-\f 12-\de-|\bt''|}\vb\wb^{\Mm+5-|\sigma''|}\rd_x^{\alp''} \rd_v^{\bt''}Y^{\sigma''} g \|_{L^2([0,T];L^2_xL^2_v)} \\
\ls &\: (1+T)^{2|\bt|} \times \ep^{\f 34} \times (\sup_{t\in [0,T]} \ep^{\f 34}(1+t)^{1+2\de-|\bt'|+2}(1+t)^{-3-\min\{2+\gamma,1\}+|\bt'|} )\times \ep^{\f 34} = \ep^{\f 94} (1+T)^{2|\bt|},
\end{split}
\end{equation*}
where in the last line we have used that $2\de <\min\{2+\gamma,1\}$ (by \eqref{def:de}).

\pfstep{Case~1(b): $|\alp'|+|\bt'|\leq \Mi$ and $\max\{|\alp'|,|\bt'|\}\geq 1$} In this case, we do not have a gain in the $\wb$ weight as in Case~1(a). Nevertheless, since $|\bt'|\geq 1$ or $|\alp'|\geq 1$, we can take advantage of the improvement in Propositions~\ref{prop:ab.Li.2} or \ref{prop:ab.Li.3}. Let us note as in Case~1(a) that $|\bt'|+|\bt''|+|\bt'''|\leq 2|\bt|+2$. Hence, using H\"older's inequality, Propositions~\ref{prop:ab.Li.2}, \ref{prop:ab.Li.3} and Lemma~\ref{lem:move.t.weights}, we obtain
\begin{equation*}
\begin{split}
&\: \max_{i,j} \| \wb^{2\Mm+10-2|\sigma|}|\rd_x^{\alp'''}\rd_v^{\bt'''}Y^{\sigma'''} g| |\rd_x^{\alp'}\rd_v^{\bt'}Y^{\sigma'} \bar{a}_{ij}| |\rd_x^{\alp''} \rd_v^{\bt''} Y^{\sigma''} g| \|_{L^1([0,T];L^1_xL^1_v)} \\
\ls &\: \max_{i,j}(1+T)^{2|\bt|} \| (1+t)^{-\f 12-\de-|\bt'''|}\vb\wb^{\Mm+5-|\sigma'''|}\rd_x^{\alp'''}\rd_v^{\bt'''}Y^{\sigma'''} g\|_{L^2([0,T];L^2_xL^2_v)}\\
&\: \times \|(1+t)^{1+2\de-|\bt'|+2}\vb^{-2}\rd_x^{\alp'}\rd_v^{\bt'}Y^{\sigma'} \bar{a}_{ij}\|_{L^\i([0,T];L^\i_xL^\i_v)}\\
&\:\times  \|(1+t)^{-\f 12-\de-|\bt''|}\vb\wb^{\Mm+5-|\sigma''|}\rd_x^{\alp''} \rd_v^{\bt''} Y^{\sigma''} g \|_{L^2([0,T];L^2_xL^2_v)} \\
\ls &\: (1+T)^{2|\bt|} \times \ep^{\f 34} \times (\sup_{t\in [0,T]} \ep^{\f 34}(1+t)^{1+2\de-|\bt'|+2}(1+t)^{-3-\min\{2+\gamma,1\}+|\bt'|} ) \times \ep^{\f 34} = \ep^{\f 94} (1+T)^{2|\bt|},
\end{split}
\end{equation*}
where in the last line we used that $2\de <\min\{2+\gamma,1\}$ (by \eqref{def:de}).

\pfstep{Case~2: $|\alp'|+|\bt'|+|\sigma'|> \Mi$} Recall again that in this case $|\alp''|+|\bt''|+|\sigma''|\leq \Mi$ and thus we control $\rd_x^{\alp''} \rd_v^{\bt''} g$ in $L^2([0,T];L^\i_xL^p_v)$ (for suitable $p$ and with appropriate weights).

Before we proceed, one checks that by definition $\Mi\geq 3$. Hence, by the pigeon hole principle, we must have $|\alp'|\geq 2$ or $|\bt'|\geq 2$ or $|\sigma'|\geq 2$. We separate into these three (non-mutually exclusive) subcases.

\pfstep{Case~2(a): $|\alp'|+|\bt'|+|\sigma'|> \Mi$ and $|\sigma'|\geq 2$} In analogy with case 1(a), we take advantage of $|\sigma'|\geq 2$ by using that it implies $|\sigma'''|+|\sigma''|\leq 2|\sigma|-2$. Hence $\wb^{2\Mm+10-2|\sigma|} \ls \wb^{\Mm+5-|\sigma'''|}\wb^{\Mm+5-|\sigma''|}\wb^{-2}$. Note also that $|\bt'|+|\bt''|+|\bt'''|\leq 2|\bt|+2$. By H\"older's inequality, Propositions~\ref{prop:ab.L2.null.cond} and \ref{main.bulk.est}, we obtain
\begin{equation*}
\begin{split}
&\: \max_{i,j} \| \wb^{2\Mm+10-2|\sigma|}|\rd_x^{\alp'''}\rd_v^{\bt'''}Y^{\sigma'''} g| |\rd_x^{\alp'}\rd_v^{\bt'}Y^{\sigma''} \bar{a}_{ij}| |\rd_x^{\alp''} \rd_v^{\bt''} Y^{\sigma''} g| \|_{L^1([0,T];L^1_xL^1_v)} \\
\ls &\: \max_{i,j}(1+T)^{2|\bt|} \| (1+t)^{-\f 12-\de-|\bt'''|}\vb\wb^{\Mm+5-|\sigma'''|}\rd_x^{\alp'''}\rd_v^{\bt'''}Y^{\sigma'''} g\|_{L^2([0,T];L^2_xL^2_v)}\\
&\: \times \|(1+t)^{1+2\de-|\bt'|+2}\vb^{-2}\wb^{-2}\rd_x^{\alp'}\rd_v^{\bt'}Y^{\sigma''} \bar{a}_{ij}\|_{L^\i([0,T];L^2_xL^\i_v+L^2_xL^{p_{**}}_v)}\\
&\:\times  \|(1+t)^{-\f 12-\de-|\bt''|}\vb\wb^{\Mm+5-|\sigma''|}\rd_x^{\alp''} \rd_v^{\bt''}Y^{\sigma''} g \|_{L^2([0,T];L^\i_xL^2_v\cap L^\i_xL^{\f{2p_{**}}{p_{**}-2}})} \\
\ls &\: (1+T)^{2|\bt|} \times \ep^{\f 34} \times (\sup_{t\in [0,T]} \ep^{\f 34}(1+t)^{1+2\de-|\bt'|+2}(1+t)^{-\min\{\f{16}{5},5+\gamma\}+|\bt'|} )\times \ep^{\f 34} = \ep^{\f 94} (1+T)^{2|\bt|},
\end{split}
\end{equation*}
where in the last line we used that $2\de <\min\{2+\gamma,\f 15\}$ (by \eqref{def:de}).

\pfstep{Case~2(b): $|\alp'|+|\bt'|+|\sigma'|> \Mi$ and $|\bt'|\geq 2$} In this case, we take advantage of $|\bt'|\geq 2$ and use Proposition~\ref{prop:ab.L2.improved}. More precisely, after noting $|\bt'|+|\bt''|+|\bt'''|\leq 2|\bt|+2$, we use H\"older's inequality, Propositions~\ref{prop:ab.L2.improved} and \ref{main.bulk.est} to obtain
\begin{equation*}
\begin{split}
&\: \max_{i,j} \| \wb^{2\Mm+10-2|\sigma|}|\rd_x^{\alp'''}\rd_v^{\bt'''}Y^{\sigma'''} g| |\rd_x^{\alp'}\rd_v^{\bt'}Y^{\sigma'} \bar{a}_{ij}| |\rd_x^{\alp''} \rd_v^{\bt''} Y^{\sigma''} g| \|_{L^1([0,T];L^1_xL^1_v)} \\
\ls &\: \max_{i,j}(1+T)^{2|\bt|} \| (1+t)^{-\f 12-\de-|\bt'''|}\vb\wb^{\Mm+5-|\sigma'''|}\rd_x^{\alp'''}\rd_v^{\bt'''}Y^{\sigma'''} g\|_{L^2([0,T];L^2_xL^2_v)}\\
&\: \times |(1+t)^{1+2\de-|\bt'|+2}\vb^{-2}\rd_x^{\alp'}\rd_v^{\bt'}Y^{\sigma'} \bar{a}_{ij}\|_{L^\i([0,T];L^2_xL^{p_{**}}_v)}\\
&\:\times  \|(1+t)^{-\f 12-\de-|\bt''|}\vb\wb^{\Mm+5-|\sigma''|}\rd_x^{\alp''} \rd_v^{\bt''} Y^{\sigma''} g \|_{L^2([0,T];L^\i_xL^{\f{2p_{**}}{p_{**}-2}}_v)} \\
\ls &\: (1+T)^{2|\bt|} \times \ep^{\f 34} \times (\sup_{t\in [0,T]} \ep^{\f 34}(1+t)^{1+2\de-|\bt'|+2}(1+t)^{-\min\{\f{16}{5},5+\gamma\}+|\bt'|} )\times \ep^{\f 34} = \ep^{\f 94} (1+T)^{2|\bt|},
\end{split}
\end{equation*}
where in the last line we again used that $2\de <\min\{2+\gamma,\f 15\}$ (by \eqref{def:de}).

\pfstep{Case~2(c): $|\alp'|+|\bt'|+|\sigma'|> \Mi$ and $|\alp'|\geq 2$} In this case, we take advantage of $|\alp'|\geq 2$ and use Proposition~\ref{prop:ab.L2.improved.2}. More precisely, after noting $|\bt'|+|\bt''|+|\bt'''|\leq 2|\bt|+2$, we use H\"older's inequality, Propositions~\ref{prop:ab.L2.improved.2} and \ref{main.bulk.est} to obtain
\begin{equation*}
\begin{split}
&\: \max_{i,j} \| \wb^{2\Mm+10-2|\sigma|}|\rd_x^{\alp'''}\rd_v^{\bt'''}Y^{\sigma'''} g| |\rd_x^{\alp'}\rd_v^{\bt'}Y^{\sigma'} \bar{a}_{ij}| |\rd_x^{\alp''} \rd_v^{\bt''} Y^{\sigma''} g| \|_{L^1([0,T];L^1_xL^1_v)} \\
\ls &\: \max_{i,j}(1+T)^{2|\bt|} \| (1+t)^{-\f 12-\de-|\bt'''|}\vb\wb^{\Mm+5-|\sigma'''|}\rd_x^{\alp'''}\rd_v^{\bt'''}Y^{\sigma'''} g\|_{L^2([0,T];L^2_xL^2_v)}\\
&\: \times \|(1+t)^{1+2\de-|\bt'|+2}\vb^{-2}\rd_x^{\alp'}\rd_v^{\bt'}Y^{\sigma'} \bar{a}_{ij}\|_{L^\i([0,T];L^2_xL^\i_v+ L^2_xL^2_v)}\\
&\:\times  \|(1+t)^{-\f 12-\de-|\bt''|}\vb\wb^{\Mm+5-|\sigma''|}\rd_x^{\alp''} \rd_v^{\bt''} Y^{\sigma''} g \|_{L^2([0,T];L^\i_xL^2_v \cap L^\i_xL^\i_v)} \\
\ls &\: (1+T)^{2|\bt|} \times \ep^{\f 34} \times (\sup_{t\in [0,T]} \ep^{\f 34}(1+t)^{1+2\de-|\bt'|+2}(1+t)^{-\min\{\f{16}{5},5+\gamma\}+|\bt'|} )\times \ep^{\f 34} = \ep^{\f 94} (1+T)^{2|\bt|},
\end{split}
\end{equation*}
where in the last line we again used that $2\de <\min\{2+\gamma,\f 15\}$ (by \eqref{def:de}). \qedhere
\end{proof}

\begin{proposition}\label{prop:II}
Let $|\alp|+|\bt|+|\sigma|\leq \Mm$. Then the term $II_e^{\alp,\bt,\sigma}$ in \eqref{def:II} is bounded as follows for every $T\in [0,T_{Boot})$: 
$$II_e^{\alp,\bt,\sigma}(T)\ls \ep^2(1+T)^{2|\bt|}.$$
\end{proposition}
\begin{proof}
Take $\alp'$, $\alp''$, $\alp'''$, $\bt'$, $\bt''$, $\bt'''$, $\sigma'$, $\sigma''$ and $\sigma'''$ satisfying the required conditions in the sum of $II_e^{\alp,\bt,\sigma}$. In particular, since $|\alp'|+|\bt'|+|\sigma'|=1$, we can put the $\rd_x^{\alp'}\rd_v^{\bt'}Y^{\sigma'} \bar{a}_{ij}$ term in $L^\i([0,T];L^\i_xL^\i_v)$. We will also make crucial use of the fact that the $\wb$ weight is one power better than the maximal weight. Finally, note that $|\bt'|+|\bt''|+|\bt'''|\leq 2|\bt|+1$. 

Hence, by H\"older's inequality, Proposition~\ref{prop:ab.Li.null.cond} and Lemma~\ref{lem:move.t.weights}, we obtain
\begin{equation*}
\begin{split}
&\: \max_{i,j}\| t \wb^{2\Mm+9-2|\sigma|} |\rd_x^{\alp'''}\rd_v^{\bt'''}Y^{\sigma'''} g| |\rd_x^{\alp'}\rd_v^{\bt'}Y^{\sigma'} \bar{a}_{ij}| |\rd_x^{\alp''} \rd_v^{\bt''} Y^{\sigma''} g| \|_{L^1([0,T];L^1_xL^1_v)} \\
\ls &\: \max_{i,j}(1+T)^{2|\bt|} \| (1+t)^{-\f 12-\de-|\bt'''|}\vb\wb^{\Mm+5-|\sigma'''|}\rd_x^{\alp'''}\rd_v^{\bt'''} Y^{\sigma'''} g\|_{L^2([0,T];L^2_xL^2_v)}\\
&\: \times \|(1+t)^{1+2\de-|\bt'|+2}\vb^{-2}\wb^{-1}\rd_x^{\alp'}\rd_v^{\bt'}Y^{\sigma'} \bar{a}_{ij}\|_{L^\i([0,T];L^\i_xL^\i_v)}\\
&\:\times  \|(1+t)^{-\f 12-\de-|\bt''|}\vb\wb^{\Mm+5-|\sigma''|}\rd_x^{\alp''} \rd_v^{\bt''} Y^{\sigma''} g \|_{L^2([0,T];L^2_xL^2_v)} \\
\ls &\: (1+T)^{2|\bt|} \times \ep^{\f 34} \times (\sup_{t\in [0,T]} \ep^{\f 34}(1+t)^{1+2\de-|\bt'|+2}(1+t)^{-3-\min\{2+\gamma,1\}+|\bt'|} ) \times \ep^{\f 34} = \ep^{\f 94} (1+T)^{2|\bt|},
\end{split}
\end{equation*}
where in the last line we used that $2\de <\min\{2+\gamma,1\}$ (by \eqref{def:de}). \qedhere
\end{proof}

\begin{proposition}\label{prop:III}
Let $|\alp|+|\bt|+|\sigma|\leq \Mm$. Then the term $III_e^{\alp,\bt,\sigma}$ in \eqref{def:III} is bounded as follows for every $T\in [0,T_{Boot})$:
$$III_e^{\alp,\bt,\sigma}(T)\ls \ep^2(1+T)^{2|\bt|}.$$
\end{proposition}
\begin{proof}
Take $\alp'$, $\alp''$, $\bt'$, $\bt''$, $\sigma'$ and $\sigma''$ satisfying the required conditions in the sum of $III_e^{\alp,\bt,\sigma}$.

We will consider separately the $T\leq 1$ and the $T\geq 1$ estimates.

\textbf{Short-time estimates:~$T\leq 1$.} This is exactly the same as the proof of the $T\leq 1$ estimates in Proposition~\ref{prop:I}, except that we replace the use of Propositions~\ref{prop:ab.Li.1} and \ref{prop:ab.L2} by Propositions~\ref{prop:ab.Li.weighted} and \ref{prop:ab.L2.weighted} respectively; we omit the details.

\textbf{Long-time estimates: ~$T\geq 1$.} We will divide into the cases $|\alp'|+|\bt'|+|\sigma'|\leq \Mi$ and $|\alp'|+|\bt'|+|\sigma'|> \Mi$.

\pfstep{Case~1: $|\alp'|+|\bt'|+|\sigma'|\leq \Mi$} In this case, we control $\rd_x^{\alp'}\rd_v^{\bt'} Y^{\sigma'} (\bar{a}_{ij}v_i)$ in $L^\i([0,T];L^\i_xL^\i_v)$ with appropriate weights. Recall that $|\bt'|+|\bt''|\leq |\bt|+1$. Hence using H\"older's inequality, Proposition~\ref{prop:ab.Li.weighted} and Lemma~\ref{lem:move.t.weights}, we obtain
\begin{equation*}
\begin{split}
&\: \max_j\| \wb^{2\Mm+10-2|\sigma|}|\rd_x^{\alp}\rd_v^{\bt}Y^{\sigma} g| |\rd_x^{\alp'}\rd_v^{\bt'}Y^{\sigma'} (\bar{a}_{ij}v_i)| |\rd_x^{\alp''} \rd_v^{\bt''} Y^{\sigma''} g| \|_{L^1([0,T];L^1_xL^1_v)} \\
\ls &\: \max_j (1+T)^{2|\bt|}\|(1+t)^{-\f 12-\de-|\bt|}\vb \wb^{\Mm+5-|\sigma|} \rd_x^{\alp}\rd_v^{\bt} Y^{\sigma} g\|_{L^2([0,T];L^2_xL^2_v)} \\
&\:\times \|(1+t)^{1+2\de-|\bt'|+1}\vb^{-2} \rd_x^{\alp'}\rd_v^{\bt'}Y^{\sigma'} (\bar{a}_{ij}v_i)\|_{L^\i([0,T];L^\i_xL^\i_v)}\\
&\: \times \|(1+t)^{-\f 12-\de-|\bt''|}\vb \wb^{\Mm+5-|\sigma''|} \rd_x^{\alp''}\rd_v^{\bt''}Y^{\sigma''} g\|_{L^2([0,T];L^2_xL^2_v)} \\
\ls &\: (1+T)^{2|\bt|}\times \ep^{\f 34} \times \ep^{\f 34}(\sup_{t\in [0,T]} (1+t)^{1+2\de-|\bt'|+1} (1+t)^{-3+|\bt'|})\times \ep^{\f 34} = \ep^{\f 94} (1+T)^{2|\bt|},
\end{split}
\end{equation*}
where in the last line we have used $2\de <1$ (which follows from \eqref{def:de}). 

\pfstep{Case~2: $|\alp'|+|\bt'|+|\sigma'|>\Mi$} In this case, we must have $|\alp''|+|\bt''|+|\sigma''|\leq \Mi$ and hence we can bound $\rd_x^{\alp''}\rd_v^{\bt''}Y^{\sigma''} g$ in $L^\i_x L^p_v$ (with weights and with $p\geq 2$). Since $\Mi\geq 3$, by the pigeon hole principle, we have either $|\alp'|\geq 1$ or $|\bt'|\geq 1$ or $|\sigma'|\geq 2$. We consider below the (non-mutually exclusive) subcases $|\sigma'|\geq 2$ and $\max\{|\alp'|,|\bt'|\}\geq 1$.

\pfstep{Case~2(a): $|\alp'|+|\bt'|+|\sigma'|>\Mi$ and $|\sigma'|\geq 2$} First note that since $|\sigma'|\geq 2$, we have $|\sigma''|\leq |\sigma|-2$. As a consequence, 
\begin{equation}\label{III.2.a.1}
\wb^{2\Mm+10-2|\sigma|} \ls \wb^{\Mm+5-|\sigma|}\wb^{\Mm+5-|\sigma''|}\wb^{-2}.
\end{equation}

In this case, we simply estimate $\rd_x^{\alp'}\rd_v^{\bt'}Y^{\sigma'} (\bar{a}_{ij}v_i)$ using the trivial pointwise estimate
$$|\rd_x^{\alp'}\rd_v^{\bt'}Y^{\sigma'} (\bar{a}_{ij}v_i)|\ls \vb (\max_{i,j}|\rd_x^{\alp'}\rd_v^{\bt'}Y^{\sigma'} \bar{a}_{ij}|)+ \sum_{\substack{|\widetilde{\bt}'|\leq |\bt'|,\,|\widetilde{\sigma}'|\leq |\sigma'| \\ |\widetilde{\bt}'| + |\widetilde{\sigma}'| =|\bt'|+ |\sigma'| -1}} (\max_{i,j}|\rd_x^{\alp'}\rd_v^{\widetilde{\bt}'}Y^{\widetilde{\sigma}'} \bar{a}_{ij}|).$$
Together with Proposition~\ref{prop:ab.L2.null.cond}, this then implies that\footnote{Note that in fact the stronger estimate with $\vb^{-2}$ replaced by $\vb^{-1}$ on the LHS holds.}
\begin{equation}\label{III.2.a.2}
\max_j \|\vb^{-2}\wb^{-2}\rd_x^{\alp'}\rd_v^{\bt'}Y^{\sigma'} (\bar{a}_{ij}v_i)\|_{L^2_xL^\i_v+ L^2_xL^{p_{**}}_v}(t) \ls \ep^{\f 34}(1+t)^{-\min\{\f{16}{5},5+\gamma\}+|\bt'|}.
\end{equation}
Recall now also that $|\bt'|+|\bt''|\leq |\bt|+1$. Therefore, using \eqref{III.2.a.1}, \eqref{III.2.a.2}, H\"older's inequality and Proposition~\ref{main.bulk.est}, we obtain
\begin{equation*}
\begin{split}
&\: \max_{j} \| \wb^{2\Mm+10-2|\sigma|}|\rd_x^{\alp}\rd_v^{\bt}Y^{\sigma} g| |\rd_x^{\alp'}\rd_v^{\bt'}Y^{\sigma'} (\bar{a}_{ij}v_i)| |\rd_x^{\alp''} \rd_v^{\bt''} Y^{\sigma''} g| \|_{L^1([0,T];L^1_xL^1_v)} \\
\ls &\: \max_{j}(1+T)^{2|\bt|} \| (1+t)^{-\f 12-\de-|\bt|}\vb\wb^{\Mm+5-|\sigma|}\rd_x^{\alp'''}\rd_v^{\bt'''}Y^{\sigma'''} g\|_{L^2([0,T];L^2_xL^2_v)}\\
&\: \times \|(1+t)^{1+2\de-|\bt'|+1}\vb^{-2}\wb^{-2}\rd_x^{\alp'}\rd_v^{\bt'}Y^{\sigma'} (\bar{a}_{ij}v_i)\|_{L^\i([0,T];L^2_xL^\i_v+ L^2_xL^{p_{**}}_v)} \\
&\:\times  \|(1+t)^{-\f 12-\de-|\bt''|}\vb\wb^{\Mm+5-|\sigma''|}\rd_x^{\alp''} \rd_v^{\bt''} Y^{\sigma''} g \|_{L^2([0,T];L^\i_xL^2_v\cap L^{\i}_xL^{\f{2p_{**}}{p_{**}-2}}_v)} \\
\ls &\: (1+T)^{2|\bt|} \times \ep^{\f 34} \times (\sup_{t\in [0,T]} \ep^{\f 34}(1+t)^{1+2\de-|\bt'|+1}(1+t)^{-\min\{\f{16}{5},5+\gamma\}+|\bt'|} )\times \ep^{\f 34} = \ep^{\f 94} (1+T)^{2|\bt|},
\end{split}
\end{equation*}
where in the last line we used that $2\de <\min\{3+\gamma,\f 65\}$ (by \eqref{def:de}).

\pfstep{Case~2(b): $|\alp'|+|\bt'|+|\sigma'|>\Mi$ and $\max\{|\alp'|,|\bt'|\}\geq 1$} While in this case we have no gain in $\wb$ powers, we use the improvement in Proposition~\ref{prop:ab.L2.weighted} when $\max\{|\alp'|,|\bt'|\}\geq 1$.

Note that $|\bt'|+|\bt''|\leq |\bt|+1$. Hence, by H\"older's inequality, \eqref{prop:ab.L2.weighted.improved} in Proposition~\ref{prop:ab.L2.weighted} and Proposition~\ref{main.bulk.est}, we obtain
\begin{equation*}
\begin{split}
&\: \max_{j} \| \wb^{2\Mm+10-2|\sigma|}|\rd_x^{\alp}\rd_v^{\bt}Y^{\sigma} g| |\rd_x^{\alp'}\rd_v^{\bt'}Y^{\sigma'} (\bar{a}_{ij}v_i)| |\rd_x^{\alp''} \rd_v^{\bt''} Y^{\sigma''} g| \|_{L^1([0,T];L^1_xL^1_v)} \\
\ls &\: \max_{j}(1+T)^{2|\bt|} \| (1+t)^{-\f 12-\de-|\bt|}\vb\wb^{\Mm+5-|\sigma|}\rd_x^{\alp}\rd_v^{\bt}Y^{\sigma} g\|_{L^2([0,T];L^2_xL^2_v)}\\
&\: \times \|(1+t)^{1+2\de-|\bt'|+1}\vb^{-2} \rd_x^{\alp'}\rd_v^{\bt'}Y^{\sigma'} (\bar{a}_{ij}v_i)\|_{L^\i([0,T];L^2_xL^\i_v+L^2_xL^{p_*}_v)}\\
&\:\times  \|(1+t)^{-\f 12-\de-|\bt''|}\vb\wb^{\Mm+5-|\sigma''|}\rd_x^{\alp''} \rd_v^{\bt''} Y^{\sigma''} g \|_{L^2([0,T];L^\i_xL^2_v\cap L^\i_xL^{\f{2p_*}{p_*-2}}_v)} \\
\ls &\: (1+T)^{2|\bt|} \times \ep^{\f 34} \times (\sup_{t\in [0,T]} \ep^{\f 34}(1+t)^{1+2\de-|\bt'|+1}(1+t)^{-\min\{\f{11}{5},4+\gamma\}+|\bt'|} )\times \ep^{\f 34} = \ep^{\f 94} (1+T)^{2|\bt|},
\end{split}
\end{equation*}
where in the last line we used that $2\de <\min\{2+\gamma,\f 15\}$ (by \eqref{def:de}). \qedhere
\end{proof}

\begin{proposition}\label{prop:IV}
Let $|\alp|+|\bt|+|\sigma|\leq \Mm$. Then the term $IV_e^{\alp,\bt,\sigma}$ in \eqref{def:IV} is bounded as follows for every $T\in [0,T_{Boot})$:
$$IV_e^{\alp,\bt,\sigma}(T)\ls \ep^2(1+T)^{2|\bt|}.$$
\end{proposition}
\begin{proof}
By H\"older's inequality, Proposition~\ref{prop:ab.Li.weighted} and Lemma~\ref{lem:move.t.weights}, we obtain
\begin{equation*}
\begin{split}
&\: \max_j \| t \wb^{2\Mm+9-2|\sigma|}|\rd_x^{\alp}\rd_v^{\bt}Y^{\sigma} g| |\bar{a}_{ij}v_i| |\rd_x^{\alp} \rd_v^{\bt}Y^{\sigma} g| \|_{L^1([0,T];L^1_xL^1_v)} \\
\ls &\: \max_j (1+T)^{2|\bt|} \|(1+t)^{-\f 12-\de-|\bt|}\vb\wb^{\Mm+5-|\sigma|}\rd_x^{\alp}\rd_v^{\bt}Y^{\sigma} g\|_{L^2([0,T];L^2_xL^2_v)}^2 \\
&\: \times \|t(1+t)^{1+2\de}\vb^{-2}\bar{a}_{ij}v_i\|_{L^\i([0,T];L^\i_xL^\i_v)} \\
\ls &\: (1+T)^{2|\bt|} \times (\ep^{\f 34})^2\times \ep^{\f 34} \sup_{t\in [0,T]} (1+t)^{1+2\de+1} (1+t)^{-3} = \ep^{\f 94} (1+T)^{2|\bt|},
\end{split}
\end{equation*}
where in the last line we have used $2\de <1$ (by \eqref{def:de}). \qedhere
\end{proof}

\begin{proposition}\label{prop:V}
Let $|\alp|+|\bt|+|\sigma|\leq \Mm$. Then the term $V_e^{\alp,\bt,\sigma}$ in \eqref{def:V} is bounded as follows for every $T\in [0,T_{Boot})$:
$$V_e^{\alp,\bt,\sigma}(T)\ls \ep^2(1+T)^{2|\bt|}.$$
\end{proposition}
\begin{proof}
The term $V_e^{\alp,\bt,\sigma}$ contains two sums, one with $\bar{a}_{ii}$ and one with $\bar{a}_{ij} v_i v_j$. To simplify the exposition, let us just estimate the terms with $\bar{a}_{ij} v_i v_j$. When we handle these terms, we will only use Propositions~\ref{prop:ab.Li.weighted} and \ref{prop:ab.L2.weighted} to control $\bar{a}_{ij} v_i v_j$ and its derivatives. Now note that since by Propositions~\ref{prop:ab.Li.1} and \ref{prop:ab.L2}, $\bar{a}_{ii}$ and its derivatives obey all the analogous estimates for $\bar{a}_{ij} v_i v_j$ and its derivatives in Propositions~\ref{prop:ab.Li.weighted} and \ref{prop:ab.L2.weighted}, the exact same argument will also apply the to terms with $\bar{a}_{ii}$ instead of $\bar{a}_{ij}v_iv_j$.

Now take $\alp'$, $\alp''$, $\bt'$, $\bt''$, $\sigma'$ and $\sigma''$ satisfying the required conditions in the sum of $V_e^{\alp,\bt,\sigma}$. We divide into the cases $|\alp'|+|\bt'|+|\sigma'|\leq \Mi$ and $|\alp'|+|\bt'|+|\sigma'|>\Mi$.

\pfstep{Step~1: $|\alp'|+|\bt'|+|\sigma'|\leq \Mi$} By H\"older's inequality, Proposition~\ref{prop:ab.Li.weighted} and Lemma~\ref{lem:move.t.weights}, we obtain
\begin{equation*}
\begin{split}
&\: \| \wb^{2\Mm+10-2|\sigma|}|\rd_x^{\alp}\rd_v^{\bt}Y^{\sigma} g| |\rd_x^{\alp'}\rd_v^{\bt'}Y^{\sigma'} (\bar{a}_{ij}v_iv_j)| |\rd_x^{\alp''} \rd_v^{\bt''} Y^{\sigma''} g| \|_{L^1([0,T];L^1_xL^1_v)}\\
\ls &\: (1+T)^{2|\bt|} \|(1+t)^{-\f 12-\de-|\bt|}\vb \wb^{\Mm+5-|\sigma|} \rd_x^{\alp}\rd_v^{\bt}Y^{\sigma} g\|_{L^2([0,T];L^2_xL^2_v)} \\
&\: \times \|(1+t)^{1+2\de-|\bt'|} \vb^{-2} \rd_x^{\alp'}\rd_v^{\bt'}Y^{\sigma'} (\bar{a}_{ij}v_iv_j)\|_{L^\i([0,T];L^\i_xL^\i_v)}\\
&\: \times \|(1+t)^{-\f 12-\de-|\bt''|}\vb \wb^{\Mm+5-|\sigma''|} \rd_x^{\alp''}\rd_v^{\bt''} Y^{\sigma''} g\|_{L^2([0,T];L^2_xL^2_v)} \\
\ls &\: (1+T)^{2|\bt|}\times \ep^{\f34} \times \ep^{\f 34}(\sup_{t\in [0,T]} (1+t)^{1+2\de-|\bt'|}(1+t)^{-3+|\bt'|})\times \ep^{\f 34} = \ep^{\f 94} (1+T)^{2|\bt|},
\end{split}
\end{equation*}
where in the last line we have used $2\de < 2$ (by \eqref{def:de}).

\pfstep{Step~2: $|\alp'|+|\bt'|+|\sigma'|>\Mi$} Note that in this case $|\alp''|+|\bt''|+|\sigma''|\leq \Mm$. Hence, by H\"older's inequality, Propositions~\ref{prop:ab.L2.weighted} and \ref{main.bulk.est}, we obtain
\begin{equation*}
\begin{split}
&\: \| \wb^{2\Mm+10-2|\sigma|}|\rd_x^{\alp}\rd_v^{\bt}Y^{\sigma} g| |\rd_x^{\alp'}\rd_v^{\bt'}Y^{\sigma'} (\bar{a}_{ij}v_iv_j)| |\rd_x^{\alp''} \rd_v^{\bt''} Y^{\sigma''} g| \|_{L^1([0,T];L^1_xL^1_v)}\\
\ls &\: (1+T)^{2|\bt|} \|(1+t)^{-\f 12-\de-|\bt|}\vb \wb^{\Mm+5-|\sigma|} \rd_x^{\alp}\rd_v^{\bt}Y^{\sigma} g\|_{L^2([0,T];L^2_xL^2_v)} \\
&\: \times \|(1+t)^{1+2\de-|\bt'|} \vb^{-2} \rd_x^{\alp'}\rd_v^{\bt'}Y^{\sigma'} (\bar{a}_{ij}v_iv_j)\|_{L^\i([0,T];L^2_xL^\i_v)}\\
&\: \times \|(1+t)^{-\f 12-\de-|\bt''|}\vb \wb^{\Mm+5-|\sigma''|} \rd_x^{\alp''}\rd_v^{\bt''}Y^{\sigma''} g\|_{L^2([0,T];L^\i_xL^2_v)} \\
\ls &\: (1+T)^{2|\bt|}\times \ep^{\f34} \times \ep^{\f 34}(\sup_{t\in [0,T]} (1+t)^{1+2\de-|\bt'|}(1+t)^{-\f 32+|\bt'|})\times \ep^{\f 34} = \ep^{\f 94} (1+T)^{2|\bt|},
\end{split}
\end{equation*}
where in the last line we have used $2\de < \f 12$ (by \eqref{def:de}). \qedhere
\end{proof}

\begin{proposition}\label{prop:VI}
Let $|\alp|+|\bt|+|\sigma|\leq \Mm$. Then the term $VI_e^{\alp,\bt,\sigma}$ in \eqref{def:VI} is bounded as follows for every $T\in [0,T_{Boot})$:
$$VI_e^{\alp,\bt,\sigma}(T)\ls \ep^2(1+T)^{2|\bt|}.$$
\end{proposition}
\begin{proof}
Take $\alp'$, $\alp''$, $\bt'$, $\bt''$, $\sigma'$ and $\sigma''$ satisfying the required conditions in the sum of $VI_e^{\alp,\bt,\sigma}$. We divide into the cases $|\alp'|+|\bt'|+|\sigma'|\leq \Mi$ and $|\alp'|+|\bt'|+|\sigma'|>\Mm$.

\pfstep{Case~1: $|\alp'|+|\bt'|+|\sigma'|\leq \Mi$} By H\"older's inequality, Proposition~\ref{prop:cb.Li} and Lemma~\ref{lem:move.t.weights}, we obtain
\begin{equation*}
\begin{split}
&\:\| \wb^{2\Mm+10-2|\sigma|}|\rd_x^\alp\rd_v^\bt Y^{\sigma} g|| \rd_x^{\alp'}\rd_v^{\bt'}Y^{\sigma'} \cb| |\rd_x^{\alp''} \rd_v^{\bt''} Y^{\sigma''} g| \|_{L^1([0,T];L^1_xL^1_v)} \\
\ls &\: (1+T)^{2|\bt|}\|(1+t)^{-\f 12-\de-|\bt|}\wb^{\Mm+5-|\sigma|} \rd_x^\alp\rd_v^\bt Y^{\sigma} g \|_{L^2([0,T];L^2_xL^2_v)} \\
&\: \times \| (1+t)^{1+2\de-|\bt'|}\rd_x^{\alp'}\rd_v^{\bt'}Y^{\sigma'} \cb \|_{L^\i([0,T];L^\i_xL^\i_v)}\\
&\: \times \|(1+t)^{-\f 12-\de-|\bt''|} \wb^{\Mm+5-|\sigma''|} \rd_x^{\alp''}\rd_v^{\bt''}Y^{\sigma''} g \|_{L^2([0,T];L^2_xL^2_v)}\\
\ls &\: (1+T)^{2|\bt|}\times \ep^{\f 34} \times \ep^{\f 34}(\sup_{t\in [0,T]} (1+t)^{1+2\de-|\bt'|}(1+t)^{-3-\gamma+|\bt'|}) \times \ep^{\f 34} = \ep^{\f 94} (1+T)^{2|\bt|},
\end{split}
\end{equation*}
where in the last line we have used that $2\de < 2+\gamma$ (by \eqref{def:de}).

\pfstep{Case~2: $|\alp'|+|\bt'|+|\sigma'|>\Mi$} Note that in this case $|\alp''|+|\bt''|+|\sigma''|\leq \Mm$. Hence, by H\"older's inequality, Propositions~\ref{prop:cb.L2} and \ref{main.bulk.est}, we obtain
\begin{equation*}
\begin{split}
&\:\| \wb^{2\Mm+10-2|\sigma|}|\rd_x^\alp\rd_v^\bt Y^{\sigma} g|| \rd_x^{\alp'}\rd_v^{\bt'}Y^{\sigma'} \cb| |\rd_x^{\alp''} \rd_v^{\bt''} Y^{\sigma''} g| \|_{L^1([0,T];L^1_xL^1_v)} \\
\ls &\: (1+T)^{2|\bt|}\|(1+t)^{-\f 12-\de-|\bt|}\wb^{\Mm+5-|\sigma|} \rd_x^\alp\rd_v^\bt Y^{\sigma} g \|_{L^2([0,T];L^2_xL^2_v)} \\
&\: \times \| (1+t)^{1+2\de-|\bt'|}\rd_x^{\alp'}\rd_v^{\bt'}Y^{\sigma'} \cb \|_{L^\i([0,T];L^2_xL^{p_{**}}_v)}\\
&\: \times \|(1+t)^{-\f 12-\de-|\bt''|} \wb^{\Mm+5-|\sigma''|} \rd_x^{\alp''}\rd_v^{\bt''}Y^{\sigma''} g \|_{L^2([0,T];L^\i_xL^{\f{2p_{**}}{p_{**}-2}}_v)}\\
\ls &\: (1+T)^{2|\bt|}\times \ep^{\f 34} \times \ep^{\f 34}(\sup_{t\in [0,T]} (1+t)^{1+2\de-|\bt'|}(1+t)^{-\min\{\f 65,3+\gamma\}+|\bt'|}) \times \ep^{\f 34} = \ep^{\f 94} (1+T)^{2|\bt|},
\end{split}
\end{equation*}
where in the last line we have used that $2\de < \min\{2+\gamma,\f 15\}$ (by \eqref{def:de}). \qedhere
\end{proof}

The terms $VII_e^{\alp,\bt,\sigma}$ and $VIII_e^{\alp,\bt,\sigma}$ are linear in $g^2$ (or the square of the derivatives of $g$). As a consequence, we will not have enough smallness if we just apply the bootstrap assumptions to control them. Therefore unlike the previous terms, we will still keep track of the precise terms on the RHS. 
\begin{proposition}\label{prop:VII}
Let $|\alp|+|\bt|+|\sigma|\leq \Mm$. Then for every $\eta>0$, there exists a constant $C_\eta>0$ (depending on $\eta$ in addition to $d_0$ and $\gamma$) such that the term $VII_e^{\alp,\bt,\sigma}$ in \eqref{def:VII} is bounded as follows for every $T\in [0,T_{Boot})$:
\begin{equation*}
\begin{split}
VII_e^{\alp,\bt,\sigma}(T)\leq &\: \eta \|\wb^{\Mm+5-|\sigma|}\rd_x^\alp\rd_v^\bt Y^\sigma g\|_{L^\infty([0,T];L^2_x L^2_v)}^2 \\
&\: + C_\eta T^2 \sum_{\substack{|\alp'|\leq |\alp|+1 \\|\bt'|\leq |\bt|-1}}\|\wb^{\Mm+5-|\sigma|}\rd_x^{\alp'}\rd_v^{\bt'}Y^\sigma g\|_{L^\infty([0,T];L^2_x L^2_v)}^2.
\end{split}
\end{equation*}
\end{proposition}
\begin{proof}
By H\"older's inequality,
\begin{equation*}
\begin{split}
&\: \sum_{|\alp'|\leq |\alp|+1,\,|\bt'|\leq |\bt|-1} \| \wb^{2\Mm+10-2|\sigma|} |\rd_x^\alp\rd_v^\bt Y^\sigma g| |\rd_x^{\alp'} \rd_v^{\bt'} Y^\sigma  g| \|_{L^1([0,T];L^1_xL^1_v)} \\
\ls &\: \sum_{\substack{|\alp'|\leq |\alp|+1 \\|\bt'|\leq |\bt|-1}} T \|\wb^{\Mm+5-|\sigma|}\rd_x^\alp\rd_v^\bt Y^\sigma g\|_{L^\infty([0,T];L^2_x L^2_v)} \|\wb^{\Mm+5-|\sigma|}\rd_x^{\alp'}\rd_v^{\bt'}Y^\sigma g\|_{L^\infty([0,T];L^2_x L^2_v)}.
\end{split}
\end{equation*}
The conclusion then follows from an application of Young's inequality. \qedhere
\end{proof}

\begin{proposition}\label{prop:VIII}
Let $|\alp|+|\bt|+|\sigma|\leq \Mm$. Then for every $\eta>0$, there exists a constant $C_\eta>0$ (depending on $\eta$ in addition to $d_0$ and $\gamma$) such that the term $VIII_e^{\alp,\bt,\sigma}$ in \eqref{def:VIII} is bounded as follows for every $T\in [0,T_{Boot})$:
\begin{equation*}
\begin{split}
VIII_e^{\alp,\bt,\sigma}(T)\leq &\: \eta \|(1+t)^{-\f 12-\f \de 2} \vb \wb^{\Mm+5-|\sigma|}\rd_x^\alp\rd_v^\bt Y^\sigma g\|_{L^2([0,T];L^2_x L^2_v)}^2 \\
&\: + C_\eta \sum_{\substack{|\bt'|\leq |\bt|,\,|\sigma'|\leq |\sigma|\\ |\bt'|+|\sigma'|\leq |\bt|+|\sigma|-1}}\|(1+t)^{-\f 12-\f \de 2} \vb \wb^{\Mm+5-|\sigma'|}\rd_x^{\alp}\rd_v^{\bt'}Y^{\sigma'} g\|_{L^2([0,T];L^2_x L^2_v)}^2.
\end{split}
\end{equation*}
\end{proposition}
\begin{proof}
This is an easy consequence of the Cauchy--Schwarz inequality and Young's inequality. \qedhere
\end{proof}

We have therefore estimated all of the terms on the RHS in the estimate in Proposition~\ref{prop:EE.with.error}.

\subsection{Putting everything together}\label{sec:EE.everything}

Combining Proposition~\ref{prop:EE.with.error} with the estimates in Propositions~\ref{prop:I}--\ref{prop:VIII}, we obtain
\begin{proposition}\label{EE.combined.1}
Let $|\alp|+|\bt|+|\sigma|\leq \Mm$. Then for every $\eta>0$, there exists a constant $C_\eta>0$ (depending on $\eta$ in addition to $d_0$ and $\gamma$) such that the following estimate holds for all $T\in [0,T_{Boot})$:
\begin{equation*}
\begin{split}
&\: \|\wb^{\Mm+5-|\sigma|}\rd_x^\alp\rd_v^\bt Y^{\sigma} g\|_{L^\infty([0,T];L^2_x L^2_v)}^2+ \|(1+t)^{-\f 12-\f \de 2} \vb \wb^{\Mm+5-|\sigma|}\rd_x^\alp\rd_v^\bt Y^{\sigma} g\|_{L^2([0,T];L^2_x L^2_v)}^2\\
\leq &\: C_\eta(\ep^2(1+T)^{2|\bt|} + T^2 \sum_{\substack{|\alp'|\leq |\alp|+1 \\|\bt'|\leq |\bt|-1}}\|\wb^{\Mm+5-|\sigma|}\rd_x^{\alp'}\rd_v^{\bt'}Y^{\sigma} g\|_{L^\infty([0,T];L^2_x L^2_v)}^2 \\
&\qquad + \sum_{\substack{|\bt'|\leq |\bt|,\,|\sigma'|\leq |\sigma|\\ |\bt'|+|\sigma'|\leq |\bt|+|\sigma|-1}}\|(1+t)^{-\f 12-\f \de 2} \vb \wb^{\Mm+5-|\sigma'|}\rd_x^{\alp}\rd_v^{\bt'}Y^{\sigma'} g\|_{L^2([0,T];L^2_x L^2_v)}^2  ) \\
&\: + \eta \|\wb^{\Mm+5-|\sigma|}\rd_x^\alp\rd_v^\bt Y^{\sigma} g\|_{L^\infty([0,T];L^2_x L^2_v)}^2 \\
&\: + \eta \|(1+t)^{-\f 12-\f \de 2} \vb \wb^{\Mm+5-|\sigma|}\rd_x^\alp\rd_v^\bt Y^{\sigma}g\|_{L^2([0,T];L^2_x L^2_v)}^2.
\end{split}
\end{equation*}
\end{proposition}
We can control the terms with an $\eta$ coefficient on the RHS of the estimate in Proposition~\ref{EE.combined.1} to obtain the following stronger bounds:
\begin{proposition}\label{EE.combined.2}
Let $|\alp|+|\bt|+|\sigma|\leq \Mm$. Then the following estimate holds for all $T\in [0,T_{Boot})$:
\begin{equation*}
\begin{split}
&\: \|\wb^{\Mm+5-|\sigma|}\rd_x^\alp\rd_v^\bt Y^{\sigma} g\|_{L^\infty([0,T];L^2_x L^2_v)}^2+ \|(1+t)^{-\f 12-\f \de 2} \vb \wb^{\Mm+5-|\sigma|}\rd_x^\alp\rd_v^\bt Y^{\sigma} g\|_{L^2([0,T];L^2_x L^2_v)}^2\\
\ls &\: \ep^2(1+T)^{2|\bt|} + T^2 \sum_{\substack{|\alp'|\leq |\alp|+1 \\|\bt'|\leq |\bt|-1}}\|\wb^{\Mm+5-|\sigma|}\rd_x^{\alp'}\rd_v^{\bt'}Y^\sigma g\|_{L^\infty([0,T];L^2_x L^2_v)}^2\\
&\:+ \sum_{\substack{|\bt'|\leq |\bt|,\,|\sigma'|\leq |\sigma|\\ |\bt'|+|\sigma'|\leq |\bt|+|\sigma|-1}}\|(1+t)^{-\f 12-\f \de 2} \vb \wb^{\Mm+5-|\sigma'|}\rd_x^{\alp}\rd_v^{\bt'}Y^{\sigma'} g\|_{L^2([0,T];L^2_x L^2_v)}^2.
\end{split}
\end{equation*}
Here, by our convention (see~Section~\ref{sec:notation}), if $|\bt|+|\sigma|=0$, then the last two terms on the RHS are not present.
\end{proposition}
\begin{proof}
Apply Proposition~\ref{EE.combined.1} with $\eta = \f 12$. We then subtract
$$\f 12(\|\wb^{\Mm+5-|\sigma|}\rd_x^\alp\rd_v^\bt Y^{\sigma}g\|_{L^\infty([0,T];L^2_x L^2_v)}^2 + \|(1+t)^{-\f 12-\f \de 2} \vb \wb^{\Mm+5-|\sigma|}\rd_x^\alp\rd_v^\bt Y^{\sigma} g\|_{L^2([0,T];L^2_x L^2_v)}^2)$$ 
from both sides of the equation. Now that $\eta$ is fixed, $C_\eta$ is simply a constant depending on $d_0$ and $\gamma$. We have thus proven the desired inequality. \qedhere
\end{proof}

We now set up an induction argument to obtain the final energy estimates from Proposition~\ref{EE.combined.2}.

\begin{proposition}\label{EE.final}
Let $|\alp|+|\bt|+|\sigma|\leq \Mm$. Then the following estimate holds for all $T\in [0,T_{Boot})$:
\begin{equation*}
\begin{split}
&\: \|\wb^{\Mm+5-|\sigma|}\rd_x^\alp\rd_v^\bt Y^{\sigma} g\|_{L^\infty([0,T];L^2_x L^2_v)}^2+ \|(1+t)^{-\f 12-\f \de 2} \vb \wb^{\Mm+5-|\sigma|}\rd_x^\alp\rd_v^\bt Y^{\sigma} g\|_{L^2([0,T];L^2_x L^2_v)}^2 \\
\ls &\: \ep^2(1+T)^{2|\bt|}.
\end{split}
\end{equation*}
\end{proposition}
\begin{proof}
We induct on $|\bt|+|\sigma|$. 

\pfstep{Step~1: Base case: $|\bt|+|\sigma|=0$} Applying Proposition~\ref{EE.combined.2} when $|\bt|+|\sigma|=0$, the last two terms on the RHS are not present. Hence we immediately obtain
$$\|\wb^{\Mm+5}\rd_x^\alp g\|_{L^\infty([0,T];L^2_x L^2_v)}^2+ \|(1+t)^{-\f 12-\f \de 2} \vb \wb^{\Mm+5}\rd_x^\alp g\|_{L^2([0,T];L^2_x L^2_v)}^2 \ls \ep^2$$ 
for all $|\alp|\leq \Mm$, as desired.

\pfstep{Step~2: Induction step} Assume as our induction hypothesis that there exists a $B\in \mathbb N$ such that whenever $|\alp|+|\bt|+|\sigma|\leq \Mm$ and $|\bt|+|\sigma|\leq B-1$, 
\begin{equation*}
\begin{split}
&\: \|\wb^{\Mm+5-|\sigma|}\rd_x^\alp\rd_v^\bt Y^{\sigma} g\|_{L^\infty([0,T];L^2_x L^2_v)}^2+ \|(1+t)^{-\f 12-\f \de 2} \vb \wb^{\Mm+5-|\sigma|}\rd_x^\alp\rd_v^\bt Y^{\sigma} g\|_{L^2([0,T];L^2_x L^2_v)}^2 \\
\ls &\: \ep^2(1+T)^{2|\bt|}.
\end{split}
\end{equation*}

Now take some multi-indices $\alp$, $\bt$ and $\sigma$ such that $|\alp|+|\bt|+|\sigma|\leq \Mm$ and $|\bt|+|\sigma| = B$. Our goal will be to show that the estimate as in the statement of the proposition holds for this choice of $(\alp,\bt,\sigma)$.

By Proposition~\ref{EE.combined.2} and the induction hypothesis,
\begin{equation*}
\begin{split}
&\: \|\wb^{\Mm+5-|\sigma|}\rd_x^\alp\rd_v^\bt Y^{\sigma} g\|_{L^\infty([0,T];L^2_x L^2_v)}^2+ \|(1+t)^{-\f 12-\f \de 2} \vb \wb^{\Mm+5-|\sigma|}\rd_x^\alp\rd_v^\bt Y^{\sigma} g\|_{L^2([0,T];L^2_x L^2_v)}^2\\
\ls &\: \ep^2(1+T)^{2|\bt|} + T^2 \sum_{\substack{|\alp'|\leq |\alp|+1 \\|\bt'|\leq |\bt|-1}}\|\wb^{\Mm+5-|\sigma|}\rd_x^{\alp'}\rd_v^{\bt'}Y^\sigma g\|_{L^\infty([0,T];L^2_x L^2_v)}^2\\
&\:+ \sum_{\substack{|\bt'|\leq |\bt|,\,|\sigma'|\leq |\sigma|\\ |\bt'|+|\sigma'|\leq |\bt|+|\sigma|-1}}\|(1+t)^{-\f 12-\f \de 2} \vb \wb^{\Mm+5-|\sigma'|}\rd_x^{\alp}\rd_v^{\bt'}Y^{\sigma'} g\|_{L^2([0,T];L^2_x L^2_v)}^2\\
\ls &\: \ep^2(1+T)^{2|\bt|} + \ep^2 (\sum_{|\bt'|\leq |\bt|-1} T^2 (1+T)^{2|\bt'|}) + \ep^2 (\sum_{|\bt'|\leq |\bt|} (1+T)^{2|\bt'|}) \ls \ep^2(1+T)^{2|\bt|}.
\end{split}
\end{equation*}

By induction, we have thus obtained the desired estimate. \qedhere
\end{proof}

\subsection{Proof of Theorem~\ref{thm:BA}}\label{sec:end.of.bootstrap}

Combining Propositions~\ref{prop:Li.improved} and \ref{EE.final}, we have now completed the proof of Theorem~\ref{thm:BA}.

\section{Putting everything together (Proof of Theorem~\ref{thm:main})}\label{sec:everything}

We now complete the proof of Theorem~\ref{thm:main}:

\begin{proof}[Proof of Theorem~\ref{thm:main}]
We assume throughout that $\underline{\ep}_0\leq \ep_0$ so that Theorem~\ref{thm:BA} applies.

Let
\begin{equation*}
\begin{split}
T_{\mathrm{max}}:= \sup \{T\in [0,+\infty): &\: \mbox{there exists a unique solution $f:[0,T]\times \mathbb R^3\times \mathbb R^3$ to \eqref{Landau} with} \\
&\: \mbox{$f\geq 0$, $f\restriction_{\{t=0\}} = f_{\mathrm{in}}$ and satisfying \eqref{eq:local.ext} for $k=\Mm$ and $N = \Mm+5$} \\
&\: \mbox{such that the bootstrap assumptions \eqref{BA}, \eqref{BA.sim.1} and \eqref{BA.sim.2} hold}\}.
\end{split}
\end{equation*} 
Note that by Corollary~\ref{cor:local.ext}, $T_{\mathrm{max}}>0$.

We will prove that $T_{\mathrm{max}}=+\infty$. Assume for the sake of contradiction that $T_{\mathrm{max}}<+\infty$.

It follows from the definition of $T_{\mathrm{max}}$ that the assumptions of Theorem~\ref{thm:BA} hold for $T_{Boot} = T_{\mathrm{max}}$. Therefore, by (the $|\sigma|=0$ case in) Theorem~\ref{thm:BA}, 
\begin{equation}\label{uniform.higher.reg.bound}
\sum_{|\alp|+|\bt|\leq \Mm} \| \wb^{\Mm+5} \rd_x^\alp \rd_v^\bt (e^{d(t)\vb^2}f)(t,x,v)\|_{L^\i([0,T_{\mathrm{max}});L^2_xL^2_v)} \ls \ep.
\end{equation}
Take an increasing sequence $\{t_n\}_{n=1}^\infty\subset [0,T_{\mathrm{max}})$ such that $t_n\to T_{\mathrm{max}}$. By the uniform bound \eqref{uniform.higher.reg.bound} and the local existence result in Corollary~\ref{cor:local.ext}, there exists $T_{\mathrm{small}}\in (0,1]$ such that a unique solution exists $[0,t_n+T_{\mathrm{small}}]\times \mathbb R^3\times \mathbb R^3$. In particular, taking $n$ sufficiently large, we have constructed a solution beyond the time $T_{\mathrm{max}}$, up to, say, time $T_{\mathrm{max}}+ \f 12 T_{\mathrm{small}}$. The solution moreover satisfies \eqref{eq:local.ext} for $k=\Mm$ and $N = \Mm+5$

Our next goal will be to show that in fact the estimates \eqref{BA}, \eqref{BA.sim.1} and \eqref{BA.sim.2} hold slightly beyond $T_{\mathrm{max}}$. Our starting point is that by the bootstrap theorem (Theorem~\ref{thm:BA}), 
\begin{equation}\label{important.fact}
\mbox{the estimates \eqref{BA}, \eqref{BA.sim.1} and \eqref{BA.sim.2} in fact all hold in $[0,T_{\mathrm{max}})$ with $\ep^{\f34}$ replaced by $C_{d_0,\gamma} \ep$.}
\end{equation} 

By the local existence result in Corollary~\ref{cor:local.ext}, for $|\alp|+|\bt|\leq \Mm$, $\wb^{\Mm+5}\rd_x^\alp \rd_v^\bt g(t,x,v) \in C^0([0,T_{\mathrm{max}}+\f 12 T_{\mathrm{small}}]; L^2_xL^2_v)$. Since $Y=t\rd_x+\rd_v$, for all $|\alp|+|\bt|+|\sigma|\leq \Mm$, we also have $\wb^{\Mm+5-|\sigma|}\rd_x^\alp \rd_v^\bt Y^\sigma g(t,x,v) \in C^0([0,T_{\mathrm{max}}+\f 12 T_{\mathrm{small}}]; L^2_xL^2_v)$. Using also \eqref{important.fact}, it follows that after choosing $\underline{\ep}_0$ smaller (so that $\ep$ is sufficiently small) if necessary, there exists $T_{\mathrm{ext},0}\in (T_{\mathrm{max}}, T_{\mathrm{max}}+\f 12 T_{\mathrm{small}}]$ such that \eqref{BA} holds up to time $T_{\mathrm{ext},0}$. It thus remains to prove that the estimates \eqref{BA.sim.1} and \eqref{BA.sim.2} hold beyond $T_{\mathrm{max}}$.

\textbf{Claim~1:} There exist $R_0>0$ and $T_{\mathrm{ext},1}\in (T_{\mathrm{max}},T_{\mathrm{ext},0}]$ such that 
$$\sum_{|\alp|+|\bt|\leq \Mm}\|\wb^{\Mm+5-|\sigma|}\rd_x^\alp \rd_v^\bt Y^\sigma g(t,x,v) \|_{L^2_xL^2_v (\{|x|^2+|v|^2\geq R_0^2\})} \leq \ep(1+T_{\mathrm{max}})^{-4}$$ 
for every $t\in [T_{\mathrm{max}}, T_{\mathrm{ext},1}]$.

\emph{Proof of Claim~1.} We showed above that \eqref{BA} holds up to time $T_{\mathrm{ext},0}$. In particular, $\sum_{|\alp|+|\bt|+|\sigma|\leq \Mm}\|\langle x-T_{\mathrm{max}}v \rangle^{\Mm+5-|\sigma|}\rd_x^\alp \rd_v^\bt Y^\sigma g(T_{\mathrm{max}},x,v) \|_{L^2_xL^2_v}$ is finite. Hence there exists $R_0'>0$ such that 
\begin{equation}\label{claim.1.pf.1}
\sum_{|\alp|+|\bt|+|\sigma|\leq \Mm}\|\langle x-T_{\mathrm{max}}v \rangle^{\Mm+5-|\sigma|}\rd_x^\alp \rd_v^\bt Y^\sigma g(T_{\mathrm{max}},x,v) \|_{L^2_xL^2_v (\{|x|^2+|v|^2\geq (R_0')^2\})} \leq \f{\ep}2(1+T_{\mathrm{max}})^{-4}.
\end{equation}

Let $\chi:\mathbb R^3\times \mathbb R^3\to \mathbb R$ be a smooth cut-off function satisfying $0\leq \chi\leq 1$ with $\chi(x,v) = 1$ when $|x|^2+|v|^2\geq (R_0'+1)^2$ and $\chi(x,v)=0$ when $|x|^2+|v|^2\leq (R_0')^2$.

By the continuity-in-time of the $L^2_xL^2_v$ norm in local existence result in Corollary~\ref{cor:local.ext}, 
\begin{equation}\label{claim.1.pf.2}
\lim_{t\to T_{\mathrm{max}}^+} \sum_{|\alp|+|\bt|+|\sigma|\leq \Mm} \| \wb^{\Mm+5-|\sigma|}\chi(x,v) \rd_x^\alp \rd_v^\bt Y^\sigma (g(t,x,v) - g(T_{\mathrm{max}},x,v))\|_{L^2_xL^2_v}  = 0.
\end{equation}
Let $R_0 = R_0'+1$. The claim follows from \eqref{claim.1.pf.1} and \eqref{claim.1.pf.2}.

\textbf{Claim~2:} Given $T_{\mathrm{ext},1}$ as in Claim~1, there exists $T_{\mathrm{ext},2} \in (T_{\mathrm{max}},T_{\mathrm{ext},1}]$ such that for every $t\in [T_{\mathrm{max}},T_{\mathrm{ext},2}]$, when
$|\alp|+|\bt|+|\sigma| \leq \Mm-4-\max\{2,\lceil \f{2}{2+\gamma} \rceil \}$, 
\begin{equation}\label{BA.final.1}
\|\wb^{\Mm+5-|\sigma|} \rd_x^{\alp} \rd_v^{\bt} Y^{\sigma} g \|_{L^\i_x L^\i_v}(t) \leq \f 12 \ep^{\f 34}(1+t)^{|\bt|};
\end{equation}
and when $\Mm-3-\max\{2,\lceil \f{2}{2+\gamma} \rceil \}\leq |\alp|+|\bt|+|\sigma|=:k \leq \Mm-5$, then
\begin{equation}\label{BA.final.2}
\|\wb^{\Mm+5-|\sigma|} \rd_x^{\alp} \rd_v^{\bt} Y^{\sigma} g \|_{L^\i_x L^\i_v}(t) \leq \f 12 \ep^{\f 34}(1+t)^{\f 32 - (\Mm-4-k)\min\{\f 34, \f{3(2+\gamma)}{4}\} +|\bt|}.
\end{equation}

\emph{Proof of Claim~2.} We first prove \eqref{BA.final.1} and \eqref{BA.final.2} for $|x|^2+|v|^2\geq (R_0+1)^2$, where $R_0$ is as in Claim~1. Let $\chi:\mathbb R^3\times \mathbb R^3\to \mathbb R$ be a smooth cut-off function\footnote{Note that this cut-off function is slightly different from that in Claim~1.} satisfying $0\leq \chi\leq 1$ with $\chi(x,v) = 1$ when $|x|^2+|v|^2\geq (R_0+1)^2$ and $\chi(x,v)=0$ when $|x|^2+|v|^2\leq R_0^2$.

We use the Sobolev embedding in Lemma~\ref{lem:stupid.Sobolev.embedding} to control the $L^\i_xL^\i_v$ norm $\chi \rd_x^\alp \rd_v^\bt Y^\sigma g$ for $t\in [T_{\mathrm{max}},T_{\mathrm{ext},1}]$
\begin{equation*}
\begin{split}
&\: \sup_{t\in [T_{\mathrm{max}},T_{\mathrm{ext},1}]}\sum_{|\alp|+|\bt|+|\sigma|\leq \Mm-5} \|\wb^{\Mm+5-|\sigma|}\chi \rd_x^\alp \rd_v^\bt Y^\sigma g\|_{L^\i_xL^\i_v} \\
\ls &\: \sup_{t\in [T_{\mathrm{max}},T_{\mathrm{ext},1}]}\sum_{|\alp|+|\bt|+|\sigma|\leq \Mm-5} \sum_{|\alp'|+|\bt'|\leq 4} \|\rd_x^{\alp'}\rd_v^{\bt'} (\wb^{\Mm+5-|\sigma|} \chi \rd_x^\alp \rd_v^\bt Y^\sigma g)\|_{L^2_xL^2_v} \\
\ls &\: \sup_{t\in [T_{\mathrm{max}},T_{\mathrm{ext},1}]} (1+T_{Boot})^4 \sum_{|\alp|+|\bt|+|\sigma|\leq \Mm-1} \| \wb^{\Mm+5-|\sigma|}\rd_x^\alp \rd_v^\bt Y^\sigma g\|_{L^2_xL^2_v(\{|x|^2+|v|^2\geq R_0^2)} \ls \ep,
\end{split}
\end{equation*}
where in the last estimate we have used Claim~1. By the properties of $\chi$, we have thus proven \eqref{BA.final.1} and \eqref{BA.final.2} for $|x|^2+|v|^2\geq (R_0+1)^2$ for every $t\in [T_{\mathrm{max}},T_{\mathrm{ext},1}]$.

It remains to prove \eqref{BA.final.1} and \eqref{BA.final.2} for $|x|^2+|v|^2\leq (R_0+1)^2$. Note that this is a spatially compact set, and we already have the estimate \eqref{important.fact}. Therefore, by the smoothness of $g$, after choosing $\underline{\ep}_0$ smaller if necessary (so that $\ep$ is also sufficiently small), \eqref{BA.final.1} and \eqref{BA.final.2} hold in the region $|x|^2+|v|^2\leq (R_0+1)^2$ for every $t\in [T_{\mathrm{max}},T_{\mathrm{ext},2}]$ for some $T_{\mathrm{ext},2}$ chosen to be sufficiently close to $T_{\mathrm{max}}$.

Combining the estimates for $|x|^2+|v|^2\geq (R_0+1)^2$ and $|x|^2+|v|^2\leq (R_0+1)^2$, we have proven Claim~2.

Claim~2 therefore established that the estimates \eqref{BA.sim.1} and \eqref{BA.sim.2} can be extend beyond $T_{\mathrm{max}}$. Together with the extension of \eqref{BA} beyond $T_{\mathrm{max}}$ that we established earlier, we have obtained a contradiction with the definition of $T_{\mathrm{max}}$. It thus follows that $T_{\mathrm{max}}=+\infty$. 

Finally, the statements of uniqueness, smoothness and positivity of $f$ follow from Theorems~\ref{thm:local.existence} and \ref{thm:smoothness.positivity}. \qedhere
\end{proof}

\section{Long-time asymptotics}\label{sec:long.time}

In this final section we prove the results about long-time asymptotics of solutions in the near-vacuum regime. In \textbf{Section~\ref{sec:9.1}}, we prove Theorem~\ref{thm:asymptotics}, in \textbf{Section~\ref{sec:9.2}}, we prove Corollary~\ref{cor:macro}, and finally in \textbf{Section~\ref{sec:9.3}}, we prove Theorem~\ref{thm:Maxwellian}.

In the rest of this section, we will work under the assumptions of Theorem~\ref{thm:main} and use the estimates established in the proof of Theorem~\ref{thm:main}.

\subsection{Existence of a large time limit (Proof of Theorem~\ref{thm:asymptotics})}\label{sec:9.1}

\begin{lemma}\label{lem:main.conv}
Assume the conditions of Theorem~\ref{thm:main} hold and suppose $f$ is a solution given by Theorem~\ref{thm:main}. Define $f^\sharp$ as in \eqref{fsharp.def}.

Given $0\leq T_1<T_2$ and $\ell\in \mathbb N\cup\{0\}$, the following estimate holds for some implicit constant depending only on $d_0$, $\gamma$ and $\ell$ (and is independent of $T_1$ and $T_2$):
$$\|\vb^\ell \xb^{\Mm+4} |f^\sharp(T_1,x,v)-f^{\sharp}(T_2,x,v)|\|_{L^\i_xL^\i_v} \ls \ep^{\f 32} (1+T_1)^{-\min\{1,2+\gamma\}}.$$
\end{lemma}
\begin{proof}
Using the definition of $f^\sharp$ in \eqref{fsharp.def} and the Landau equation \eqref{Landau.3}, we obtain
$$\rd_t f^{\sharp}(t,x,v) = (\rd_t f+v_i\rd_{x_i} f)(t,x+tv,v) = (\bar{a}_{ij} \rd^2_{v_iv_j} f - \bar{c} f)(t,x+tv,v).$$
This implies 
$$(\rd_t (\vb^\ell \xb^{\Mm+4} f^{\sharp}))(t,x,v)= \vb^{\ell} \xb^{\Mm+4}(\bar{a}_{ij} \rd^2_{v_iv_j} f - \bar{c} f)(t,x+tv,v).$$
Integrating in $t$ from $t=T_1$ to $t=T_2$, we thus obtain
\begin{align}
&\: \|\vb^\ell \xb^{\Mm+4} |f^{\sharp}(T_1,x,v)- f^{\sharp}(T_2,x,v)|\|_{L^\i_xL^\i_v} \notag\\
\ls & \: \int_{T_1}^{T_2} \| \vb^{\ell} \xb^{\Mm+4} |\bar{a}_{ij} \rd^2_{v_iv_j} f|(t,x+tv,v) \|_{L^\i_xL^\i_v} \, \ud t \label{asymp.error.1}\\ 
&\: + \int_{T_1}^{T_2} \| \vb^{\ell} \xb^{\Mm+4} |\bar{c} f| (t,x+tv,v) \|_{L^\i_xL^\i_v} \, \ud t . \label{asymp.error.2}
\end{align}

To bound the terms \eqref{asymp.error.1} and \eqref{asymp.error.2}, we use the fact that $\sup_{x\in \mathbb R^3} \sup_{v\in \mathbb R^3} = \sup_{x-tv\in \mathbb R^3}\sup_{v\in \mathbb R^3}$. To control \eqref{asymp.error.1}, we first use Lemma~\ref{lem:Li.4} to obtain
\begin{equation}\label{asymp.error.est.1}
\begin{split}
&\: |\eqref{asymp.error.1}| \\
\ls &\:  (\max_{i,j}\|(1+t)^{3+\min\{1,2+\gamma\}}\vb^{-(2+\gamma)} \wb^{-\min\{1,2+\gamma\}} \bar{a}_{ij}(t,x,v)\|_{L^\i([T_1,T_2];L^\i_x L^\i_v)}) \\
&\: \times (\sum_{|\bt|=2} \|(1+t)^{-2} \vb^{\ell+1} \wb^{\Mm+5} \rd^\bt_{v} f(t,x,v)\|_{L^\i([T_1,T_2];L^\i_xL^\i_v)}) \|(1+t)^{-1-\min\{1,2+\gamma\}}\|_{L^1([T_1,T_2])}\\
\ls &\: (\max_{i,j}\|(1+t)^{3+\min\{1,2+\gamma\}}\vb^{-(2+\gamma)} \wb^{-\min\{1,2+\gamma\}} \bar{a}_{ij}(t,x,v)\|_{L^\i([T_1,T_2];L^\i_x L^\i_v)}) \\
&\: \times (\sum_{|\bt|= 2} \|(1+t)^{-2}\wb^{\Mm+5} \rd^\bt_v g(t,x,v)\|_{L^2([T_1,T_2];L^2_xL^2_v)})(1+T_1)^{-\min\{1,2+\gamma\}}.
\end{split}
\end{equation}
Now by Propositions~\ref{prop:ab.Li.1} and \ref{prop:ab.Li.null.cond} (used for $t\leq 1$ and $t>1$ respectively), the first factor is bounded above by $\ep^{\f 34}$. By \eqref{BA.Z}, the second factor is bounded by $\ep^{\f 34}$. Combining, we see that 
$$|\eqref{asymp.error.1}|\ls |\mbox{RHS of \eqref{asymp.error.est.1}}|\ls \ep^{\f 32} (1+T_1)^{-\min\{1,2+\gamma\}}.$$

We next bound \eqref{asymp.error.2}. Using Lemma~\ref{lem:Li.4}, Proposition~\ref{prop:cb.Li} and \eqref{BA.Z}, we obtain
\begin{equation*}
\begin{split}
|\eqref{asymp.error.2}| \ls &\: \|(1+t)^{-(3+\gamma)}\|_{L^1([T_1,T_2])} \|(1+t)^{3+\gamma}\bar{c}(t,x,v)\|_{L^\i([T_1,T_2];L^\i_x L^\i_v)} \\
&\:\qquad \times \|\vb^{\ell} \wb^{\Mm+4} f(t,x,v)\|_{L^\i([T_1,T_2];L^\i_xL^\i_v)} \ls \ep^{\f 32} (1+T_1)^{-(2+\gamma)}.
\end{split}
\end{equation*}

Finally, using the estimates for \eqref{asymp.error.1} and \eqref{asymp.error.2} above, we obtain the desired estimate. \qedhere
\end{proof}

We are now in a position to prove Theorem~\ref{thm:asymptotics}.
\begin{proof}[Proof of Theorem~\ref{thm:asymptotics}]
By Lemma~\ref{lem:main.conv}, for any $\ell \in \mathbb N$ and for any sequence $t_n\to +\infty$, $\{f^\sharp(t_n,x,v)\}_{n=1}^\infty$ is Cauchy in the Banach space with the norm $\|\vb^\ell \xb^{\Mm+4} (\cdot)\|_{L^\i_x L^\i_v}$. Therefore there exists a unique $f^{\sharp}_\infty:\mathbb R^3\times \mathbb R^3 \to \mathbb R$ such that for any $\ell \in \mathbb N \cup\{0\}$,
$$\lim_{t\to +\infty} \|\vb^\ell \xb^{\Mm+4}|f^\sharp(t,x,v)-f^{\sharp}_\infty(x,v)|\|_{L^\i_xL^\i_v} = 0.$$
Using the estimate in Lemma~\ref{lem:main.conv} again, it then follows that
$$\sup_{t\geq 0} (1+t)^{\min\{1,2+\gamma\}}\|\vb^\ell \xb^{\Mm+4} (f^\sharp(t,x,v) - f^\sharp_\infty(x,v))\|_{L^\i_x L^\i_v }(t) \ls \ep^{\f 32},$$
which is what we wanted to prove. \qedhere
\end{proof}

\subsection{Large time asymptotics for macroscopic quantities (Proof of Corollary~\ref{cor:macro})}\label{sec:9.2}

We will prove slightly more general estimates than Corollary~\ref{cor:macro}. The following proposition gives the main estimates.
\begin{proposition}\label{prop:marco.more.general}
Assume the conditions of Theorem~\ref{thm:main} hold and suppose $f$, $f^\sharp_\infty$ are as given by Theorems~\ref{thm:main} and \ref{thm:asymptotics} respectively.

For any $\ell \in \mathbb N\cup \{0\}$, the following estimate holds for all $t\in [0,+\infty)$ with an implicit constant depending on $d_0$, $\gamma$ and $\ell$:
\begin{equation}\label{eq:macro.more.general.1}
\|\vb^\ell |f(t,x,v) - f^\sharp_\infty(x-tv,v)| \|_{L^\i_xL^1_v}\ls \ep^{\f 32}(1+t)^{-3-\min\{1,2+\gamma\}}
\end{equation}
and
\begin{equation}\label{eq:macro.more.general.2}
\|\vb^\ell f(t,x,v) \|_{L^\i_xL^1_v}\ls \ep(1+t)^{-3}.
\end{equation}
\end{proposition}
\begin{proof}
Using Lemma~\ref{lem:Li.2}, then noting $\sup_{x\in \mathbb R^3} \sup_{v\in \mathbb R^3} = \sup_{x-tv\in \mathbb R^3}\sup_{v\in \mathbb R^3}$, and finally using Theorem~\ref{thm:asymptotics}, we obtain
\begin{equation*}
\begin{split}
&\: \|\vb^\ell |f(t,x,v) - f^\sharp_\infty(x-tv,v)| \|_{L^\i_xL^1_v} \\
\ls &\: (1+t)^{-3} (\|\vb^{\ell+4} |f(t,x,v) - f^\sharp_\infty(x-tv,v)| \|_{L^\i_xL^\i_v} + \|\vb^\ell \wb^4 |f(t,x,v) - f^\sharp_\infty(x-tv,v)| \|_{L^\i_xL^\i_v}) \\
= &\: (1+t)^{-3} (\|\vb^{\ell+4} |f^\sharp(t,x,v) - f^\sharp_\infty(x,v)| \|_{L^\i_xL^\i_v} + \|\vb^\ell \xb^4 |f^\sharp(t,x,v) - f^\sharp_\infty(x,v)| \|_{L^\i_xL^\i_v}) \\
\ls &\: \ep^{\f 98} (1+t)^{-3-\min\{1,2+\gamma\}}.
\end{split}
\end{equation*}
This proves \eqref{eq:macro.more.general.1}.

Now using a similar argument as above, we obtain
\begin{equation}\label{eq:main.step.in macro.more.general.2.0}
\begin{split}
&\: \| \vb^{\ell} f^\sharp_\infty(x-tv,v) \|_{L^\i_xL^1_v} \\
\ls &\: (1+t)^{-3} ( \|\vb^{\ell+4}  f^\sharp_\infty(x-tv,v) \|_{L^\i_xL^\i_v} + \|\vb^{\ell} \wb^4 f^\sharp_\infty(x-tv,v) \|_{L^\i_xL^\i_v})\\
\ls &\: (1+t)^{-3} ( \|\vb^{\ell+4}  f^\sharp_\infty(x,v) \|_{L^\i_xL^\i_v} + \|\vb^{\ell} \xb^4 f^\sharp_\infty(x,v) \|_{L^\i_xL^\i_v}).
\end{split}
\end{equation}
To proceed, note that applying the estimate in Theorem~\ref{thm:asymptotics} with $t=0$, we obtain, for every $\ell'\in \mathbb N\cup \{0\}$,
$$\|\vb^{\ell'} \xb^{\Mm+4} (f_{\mathrm{in}}(x,v) - f^\sharp_\infty(x,v))\|_{L^\i_x L^\i_v}\ls \ep^{\f 32}.$$
Combining this estimate with the assumptions on $f_{\mathrm{in}}$ in Theorem~\ref{thm:main} and using the triangle inequality, we in particular have
$$\|\vb^{\ell'} \xb^{4} f^\sharp_\infty(x,v)\|_{L^\i_x L^\i_v}\ls \ep.$$
Plugging this into \eqref{eq:main.step.in macro.more.general.2.0}, we then obtain
\begin{equation}\label{eq:main.step.in macro.more.general.2}
\begin{split}
\| \vb^{\ell} f^\sharp_\infty(x-tv,v) \|_{L^\i_xL^1_v}  \ls \ep(1+t)^{-3}.
\end{split}
\end{equation}
Combining \eqref{eq:main.step.in macro.more.general.2} with \eqref{eq:macro.more.general.1} and using the triangle inequality yields \eqref{eq:macro.more.general.2}. \qedhere
\end{proof}

Using Proposition~\ref{prop:marco.more.general}, we can immediately prove Corollary~\ref{cor:macro}:
\begin{proof}[Proof of Corollary~\ref{cor:macro}]
Note that \eqref{cor:macro.1} is an immediate corollary of \eqref{eq:macro.more.general.2}; while \eqref{cor:macro.2} is an immediate corollary of \eqref{eq:macro.more.general.1}. \qedhere
\end{proof}

\subsection{The large time limit is in general not a traveling global Maxwellian (Proof of Theorem~\ref{thm:Maxwellian})}\label{sec:9.3}

In this final subsection, we prove Theorem~\ref{thm:Maxwellian}. The reader may find it useful to recall Definition~\ref{def:GM}.

\begin{proof}[Proof of Theorem~\ref{thm:Maxwellian}]
By Lemma~\ref{lem:main.conv} (with $\ell = 2$), given initial data $f_{\mathrm{in}}$ with 
\begin{equation}\label{initial.assumption}
\begin{split}
&\: \sum_{|\alp|+|\bt|+|\sigma|\leq \Mm}  \|\xb^{\Mm+5}\rd_x^{\alp}\rd_v^{\bt} (e^{2d_0(1+|v|^2)} f_{\mathrm{in}}) \|_{L^2_xL^2_v} \\
&\: \qquad + \sum_{|\alp|+|\bt|+|\sigma|\leq \Mm-5} \|\xb^{\Mm+5}\vb \rd_x^{\alp}\rd_v^{\bt} (e^{2d_0(1+|v|^2)} f_{\mathrm{in}}) \|_{L^\i_xL^\i_v} \leq \ep 
\end{split}
\end{equation}
(where $\ep\in (0,\underline{\ep}_0]$ with $\underline{\ep}_0$, $\Mm$, $d_0$ as in Theorem~\ref{thm:main}), there exists $C_{\infty}>0$ (depending only on $d_0$ and $\gamma$) such that the unique solution arising from $f_{\mathrm{in}}$ satisfies 
$$\|\vb^2\xb^2|f^\sharp(T_1,x,v)-f^{\sharp}(T_2,x,v)|\|_{L^\i_xL^\i_v} \leq C_{\infty}\ep^{\f 32}(1+T_1)^{-\min\{1,2+\gamma\}}$$
for all $0\leq T_1<T_2$. In particular, taking $T_1=0$ and $T_2\to +\infty$, and using the definition of $f_\infty^\sharp$, we obtain,
$$\|\vb^2\xb^2|f_{\mathrm{in}}(x,v)-f_{\infty}^{\sharp}(x,v)|\|_{L^2_xL^2_v} \leq C_{\infty}\ep^{\f 32}.$$

In view of the above inequality, in order to prove the present proposition, it suffices to exhibit a function $f_{\mathrm{in}}$ such that for some $\ep\in [0,\ep_0]$, the following two conditions are simultaneously satisfied:
\begin{enumerate}
\item \eqref{initial.assumption} holds, and
\item $\inf_{\mathcal M\in {\bf \mathfrak M}} \|\vb^2\xb^2 |f_{\mathrm{in}}(x,v)- \mathcal M^\sharp(x,v)|\|_{L^2_xL^2_v} > C_{\infty}\ep^{\f 32}$.
\end{enumerate}

To show that such an $f_{\mathrm{in}}$ exists, take an arbitrary ``seed function'' $\underline{f}:\mathbb R^3\times \mathbb R^3\to \mathbb R_{>0}$ which 
\begin{itemize}
\item satisfies 
\begin{equation}\label{upper.bd.of.seed}
\begin{split}
&\: \sum_{|\alp|+|\bt|+|\sigma|\leq \Mm}  \|\xb^{\Mm+5}\rd_x^{\alp}\rd_v^{\bt} (e^{2d_0(1+|v|^2)} f_{\mathrm{in}}) \|_{L^2_xL^2_v} \\
&\: \qquad + \sum_{|\alp|+|\bt|+|\sigma|\leq \Mm-5} \|\xb^{\Mm+5}\vb \rd_x^{\alp}\rd_v^{\bt} (e^{2d_0(1+|v|^2)} f_{\mathrm{in}}) \|_{L^\i_xL^\i_v} \leq \underline{C}<+\infty 
\end{split}
\end{equation}
for some $\underline{C}>0$, and
\item is not $\mathcal M^\sharp$ for any global Maxwellian.
\end{itemize}
Note that there must exist a constant $\underline{c}>0$ such that
\begin{equation}\label{away.from.Max}
\inf_{\mathcal M\in\mathfrak M} \|\vb^2 \xb^2 (\underline{f}(x,v)- \mathcal M^\sharp(x,v)) \|_{L^2_xL^2_v} \geq \underline{c} >0.
\end{equation}
(This is an easy consequence of the fact that global Maxwellians are parametrized by a finite dimensional space of parameters. More precisely, if \eqref{away.from.Max} were not true, then there exists a sequence of global Maxwellians $\mathcal M_n$ parametrized by $(\alp_n,\bt_n,\sigma_n,m_n,B_n)$ such that $\lim_{n\to+\infty} \|\xb^2 \vb^2(\underline{f}(x,v)- \mathcal M_n^\sharp(x,v))\|_{L^2_xL^2_v} = 0$.  This convergence in particular implies that all the second moments of $\underline{f}$ are bounded. Hence $(\alp_n,\bt_n,\sigma_n,m_n,B_n)$ stays in a compact set of $\mathbb R\times \mathbb R\times\mathbb R\times\mathbb R\times\mathbb R^{3\times 3}$. Therefore there exists a convergent subsequence which converges, i.e.~$(\alp_n,\bt_n,\sigma_n,m_n,B_n)\to (\alp,\bt,\sigma,m,B)$. This then implies $f_{\mathrm{in}}(x,v) = \mathcal M^\sharp$ for some $\mathcal M\in \mathfrak M$, contradicting our assumptions.)

For $\eta>0$ to be chosen below, we now let
$$f_{\mathrm{in}}:=\eta \underline{f}.$$
Given \eqref{upper.bd.of.seed} and \eqref{away.from.Max}, and noting that the family of global Maxwellians $\mathfrak M$ is invariant under rescaling (i.e.~$\mathcal M\in \mathfrak M\iff \lambda \mathcal M\in \mathfrak M, \forall \lambda>0$), the conditions (1) and (2) above for $f_{\mathrm{in}}$ therefore translates to the two conditions
$$\underline{C}\eta \leq \ep,\quad \underline{c}\eta>C_{\infty} \ep^{\f 32}.$$
for some $\ep\in [0,\ep_0]$. It is then easy to see that this can be satisfied if we take $\ep = \underline{C}\eta$ and $\eta< \min\{\f{\underline{c}^2}{C_\infty^2\underline{C}^3}, \f{\ep_0}{\underline{C}}\}$. \qedhere
\end{proof}

\begin{remark}\label{rmk:generally.not.Max}
We have in fact proven slightly more. Given any function $\underline{f}$ which does not correspond to $\mathcal M^\sharp$ for any $\mathcal M\in \mathfrak M$, there exists $\eta_0>0$ depending on $\underline{f}$ such that if $\eta \in (0,\eta_0)$, then the solution arising from $f_{\mathrm{in}} = \eta \underline{f}$ does not converge to a zero solution or a traveling global Maxwellian.
\end{remark}

\bibliographystyle{plain}
\bibliography{Landau}

\end{document}